\documentclass[12pt]{article}
\usepackage{amsmath,amssymb,amsfonts,amsthm}
\usepackage{makeidx}

\makeindex
\begin{document}

 \newcommand {\ID}{\mathbb{D}}  
 \newcommand {\IN}{\mathbb{N}}  
 \newcommand {\IR}{\mathbb{R}}   
 \newcommand {\IQ}{\mathbb{Q}}   
 \newcommand {\IRR}{\mathbb{R}} 
 \newcommand {\IIQ}{\mathbb{Q}} 
 \newcommand {\IC}{\mathbb{C}}
 \newcommand {\ZZ}{\mathbb{Z}}

\newcommand {\rtoo}{\IR_+^2}
\newcommand {\TT}{T \!\!\!T}
\newcommand {\B}{\mbox{${\cal B}$}}
\newcommand {\cC}{\mbox{${\cal C}$}}
\newcommand {\D}{\mbox{${\cal D}$}}

\newcommand {\E}{\mbox{${\cal E}$}}
\newcommand {\F}{\mbox{${\cal F}$}}
\newcommand {\G}{\mbox{${\cal G}$}}
\newcommand {\cH}{\mbox{${\cal H}$}}
\newcommand {\IH}{\mbox{${\mathcal H}$}}
\newcommand {\cI}{\mbox{${\cal I}$}}
\newcommand {\cL}{\mbox{${\cal L}$}}
\newcommand {\N}{\mbox{${\cal N}$}}
\newcommand {\cO}{\mbox{${\mathcal O}$}}
\newcommand {\cP}{\mbox{${\cal P}$}}
\renewcommand {\P}{{\mathbb P}} 
\newcommand {\IP}{\mbox{${\mathcal P}$}}
\newcommand{\cR}{\mathcal{R}} 
\newcommand {\R}{\mbox{${\mathcal R}$}}
\newcommand {\cS}{\mbox{${\cal S}$}}
\newcommand {\w}{\mbox{${\omega}$}}
\newcommand {\ep}{\varepsilon}
\newcommand{\half}{\frac{1}{2}}
\newcommand{\ti}[1]{\tilde{#1}}
\newcommand{\ul}{\underline}

\newcommand{\ulsigma}{\underline{\sigma}}
\newcommand{\ultau}{\underline{\tau}}
\newcommand{\ulT}{\underline{T}}
\newcommand{\ulS}{\underline{S}}
\newcommand{\ulu}{\underline{u}}
\newcommand{\ulinfty}{\underline{\mbox{$\infty$}}}
\newcommand{\qbar}{\bar{q}}
\newcommand{\tbar}{\bar{t}}
\newcommand{\wbar}{\bar{w}}
\newcommand{\vbar}{\bar{v}}
\newcommand{\xbar}{\bar{x}}
\newcommand{\Xbar}{\bar{X}}
\newcommand{\ybar}{\bar{y}}
\newcommand{\Ybar}{\bar{Y}}

\newtheorem{stat}{Statement}[section]
\newtheorem{examp}[stat]{Example}
\newtheorem{assump}{Assumption}[section]
\newtheorem{decth}[stat]{Theorem}
\newtheorem{prop}[stat]{Proposition}
\newtheorem{cor}[stat]{Corollary}
\newtheorem{thm}[stat]{Theorem}
\newtheorem{lemma}[stat]{Lemma}
\newtheorem{remark}[stat]{Remark}
\newtheorem{def1}{Definition}[section]

\numberwithin{equation}{section}

\begin{center}
{\bf \large Hausdorff dimension of the boundary of bubbles
\vskip 12pt
of additive Brownian motion and of the Brownian sheet}
\vskip .5in
{\large Robert C. Dalang\footnote[1]{Institut
de Math\'ematiques, Ecole Polytechnique F\'ed\'erale,
1015 Lausanne, Switzerland. 

\noindent robert.dalang@epfl.ch, thomas.mountford@epfl.ch
\vskip 10pt

\noindent The research of each author is
partially supported by
the Swiss National Foundation for Scientific Research.
\vskip 10pt

\noindent MSC 2010 subject classifications. Primary 60G60;
Secondary 60G17, 60G15.
\vskip 10pt
\noindent Key words and phrases: Brownian sheet, Brownian bubble, excursions, level sets.
} and
T.~Mountford$^1$
}
\vskip 12pt
Ecole Polytechnique F\'ed\'erale de Lausanne
\end{center}
\vskip .5 truein

\begin{center}{\em
   Dedicated to John B.~Walsh, whose work initiated the authors to the Brownian sheet
}
\end{center}
\vskip .5 truein

\begin{abstract} We first consider the additive Brownian motion process $(X(s_1,s_2),\ (s_1,s_2) \in \IR^2)$ defined by $X(s_1,s_2) = Z_1(s_1) - Z_2 (s_2)$, where $Z_1$ and $Z_2 $ are two independent (two-sided) Brownian motions. We show that with probability one, the Hausdorff dimension of the boundary of any connected component of the random set $\{(s_1,s_2)\in \IR^2: X(s_1,s_2) >0\}$ is equal to
$$
  \frac{1}{4}\left(1 + \sqrt{13 + 4 \sqrt{5}}\right) \simeq 1.421\, . 
$$
Then the same result is shown to hold when $X$ is replaced by a standard Brownian sheet indexed by the nonnegative quadrant.
\end{abstract}
\vskip 1 truein
\eject
\vskip 1 truein

\tableofcontents

\eject

\begin{section}{Introduction}

  In this paper, we consider two closely related stochastic processes: a standard additive Brownian motion (ABM) $\ti X=(\ti X(s_1,s_2),\ (s_1,s_2) \in \IR^2)$, defined by
$$
  \ti X(s_1,s_2) = \ti{Z}_1(s_1) \ - \ \ti{Z}_2(s_2), \qquad (s_1 , s_2 ) \in \IR^2,
$$
where the $\ti{Z}_i $ are standard independent (two-sided) Brownian motions, and a standard Brownian sheet\index{$W$} indexed by the nonnegative quadrant:
$$
   W=(W(s_1,s_2),\ (s_1,  s_2) \in [0, \infty [^2). 
$$
We recall that this is a mean-zero Gaussian process with continuous sample paths and covariance 
$$
   E(W(s_1,s_2) W(t_1,t_2)) = \min(s_1,t_1) \min(s_2,t_2).
$$
The additive Brownian motion is a process of interest in its own right (see for instance \cite{gabor} and the references therein), but it also demands attention due to the fact that, locally, its behaviour is very close to that of the Brownian sheet. In fact,  arguments dealing with the Brownian sheet (e.g. \cite {DM2,KhShi,M}) often carry over immediately to give analogous results for additive Brownian motion. Indeed, typically, arguments for the sheet, while conceptually the same as for additive Brownian motion, have an extra layer of technicalities (in the case of the results of the present paper, the extra technicalities will be extensive). However, the Brownian sheet does exhibit behaviors that are different from those of ABM (for instance the existence of ``points of increase along lines" established in \cite{DM1}, or the results of \cite{M2} concerning quasi-everywhere upper functions), so one cannot simply expect that results established for ABM will necessarily carry over to the Brownian sheet. 

   In this paper, we are interested in the connected components of the random open set $\{ s\in\IR^2 : \ti X(s) \ne q \}$ (respectively $\{s\in\rtoo : W(s) \ne q \}$) for some fixed level $q\in \IR$.  Such a component is called a $q$-{\em bubble}, or simply a {\em bubble} if $q$ is fixed. By analogy with ordinary Brownian motion, these are excursion sets of $\ti X$ (resp.~$W$) away from the level $q$. An {\em upwards} (respectively {\em downwards}) $q$-bubble is defined with ``$\neq$" replaced by ``$>$" (respectively ``$<$").
	
	The level set of $W$ at level $q$ is the random closed set $\{ s\in\rtoo : W(s) = q \}$. It is known since \cite{adler,E,R} that the Hausdorff dimension of this set is $\frac{3}{2}$ a.s. We refer the reader to \cite[Appendix C]{davar} or \cite{landkof}, for instance, for all required information about Hausdorff dimension. In \cite{kendall}, it was observed that typical points on a level set of the Brownian sheet are disconnected from the rest of the level set, even though the level set has non-degenerate connected components. Beginning in the early 1990's, substantial efforts were made to understand the structure of bubbles and level sets.

   In \cite{DW0}, the distribution and size of bubbles in the neighborhood of certain points on the boundary of a bubble were analyzed. The paper \cite{DW} describes bubbles of additive Brownian motion and gives a formula for the expected area of a bubble given its height. In \cite{DM96}, the authors showed that a Jordan curve contained in a level set of the Brownian sheet must be nowhere differentiable, which indicates that connected components of the level set must be highly irregular. Whether or not the level sets of the Brownian sheet actually do contain a Jordan curve remains an open problem. However, for additive Brownian motion, this question was resolved affirmatively in \cite{DM01}. The Hausdorff dimension of this Jordan curve in a level set of an ABM has not yet been determined but is conjectured in \cite{gabor} to be $(\sqrt{17} +1)/4$.
	
	In \cite{M}, T.~Mountford showed that there exist points in $\rtoo$ which are on the boundary of both a positive and a negative bubble of the Brownian sheet, a situation which does not arise for standard Brownian motion. He also showed that the Hausdorff dimension of the boundary of any bubble is at least $1.25$ and is strictly smaller than $\frac{3}{2}$.
	
	The first result of \cite{M} was improved in \cite{DM97}, where it was shown that there exist monotone curves along which the Brownian sheet has a point of increase at a given level $q$. Several refinements of this were given in \cite{DM1}. Finally, the authors showed in \cite{DM2} that, for ABM and for the Brownian sheet, given the level set at level $q$, distinct excursions away from $q$ are {\em not} independent. An overview of these results can be found in \cite{Dalang} and the references therein. Other properties of the Brownian sheet can be found in \cite[Chapter 12]{davar}.
 
   In recent years, there has been much interest in level lines (contours) of the Gaussian random field known as the two dimensional Gaussian Free Field GFF (see \cite{SS2009} for the discrete version of this random field). In particular, \cite{SS2013} shows that the level lines of a GFF correspond to a chordal Schramm-Loewner evolution SLE$_4$ (see also \cite{dubedat}). The Hausdorff dimension of such a curve is known to be $\frac{3}{2}$ (see \cite{beffara}), which happens to be the same dimension as that of the level sets of the Brownian sheet. 
However, the issues that we discuss here for ABM and the Brownian sheet do not seem to have been yet discussed for the GFF. There has also been much interest in level sets of smooth Gaussian random fields: see \cite{AT,AW}, for instance.

   The main objective of this paper is to improve the second result of \cite{M}, namely, to determine, for ABM and for the Brownian sheet, the exact value of the Hausdorff dimension of the boundary of a bubble. 

   Central to our methodology is an algorithm which was introduced in \cite{DW} under the name ``Algorithm A" and which applies to additive Brownian motion $\ti X$. This algorithm, which we term the ``DW-algorithm", constructs, assuming that $\ti X(0,0) >0$, a path in $\IR^2$ along which $\ti X$ is positive, with one extremity at $(0,0)$ and the other at the highest point on the excursion over the bubble that contains the origin. This algorithm was used in \cite{DW} to determine the expected area of a bubble of ABM given the height of the excursion over this bubble.
   
   Here, we analyze this algorithm carefully in order to compute {\em exact} gambler's ruin probabilities for ABM: given $0 < x_0 < 1$, we calculate in Theorem \ref{prop1} an exact and explicit formula for the probability, given that $\ti X(0,0) = x_0$, that ``there exists a path in $\IR^2$ starting at the origin along which $\ti X$ hits level 1 before level 0," {\em with no a priori restriction on the path.} In particular, for $x$ near $0$, it turns out that this probability is of order $x_0^{\lambda_1}$, where
$$
   \lambda_ 1 = \half\left(5 - \sqrt{13 + 4\sqrt{5}}\right) \simeq 0.157764\, .
$$
This implies the following result.

\begin{thm}
There exists a constant $c_1$ such that for all $x \geq 1$ and for a standard ABM $\ti X$ such that $\ti X(0,0) = 1$, the probability that ``there exists a path starting at the origin with $\ti X \geq x$ at the other extremity of the path" is equal to $c_1 x^{-\lambda _ 1 } +o( x ^ {- \lambda _1 })$.
\label{thm1}
\end{thm}

   In Section \ref{sec3}, we build on the result of Theorem \ref{prop1} to determine {\em escape probabilities} of a standard ABM such that $\ti X(0,0) = x_0$, for $0 < x_0 < 1$: we obtain sharp estimates on the probability that the bubble containing the origin extends at least $a$ units away from the origin. We show in particular in Theorem \ref{thm3} that if $x_0 / \sqrt{a} \leq 1$, then this probability is of order $(x_0/\sqrt{a})^{\lambda_1}$.
   
   Sections \ref{sec4} to \ref{sec7} study the Hausdorff dimension of the boundary of $q$-bubbles of ABM. In Section \ref{sec4}, we show in Proposition \ref{thm4} that $(3 - \lambda_1)/2$ is an upper bound for this Hausdorff dimension. The proof of this result uses a covering argument and is a fairly straightforward consequence of the results of Section \ref{sec3}, and, in particular, of Proposition \ref{univariateup}.

   The objective of Sections \ref{sec5}--\ref{sec7} is to establish the corresponding lower bound. This is done via a so-called ``second-moment argument." This requires two important ingredients. The first is an upper bound on a two-point escape probability, that is, the probability that the bubbles containing two distinct points $s$ and $t$ both have a diameter of order $1$. Of course, if $s$ and $t$ are far apart, then the two bubbles are essentially independent, and so the main objective is to understand how this probability behaves for $s$ near $t$. The required upper bound is obtained in Proposition \ref{lembivariate}. The second ingredient is a lower bound on escape probabilities with the additional constraint that as one moves away from the origin along the path constructed by the DW-algorithm, the value of the ABM grows quickly enough. This is done in Proposition \ref{rdprop21}. Proposition \ref{lowerlem1} contains the key ingredients for the second-moment argument, which is implemented in Section \ref{sec7}, culminating in the following result (in which the lower bound comes from Theorem \ref{thm7.3} and the upper bound from Proposition \ref{thm4}). This result was announced in \cite{Dalang}.
   
\begin{thm} Fix $q \in \IR$. For standard additive Brownian motion, with probability one, the Hausdorff dimension of the boundary of every $q$-bubble is equal to 
 $$
   \frac{3 - \lambda_1}{2} =  \frac{1}{4}\left(1 + \sqrt{13 + 4 \sqrt{5}}\right) \simeq 1.421.
$$
\label{thm2}
\end{thm}

   Sections \ref{sec8} to \ref{sec11} deal with the analogue of Theorem \ref{thm2} for the Brownian sheet. The underlying idea is that in the neighborhood of a point, the Brownian sheet is the sum of a standard ABM and an error term (see \cite{DW0}), and appropriate arguments and estimates are needed to control the contribution of this error term. A key difficulty is that adding a small quantity to a process can substantially change the size of some bubbles.
   
   Section \ref{sec8} establishes in Proposition \ref{bs_ubprop1} that $(3-\lambda_1)/2$ is an upper bound for the Hausdorff dimension of bubbles of the Brownian sheet. This uses a first of four variants on the DW-algorithm, which we call the $\delta$-DW-algorithm, where $\delta >0$ is a parameter. Indeed, when the DW-algorithm for a standard ABM associated with the Brownian sheet terminates, it has constructed a rectangle on which the ABM is negative. However, the ABM may be ``barely negative", and so the bubble for the Brownian sheet may extend outside of this rectangle, due to the contribution of the error term. In order to show that this is unlikely, the $\delta$-DW-algorithm continuous on, provided the ABM becomes again sufficiently positive rather quickly, before going sufficiently negative (see Section \ref{sec8}). We show in Section \ref{sec8} that the gambler's ruin and escape probabilities are essentially the same for the DW-algorithm and for the $\delta$-DW-algorithm (Lemma \ref{bs_ub_lem4}), and that it is unlikely that the error term will have the effect that the bubble for the ABM and the bubble of the Brownian sheet are of significantly different size. The proof of Proposition \ref{bs_ubprop1} consists in making all these statements precise.
   
   The objective of Sections \ref{sec9}--\ref{sec11} is to establish that $(3-\lambda_1)/2$ is also a lower bound for the Hausdorff dimension of bubbles of the Brownian sheet. As for ABM, this is done via a second-moment argument. With the two-point escape probability in mind, and since the expression for the standard ABM that approximates the Brownian sheet is not so simple (see \eqref{08_08_14_1}), we prefer to approximate the Brownian sheet using an ABM that is {\em not standard,} that is, the two Brownian motions $\ti{Z}_1$ and $\ti{Z}_1$ are {\em not} independent (see \eqref{rd9.11a}). However, this (non standard) ABM can  be well-approximated by a standard ABM, as we show in Section \ref{sec10} by using a second variant on the DW-algorithm. This relies on a result established in  Section \ref{sec9}, which states that if two ABM's are close together and if one of them behaves in a ``typical way", then the other also behaves in this typical way. This statement, which is a sort of continuity property of the DW-algorithm on a subset of ``typical" paths, is made precise in Proposition \ref{rdprop37}, and we show in Proposition \ref{rdprop38} that this ``typical" behavior does indeed occur with high probability. Finally, Section \ref{sec10} also addresses the issue of the lower bound on escape probabilities for the Brownian sheet, by comparing the behavior of the Brownian sheet with that of a non-standard ABM, and the behavior of the latter with that of a standard ABM.
   
   In Section \ref{sec11}, we establish the necessary upper bound on a two-point escape probability. This requires approximating the Brownian sheet simultaneously in the neighborhood of two points $s$ and $t$ by two standard ABM's. The construction of the two standard ABM's builds on the ideas developed in Section \ref{sec10} and uses a third variant of the DW-algorithm. It is also necessary to obtain an upper bound on the probability that the bubble containing $s$ and the bubble containing $t$ both correspond to sufficiently high excursions. Because of the correlations in certain overlapping rectangular increments of the Brownian sheet, this requires a fourth and final variant on the DW-algorithm, that we call the ``boosted DW-algorithm." Gambler's ruin probabilities and escape probabilities for this boosted DW-algorithm are analyzed in Lemma \ref{08.lem5}. Proposition \ref{rdprop42} contains the needed upper bound on the two-point escape probability. In Section \ref{sec12}, we extend the results of Sections \ref{sec10} and \ref{sec11} to a family of processes that are obtained from certain scaled increments of the Brownian sheet and are themselves almost Brownian sheets, though not ``standard Brownian sheets." This leads to a proof of the following theorem (which was conjectured in \cite{DM0}).
 
\begin{thm} Fix $q \in \IR$. For the Brownian sheet, with probability one, the Hausdorff dimension of the boundary of every $q$-bubble is equal to $(3-\lambda_1 )/2$.
\label{thm3a1}
\end{thm}

   This theorem is the main result of this paper.  Its proof is given in the first part of Section \ref{sec12}. 
	
   
\end{section}
\eject

\begin{section}{Gambler's ruin probabilities for additive Brownian motion}\label{sec2}

   A {\em standard additive Brownian motion process} $\ti{X}= (\ti{X}(s_1,s_2),\ (s_1,s_2) \in \IR^2)$\index{$\ti{X}$} is given by
\begin{equation}\label{abm}
   \ti{X}(s_1,s_2) = \ti{Z}_1(s_1) - \ti{Z}_2 (s_2)
\end{equation} 
(note the minus sign), where $\ti{Z}_1$\index{$\ti{Z}_1$}, $\ti{Z}_2$\index{$\ti{Z}_2$} are two independent (two-sided) standard Brownian motions, defined on a probability space $(\Omega, \F,P)$, such that $\ti{Z}_1(0) = \ti{Z}_2(0) =0$. The processes $\ti{Z}_1$ and $\ti{Z}_2$ are the {\em components} of $\ti{X}$. An {\em additive Brownian motion}, without the qualification that it is {\em standard}, will refer to the case where the two Brownian motions may be correlated.

    When considering local behavior around 
the origin, it is sometimes convenient to regard a standard ABM as derived from four independent standard Brownian motions $B_i$\index{$B_1$}\index{$B_2$}\index{$B_3$}\index{$B_4$}, where
$$
   \ti{Z}_1(u) = \left\{\begin{array}{ll}
                  B_1 (u),&  \mbox{if } u \geq 0, \\
                  B_3 (-u),& \mbox{if } u \leq 0,
           \end{array}\right.
  \qquad
   \ti{Z}_2(v)  = \left\{\begin{array}{ll} 
                  B_2 (v),&  \mbox{if } v \geq 0, \\
                  B_4 (-v),& \mbox{if } v \leq 0.
           \end{array}\right.
$$

   For small $x_0 > 0$, we are interested in estimating the probability $\E(x_0)$ of the event ``there exists a path in $\IR^2$ starting at the origin along which $x_0 + \ti X$ hits level 1 before level 0," with no a priori restriction on the path. Formally, for a continuous path $\Gamma: \IR_+ \to
\IR^2$, set\index{$\tau^\Gamma$}
$$
   \tau^\Gamma = \inf\{u >0: x_0 + \ti X(\Gamma(u)) \in \{0,1\}\}.
$$
Then\index{$\E(x_0)$}
$$
   \E(x_0) = P\{\exists \Gamma: \Gamma(0) = (0,0) \mbox{ and }
   x_0 + \ti X(\Gamma(\tau^\Gamma)) = 1\}.
$$

   In the classical case where $\ti X$ is replaced by
(two-sided) Brownian motion, then there are essentially only two possible paths and this
probability is $1 - (1-x)^2 \sim 2x$ for $x$ near $0$. For the additive Brownian motion process
$\ti X$, there are uncountably many possible paths, so one expects that this escape
probability will be of order substantially larger than $x$, say of order
$x^\lambda$, for some $0 < \lambda <1$. The following result confirms this
intuition, and in addition, gives an {\em exact and explicit} formula for $\E(x)$.

\begin{thm} There are positive real numbers $\lambda_1 < \lambda_2 < \lambda_3 < \lambda_4$ and
constants $\alpha_i$, 
$i = 1,\dots,4,$ such that for $0 < x < 1$,
\begin{equation}\label{formE}
   \E(x) = \sum_{i=1}^4 \alpha_i\, x^{\lambda_i} .
\end{equation}
In fact, 
$$
   \{\lambda_1,\lambda_2,\lambda_3,\lambda_4\} = \left\{\half\left(5 \pm \sqrt{13 \pm
   4\sqrt{5}}\,\right)\right\},
$$
so $\lambda_1 = \half\left(5 - \sqrt{13 + 4\sqrt{5}}\right) \simeq 0.157764$\index{$\lambda_1$},
$\lambda_2 \simeq 1.49306$, $\lambda_3 \simeq 3.50694$, and $\lambda_4 \simeq 4.84224$.
Further, $\alpha_1 \simeq 0.938911$ and $\alpha_2$, $\alpha_3$ and $\alpha_4$ 
are given explicitly in (\ref{explalpha}) 
and numerically in \eqref{remnum}. In particular,
$$
   \lim_{x \downarrow 0} \frac{\E(x)}{\alpha_1 x^{\lambda_1}} = 1.
$$
\label{prop1}
\end{thm}

   The difficulty in this theorem is that {\em a priori}, there are many
possible paths to consider. However, this can be addressed using the
DW-algorithm, introduced as Algorithm A in \cite{DW}. This algorithm constructs a specific path $\Gamma^*$ with the remarkable property that either $x_0 + \ti X(\Gamma^*(\tau^{\Gamma^*})) = 1$, or there is no path with this property, that is, for all paths $\Gamma$,
either $\tau^\Gamma = \infty$ or $x_0 + \ti X(\Gamma(\tau^\Gamma)) = 0$. So after
recalling the DW-algorithm (stated in a form suitable for our purposes) 
and some of its properties, our proof of Theorem \ref{prop1} will analyze the probability that $x_0 + \ti X(\Gamma^*(\tau^{\Gamma^*})) = 1$. This algorithm, and four variants that will be described later, will play a fundamental role throughout this paper.

   For two points $(s_1,s_2)$ and $(t_1,t_2)$ in $\IR^2$, we denote by 
$\langle (s_1,s_2), (t_1,t_2)\rangle$ the straight line segment that connects
these two points.
\vskip 12pt

\noindent{\em The DW-algorithm ``started at $r$ with value $x_0$"} 
\vskip 12pt

  Fix $r = (r_1,r_2)\in \IR^2$ and $x_0 >0$. Let\index{$X^r$}
$$
   X^{r}(t) = \ti X(t) - \ti X(r).
$$
This is the increment process from $r$.

   Set $T^0_1 = r_1$, $T^0_2 = r_2$, $U_0 = U_0' = r_1$, $V_0 = V_0' = r_2$, $H_0 = x_0$ and
$\Gamma^{*,x_0,r}_0 = \{r\}$. The algorithm proceeds in stages, beginning with 
$n = 1$.
\vskip 12pt

   {\em Stage $2n-1$.} Let\index{$H_{2n-1}$}
\begin{eqnarray*}
   U_n &=& \sup\{u < T^{n-1}_1: x_0 + X^{r}(u,T^{n-1}_2) = 0\}, \\
   U_n' &=& \inf\{u > T^{n-1}_1: x_0 + X^{r}(u,T^{n-1}_2) = 0\}, \\
   H_{2n-1} &=& x_0 + \sup_{U_n < u < U_n'} X^{r}(u,T^{n-1}_2),
\end{eqnarray*}
and let $T^n_1$ be the unique time point in $[U_n,U_n']$\index{$U_n$}\index{$U_n'$} such that
$$
   x_0 + X^{r}(T^n_1, T^{n-1}_2) = H_{2n-1}.
$$
If $H_{2n-1} = H_{2n-2}$ (or, equivalently, $T^n_1 = T^{n-1}_1$), then STOP.
Otherwise, set
$$
  \Gamma^{*,x_0,r}_{2n-1} = \Gamma^{*,x_0,r}_{2n-2} \cup
   \langle(T^{n-1}_1,T^{n-1}_2),(T^n_1,T^{n-1}_2)\rangle,
$$
and
$$
   R^{x_0,r}_{2n-1} = [U_n,U_n']\times[V_{n-1},V_{n-1}'],
$$
and proceed to Stage $2n$.
\vskip 12pt

   {\em Stage $2n$.} Let \index{$H_{2n}$}
\begin{eqnarray*}
   V_n &=& \sup\{v < T^{n-1}_2: x_0 + X^{r}(T^{n}_1, v) = 0\}, \\
   V_n' &=& \inf\{v > T^{n-1}_2: x_0 + X^{r}(T^{n}_1, v) = 0\}, \\
   H_{2n} &=& x_0 + \sup_{V_n < v < V_n'} X^{r}(T^{n}_1, v),
\end{eqnarray*}
and let $T^n_2$ be the unique time point in $[V_n,V_n']$\index{$V_n$}\index{$V_n'$} such that
$$
   x_0 + X^{r}(T^n_1, T^{n}_2) = H_{2n}.
$$
If $H_{2n} = H_{2n-1}$ (or, equivalently, $T^n_2 = T^{n-1}_2$), then STOP.
Otherwise, set
$$
   \Gamma^{*,x_0,r}_{2n} = \Gamma^{*,x_0,r}_{2n-1} \cup
   \langle(T^{n}_1,T^{n-1}_2),(T^n_1,T^{n}_2)\rangle,
$$
and
$$
   R^{x_0,r}_{2n} = [U_n,U_n']\times[V_{n},V_{n}'],
$$
and proceed to Stage $2n+1$.
\vskip 12pt

   In the case where $r=(0,0)$, we omit the superscript $r$ from the notation
above, and we do the same with $x_0$ if its value is clear from the context. 
The reader can check that with $r = (0,0)$, this reproduces exactly 
Algorithm A of \cite{DW}. The DW-algorithm terminates after a finite (random) 
number $N^{x_0,r}$\index{$N^{x_0,r}$} of stages (see the proof of
Proposition 2.2 in \cite{DW}).  The rectangle $R^{x_0,r}_n$\index{$R^{x_0,r}_n$} represents the 
region of $\IR^2$ that the algorithm has explored up to stage $n$. Clearly, 
$R^{x_0,r}_n \subset R^{x_0,r}_{n+1}$. 

   The DW-algorithm constructs the path $\Gamma^*_N$,\index{$\Gamma^*_N$} which is a finite union of horizontal and vertical segments. At the points $(T^n_1, T^{n-1}_2)$ (resp.~$(T^n_1, T^{n}_2)$),\index{$(T^n_1, T^{n}_2)$} the process $x_0 + X^r$ restricted to this path achieve a new maximum value $H_{2n-1}$ (resp.~$H_{2n}$) and $\Gamma^*_N$ changes direction, so we call these points {\em corners} of $\Gamma^*_N$.

   The conclusion of the following proposition explains why this algorithm is 
relevant for the computation of $\E(x)$.
 
\begin{prop} Let $r = (0,0)$. Fix $x_0 >0$ and let $\Gamma^*$ be the path constructed by the DW-algorithm. Then:

  (a) $x_0 + X \leq 0$ on $\partial R^{x_0,r}_{N^{x_0,r}}$;
  
  (b) A.s., either $x_0 + X(\Gamma^*(\tau^{\Gamma^*})) = 1$, or  
for all paths $\Gamma$ with $\Gamma(0) = (0,0)$, either 
$\tau^\Gamma = \infty$ or $\tau^\Gamma < \infty$ and 
$x_0 + X(\Gamma(\tau^\Gamma)) = 0$;

   (c) If there is a path $\Gamma$ with $\Gamma(0) = (0,0)$ and $\Gamma(1) \not\in [-a,a]^2$ along which $x_0 + X$ is positive, then $R^{x_0,r}_{N^{x_0,r}} \not\subset [-a,a]^2$.
\label{prop2}
\end{prop}

\proof Let $E$ be the component of $\{(s_1,s_2)\} \in \IR^2: x_0 + X(s_1,s_2) >0\}$
that contains $(0,0)$. As mentioned above, the DW-algorithm stops a.s., after a finite number $N^{x_0,r}$ of stages. Suppose without loss
of generality that this occurs at an even stage $2n = N^{x_0,r}$. 

   (a) As explained in \cite[Section 2]{DW},
\begin{equation}\label{supX}
   X(T^n_1,T^n_2) = \sup_{s\in E} X(s),
\end{equation}
$x_0 + X$ is positive on the union of the two segments $]U_n,U_n'[\, \times \{T^n_2\}$
and $\{T^n_1\} \times\, ]V_n,V_n'[$, and
\begin{equation}\label{rectincr}
   X(U_n,T^n_2) = X(U'_n,T^n_2) = X(T^n_1,V_n) = X(T^n_1,V_n') = -x_0.
\end{equation}
In addition, $x_0 + X \leq 0$ on $\partial R_{2n}$. Indeed, for any rectangle $R =
[s_1,s_2]\times[t_1,t_2]$, set
\begin{equation}\label{defdelta}
   \Delta_R X = X(t_1,t_2) - X(t_1,s_2) - X(s_1,t_2) + X(s_1,s_2),
\end{equation}
and notice from (\ref{abm}) that $\Delta_R X = 0$ for any rectangle $R$. For
$U_n \leq s_1 < T^n_1$,
$$
   X(T^n_1,T^n_2) - X(T^n_1,V_n) - X(s_1,T^n_2) + X(s_1,V_n) = 0.
$$
By (\ref{supX}) and (\ref{rectincr}), we conclude that
$$
   X(s_1,V_n) = X(T^n_1,V_n) + X(s_1,T^n_2) - X(T^n_1,T^n_2) < -x_0.
$$
Therefore, $x_0 + X <0$ on the segment $[U_n,T^n_1[\, \times \{V_n\}$, and one checks
similarly that $x_0 + X\leq 0$ on the remainder of $\partial R_{2n}$, with equality only
at the four points that appear in (\ref{rectincr}). 

  (b) It follows in particular from (a) that $E \subset R_{2n}$, a.s. When $E \subset R_{2n}$,
two cases are possible: either $x_0 + X(T^n_1,T^n_2) \geq 1$, in which case $x_0 + X(\Gamma^*(\tau^{\Gamma^*})) = 1$, or $x_0 + X(T^n_1,T^n_2) < 1$, in which case
\begin{equation}\label{sup}
   x_0 + \sup_{s \in E} X(s) < 1.
\end{equation}
If $\Gamma$ is a path with $\Gamma(0) = (0,0)$, if $\Gamma$ does not exit $E$,
then $\tau^\Gamma = \infty$, while if $\Gamma$ does exit $E$, then $x_0 + X=0$ at any
point on $\Gamma \cap \partial E$, so by (\ref{sup}), along $\Gamma$, $x_0 + X$ hits 0
before 1 and $x_0 + X(\Gamma(\tau^\Gamma)) = 0$. This completes the proof of (b).

   (c) We have noted above that $x_0 + X \leq 0$ on $\partial R_{2n}$, where $2n = N^{x_0,r}$. If $R^{x_0,r}_{2n}$ were contained in $[-a,a]^2$, then there could not exist a path $\Gamma$ with the properties stated in (c). Therefore, (c) holds.
\hfill $\Box$
\vskip 16pt

\begin{remark} We notice from the proof of Proposition \ref{prop2} that when the DW-algorithm stops at an even stage $2n$, then the connected component of $\{s \in \IR^2: x_0 + X(s) > 0\}$ that contains $(T^n_1,T^n_2)$ contains the two segments $]U_n,U_n'[\, \times \{ T^n_2\}$ and $\{T^n_1\} \times\, ]V_n,V_n'[$, and is contained in the rectangle $[U_n,U_n']\times [V_n,V_n']$. Further, $(T^n_1,T^n_2)$ is the highest point of the excursion of $x_0 + X$ over this component.
\end{remark}

\noindent{\em Some further properties of the DW-algorithm}
\vskip 12pt

   We first introduce some terminology associated with the algorithm. For 
$r \in \IR^2$ fixed, we define a one-to-one correspondence $\cR_r: \IR_+^4 \to
\{$rectangles of$\ \IR^2\}$ by\index{$\cR_r(\ul{u})$}\index{$\ul{u}$}
$$
   \cR_r(\ul{u}) = [r_1 - u_3,r_1+u_1]\times[r_2 - u_4,r_2+u_2] \qquad
    \mbox{if } \ul{u} = (u_1,u_2,u_3,u_4).
$$
In this way, each $\ul{u}\in \IR_+^4$ represents a rectangle around 
$r \in \IR^2$. We set $\Vert \ul{u} \Vert = |u_1| + \cdots + |u_4|$. When $u_1 =
u_2 = u_3 = u_4 = a$, we write $\cR_r(a)$ instead of $\cR_r(a,a,a,a)$.\index{$\cR_r(a)$}

   Define a $4$-parameter filtration $(\F^r_{\ul{u}},\ \ul{u} \in \IR_+^4)$\index{$\F^r_{\ul{u}}$} by
$$
    \F^r_{\ul{u}} = \sigma\{X^r(s),\ s \in \cR_r(\ul{u})\}.
$$
It should be noted that $\F^r_{\ul{u}}$ is generated by increments:
typically, for $s \in \cR_r(\ul{u})$, $\ti X(s)$ is not measurable with respect to $\F^{r} _{\ul{u}}$, though
this {\em is} the case when $r = (0,0)$.

  For the general theory of multiparameter processes, and in particular, notions such as stopping points, we refer to \cite{meyer}. In particular, $\ul{u} \leq \ul{v}$ means $u_i \leq v_i$, $i=1,\dots,4$, while $\ul{u} \wedge \ul{v}= (u_1 \wedge v_1,\dots,u_4 \wedge v_4)$ denotes the coordinate-wise minimum of $\ul{u}$ and $\ul{v}$.

   In particular, if $\ul{T} = (\ul{T}_1,\ul{T}_2,\ul{T}_3,\ul{T}_4)$ is a 
stopping point relative to $(\F^r_{\ul{u}})$, then
\begin{equation}\label{startp6}
   (B^{r,\ulT}_1(u_1) = X^r(r_1+ \ul{T}_1 + u_1, r_2), \ u_1 \geq 0),
\end{equation}
$$
   (B^{r,\ulT}_2(u_2) = X^r(r_1,r_2 + \ul{T}_2 + u_2),\ u_2 \geq 0), 
$$
$$
   (B^{r,\ulT}_3(u_3) = X^r(r_1- \ul{T}_3 - u_3, r_2), \ u_3 \geq 0),
$$
$$
   (B^{r,\ulT}_4(u_4) = X^r(r_1,r_2 + \ul{T}_4 - u_4),\ u_4 \geq 0),
$$
are such that $(B^{r,\ulT}_i(u_i) - B^{r,\ulT}_i(0),\ u_i \geq 0)$,\index{$B^{r,\ulT}_i$} $i = 1,\dots,4$, are four independent standard Brownian motions, independent of $\F^r_{\ul{T}}$. These processes can be used to define two two-sided Brownian motions $\tilde Z_1^{r,\ulT}$ and $\tilde Z_2^{r,\ulT}$\index{$\tilde Z_i^{r,\ulT}$} by
$$
   \tilde{Z}_1^{r,\ulT}(u) = \left\{\begin{array}{ll}
        B_1^{r,\ulT}(u) - B^{r,\ulT}_1(0), & \mbox{ if } u \geq 0,\\
        B_3^{r,\ulT}(-u) - B^{r,\ulT}_3(0), & \mbox{ if } u < 0,
           \end{array}\right. 
$$
$$  
   \tilde{Z}_2^{r,\ulT}(u) = \left\{\begin{array}{ll}
         B^{r,\ulT}_2(u) - B^{r,\ulT}_2(0), & \mbox{ if } u \geq 0,\\
         B_4^{r,\ulT}(-u) - B_4^{r,\ulT}(0), & \mbox{ if } u < 0,
\end{array}\right.
$$
and an additive Brownian motion $X^{r,\ulT}(s_1, s_2) = 
\tilde{Z}_1^{r,\ulT}(s_1) + \tilde{Z}_2^{r,\ulT}(s_2)$. Then $X^{r,\ulT}$\index{$X^{r,\ulT}$} is a standard ABM that is independent of ${\F}^r_{\ulT}$.

   For $n \geq 1$, we set $\ul{\tau}_{(n)}^{x_0,r} = \cR_r^{-1}(R_n^{x_0,r})$.\index{$\ul{\tau}_{(n)}^{x_0,r}$} 
In particular, $\ul{\tau}_{(N)}^{x_0,r} = 
\cR_r^{-1}(R_{N}^{x_0,r})$. This represents the rectangle in 
$\IR^2$ that is explored by the DW-algorithm up to termination. Clearly, 
$\ul{\tau}_{(N)}^{x_0,r}$ is a stopping point
relative to $(\F^r_{\ul{u}})$, and $\F^r_{\ul{\tau}_{(N)}^{x_0,r}}$ 
contains information about the increments of $X$ restricted to 
$R^r_{N^{x_0,r}}$. 
\vskip 12pt

\noindent{\em The probability of stopping at a given stage}
\vskip 12pt

   For the remainder of this section, we set $r = (0,0)$, and use the notations
above without the superscript $r$. For $n \geq 1$ and $i = 1,\dots,4$, set $\ul{\tau}_{(n)} = \ul{\tau}_{(n)}^{x_0,r}$ and
$$
   Z^n_i(u_i) = Z_i^{r,\ultau_{(n)}}(u_i) - Z_i^{r,\ultau_{(n)}}(0).
$$
Then $(Z^n_1,Z^n_2,Z^n_3,Z^n_4)$ is independent of $\F_{\ul{\tau}_{(n)}}$, and is a 
standard 4-dimensional Brownian motion started at the origin.

   We now evaluate the probability that the DW-algorithm started at $r=(0,0)$ with value $x_0$ does not STOP at a given stage, that is, the conditional probability that $H_n > H_{n-1}$ given $\F_{\ul{\tau}_{(n-1)}}$. We use $P_{x_0}$ to denote probabilities for the DW-algorithm started with value $x_0$.

\begin{prop} Fix $x_0 >0$. For $n \geq 2$, $H_n$ is conditionally independent of
$\F_{\ul{\tau}_{(n-1)}}$ given $H_{n-1}$ and $H_{n-2}$, and
\begin{equation}\label{d1}
   P_{x_0}\{H_n > H_{n-1} \mid H_{n-1},\ 
      H_{n-2}\} = 1 - \left(\frac{H_{n-2}}{H_{n-1}}\right)^2,
\end{equation}
and for $z \geq y > x >0$,
\begin{equation}\label{d2}
  P_{x_0}\{H_n > z \mid H_{n-1} = y,\ H_{n-2} = x\} = 
     \left(\frac{2y}{z} - \frac{y^2}{z^2}\right)
     \left(1 - \frac{x}{y}\right)^2 + \frac{2x}{z}\left(1 - \frac{x}{y}\right).
\end{equation}
\label{noSTOP}
\end{prop}

\proof We assume without loss of generality that $n$ is odd, so that $n$
equals $2m+1$. Fix $u \in [U_m,U_m']$, and let $R$ be the rectangle with corners
$(T^m_1,T^{m-1}_2)$ and $(u,T^m_2)$. As $\Delta_R X = 0$, we see that
$$
   X(u,T^m_2) - X(T^m_1,T^m_2) - X(u,T^{m-1}_2) + X(T^m_1,T^{m-1}_2) = 0,
$$
that is,
\begin{equation}\label{H2n}
   x_0 + X(u,T^m_2) = H_{2m} - (H_{2m-1} - x_0 - X(u,T^{m-1}_2)) \leq H_{2m}
\end{equation}
because $x_0 + X(u,T^{m-1}_2) \leq H_{2m-1}$. Therefore,
$$
   x_0 + \sup_{u \in [U_m,U_m']} X(u,T^m_2) \leq H_{2m}.
$$
Notice also that for $u \in\, ]U_m,U_m'[$,
\begin{equation}\label{X_0}
   x_0 + X(u,T^m_2) = H_{2m} - H_{2m-1} + x_0 + X(u,T^{m-1}_2) >0,
\end{equation}
because $H_{2m} > H_{2m-1}$ and $x_0 + X(u,T^{m-1}_2) >0$ by definition of $U_m$ and
$U_m'$. Finally, note, as in (\ref{H2n}), that 
\begin{equation}\label{starrd1}
   x_0 + X(U_m,T^m_2) = x_0 + X(U'_m,T^m_2) = H_{2m} - H_{2m-1},  
\end{equation}
because $x_0 + X(U_m,T^{m-1}_2) = x_0 + X(U_m',T^{m-1}_2) = 0$.

  It follows from (\ref{H2n}) and (\ref{X_0}) that the DW-algorithm does not STOP at
stage $2m+1 = 2(m+1)-1$ if and only if $T^{m+1}_1 \in [U_{m+1}, U_{m+1}'] \setminus
[U_m,U_m']$. Finally, note that for $u >0$, 
\begin{eqnarray*}
   x_0 + X(U_m'+u,T^m_2) &=& x_0 + X(U_m', T^m_2) + (X(U_m'+u,T^m_2) - X(U_m', T^m_2)) \\
      &=& H_{2m} - H_{2m-1} + Z^{2m}_1(u),
\end{eqnarray*}
and
\begin{eqnarray*}
   x_0 + X(U_m - u,T^m_2) &=& x_0 + X(U_m, T^m_2) + (X(U_m - u,T^m_2) - X(U_m, T^m_2)) \\
      &=& H_{2m} - H_{2m-1} + Z^{2m}_3(u),
\end{eqnarray*}
Let 
$$
   Y_1(u) = H_{2m} - H_{2m-1} + Z^{2m}_1(u),\qquad 
   Y_2(u) = H_{2m} - H_{2m-1} + Z^{2m}_3(u),
$$
and for $i = 1,2$, set
$$
   \tau_i = \inf\{u \geq 0: Y_i(u) = 0\}.
$$
Then
$$
   H_{2m+1} = \max\left(H_{2m},\ \sup_{0 \leq u \leq \tau_1} Y_1(u),\ 
               \sup_{0 \leq u \leq \tau_2} Y_2(u)\right).
$$
Because $H_{2m}$ and $H_{2m-1}$ are $\F_{\ul{\tau}_{(2m)}}$-measurable, and 
$Z^{2m}_1$ and $Z^{2m}_3$ are independent of $\F_{\ul{\tau}_{(2m)}}$, it 
follows that $H_{2m+1}$ is conditionally independent of 
$\F_{\ul{\tau}_{(2m)}}$ given $H_{2m}$ and $H_{2m-1}$.

   Since $Z^{2m}_1$ and $Z^{2m}_3$ are independent standard Brownian motions,
using the gambler's ruin probabilities for Brownian motion, we see that
\begin{eqnarray*}
   P_{x_0}\{H_{2m+1} > H_{2m} \mid H_{2m},\ H_{2m-1}\} &=&
        P\{\max\left(\sup_{0 \leq u \leq \tau_1} Y_1(u),\ 
               \sup_{0 \leq u \leq \tau_2} Y_2(u)\right) >H_{2m}\} \\
     &=& 1 - \left(\frac{H_{2m-1}}{H_{2m}}\right)^2.
\end{eqnarray*}
This proves (\ref{d1}).

   In order to prove (\ref{d2}), let
$$
   Y_1(u) = y - x + Z^{2m}_1(u),\qquad Y_2(u) = y - x + Z^{2m}_3(u),
$$
set
$$
  \sigma_i = \inf\{u \geq 0: Y_i(u) = y\},
$$
and consider the two independent events
$$
   F_1 = \{\sigma_1 < \tau_1\}, \qquad F_2 = \{\sigma_2 < \tau_2\}.
$$
Notice that $\{H_{2m+1} > H_{2m}\} = F_1 \cup F_2$. In addition,
$$
   P(\{H_{2m+1} > z\} \cap F_1 \cap F_2^c \mid H_{2m} = y,\ H_{2m-1} = x)
      = P(\{ \sup_{\sigma_1 \leq u \leq \tau_1} Y_1(u) > z\} \cap F_1\cap
      F^c_2).
$$
By the independence of $\sigma(Y_1,\ F_1)$ and $F_2$, and by the strong Markov
property, this is equal to
\begin{equation}\label{F1F2c}
   \frac{y}{z}\, \frac{x}{y} \, \left(1 - \frac{x}{y}\right) =
      \frac{x}{z} \left(1 - \frac{x}{y}\right).
\end{equation}
Similarly,
\begin{equation}\label{F1cF2}
   P(\{H_{2m+1} > z\} \cap F^c_1 \cap F_2 \mid H_{2m} = y,\ H_{2m-1} = x\} =
      \frac{x}{z} \left(1 - \frac{x}{y}\right).
\end{equation}
Finally,
\begin{eqnarray*}
 && P(\{H_{2m+1} > z\} \cap F_1 \cap F_2 \mid H_{2m} = y,\ H_{2m-1} = x)\\
 &&\qquad = P(\{\max(\sup_{\sigma_1 \leq u \leq \tau_1} Y_1(u),
           \sup_{\sigma_2 \leq u \leq \tau_2} Y_2(u)) > z\} \cap F_1\cap F_2).
\end{eqnarray*}
Let $\ti B_1(\cdot)$ and $\ti B_2(\cdot)$ be independent Brownian
motions starting at $y$. By the strong Markov property for the process $((Y_1(\cdot), Y_2(\cdot))$ at the
stopping point $(\sigma_1,\sigma_2)$, this probability is equal to
\begin{eqnarray}\nonumber
 &&  P\{\ti B_1(\cdot) \mbox{ hits } z \mbox{ before } 0 \mbox{ or }
        \ti B_2(\cdot) \mbox{ hits } z \mbox{ before } 0\} P(F_1)P(F_2) \\
  \nonumber 
 &&\qquad = \left(1 - \left(1 - \frac{y}{z}\right)^2\right) \left(1 -
  \frac{x}{y}\right)^2 \\
  &&\qquad = \left(\frac{2y}{z} - \frac{y^2}{z^2}\right) \left(1 -
  \frac{x}{y}\right)^2.
  \label{F1F2}
\end{eqnarray}
 Adding up (\ref{F1F2c}), (\ref{F1cF2}) and (\ref{F1F2})
establishes (\ref{d2}).
\hfill $\Box$
\vskip 16pt

\noindent{\em A Markov chain}
\vskip 12pt

   For $n \geq 1$, set $\Theta_n = (H_{n-1}, H_n)$.\index{$\Theta_n$} By Proposition \ref{noSTOP},
$(\Theta_n,\ n \geq 1)$ is a discrete time Markov chain with state space $\cS =
\{(x,y): 0 < x \leq y\}$. Given that $\Theta_n = (x,y)$, with $x \leq y$, 
$\Theta_{n+1} = (y, Z_{n+1})$, where $Z_{n+1}$ is a random variable such that
for $x \leq y < z$,
\begin{equation}\label{PZn}
   P\{Z_{n+1} \geq z \mid \Theta_n = (x,y)\} = 
     \left(\frac{2y}{z} - \frac{y^2}{z^2}\right)\left(1 - \frac{x}{y}\right)^2 + 
         \frac{2x}{z}\left(1 - \frac{x}{y}\right),
\end{equation}
while for $x \leq y = z$,
$$
   P\{Z_{n+1} = y \mid \Theta_n = (x,y)\} = \left(\frac{x}{y}\right)^2.
$$
Let 
$$
   D = \{(x,y)\in \cS: x = y\} \quad\mbox{and}\quad 
   S = \{(x,y)\in \cS: y \geq 1\}.
$$
For $0< x < y < 1$, set
$$
   \alpha(x,y) = P\{\mbox{the chain } (\Theta_n)\mbox{ visits } D 
      \mbox{ before } S \mid \Theta_1 = (x,y)\}.
$$
We are going to compute $\alpha(x,y)$ explicitly, and this will lead to the
formula for $\E(x)$ in Theorem \ref{thm1}.

   Consider the matrix and column vector
\begin{equation}\label{Ad}
A = \left( \begin{array}{rrrrrr}
0 & 1 & 0 & 0 & 0 & 0\\
0 & 0 & 1 & 0 & 0 & 0\\
0 & -9 & 6 & 4 & 0 & 0\\
0 & 0 & 0 & 0 & 1 & 0\\
0 & 0 & 0 & 0 & 0 & 1\\
-8 & 2 & 0 & 28 & -26 & 9
\end{array}\right),
\qquad d = \left( \begin{array}{r}
-1\\
-1\\
-3\\
-\half \\
-1\\
-4
\end{array}\right).
\end{equation}
It is not difficult to determine (by hand or using Mathematica, for instance) that $A$ has
six eigenvalues, the first four of which are real and are the numbers
$\lambda_1,\dots,\lambda_4$ defined in Theorem \ref{thm1}, and two complex
eigenvalues $\lambda_5 = \half(5 + i \sqrt{7})$ and $\lambda_6 = \half( 5 - i
\sqrt{7})$. Let $v_1,\dots,v_6$ be the associated eigenvectors, which we
normalize by setting the last entry of each equal to 1, and $V$ the
matrix whose columns are $v_1,\dots,v_6$. The $j$-the entry of $v_k$ will be
denoted $v_{j,k}$. 

   Let $c$ be the column vector with entries $c_1,\dots,c_6$ that is the solution
of the linear system
\begin{equation}\label{Vcd}
   V \cdot c = d.
\end{equation}
We will show below that $c_5 = c_6 = 0$.

\begin{prop} For $0 < x \leq y < 1$,
\begin{equation}\label{formalpha}
   \alpha(x, y) = \left(\frac{x}{y}\right)^2 
             + 2 (y - x) \varphi_1(y) - 2(y-x)^2 \varphi_2(y),
\end{equation}
where
\begin{equation}\label{formphi}
 \varphi_1(y) = \frac{1}{y} + \sum_{k=1}^4 c_k\, y^{\lambda_k -1}\, v_{1,k}, 
 \qquad 
 \varphi_2(y) = \frac{1}{2 y^2} + \sum_{k=1}^4 c_k\, y^{\lambda_k -2}\, v_{4,k}.
\end{equation}
\label{propalpha}
\end{prop}

\proof Let $f(x,y;z)$ be the conditional density of $Z_{n+1}$ given $\Theta_n =
(x,y)$ (this density is defined for $x < y < z$). Differentiate the right-hand
side of (\ref{PZn}) to find, after simplification and regrouping of terms, that
\begin{equation}\label{fZn}
f(x, y; z) = \frac{2(y-x)}{z^2} - \frac{2(x-y)^2}{z^3}.
\end{equation}
By the Markov property of $\Theta_n$, the absorption probability $\alpha(x,y)$
satisfies the relation
\begin{equation}\label{eqalpha}
  \alpha(x, y) = \left(\frac{x}{y}\right)^2 + \int_y^1 dz\, f(x, y; z)\, 
                   \alpha(y, z),
\end{equation}
for all $0 < x < y < 1$. Indeed, at the first step, either the chain visits $D$, 
or it moves to a state $(y,z)$ with $y < z <1$ and then is later on absorbed in
$D$, which occurs with probability $\alpha(y,z)$.

   We now analyze the integral equation (\ref{eqalpha}). Set
\begin{equation}\label{eqphi12}
    \varphi_1(y) = \int_y^1 dz\, \frac{\alpha(y,z)}{z^2}, 
    \qquad \varphi_2(y) = \int_y^1 dz\, \frac{\alpha(y, z)}{z^3}.
\end{equation}
By (\ref{fZn}), (\ref{eqalpha}) can be written
\begin{equation}\label{neqalpha}
 \alpha(x, y) = \left(\frac{x}{y}\right)^2 
             + 2 (y - x) \varphi_1(y) - 2(y-x)^2 \varphi_2(y).
\end{equation}
Substitute (\ref{neqalpha}) into (\ref{eqphi12}) to get the two equations
$$
\left\lbrace 
   \begin{array}{l}
      \varphi_1(y) = \displaystyle \int_y^1 \frac{dz}{z^2} 
     \left[\left(\frac{y}{z}\right)^2 + 2(z-y) \varphi_1(z) - 2(z-y)^2 
     \varphi_2(z)\right],\\
      \\
     \varphi_2(y) = \displaystyle \int_y^1 \frac{dz}{z^3} 
     \left[\left(\frac{y}{z}\right)^2 
      + 2(z-y)  \varphi_1(z) - 2(z-y)^2 \varphi_2(z)\right].
   \end{array} \right.
$$
Rearrange according to powers of $y$ to get:
\begin{eqnarray}\nonumber
   \varphi_1(y) &=& \int_y^1 \frac{dz}{z^2}\, 
        [ 2z \varphi_1(z) - 2z^2 \varphi_2(z)] \\ \label{pphi1}
     &&\qquad +\, y \int_y^1 \frac{dz}{z^2}\, [-2 \varphi_1(z) + 4z \varphi_2(z)] 
      +\, y^2 \int_y^1 \frac{dz}{z^2} \left[ \frac{1}{z^2} - 2 \varphi_2
      (z)\right],\\ \nonumber
  \varphi_2 (y) &=& \int_y^1 \frac{dz}{z^3} \,
    [ 2z \varphi_1(z) - 2 z^2 \varphi_2(z)] \\
    &&\qquad + y \int_y^1 \frac{dz}{z^3}\, [-2 \varphi_1(z) + 4 z \varphi_2(z)] 
       +\, y^2 \int_y^1 \frac{dz}{z^3}\, \left[ \frac{1}{z^2} - 2 \varphi_2
       (z)\right].\quad \label{pphi2}
\end{eqnarray}
Differentiate (\ref{pphi1}) and (\ref{pphi2}), to get after simplification:
\begin{equation}\label{1diff}
  \begin{array}{rcl}
   \varphi_1^\prime(y) &=& \displaystyle{- \frac{1}{y^2} + \int_y^1 \frac{dz}{z^2}
        (-2 \varphi_1(z) + 4 z \varphi_2(z)) 
	+ 2y \int_y^1 \frac{dz}{z^2}\left(\frac{1}{z^2} - 2 \varphi_2(z)\right),}
	\\
	& & \\
   \varphi_2^\prime(y) &=& \displaystyle{- \frac{1}{y^3} + \int_y^1 \frac{dz}{z^3} (-2
   \varphi_1(z) + 4 z \varphi_2(z)) + 2y \int_y^1 \frac{dz}{z^3}
   \left(\frac{1}{z^2} - 2 \varphi_2(z)\right).}
  \end{array}
\end{equation}
Differentiate again, to get
\begin{equation}\label{2diff}
  \begin{array}{rcl}
   \varphi_1^{\prime \prime}(y) &=& \displaystyle{\frac{2 \varphi_1(y)}{y^2} + 2 \int_y^1
   \frac{dz}{z^2} \left(\frac{1}{z^2} - 2 \varphi_2(z)\right),}\\
    & & \\
   \varphi_2^{\prime\prime}(y) &=& \displaystyle{\frac{1}{y^4} + \frac{2 \varphi_1(y)}{y^3} +
   2 \int_y^1 \frac{dz}{z^3} \left(\frac{1}{z^2} - 2 \varphi_2(z)\right),}
  \end{array}
\end{equation}
and differentiate a third time, to get after simplification:
\begin{equation}\label{3diff}
 \left\{
  \begin{array}{rcl}
 \varphi_1^{\prime \prime \prime}(y) &=& \frac{-4}{y^3}\, \varphi_1(y) 
       + \frac{2}{y^2}\, \varphi^\prime_1(y) 
       - \frac{2}{y^4} + 4 \frac{\varphi_2(y)}{y^2}\,,\\
       && \\
 \varphi_2^{\prime \prime \prime}(y) &=& \frac{-6}{y^5} 
       - \frac{6}{y^4}\, \varphi_1(y) + \frac{2}{y^3}\, \varphi^\prime_1 (y) 
       + \frac{4}{y^3}\, \varphi_2(y).
  \end{array}\right.
\end{equation}
Let $\psi_1(\cdot)$ and $\psi_2(\cdot)$ be the two functions defined by
$$
   \varphi_1(y) = \frac{1}{y}\, \psi_1(\ln y), \qquad 
   \varphi_2(y) = \frac{1}{y^2}\, \psi_2(\ln y).
$$
The two equations in (\ref{3diff}) translate into the following two 
equations for $\psi_1$ and $\psi_2$:
\begin{equation}
\left\{
  \begin{array}{rcl}
   \psi_1^{\prime \prime \prime}(z) &=& -9 \psi_1^\prime(z) 
       + 6 \psi_1^{\prime \prime}(z) + 4 \psi_2(z) - 2,\\ 
   \psi_2^{\prime \prime \prime}(z) &=& -8 \psi_1(z) 
      + 2 \psi_1^\prime(z) 
         + 28 \psi_2(z) - 26 \psi_2^\prime(z) + 9 \psi_2^{\prime \prime}(z) - 6.
   \end{array}\right.
\end{equation}
This third order system of ordinary differential equations translates into the
first order system of six differential equations and six unknowns
\begin{equation}\label{diffsyst}
\mu^\prime(z) = A \cdot \mu(z) + b,
\end{equation}
where $A$ is defined in (\ref{Ad}),
$$
   b = \left(\begin{array}{r}
             0\\
             0\\
             -2\\
             0\\
             0\\
             -6
        \end{array}\right),\qquad\mbox{and}\qquad 
  \mu (z) = \left( 
        \begin{array}{r}
             \psi_1(z)\\
             \psi_1^\prime(z)\\
             \psi_1^{\prime \prime}(z)\\
             \psi_2(z)\\
             \psi_2^\prime(z)\\
             \psi_2^{\prime \prime}(z)
        \end{array}\right).
$$
The initial condition for (\ref{diffsyst}) is $\mu(0)^T = (0,-1,-3,0,1,-4)$. 
Indeed, we see from (\ref{eqphi12}) that $\varphi_1(1) = 0 = \varphi_1(2)$, and 
from (\ref{1diff}) and (\ref{2diff}), we find that
$$
   \varphi_1^\prime(1) = -1, \qquad \varphi_1^{\prime \prime}(1) = 0, 
    \qquad 
   \varphi_2^\prime(1) = -1, \qquad \varphi_2^{\prime \prime}(1) = 1,
$$
which translates into
\begin{eqnarray*}
   \psi_1(0) = 0, \qquad \psi^\prime_1(0) = -1, 
      \qquad \psi^{\prime \prime}_1(0) = -3,\\
   \psi_2(0) = 0, \qquad \psi_2^\prime(0) = -1, 
      \qquad \psi_2^{\prime \prime}(0) = -4.
\end{eqnarray*}
The solution of (\ref{diffsyst}) is of the form
\begin{equation}\label{muform}
   \mu (z) = \nu + \sum^6_{k=1} c_k\, e^{\lambda_k z} v_k,
\end{equation}
where $\nu^T = (\nu_1,\dots,\nu_6)$ and $c_1, \ldots, c_6$ are determined as follows. 
Given $\mu(z)$ as in (\ref{muform}),
$$
   \mu^\prime(z) = \sum_{k=1}^6 c_k\, e^{\lambda_k z} (\lambda_k v_k) 
   = A\cdot\left(\sum_{k=1}^6 c_k e^{\lambda_k z} v_k \right),
$$
while
$$
   A\cdot \mu(z) + b = A\cdot \nu + A \cdot 
   \left(\sum_{k=1}^6 c_k e^{\lambda_k z} v_k \right) + b.
$$
Therefore, $\mu(z)$ will solve (\ref{diffsyst}) if and only if
$$
   A\cdot \nu + b = 0.
$$
This determines $\nu$. One immediately checks that
$$
   \nu^T = (1,0,0,\half,0,0).
$$
The $c_i$ are now determined from the initial conditions and (\ref{muform}), by
solving the linear system
$$
   \sum_{i=1}^6 c_i\, v_i = \mu(0) - \nu,
$$
which reduces to (\ref{Vcd}), since $\mu(0) - \nu = d$.

  The functions $\psi_1(z)$ and $\psi_2(z)$ are given by rows 1 and 4 of 
$\mu(z)$:
$$
   \psi_1(z) = \nu_1 + \sum_{k=1}^6 c_k e^{\lambda_k z} v_{1,k}\qquad
   \psi_2(z) = \nu_4 + \sum_{k=1}^6 c_k e^{\lambda_k z} v_{4,k},
$$
so $\varphi_1(y)$ and $\varphi_2(y)$ are given by:
$$
   \varphi_1(y) = \frac{\nu_1}{y} 
      + \sum_{k=1}^6 c_k y^{\lambda_k -1} v_{1,k},\qquad 
   \varphi_2(y) = \frac{\nu_4}{y^2} 
      + \sum_{k=1}^6 c_k y^{\lambda_k -2} v_{4,k}.
$$
By (\ref{neqalpha}), this establishes (\ref{formalpha}) and (\ref{formphi}),
because $\nu_1 = 1$ and $\nu_4 = \half$, except that the summations in these
expressions have six terms instead of the four indicated in (\ref{formphi}). We will see below that the two omitted terms are in fact equal to zero.

\vskip 12pt

\noindent{\em Numerical values}
\vskip 12pt

   By hand, or using computer software such as Mathematica, one easily obtains
the exact values of the eigenvectors $v_1,\dots,v_6$, and it is then
straightforward to check that these are indeed eigenvalues and eigenvectors. For
instance
$$
   v_1 = \left(\begin{array}{c}
      4(3+\sqrt{5})(5 + \sqrt{13+4 \sqrt{5}})/(x_1^2 x_2) \\
      32 /(x_1^2 x_2) \\
      -16/(x_1 x_2) \\
      4 / x_1^2 \\
      -2 / x_1 \\
      1
    \end{array} \right)
$$
where
$$
   x_1 = - 5 + \sqrt{13+4 \sqrt{5}}\qquad\mbox{and}\qquad
   x_2 = 7 + 2 \sqrt{5} + \sqrt{13+4 \sqrt{5}}.
$$
A numerical approximation of the first four columns of $V$ is
\begin{equation}\label{numV}
   \left(\begin{array}{rrrr}
      126.09952 & 0.52922 & 0.36088 &  0.010381 \\
       19.89401 & 0.79016 & 1.26557 &  0.050266 \\
        3.13856 & 1.17975 & 4.43828 &  0.243402  \\
       40.17745 & 0.44859 & 0.08131 &  0.042649 \\
        6.33857 & 0.66977 & 0.28515 &  0.206516 \\
	1.00000 & 1.00000 & 1.00000 &  1.000000
    \end{array} \right)
\end{equation}
One easily checks that this matrix multiplied on the right by
\begin{equation}\label{numc}
\left(\begin{array}{c}
     c_1 \\
     c_2 \\
     c_3 \\
     c_4
    \end{array} \right)
    =
    \left(\begin{array}{c}
     -0.00546374 \\
     -0.230522 \\
     -0.427837 \\
     -3.33618
    \end{array} \right)
\end{equation}
equals $d$. Now (\ref{Vcd}) has a unique solution because $\{v_1,\dots,v_6\}$ are
linearly independent since the six eigenvalues of $A$ are distinct. This shows
that $c_5 = c_6 = 0$. The fact that this is indeed an equality and not an
approximate equality can be tediously checked by hand using the exact values of
$v_1,\dots,v_4$, or by using Mathematica. This proves Proposition 
\ref{propalpha}.
\hfill $\Box$
\vskip 16pt

\noindent{\sc Proof of Theorem \ref{prop1}.} By Proposition \ref{prop2},
$$
   \E(x) = P_x\{x+ X(\Gamma^*(\tau^{\Gamma^*})) = 1\},
$$
and $\{x+ X(\Gamma^*(\tau^{\Gamma^*})) = 1\}$ occurs if and only if the sequence
$(H_n)$ exceeds 1 before the DW-algorithm STOPS, that is, before two consecutive
values of this sequence coincide. By definition of the Markov chain $(\Theta_n)$
and the absorption probability $\alpha(x,y)$, this implies
$$
   \E(x) = 1 - E_x\left(\alpha(\Theta_1)\, 1_{\{H_1 < 1\}}\right) = 1 - E_x\left(\alpha(x,H_1)\, 1_{\{H_1 < 1\}}\right).
$$
Let $U_1$ and $U_1'$ be defined as in Stage 1 of the DW-algorithm. Then
$$
   U_1 = - \inf\{u \geq 0: B_3(u) = -x\}, \qquad
   U_1' = \inf\{u \geq 0: B_1(u) = -x\},
$$
and therefore, under $P_x$,
$$
   H_1 = x + \max\left(\sup_{U_1 \leq u \leq 0} B_3(-u),\ 
                      \sup_{0 \leq u \leq U_1'} B_1(u)
                 \right).
$$
Let $\ti{B}_1$ be a standard Brownian motion starting at 0. Then for $y \geq x$,
\begin{eqnarray*}
   P_x\{H_1 \leq y\} &=& \left(P\{\ti{B}_1 \mbox{ hits } -x \mbox{ before }
              y-x \} \right)^2 \\
	&=& \left(\frac{y-x}{y}\right)^2 = \left(1 - \frac{x}{y}\right)^2.
\end{eqnarray*}
The density of $H_1$ under $P_x$ is therefore
$$
   2 \left(1 - \frac{x}{y}\right)\, \frac{x}{y^2},
$$
and so
$$
   E_x\left(\alpha(x,H_1)\, 1_{\{H_1 < 1\}}\right) = 2 \int_x^1 dy\, \left(1 - \frac{x}{y}\right)\, 
           \frac{x}{y^2}\, \alpha(x,y).
$$
Replace $\alpha(x,y)$ by its value as expressed in (\ref{formalpha}) and 
(\ref{formphi}), then
multiply out the factors in the integrand and reorder according to powers of
$y$, to see that this equals
\begin{eqnarray*}
    2x \int_x^1 dy\,\left(\frac{1}{y^2} - \frac{x}{y^3}\right. &+&
                   2 \sum_{k=1}^4 c_k\,(v_{1,k} - v_{4,k})\, y^{\lambda_k-2} \\
    \qquad &-& x  \sum_{k=1}^4 c_k\,(4 v_{1,k}- 6 v_{4,k})\, y^{\lambda_k-3} \\
    \qquad &-& x^2  \sum_{k=1}^4 c_k\,(6v_{4,k} - 2 v_{1,k})\, y^{\lambda_k-4} \\
    \qquad &+& 2 x^3  \left. \sum_{k=1}^4 c_k\, v_{4,k}\, y^{\lambda_k-5}\right).
\end{eqnarray*}
The antiderivative of the integrand is now trivially computed and one finds that
$\E(x)$ is equal to
\begin{eqnarray*}
   2x - x^2 \!\!&+&\! 4 \sum_{k=1}^4 c_k\, (v_{1,k} - v_{4,k}) \frac{x^{\lambda_k}}{\lambda_k -1}
     - 2 \sum_{k=1}^4 c_k\, (4 v_{1,k} - 6v_{4,k}) 
                  \frac{x^{\lambda_k}}{\lambda_k -2}\\
   && -2 \sum_{k=1}^4 c_k\, (6v_{4,k} - 2 v_{1,k}) 
                  \frac{x^{\lambda_k}}{\lambda_k -3}
       +4 \sum_{k=1}^4 c_k\, v_{4,k} 
                  \frac{x^{\lambda_k}}{\lambda_k -4} \\
   && - 4x \sum_{k=1}^4 c_k \frac{v_{1,k} - v_{4,k}}{\lambda_k-1} 
      + 2 x^2 \sum_{k=1}^4 c_k \frac{4 v_{1,k} - 6 v_{4,k}}{\lambda_k-2} \\
   && + 2 x^3 \sum_{k=1}^4 c_k \frac{6 v_{4,k} - 2 v_{1,k}}{\lambda_k-3}
      - 4 x^4 \sum_{k=1}^4 c_k \frac{v_{4,k}}{\lambda_k-4}.
\end{eqnarray*}
This has the form 
$$
   \sum_{i=1}^4 \alpha_i\, x^{\lambda_i} + \sum_{i=1}^4 \beta_i\, x^{i},
$$
with, for $k = 1,\dots,4$,
\begin{equation}\label{explalpha}
   \alpha_k = c_k\left(4 \frac{v_{1,k}- v_{4,k}}{\lambda_k - 1}
      - 4 \frac{2 v_{1,k} - 3 v_{4,k}}{\lambda_k - 2}
      -4 \frac{3v_{4,k}- v_{1,k}}{\lambda_k - 3}
      +4 \frac{v_{4,k}}{\lambda_k - 4}\right),
\end{equation}
and
\begin{equation}\label{explbeta}
  \begin{array}{ll}
   \beta_1 = 2 -4 \displaystyle{\sum_{k=1}^4}\, c_k \frac{v_{1,k} - v_{4,k}}{\lambda_k-1}, &
   \beta_2 = -1 + 4 \displaystyle{\sum_{k=1}^4}\, c_k 
            \frac{2 v_{1,k} - 3 v_{4,k}}{\lambda_k-2},\\
   & \\
   \beta_3 = 4 \displaystyle{\sum_{k=1}^4}\, c_k \frac{3 v_{4,k} - v_{1,k}}{\lambda_k-3}, &
   \beta_4 = -4 \displaystyle{\sum_{k=1}^4}\, c_k \frac{v_{4,k}}{\lambda_k-4}.
  \end{array}
\end{equation}
If we substitute in the numerical approximations of 
$v_{1,k}$, $v_{4,k}$ and $c_k$ from (\ref{numV}) and (\ref{numc}), we find
\begin{equation}\label{remnum}
\begin{array}{llll}
   \alpha_1 \simeq 0.938911, & \alpha_2 \simeq 0.037177, & \alpha_3 \simeq 0.239215, 
   & \alpha_4 \simeq -0.215302, \\
   \beta_1 \simeq 0.00000, & \beta_2 \simeq 0.00000, & \beta_3 \simeq 0.00000, 
  & \beta_4 \simeq 0.00000.
\end{array}
\end{equation}
The fact that the $\beta_i$ are exactly equal to $0$ can be checked tediously by hand or by using Mathematica.
This yields the formula in (\ref{formE}) and completes the proof of Theorem \ref{prop1}.
\hfill $\Box$
\vskip 16pt


\end{section}
\eject

\begin{section}{Escape probabilities for additive Brownian motion}\label{sec3}
In the previous section, we computed probabilities of the type ``there is a path
starting at $(0,0)$ along which a standard ABM started at $x>0$ reaches level $M >0$ before $0$." In this section,
we shall estimate related probabilities of the type ``there is a path starting 
at $(0,0)$ and ending at least $a$ units away from $(0,0)$ along which a standard ABM started at $x>0$ is 
positive." This will be useful in the next section, because of measurability issues: the former probabilities may depend on values of the ABM at points that are arbitrarily far away from the origin, whereas the latter only depend on the behavior of the ABM in a ball around the origin. Our first application of this result will be Proposition \ref{univariateup}, which is a key ingredient in the proof of the first half of Theorem \ref{thm2} (see Section \ref{sec4}).

   We begin by introducing some additional terminology related to the DW-algorithm. We say that the DW-algorithm started at $r$ with value $x_0$ {\em reaches level $M$ during the stage $n$} if
$H_{n-1} < M \leq H_n$, or equivalently, $\sup_{R_{n-1}} X^r < M \leq
\sup_{R_n} X^r$. We define the {\em point $\ul{\tau}^{x_0,r,M}= (\ul{\tau}^{x_0,r,M}_1,
\ul{\tau}^{x_0,r,M}_2,\ul{\tau}^{x_0,r,M}_3,\ul{\tau}^{x_0,r,M}_4)$\index{$\ul{\tau}^{x_0,r,M}$} at which the
DW-algorithm reaches level $M$} as follows: if the algorithm reaches level $M$
during an odd stage $2n+1$, then $\ul{\tau}^{x_0,r,M}_2 = V_n' - r_2$,
$\ul{\tau}^{x_0,r,M}_4 = r_2 - V_n$,
\begin{eqnarray*}
   \ul{\tau}^{x_0,r,M}_1 &=& -r_1 + \inf\{u > U_n': x_0 + X^r(u,T^n_2) \in \{0,M\}\},\\
   \ul{\tau}^{x_0,r,M}_3 &=& r_1 - \sup\{u < U_n: x_0 + X^r(u,T^n_2) \in\{0,M\}\},
\end{eqnarray*}
while if the algorithm reaches $M$ during an even stage $2n$, then 
$\ul{\tau}^{x_0,r,M}_1 = U_n' - r_1$, $\ul{\tau}^{x_0,r,M}_3 = r_1 - U_n$,
\begin{eqnarray*}
   \ul{\tau}^{x_0,r,M}_2 &=& -r_2 + \inf\{v > V_{n-1}': x_0 + X^r(T^n_1,v) \in
                          \{0,M\}\},\\
   \ul{\tau}^{x_0,r,M}_4 &=& r_2 - \sup\{v < V_{n-1}: x_0 + X^r(T^n_1) \in\{0,M\}\}.
\end{eqnarray*}
If the algorithm never reaches level $M$, then $\ul{\tau}^{x_0,r,M}$ is taken to be 
$\ulinfty = (\infty, \infty, \infty, \infty)$.\index{$\ulinfty$} 
We note here that $\ul{\tau}^{x_0,r,M}$ is a stopping point and if, for instance 
the algorithm reaches level $M$ during stage  $2n+1$, then at least one (and 
possibly both) of $x_0 + X^{r}(r_1+ \ul{\tau}^{x_0,r,M}_1, T^n_2)$ and 
$x_0 + X^{r}(r_1-\ul{\tau}^{x_0,r,M}_3, T^n_2)$ is equal to $M$.  

   If $r \in R= [a_1,b_1]\times[a_2,b_2]$, we say that the DW-algorithm started at $r$ with value $x_0$ {\em
escapes} $R$ if $\ul{\tau}_{(N)}^{x_0,r} \not\leq (b_1-r_1,b_2-r_2,r_1-a_1,r_2-a_2)$, or
equivalently, $\cR_r(\ul{\tau}_{(N)}^{x_0,r}) \not\subset R$. We let 
$\ul{\sigma}^{x_0,r,R}$\index{$\ul{\sigma}^{x_0,r,R}$}
represent {\em the portion of $\IR^2$ explored up to escaping $R$,} that is, if for some
$n \geq 1$, $R^{x_0,r}_n \subset R$ but $R^{x_0,r}_{n+1} \not\subset R$, then 
$\ul{\sigma}^{x_0,r,R} = \cR_r^{-1}(R \cap R^{x_0,r}_{n+1})$, and if 
$\cR_r(\ul{\tau}_{(N)}^{x_0,r}) \subset R$, then $\ul{\sigma}^{x_0,r,R} = 
\ul{\tau}_{(N)}^{x_0,r}$. We note that $\ul{\sigma}^{x_0,r,R} < \ulinfty$ 
always holds.

   Finally, we say that the DW-algorithm {\em reaches level $M$ before escaping
$R$} if $\ul{\tau}^{x_0,r,M} \leq \ul{\sigma}^{x_0,r,R}$ (recall the definition of $\leq$ introduced just before \eqref{startp6}).

   As in the previous section, if $r=(0,0)$, then we omit the sub/superscript $r$ from the notations
introduced above, as well as $x_0$ if it is determined from the context.

   The next theorem, which is based on Theorem \ref{prop1} and concerns the probability that the DW-algorithm escapes a given square, is the key ingredient in the proof of the upper bound on the Hausdorff dimension of the boundaries of bubbles.
  
\begin{thm} Consider the DW-algorithm started at $(0,0)$ with value $x_0$. We use $P_{x_0}$ to refer to probabilities for this algorithm. There exist constants $c$ and $C$ such that for all $x_0 >0$ and all $a \geq x_0^2$,
$$
  c \left(\frac{x_0}{\sqrt{a}}\right)^{\lambda _1} \leq
  P_{x_0}\left\{\cR(\ultau_{(N)}) \not\subset [-a,a]^2 \right\} \leq
   C \left(\frac{x_0}{\sqrt{a}}\right)^{\lambda _1}.
$$
\label{thm3}
\end{thm}
 
   Theorem \ref{thm3} is natural in view of the scaling property of Brownian
motion and the fact that, starting from level 1, say, it is to be expected that 
escaping the square 
$[-a,a]^2$ without hitting $0$ has about the same probability as reaching
level $\sqrt a$ without hitting $0$. 
\vskip 12pt

\noindent{\em Proof of Theorem \ref{thm3} (lower bound)}. In view of the scaling
property of Brownian motion, it suffices to consider the case $x_0 = 1$. We
begin with the lower bound. Clearly,
$$
   P_1\{\cR(\ul{\tau}_{(N)}) \not\subset [-a,a]^2\} \geq 
      \sum_{i=1}^6 P_1(\{\cR(\ul{\tau}_{(N)}) \not\subset [-a,a]^2,\
      \ul{\tau}^{\sqrt{a}} < \infty\} \cap F_i),
$$
where $F_1,\dots,F_6$ represent the six possible configurations which we
describe informally as follows: level $\sqrt{a}$ is reached during an odd stage,
and $x_0 + X$ equals $\sqrt{a}$ at the upper right corner only, or at the upper left corner only, or at both
of these corners, or, level $\sqrt{a}$ is reached during
an even stage, etc. Let $F_1$ be the event ``level $\sqrt{a}$ is reached during 
an odd stage, $x_0 + X(\ul{\tau}_1^{\sqrt{a}},\ul{\tau}_2^{\sqrt{a}}) = \sqrt{a}$ and
$x_0 + X(-\ul{\tau}_3^{\sqrt{a}},\ul{\tau}_2^{\sqrt{a}}) < \sqrt{a}$." Then
\begin{eqnarray*}
 && \{\cR(\ul{\tau}_{(N)}) \not\subset [-a,a]^2,\ \ul{\tau}^{\sqrt{a}} < \infty\} 
    \cap F_1  \\
 &&\qquad\qquad\qquad \supset \{\ul{\tau}^{\sqrt{a}} < \infty\} \cap F_1 \cap
       \{x_0 + X(\ul{\tau}^{\sqrt{a}}_1 + u, \ul{\tau}^{\sqrt{a}}_2) >0,
       \ 0 \leq u \leq a\}.
\end{eqnarray*}
Because $\ul{\tau}^{\sqrt{a}}$ is a stopping point, we can use the strong Markov
property of additive Brownian motion at this point to conclude that for
$i=1,\dots,6$,
$$
 P_1(\cR(\ul{\tau}_{(N)}) \not\subset [-a,a]^2,\ \ul{\tau}^{\sqrt{a}} < \infty\} 
    \cap F_i) \geq c P_1(\{\ul{\tau}^{\sqrt{a}} < \infty\} \cap F_i),
$$
where $c = P_{\sqrt{a}}\{B_1(u) >0,\ 0 \leq u \leq a\} > 0$ and does not
depend on $a$. Summing over $i=1,\dots,6$, we find that
$$
    P_1\{\cR(\ul{\tau}_{(N)}) \not\subset [-a,a]^2\} 
        \geq c P_1\{\ul{\tau}^{\sqrt{a}} < \infty\}
	= c P_{a^{-1/2}}\{\ul{\tau}^1 < \infty\} = c a^{-\lambda_1/2}
$$
by Theorem \ref{prop1}, which proves the desired lower bound.
\hfill $\Box$
\vskip 16pt

   We now turn to the upper bound, again in the case $x_0 = 1$. We need two
lemmas.

\begin{lemma} For $\ell \geq 0$, let $\nu_\ell$\index{$\nu_\ell$} be the (random) number of stages
needed by the DW-algorithm to pass from level $2^\ell$ to $2^{\ell+1}$ before terminating, that is,
$$
   \nu_\ell = 1 + \sum_{n \geq 1} 1_{\{2^\ell < H_n < 2^{\ell+1},\ N > n\}}.
$$
For all $\ell, n \geq 1$,
$$
   P_1\{\nu_\ell \geq n \mid \F_{\ul{\tau}^{2^\ell}}\} \leq \left(\frac{3}{4}\right)^{n-2}, \qquad
   a.s.\mbox{ on } \{\ul{\tau}^{2^\ell} < \ul{\infty}\}.
$$
\label{rdlem9}
\end{lemma}

\proof Suppose level $2^\ell$ is reached during stage $N_0$, that is, $H_{N_0
-1} < 2^\ell \leq H_{N_0}$. Notice that $H_m - H_{m-1} < 2^\ell$, for any $m > N_0$ such that $2^\ell \leq
H_{N_0} \leq H_{m-1} < H_m < 2^{\ell +1}$.

   Consider now an even stage $2m > N_0$ such that $2^\ell < H_{2m} < 2^{\ell
+1}$. Let
$$
   W_1(u) = x_0 + X(U_m' +u, T^m_2), \qquad W_2(u) = x_0 + X(U_m - u,T^m_2).
$$
Given $H_{2m} - H_{2m-1}$, these two processes are conditionally independent
Brownian motions started at $H_{2m} - H_{2m-1}$, which are conditionally
independent of $\F_{\ul{\tau}_{(2m)}}$. Observe that the event $F_m = \{N =
2m+1\}$ contains the event ``$W_1$ and $W_2$ hit $0$ before $H_{2m}$," so on 
$\{H_{2m} - H_{2m-1} \leq 2^{\ell - 1},\ N_0 < 2m \leq N\}$, $P_1(F_m \mid 
\F_{\ul{\tau}_{(2m)}}) \geq (1/2)^2 = 1/4$. On the other hand, the event 
$G_m = \{N>2m+1,\ H_{2m+1} \geq 2^{\ell +1}\}$ contains the event ``$W_1$ or
$W_2$ hits $2^{\ell +1}$ before $0$," so on $\{H_{2m} - H_{2m-1} \geq 
2^{\ell - 1},\ N_0 < 2m \leq N\}$, $P_1(G_m \mid \F_{\ul{\tau}_{2m}}) \geq 1 - (3/4)^2 = 7/16 \geq
1/4$. Therefore, on $\{N_0 < 2m \leq N,\ H_{2m} < 2^{\ell +1}\}$,
\begin{eqnarray*}
  && P_1(\{N = 2m+1\}\cup\{N > 2m+1,\ H_{2m+1} \geq 2^{\ell +1} \} \mid 
       \F_{\ul{\tau}_{(2m)}}) \\
  && \qquad\geq P(F_m \mid \F_{\ul{\tau}_{(2m)}}) 
         1_{\{H_{2m} - H_{2m-1} \leq 2^{\ell - 1}\}}
	  + P(G_m \mid \F_{\ul{\tau}_{(2m)}}) 
	   1_{\{H_{2m} - H_{2m-1} \geq 2^{\ell - 1}\}} \\
  &&\qquad\geq \frac{1}{4}.
\end{eqnarray*}
The same inequality holds for odd stages, and so for $m$ even or odd,
$$
   P_1\{N > m+1,\ H_{m+1} < 2^{\ell+1} \mid \F_{\ul{\tau}_{(m)}} \} 
   \leq \frac{3}{4} \qquad \mbox{on } \{N_0 < m \leq N,\ H_m < 2^{\ell+1} \}.
$$
In words, given that level $2^\ell$ has been reached but level $2^{\ell +1}$ has
not been reached at stage $m$, the probability that at stage $m+1$, the 
DW-algorithm has not stopped and level $2^{\ell +1}$ is not reached, is bounded 
above by $3/4$.
Therefore, conditional tail probabilities for $\nu_\ell$ are bounded above by
those of a geometric random variable with success probability $1/4$, which
proves the lemma.
\hfill $\Box$
\vskip 16pt

\begin{lemma} There exist $c>0$ and $C < \infty$ such that for all $\ell \in
\IN$ and $x \geq 1$,
$$
   P_1\{\Vert \ul{\tau}^{2^{\ell+1}} \wedge \ul{\tau}_{(N)} -
   \ul{\tau}^{2^\ell}\Vert \geq x2^{2\ell} \mid \F_{\ul{\tau}^{2^\ell}} \} 
       \leq C e^{-cx} \qquad\mbox{on } \{\ul{\tau}^{2^\ell} < \infty \}.
$$
\label{rdlem10}
\end{lemma}

\proof Let $B$ be a standard Brownian motion, and set $\kappa = \inf\{u \geq 0:
B_u \leq \sup_{v \leq u} B_v - 2\}$. We shall show below that for all $n \in
\IN_*$, $\ell \in \IN$ and $x >0$,
\begin{equation}\label{rd35}
 P_1\{ ((\ul{\tau}^{2^{\ell+1}}_1 \wedge U'_{n+1}) - 
 (\ul{\tau}^{2^{\ell}}_1 \vee U'_n) )_+ \geq x2^{2\ell} 
    \mid \F_{\ul{\tau}^{2^{\ell}} \vee \ul{\tau}_{(n)}}\} 
    \leq P\{\kappa > x\}.
\end{equation}
on $\{\ul{\tau}^{2^{\ell}} < \infty,\ N > n \}$. Assuming this for the moment, we now prove the lemma. 

   It suffices to show that on $\{\ul{\tau}^{2^{\ell}} < \infty \}$, for 
$i=1,\dots,4$, $P(G(i,\ell,x) \mid \F_{\ul{\tau}^{2^{\ell}}}) \leq C e^{-cx}$, 
where
$$
   G(i,\ell,x) = \{\ul{\tau}^{2^{\ell+1}}_i \wedge (\ul{\tau}_{(N)})_i - 
 \ul{\tau}^{2^{\ell}}_i \geq \frac{x}{4} 2^{2\ell} \}.
$$
By L\'evy's theorem \cite[Chap.VI, Theorem (2.3)]{RY}, $E(\kappa) < \infty$. Fix $K < \infty$ such that $E(\kappa) < K$, and let $\nu_\ell$ be as in Lemma
\ref{rdlem9}. Then $G(i,\ell,x)$ is contained in
$$
 \left\{\nu_\ell \geq \frac{x}{K}\right\} \cup \left\{\nu_\ell < \frac{x}{K},\ 
      \sum_n (\ul{\tau}^{2^{\ell+1}}_i 
    \wedge (\ul{\tau}_{(N)})_i \wedge U'_{n+1} - \ul{\tau}^{2^{\ell}}_i \vee
    U'_n)_+ \geq \frac{x}{4} 2^{2\ell}\right\}.
$$
By Lemma \ref{rdlem9}, the probability of the first event is $\leq (3/4)^{-2+x/K}$
for all $\ell$, and by (\ref{rd35}), the probability of the second event is bounded
above by
$$
 P\left\{\sum_{i=1}^{x/K} \kappa_i \geq x\right\} 
    = P\left\{\frac{1}{x/K}\sum_{i=1}^{x/K} (\kappa_i -
 E(\kappa)) \geq K (1-\frac{E(\kappa)}{K}) \right\},
$$
where the $\kappa_i$ are i.i.d.~copies of $\kappa$. Because $1-E(\kappa)/K >0$ and $\kappa$ has some positive exponential moments,
Cramer's theorem \cite[Section 2.2]{DZ} applied to $\kappa$ yields the desired exponential bound.

   We now prove (\ref{rd35}). The event on the left-hand side occurs only if
$\ul{\tau}^{2^\ell}_1 \vee U_n' < \ul{\tau}^{2^{\ell +1}}_1 \wedge U_{n+1}'$,
and in this case, setting $Z(u) = x_0 + X(\ul{\tau}^{2^\ell}_1 \vee U_n' +u, T^n_2)$, 
this is the event ``it takes $Z$ at
least $x2^\ell$ units of time to hit $\{0,2^{\ell+1}\}$." Since $Z$ starts at a
value in $]0,2^{\ell+1}[$, this amount of time is bounded above by $\kappa(\ell)
= \inf\{u \geq 0: Z(u) \leq \sup_{v \leq u} Z(v) - 2^{\ell+1} \}$. Since $Z(u) -
Z(0)$ is independent of $\F_{\ul{\tau}^{2^\ell} \vee \ul{\tau}_{(n)}}$, $\kappa(\ell)$ is too, and $\kappa(\ell)$ and
$2^{2\ell}\kappa$ are identically distributed. This proves (\ref{rd35}) and
completes the proof of the lemma.
\hfill $\Box$
\vskip 16pt

\noindent{\em Proof of Theorem \ref{thm3} (upper bound).} Since the upper bound
is a fixed power of $a$, it suffices to prove it for $a$ of the form $a = 2^{2m}$
for all large integers $m$.

   For $j \in \IN$, set $R^{(j)} = [-2^j,2^j]^2$, and let $F = \{R_{N}
\not\subset R^{(2m)}\}$, which is the event ``the DW-algorithm started at
$(0,0)$ with value $x_0$ escapes the rectangle $[-2^{2m},2^{2m}]^2$." Let
$$
 F_{m,0} = \{R_{N} \not\subset R^{(2m)},\ \cR(\ul{\tau}^{2^m}) \subset R^{(2m)}\},
$$
for $j = 1,\dots,m-1$, let
$$
 F_{m,j} = \{R_{N} \not\subset R^{(2m-j+1)},\ \cR(\ul{\tau}^{2^{m-j+1}} \wedge
 \ul{\tau}_{(N)}) \not\subset R^{(2m-j+1)},\  \cR(\ul{\tau}^{2^{m-j}}) 
    \subset R^{(2m-j)}\},
$$
and for $j = 0,\dots,m$, let
$$
 G_{m,j} = \{R_{N} \not\subset R^{(2m-j)},\ \cR(\ul{\tau}^{2^{m-j}} \wedge
 \ul{\tau}_{(N)}) \not\subset R^{(2m-j)} \}.
$$
Clearly, $F \subset F_{m,0} \cup G_{m,0}$, and for $j =0,\dots,m-1$, one easily
checks that $G_{m,j} \subset F_{m,j+1}\cup G_{m,j+1}$, and therefore,
$$
   F \subset F_{m,0} \cup F_{m,1} \cup \cdots \cup F_{m,m} \cup G_{m,m}.
$$
Note that $F_{m,0} \subset \{\ul{\tau}^{2^m} < \infty\}$, $G_{m,m} \subset
\{\Vert \ul{\tau}^1 \wedge \ul{\tau}_{(N)} \Vert \geq 2^m\}$, and for $j =
1,\dots,m$,
$$
   F_{m,j} \subset \{\ul{\tau}^{2^{m-j}} < \infty,\ \Vert \ul{\tau}^{2^{m-j+1}}
       \wedge \ul{\tau}_{(N)} - \ul{\tau}^{2^{m-j}} \Vert \geq 2^{2m-j}\}.
$$
Therefore, by Theorem \ref{prop1}, Lemma \ref{rdlem10} and the Markov property
at $\ultau^{2^{m-j}}$,
\begin{equation}\label{rd35a}
   P_1(F) \leq (2^m)^{-\lambda_1} + \sum_{j=1}^m (2^{m-j})^{-\lambda_1}
         C \exp(-c 2^j) + C \exp(-c2^m).
\end{equation}
Clearly, there is $K < \infty$, not depending on $m$, such that the right-hand
side is $\leq K (2^m)^{-\lambda_1}$. This completes the proof of Theorem
\ref{thm3}.
\hfill $\Box$
\vskip 16pt

   The next result is the key ingredient to our proof of the upper bound on the Hausdorff dimension of boundaries of bubbles.

\begin{prop} Fix $q \in \IR$. For $t \in [2,3]^2$ and $\ep >0$, set $F(t,\ep) = F_1(t,\ep) \cap F_2(t,\ep)$, where $F_1(t,\ep) = \{ |\ti X(t) - q| \leq 2\ep\}$ and $F_2(t,\ep) =
\{\cR_t(\ultau^{2\ep,t}_{(N)}) \not\subset \cR_t(\half)\} \cup \{\ultau^{2\ep,t,1} \leq
\ulsigma^{2\ep,t,\cR_t(\half)}\}$ (this is the event ``the DW-algorithm started
at $t$ with value $2\ep$ escapes $\cR_t(\half)$ or reaches level $1$ within 
this rectangle"). There is $C < \infty$ such that for all $t \in [2,3]^2$ and all
sufficiently small $\ep >0$, 
$$
   P(F(t,\ep)) \leq C\, \ep^{1+\lambda_1}.
$$
\label{univariateup}
\end{prop}

\proof Let $\G = \F^{(5/2,5/2)}_{(1,1,1,1)}$. Note that 
$F_2(t,\ep) \in \F^t_{(\half,\half,\half,\half)} \subset \G$, and $\ti X(t) = \ti X(\frac{3}{2},\frac{3}{2}) +Y$, 
where $Y = \ti X(t) - \ti X(\frac{5}{2},\frac{5}{2}) - (\ti X(\frac{3}{2},\frac{3}{2}) - \ti X(\frac{5}{2},\frac{5}{2}))$ is $\G$-measurable. Since $\ti X(\frac{3}{2},\frac{3}{2})$ 
is independent of $\G$,
\begin{eqnarray*}
   P(F_1(t,\ep) \cap F_2(t,\ep)) &\leq& 
      E(1_{F_2(t,\ep)} P\{|Y + \ti X(\frac{3}{2},\frac{3}{2}) - q| \leq 2\ep \mid \G\}) \\
  &\leq& P(F_2(t,\ep)) \sup_{y \in \IR} P\{|y + \ti X(\frac{3}{2},\frac{3}{2})| \leq 2\ep\}.
\end{eqnarray*}
Using the fact that $\ti X(\frac{3}{2},\frac{3}{2})$ is $N(0,3)$, we conclude from Theorems \ref{prop1} and \ref{thm3}  that the last right-hand side is $\leq C\ep^{\lambda_1}\ep = C \ep^{1+\lambda_1}$, where $C$ does not depend on
$t$ or $\ep$.
\hfill $\Box$
\vskip 16pt

   We conclude this section with a property that will be needed in Section \ref{sec8}.

\begin{prop} There exist $K < \infty$ and $c > 0$ such that for $r=(0,0)$, $x \geq 1$ and $y \geq 1,$ 
\begin{equation}\label{bs_ubeq4}
   P_1\{{\cal{R}}(\underline{\tau}^{1,r,y} \wedge \underline{\tau}^{1,r}_{(N)}) \not\subset [-x^2, x^2]^2\} \leq K \ x^{-\lambda_1} e^{-cx/y} 
\end{equation}
(note that the event on the left-hand side is ``the DW-algorithm started at $r$ with value $1$ escapes $[r_1 -x^2, r_1 + x^2] \times [r_2 -x^2, r_2 + x^2]$ before reaching level $y$'').
\label{bs_ubprop6}
\end{prop}

\proof If $y \geq x \geq 1,$ then by Theorem \ref{thm3},
\begin{eqnarray*}
   P_1\{{\cal{R}}(\underline{\tau}^{1,r,y} \wedge \underline{\tau}^{1,r}_{(N)}) 
       \not\subset [-x^2,x^2]\} 
   &\leq& P_1\{{\cal{R}}(\underline{\tau}^{1,r}_{(N)}) \not\subset [-x^2, x^2]\}\\
   &\leq& K \ x^{-\lambda_1}\\
   &\leq& K \ e^c x^{-\lambda_1} e^{-cx/y},
\end{eqnarray*}
so it suffices to consider the case $x > y \geq 1.$ In this case, by the scaling property of Brownian motion, the probability on the left-hand side of $(\ref{bs_ubeq4})$ is equal to
\begin{eqnarray*}
  \lefteqn{
   P_{x/y} \{ {\cal{R}}(\underline{\tau}^{x/y,r,x} \wedge \underline{\tau}^{x/y,r}_{(N)})
      \not\subset [-(x/y)^2 x^2, (x/y)^2 x^2]^2\}
      }\\
   &\leq& P_{x/y} \{{\cal{R}}(\underline{\tau}^{x/y,r,x} \wedge
    \underline{\tau}^{x/y,r}_{(N)} ) \not\subset [-(kx)^2, (kx)^2]^2\},
\end{eqnarray*}
where $k=[x/y] \geq 1$. For $\ell = 1, \ldots, k,$ set $\underline{\sigma}_\ell = \underline{\sigma}^{x/y,r, {\cal{R}}_r((\ell x)^2)}$. On $\{\underline{\tau}^{x/y,r,x} \wedge \underline{\tau}^{x/y,r}_{(N)} \geq \underline{\sigma}_{k-1}\},$
$$
   P_{x/y} \{{\cal{R}}(\underline{\tau}^{x/y,r,x} \wedge \underline{\tau}_{(N)}^{x/y,r})
    \not\subset {\cal{R}}((kx)^2) \mid {\cal{F}}_{\underline{\sigma}_{k-1}}\}
$$
is bounded above by the probability that the DW-algorithm started at $r$ with some value in $[0,x]$ escapes ${\cal{R}}((kx)^2) \setminus {\cal{R}}((k-1)^2x^2)$ before hitting level $x$. This is bounded above by
$$
   c_0 = \sup_{z \in [0,x]} P_z\{\cR(\ul{\tau}^{z,r,x}) \not\subset [-x^2,x^2]^2 \}.
$$
The probability on the right-hand side does not depend on $x$, so we put $x=1$ here, and it is a continuous function of $z$, which attains its maximum at some $z_0 \in [0,1]$. Therefore,
$$
   c_0 = P_{z_0}\{\cR(\ul{\tau}^{z_0,r,1}) \not\subset [-1,1]^2 \} < 1.
$$
Repeating this argument for $\ell = k-2, \ldots, 2,$ we see that the left-hand side of $(\ref{bs_ubeq4})$ is bounded above by
$$
   P_{x/y} \{{\cal{R}}(\underline{\tau}^{x/y,r,x} \wedge \underline{\tau}^{x/y,r}_{(N)})
    \not\subset [-x^2, x^2]^2\} \cdot c_0^{k-2}  
   \leq P_{x/y}\{{\cal{R}}(\underline{\tau}^{x/y,r}_{(N)}) \not\subset [-x^2,x^2]\} c_0^{k-2}.
$$
By Theorem \ref{thm3}, this is
$$
   \leq C \left(\frac{x/y}{x}\right)^{\lambda_1} c_0^{k-2} 
     \leq C \ x^{-\lambda_1} (x/y)^{\lambda_1} e^{-\tilde{c}[x/y]},
$$
for an appropriate of $\tilde{c} > 0.$ By replacing $\tilde{c}$ by a smaller constant $c > 0$ and increasing the constant $C$, this is bounded above by $K x^{-\lambda_1} e^{-c x/y}$ for some universal constant $K$.
\hfill $\Box$
\vskip 16pt

\end{section}
\eject

\begin{section}{ABM: Upper bound on the Hausdorff dimension}\label{sec4}
In this section, we shall use Proposition \ref{univariateup} to derive part of Theorem \ref{thm2} (the upper 
bound). 

\begin{prop} Fix $q \in \IR$. A.s., the Hausdorff dimension of the boundary of any $q$-bubble of the standard ABM $\ti X$ is $\leq (3 - \lambda_1)/2$.
\label{thm4}
\end{prop}

\proof By elementary scaling considerations, it will suffice to show 
that $\dim(H_1 \cap [1, \infty )^2 ) \leq (3 - \lambda_1)/2$, where $H_v$ is the set of points which are in the boundary of some upwards $q$-bubble of diameter at least $v$. 

\begin{remark} For $t \in H_v$, there is not necessarily a curve $\Gamma$ with one extremity at $t$ and the other about one unit away from $t$ with $\ti X(s) > q$ for $s \in \Gamma \setminus \{t\}$ (the reader can easily construct simple planar domains with this property). However, $t \in H_v$ if and only if, for all $k\geq 1$, there is $t_k$ such that $\vert t_k - t \vert < 1/k$ and a curve $\Gamma_k$ with one extremity at $t_k$ and the other $1-2/k$ units away from $t$, with $\ti X(s) > q$ for all $s \in \Gamma$.
\label{rem4.2}
\end{remark}

   For $r \in \ZZ^2$, set $D_r = [r_1, r_1+1] \times [r_2, r_2 + 1]$.  Since the Hausdorff dimension of $H_1 \cap [1, \infty )^2$ is equal
to $\sup_{r \geq (1,1)} \dim(H_1 \cap D_r)$, it will be sufficient to show that for all $r \geq (1,1)$, $\dim(H_1 \cap D_{r}) \leq (3- \lambda_1)/2$.  We treat the case $r = (1,1) $ explicitly and leave it to the reader to check that the argument carries over to all other $r$.

  It will suffice to show that for any $\alpha > (3- \lambda_1)/2$,
$\dim(H_1 \cap D_{(1,1)}) \leq \alpha$ a.s.  Fix such an $\alpha$.

   The Hausdorff dimension of a set $E$ is bounded above by $\alpha$ if, for 
every $\delta > 0$, there is a covering $\{I_i,\ i \geq 1\}$ of $E$ by
squares $I_i$ such that $|I_i| \leq \delta$
and $\sum_i |I_i|^\alpha < \delta $, where $|I|$ denotes the diameter of set 
$I$ (see e.g. \cite{F}).  Thus, by Fatou's lemma, it suffices to find a sequence of
random coverings $\{I^n_i,\ i \geq 1\}$ of $H_1 \cap D_{(1,1)}$ such that a.s.,

\indent (a) $\sup _{i} |I^n_i| \ \rightarrow \ 0 $ a.s. as $n \to \infty$,

\indent (b) $E \left( \sum _i |I^n_i|^ \alpha  \right) \ \rightarrow \ 0$ as 
$n \to \infty$.

\noindent For this, we divide the square $D_{(1,1)}=[1,2]^2$ into the union of the squares
\begin{equation}\label{rdsquares}
   D^n_{i,j} = [1+ i2^{-2n}, 1 + (i+1)2^{-2n}] \times 
    [1+ j2^{-2n}, 1 + (j+1)2^{-2n}],
\end{equation}
with $0 \leq i,j\leq 2^{2n}-1$. We will simply take as 
random covering the collection of $D^n_{i,j}$ which intersect the set
$H_1$. Note that for this collection of random coverings,
condition (a) above is automatically satisfied, 
and 
\begin{eqnarray*} 
  E \left( \sum_{i,j} |D^n_{i,j}|^\alpha 1_{\{D^n_{i,j} \cap H_1 \ne \emptyset\}} \right) &\leq&
   \sum_{i,j=0}^{2^{2n}-1} (\sqrt{2}\, 2^{-2n})^{\alpha} 
           P\{D^n_{i,j} \cap H_1 \ne \emptyset\} \\
    &\leq& C2^{4n} 2^{-2n \alpha } \sup_{i,j} 
                P\{D^n_{i,j} \cap H_1  \ne \emptyset\}.
\end{eqnarray*}
Therefore, to show (b) above, it suffices to show that 
\begin{equation}\label{aeq1}
    \lim_{n\to \infty} 2^{(4-2 \alpha)n} \sup_{i,j} P\{D^n_{i,j} \cap
         H_1  \ne \emptyset\} = 0.
\end{equation}

   To this end, for given $n,i,j$, set $r = r_{i,j,n} = (1+ i2^{-2n} , 1+ j2^{-2n})$, $X^r(s) = \ti X(s) - \ti X(r_{i,j,n})$, and
$$
   F_n(i,j) = \left\{\sup_{s \in D^n_{i,j}} |X^r(s)| < n 2^{-n} \right\}.
$$
By the reflection principle applied to the individual Brownian motion components of $X^r$ \cite{IM}, we easily see that $P\{F_n(i,j)^c\} \leq 2 e^{-n^2/8}$ for $n$ sufficiently large. On the other hand, the key
observation is that
\begin{equation}\label{key1}
   \{D^n_{i,j} \cap H_1  \ne \emptyset\} \cap F_n(i,j) \subset
   \{ |\ti X(r_{i,j,n}) - q| < n2^{-n} \} \cap B_n(i,j),
\end{equation}
where $B_n(i,j)$ is the event ``the DW-algorithm started at $r_{i,j,n}$ with value $x_0 = 2n2^{-n}$ escapes the square $\cR_{r_{i,j,n}}(\half)$." Indeed, on $\{D^n_{i,j} \cap H_1  \ne \emptyset\} \cap F_n(i,j)$, there is $t \in D^n_{i,j} \cap H_1$, so by the definitions of $H_1$ and $F_n(i,j)$, $|\ti X(r_{i,j,n}) - q| < n2^{-n}$ and there is a path $\Gamma$, with one extremity at $t$ and the other at least one-half unit away from $t$, along which $\ti X > q - n 2^{-n}$ (by Remark \ref{rem4.2}, we cannot guarantee that $\ti X > q$ along $\Gamma$). Therefore, on $\{D^n_{i,j} \cap H_1  \ne \emptyset\} \cap F_n(i,j)$, for $s$ on the segment with extremities $r_{i,j,n}$ and $t$, $X^r(s) > -n 2^{-n}$;  for $s \in \Gamma$, $X^r(s) = \ti X(s) - \ti X(r) > q - n 2^{-n} - (q + n 2^{-n}) = -2 n 2^{-n}$. It follows that
$B_n(i,j)$ occurs by Proposition \ref{prop2}(c).

   By (\ref{key1}), we see that
$$
   P\{D^n_{i,j} \cap H_1  \ne \emptyset\} \leq P(\{ |\ti X(r_{i,j,n}) -q| < n2^{-n} \} \cap B_n(i,j)) + P( F_n(i,j)^c)
$$
and hence for $n$ large, by Proposition \ref{univariateup}, we obtain
$$
   2^{(4-2\alpha)n} \sup_{i,j} P\{D^n_{i,j} \cap H_1  \ne \emptyset\} 
    \leq 2^{(4-2\alpha)n} \left(C(n2^{-n})^{1+\lambda_1}
     +  2 e^{-n^2/8} \right).
$$
Because $\alpha > (3-\lambda_1)/2$, this proves (\ref{aeq1}) and completes the
proof of Proposition \ref{thm4}. 
\hfill $ \Box $

\end{section}
\eject

\begin{section}{ABM: Conditional and bivariate escape probabilities, upper bounds}\label{sec5}

   In order to prove the remaining part of Theorem \ref{thm2} (the lower bound), we will have to consider two DW-algorithms run simultaneously from two starting points.  This demands an understanding of the algorithm conditional on some information about the Brownian motions' behaviors elsewhere. The following result, which is the first objective of this section, is a step in this direction. It will lead to Proposition \ref{lembivariate}, which is an essential upper bound needed in the proof of Theorem \ref{thm2}. 
   
   We will use the notation $X(t) = X^{(0,0)}(t) = \ti X(t)$.
   
\begin{prop} For $s_1 < s_2$, set $V = \sup_{s_1 < u < s_2}
(X(u, 0) - X(s_1, 0))$. Consider the DW-algorithm started at $(0,0)$ with value $1$. There is $C < \infty$ such that for all 
$\ell \geq 1$ and $1 < s_1 < s_2 < 2s_1$,
\begin{eqnarray*}
   &&P_1(\{ \ultau^{2^\ell} < \infty\} \cup \{ \cR(\ultau_{(N)}) \not\subset 
     [-2^{2\ell}, 2^{2\ell}]^2\} \mid X(u,0)-X(s_1,0),\ s_1 \leq u \leq s_2) \\
     &&\qquad \qquad \leq C\,2^{-\ell \lambda_1} 
	   \left(1 + \frac{V}{\sqrt{s_1}}\right)^{\lambda_1}.
\end{eqnarray*}
\label{prop1lem16}
\end{prop}

\begin{remark} If $s_1 $ and $s_2$ are random (but nonetheless the condition $1 < s_1 < s_2 < 2s_1$ holds a.s.) but such that $(X(s_1-u, 0) - X(s_1,0),\ u \geq 0)$ and $(X(s_2+u, 0) - X(s_2,0),\ u \geq 0)$ are independent of $ \sigma(X(s_1,0)-X(u,0),\ 
s_1 \leq u \leq s_2)$, then naturally, the conclusion of Proposition \ref{prop1lem16} will still hold.  Also, if the conditioning $\sigma$-field is augmented by independent information, then the bound remains valid.
\label{remprop1lem16}
\end{remark}

   In order to establish this proposition, we need the following four lemmas.
   
\begin{lemma} Let $\ulT$ be a stopping point relative to $(\F_{\ulu})$ and set 
$$
   Y = x_0 + \max_{s \in \cR(\ulT)} \vert X(s)\vert.
$$ 
Let $X^{\ulT}$ be the standard ABM associated with $\ulT$ as below (\ref{startp6}) (here, $r = (0,0)$ so it is omitted from the notation). For $a >0$, let $F^{\ulT}(a)$ be the event ``the DW-algorithm for $X^{\ulT}$, started at $(0,0)$ with value $Y$, reaches level $a$." Then
$$
   \{Y < a\} \cap \{ \ultau^{x_0,a} < \ulinfty \} \subset F^{\ulT}(a)
$$
(recall that $\{ \ultau^{x_0,a} < \ulinfty \}$ is the event ``the DW-algorithm for $X$, started at $(0,0)$ with value $x_0$, reaches level $a$).
\label{lem5.2prime}
\end{lemma} 

\proof Let $R^{\ulT} = [U_1, V_1] \times[U_2, V_2]$ be 
the rectangle explored by the DW-algorithm for $X^{\ulT}$ (started at $(0,0)$ with value $Y$) up to termination. If this algorithm 
does not reach level $a$, then
\begin{equation}\label{abmXT}
   \sup_{s \in R^{\ulT}} (Y + X^{\ulT}(s_1, s_2)) < a \qquad\mbox{ and
   }\qquad Y + X^{\ulT} \leq 0 \mbox{ on } \partial R^{\ulT}.
\end{equation}
But this implies that $x_0 + X < a$ on $[-\ulT_3 + U_1, \ulT_1 + V_1] 
\times [-\ulT_4 + U_2, \ulT_2 + V_2]$ and $x_0 + X  \leq 0$ on the boundary of this rectangle. Indeed, we check the first inequality, since the second is checked analogously.
Consider without loss of generality the subrectangle $[0, \ulT_1 + V_1] \times [0,\ulT_2 + V_2]$.
This rectangle itself subdivides naturally into four subrectangles: $[0, \ulT_1 ] \times [0,\ulT_2 ]$,
 $[\ulT_1, \ulT_1  + V_1] \times [0,\ulT_2]$, $[0, \ulT_1 ] \times [\ulT_2 ,\ulT_2 + V_2]$ and 
 $[\ulT_1, \ulT_1 + V_1] \times [\ulT_2,\ulT_2 + V_2]$.  On the first subrectangle, $x_0 + X < a$ simply because the first constraint in \eqref{abmXT} implies that $Y<a$.  The cases of the second and third subrectangles are essentially the same so we just consider the second subrectangle: here for $(s,t) \in [\ulT_1, \ulT_1  + V_1] \times [0,\ulT_2]$, by \eqref{abmXT},
$$
   x_0 + X(s,t) = X^{\underline{T}}(s-\underline T_1, 0) +x_0 + X(\underline T_1,t) < (a -Y)+Y = a.
$$  
For the fourth subrectangle $(s,t) \in [\ulT_1, \ulT_1  + V_1] \times [\ulT_2,\ulT_2+V_2]$, we have similarly
$$
   x_0 + X(s,t) = X^{\underline{T}}(s-\underline T_1, t- \underline T_2) +x_0 + X(\underline T_1,\underline T_2) < (a -Y)+Y = a.
$$
This proves the lemma.
\hfill $\Box$
\vskip 16pt

\begin{lemma} Let $\ulT$ be a stopping point relative to 
$(\F_{\ulu})$ and let $Y$ be as in Lemma \ref{lem5.2prime}, with $x_0 = 1$.
There is 
$C < \infty$ such that for all $\ell \geq 0$, 
\begin{equation}\label{prop16lem2eq1}
   P_1(\ultau^{2^\ell} < \ulinfty \mid \F_{\ulT}) 
      \leq C\, (Y 2^{-\ell})^{\lambda_1},
\end{equation}
and for $\ell \geq 1$, on $\{\vert \ulT \vert \leq 2^{2(\ell-1)}\}$,
\begin{equation}\label{prop16lem2eq2}
   P_1(\cR(\ultau_{(N)}) \not\subset [-2^{2\ell},2^{2\ell}]^2 \mid {\F}_{\ulT})
    \leq C\, (Y 2^{-\ell})^{\lambda_1}
\end{equation}
(here, $\ultau^{2^\ell}$ and $\ultau_{(N)}$ are associated with the DW-algorithm for $X$ started at $(0,0)$ with value $1$).
\label{lem2lem16}
\end{lemma}

\proof We first prove (\ref{prop16lem2eq1}). On $\{Y \geq 2^\ell\}$, the inequality (\ref{prop16lem2eq1}) is satisfied with $C = 1$, so we focus on the event $\{Y < 2^\ell\}$. 

   Let $X^{\ulT}$ be the standard ABM defined in Lemma \ref{lem5.2prime}. Apply the DW-algorithm started at $(0,0)$ with value $Y$ to the ABM $X^{\ulT}$, and let $F^{\ulT}$ be the event that this 
algorithm reaches level $2^\ell$. By Theorem \ref{prop1}, $P_Y(F^{\ulT}) \leq 
C(Y 2^{-\ell})^{\lambda_1}$. 

   By Lemma \ref{lem5.2prime}, on $\{Y <2^\ell\}$,
$$
   P_1(\ultau^{2^\ell} < \infty \mid \F_{\ulT}) 
     \leq P_Y (F^{\ulT}) \leq C(Y 2^{-\ell})^{\lambda_1}.
$$
This proves \eqref{prop16lem2eq1}.

   We now check that on $\{\vert \ulT \vert \leq 2^{2(\ell-1)}\}$, \eqref{prop16lem2eq2} holds.
For this, we proceed as above. On $\{Y \geq 2^\ell\}$, it suffices to set $C = 1$. On $\{Y < 2^\ell\}$, we again apply the DW-algorithm started at (0,0) with 
value $Y$ to $X^{\ulT}$, and let $\tilde{F}^{\ulT}$ be the event that this 
algorithm escapes $[-(2^{2 \ell} - \vert \ulT \vert), 2^{2\ell} 
- \vert \ulT \vert]^2$. Then 
$$
   P_1(\cR(\ultau_{(N)}) \not\subset [-2^{2\ell}, 2^{2\ell}]^2 
      \mid {\F}_{\ulT}) \leq P_Y(\ti F^T).
$$
By Theorem \ref{thm3}, on $\{Y < 2^\ell\} \cap \{ \vert T \vert \leq 2^{2(\ell-1)}\} $, 
$$
   P_Y(\tilde{F}^{\underline{T}}) \leq C(Y(2^{2\ell}- \vert T \vert)^{-1/2})^{\lambda_1} \leq \tilde{C} (Y 2^{-\ell})^{\lambda_1}.
$$
This proves \eqref{prop16lem2eq2}.
\hfill $\Box$
\vskip 16pt

\begin{lemma} For $a > 1$, for the DW-algorithm started at $(0,0)$ with value $1$, 
$$
   \sup_{t \in \cR(\ultau^a \wedge \ultau_{(N)})} (-X(t)) \leq 1+a.
$$
\label{lem5.3prime}
\end{lemma}

\proof Let $\ulT = \ultau^a \wedge \ultau_{(N)}$ and assume that $\ulT$ occurs during an odd stage $2m-1$ in the interval $[U_{m-1}^\prime,U_m^\prime]$. We distinguish two cases.
\vskip 12pt

\noindent{\em Case 1.} $\ultau^a < \ulinfty$. In this case, $1+\tilde Z_1(\ulT_1) - \tilde Z_2(T_2^{m-1}) = a$. By definition of $\ultau^a$, for all $u \in [- \underline T _3, \underline T_1]$,
$$
   1+ \tilde Z_1(u) - \tilde Z_2(T_2^{m-1}) \geq 0,
$$ 
that is,
\begin{equation}\label{e5.4}
   \tilde Z_1(u) \geq \tilde Z_2(T_2^{m-1})  -1 .
\end{equation}
Similarly, for all $v \in [- \underline T _4, \underline T_2]$,
$$
    1+ \tilde Z_1(T^{m-1}_1) - \tilde Z_2(v) \geq 0,
$$
and since we are in Case 1, $\tilde Z_1(\ulT_1) \geq \tilde Z_1(T^{m-1}_1)$, so 
\begin{equation}\label{e5.6}
   1+\tilde Z_1(\ulT_1) \geq \tilde Z_2(v).
\end{equation}
Therefore, for such $u,v$, by \eqref{e5.6}, \eqref{e5.4} and since we are in Case 1,
$$
   -(\tilde Z_1(u) - \tilde Z_2(v)) = \tilde Z_2(v) - \tilde Z_1(u) \leq  1+\tilde Z_1(\ulT_1) + 1 - \tilde Z_2(T_2^{m-1}) = 1+a.
$$
\vskip 12pt

\noindent{\em Case 2.} $\ultau^a = \ulinfty$. In this case,
\begin{equation}\label{e5.7}
    1+ \tilde Z_1(T^{m-1}_1) - \tilde Z_2(T_2^{m-1})  < a,
\end{equation}
\eqref{e5.4} still holds, and we still have
\begin{equation}\label{e5.8}
   1+ \tilde Z_1(T^{m-1}_1) \geq \ti Z_2(v),
\end{equation}
so by \eqref{e5.8}, \eqref{e5.4} and \eqref{e5.7},
$$
   \tilde Z_2(v) - \tilde Z_1(u) \leq 1 + \tilde Z_1(T^{m-1}_1) + 1 - \tilde Z_2(T_2^{m-1}) < 1+a.
$$
The lemma is proved.
\hfill $\Box$
\vskip 16pt 

\begin{lemma} Consider the DW-algorithm started at $(0,0)$ with value 1. There are $c > 0$ and $C < \infty$ such that, for all $s_1 > 1$ 
and $x > 0$,
$$
   P_1(1+ \max_{t \in [-s_1, s_1]^2} \vert X(t) \vert \geq x \sqrt{s_1} 
      \mid \cR(\ultau_{(N)}) \not\subset [-s_1, s_1]^2) \leq C e^{-cx}.
$$
\label{lem3lem16}
\end{lemma}

\proof Set $\ulT = \ultau^{\sqrt{s_1}} \wedge \ulsigma^{[-s_1,s_1]^2}$. Then $\ulT$ is a stopping point relative to $(\F_{\ulu})$. For 
$i = 1, \ldots, 4$, let $(B_i^{\ulT})$ be the four Brownian motions defined in 
(\ref{startp6}). Then
$$
   \sup_{t \in [-s_1,s_1]^2} \vert X(t)\vert 
   \leq \sup_{t \in \cR(\ulT)} \vert X(t) \vert +M,
$$
where
$$
   M = \sum_{i=1}^4 \ \sup_{0 \leq u \leq s_1} \vert 
    B_i^{\ulT}(u) - B_i^{\ulT}(0) \vert.
$$
By construction,
\begin{equation}\label{supX1}
   1+ \sup_{t \in \cR(\ulT)} X(t) \leq \sqrt{s_1},
\end{equation}
and by Lemma \ref{lem5.3prime},
\begin{equation}\label{supX2}
   \sup_{t \in \cR(\ulT)} (- X(t)) \leq 1+ \sqrt{s_1}.
\end{equation}

   By (\ref{supX1}) and (\ref{supX2}), $\sup_{t \in \cR(\ulT)} \vert X (t) \vert \leq 1+ \sqrt{s_1}$. It follows that
\begin{eqnarray*}
  && P_1\left\{1+\sup_{t \in [-s_1, s_1]^2} \vert X(t) \vert \geq x \sqrt{s_1}, 
  \ \cR(\ultau_{(N)}) \not\subset [-s_1, s_1]^2 \right\}\\
  &&\qquad \leq P_1\left\{1+ \sup_{t \in[-s_1, s_1]^2} \vert X(t)\vert 
      \geq x \sqrt{s_1},
    \  \cR(\ultau^{\sqrt{s_1}}) \subset [-s_1, s_1]^2\right\}\\
  &&\qquad \qquad + P_1\left\{1+ \sup_{t \in [-s_1, s_1]^2} \vert X(t) \vert 
    \geq x \sqrt{s_1},\ \ulT = \ulsigma^{[-s_1, s_1]^2},\ \cR(\ultau_{(N)}) 
    \not\subset [s_1, s_1]^2\right\}\\
  &&\qquad \leq P_1(\{M \geq (x-1) \sqrt{s_1}\} \cap \{\cR(\ultau^{\sqrt{s_1}})
         \subset [-s_1, s_1]^2\})\\
  &&\qquad\qquad + P_1(\{M \geq (x-1) \sqrt{s_1}\} \cap \{\ulT = 
    \ulsigma^{[-s_1, s_1]^2},\ \cR(\ultau_{(N)}) \not\subset [s_1, s_1]^2\}).
\end{eqnarray*}
In each term, the event $\{M \geq (x-1) \sqrt{s_1}\}$ is independent of the 
other event (which belongs to ${\cal{F}}_{\ulT})$, so this is
$$
   \leq P_1\{M \geq (x-1) \sqrt{s_1}\}\, (P\{\ultau^{\sqrt{s_1}} < \infty \} 
     + P\{\cR(\ultau_{(N)}) \not\subset [-s_1,s_1]^2\}).
$$
Standard bounds for Brownian motion show that for $x > 0$, the first factor is 
bounded by $ C e^{-cx}$ for universal $c$ and $C$, and therefore, the conditional probability in the statement 
of the lemma is
$$
   \leq C e^{-cx} \left(\frac{P_1\{\ultau^{\sqrt{s_1}} < \infty \}}
       {P_1\{\cR(\ultau_{(N)}) \not\subset [-s_1, s_1]^2\}} + 1\right).
$$
The conclusion now follows from Theorems \ref{prop1} and \ref{thm3}. 
\hfill $\Box$
\vskip 16pt

\noindent{\sc Proof of Proposition \ref{prop1lem16}.} Set $\G = \sigma(X(u,0)-X(s_1,0),\ 
s_1 \leq u \leq s_2)$. 
Let $A \ = \{\ultau^{2^\ell} < \ulinfty \} \cup \{\cR(\ultau_{(N)}) \not\subset [-2^{2\ell}, 2^{2\ell}]^2\}$. We distinguish two cases.
\vskip 12pt

  {\em Case 1}: $s_1 \geq 2^{2(\ell-1)}$. In this case, we shall get a better bound, which does not involve $V$:
\begin{eqnarray*}
   P_1(A \mid \G) &\leq& P_1(\cR(\ultau_{(N)}) 
       \not\subset [-2^{2(\ell-1)}, 2^{2(\ell-1)}]^2 \mid \G) \\
    &&+ P_1(\ultau^{2^\ell} < \ulinfty,\ \cR(\ultau^{2^\ell}) \subset 
      [-2^{2(\ell-1)}, 2^{2(\ell-1)}]^2 \mid \G).
\end{eqnarray*}
By the independence of increments of Brownian motion, the two conditional 
probabilities on the right-hand side are respectively equal to the 
unconditional probability of the same event. By Theorem \ref{thm3}, the first 
term is therefore bounded by $C 2^{-\ell \lambda_1}$, and by Theorem 
\ref{prop1}, the second term is bounded by $P_1\{\ultau^{2^\ell} < \ulinfty\} 
\leq C 2^{-\ell \lambda_1}$, which yields the desired upper bound.
\vskip 12pt

{\em Case 2}: $s_1 < 2^{2(\ell-1)}$. Set $F(s_1) = \{\cR(\ultau_{(N)}) 
\not\subset [-s_1, s_1]^2\}$. Then $F(s_1)$ and $\G$ are independent, and by 
Theorem \ref{thm3}, $P_1(F(s_1)) \leq C s_1^{-\lambda_1/2}$.

   Set
$$
   Y = \sup_{t \in[-s_1,s_1]^2} X(t) \qquad\mbox{ and }\qquad 
   Z = \sup_{t \in [-s_1,s_2] \times [-s_1,s_1]} X(t)
$$
(note the $s_2$ in the definition of $Z$), so that $Z \leq V + Y$. Let 
$\ulT = (s_2, s_1, s_1, s_1)$. Then
$$
   P_1(A \vert \G) = P_1(A \cap F(s_1) \vert \G) 
     + P_1(A \cap F(s_1)^c \vert \G).
$$
Notice that $A \cap F(s_1)^c$ is independent of $\G$, so the second term is equal to $P_1(A \cap F(s_1)^c) \leq P_1\{\ultau^{2^\ell} < \ulinfty \} \leq C 2^{-\ell \lambda_1}$ by Theorem \ref{prop1}. On the other hand,
$$
   P_1(A \cap F(s_1)\mid \G) = E_1(P_1(A \cap F(s_1) \mid {\F}_{\ulT}) 
      \mid \G) = E_1(1_{F(s_1)} P_1(A \mid {\F}_{\ulT}) \mid \G).
$$
By Lemma \ref{lem2lem16}, this is
$$
   \leq C E_1(1_{F(s_1)} ((1+Z) 2^{-\ell})^{\lambda_1} \mid \G).
$$
Because $Z \leq V + Y$, this is
$$
   \leq C E_1(1_{F(s_1)} (1+ V + Y)^{\lambda_1} 2^{-\ell \lambda_1} \mid \G) 
   \leq C 2^{-\ell \lambda_1} \sum_{i=0}^\infty E_1(1_{F(s_1)\cap F_i} 
   ((i+1) \sqrt{s_1} + V)^\lambda \mid \G),
$$
where $F_i = \{i \sqrt{s_1} \leq 1+ Y < (i+1) \sqrt{s_1} \}$. Since $V$ is 
$\G$-measurable, this is equal to
$$
   C 2^{-\ell \lambda_1} \sum_{i=0}^\infty P_1(F(s_1) \cap F_i \mid \G) 
   ((i + 1) \sqrt{s_1} + V)^{\lambda_1}.
$$
Because $F(s_1) \cap F_i$ is independent of $\G$,
$$
   P_1(F(s_1) \cap F_i \mid \G) = P_1(F(s_1) \cap F_i) 
   = P_1(F_i \vert F(s_1))\, P_1(F(s_1)).
$$
By Lemma \ref{lem3lem16} and Theorem \ref{thm3}, this is 
$\leq C e^{-ci} s_1^{-\lambda_1/2}$, and therefore
\begin{eqnarray*}
   P_1(A \mid \G) &\leq& C 2^{-\ell \lambda_1} + C^\prime 2^{-\ell \lambda_1} 
         \sum_{i=0}^\infty e^{-ci} s_1^{-\lambda_1/2} 
         (\sqrt{s_1} + V)^{\lambda_1} (i+1)^{\lambda_1}\\
   &\leq& C^{\prime \prime} 2^{-\ell \lambda_1} 
          \left(1 + \frac{V}{\sqrt{s_1}}\right)^{\lambda_1} 
          \sum_{i=0}^\infty e^{-ci} (i+1)^{\lambda_1}.
\end{eqnarray*}
The series converges, so Proposition \ref{prop1lem16} is proved. 
\hfill $\Box$
\vskip 16pt

   For $n \in \IN,$ let $\ID_n$\index{$\ID_n$} denote the set of elements of $[2, 3]^2$ with dyadic coordinates of order $2n$, and for $1 \leq k \leq \ell \leq n,$ let $\ID_n(k, \ell)$\index{$\ID_n(k, \ell)$} be the set of $(s, t) \in \ID_n \times \ID_n$ such that
$$
  2^{2(k-n-1)} \leq \min(\vert s_1-t_1\vert, \vert s_2-t_2\vert) \leq 2^{2(k-n)}
$$
and
$$
2^{2(\ell-n-1)} \leq \max(\vert s_1-t_1\vert, \vert s_2-t_2\vert) \leq 
  2^{2(\ell-n)}
$$
(the lower bound is replaced by $0$ if $k=1$ or $\ell = 1$). 

\begin{prop} For $t \in [2,3]^2$, let $F(t,\ep) = F_1(t,\ep) \cap
F_2(t,\ep)$ be as defined in Proposition \ref{univariateup}. There is $C < \infty$ such 
that for all large $n \in \IN$, all $1 \leq k \leq \ell \leq n$ and all 
$(s,t) \in \ID(k,\ell)$,
\begin{equation}\label{bivupper}
  P(F(s,2^{-n}) \cap F(t,2^{-n})) \leq C \,2^{-(1+\lambda_1)n -\ell-k\lambda_1}.
\end{equation}
\label{lembivariate}
\end{prop}
 
\proof In the case where $1 \leq k \leq \ell \leq 3$, the trivial inequality
$$
   P(F(s, 2^{-n}) \cap F(t, 2^{-n})) \leq P(F(t, 2^{-n})),
$$
together with the bound from Proposition \ref{univariateup}, is sufficient for (\ref{bivupper}), since 
$$
   2^{-(1 + \lambda_1)n}  \leq 2^{3+3\lambda_1} 
      2^{-(1 + \lambda_1)n-\ell-k\lambda_1}
$$
in this case.

   We consider the case where $3 < k \leq \ell$, since the case $k \leq 3 < \ell$ is actually easier, as we will detail at the end of the proof. We introduce some notation.

   For $s \in [2,3]^2$ and $\ell \leq n$, set
\begin{eqnarray*}
   R(s, n, \ell) &=& \cR_s(2^{2(\ell-n)}/9), \\
   \E(s,n,\ell) &=& \{\cR_s(\ultau_{(N)}^{2^{-n},s}) \not\subset R(s,n,\ell)\} 
      \cup\{\ultau^{2^{-n}, s,2^{\ell-n}} \leq \ulsigma^{2^{-n},s,
      R(s,n,\ell)}\},
\end{eqnarray*}
and let
$$
   \ulT^{s,\ell} = \ultau^{2^{-n}, s, 2^{\ell-n}} \wedge \underline{\sigma}^{2^{-n},s,
    R(s,n,\ell)}.
$$
Note that $\E(s,n,\ell)$ is the event ``the DW-algorithm started at $s$ with 
value $2^{-n}$ escapes $R(s, n, \ell)$ or reaches level $2^{\ell-n}$ within 
this rectangle,'' and $\ulT^{s,\ell}$ represents the portion of $\IR^2$ that the 
algorithm explores up to escaping $R(s, n, \ell)$ or reaching level 
$2^{\ell-n}$.
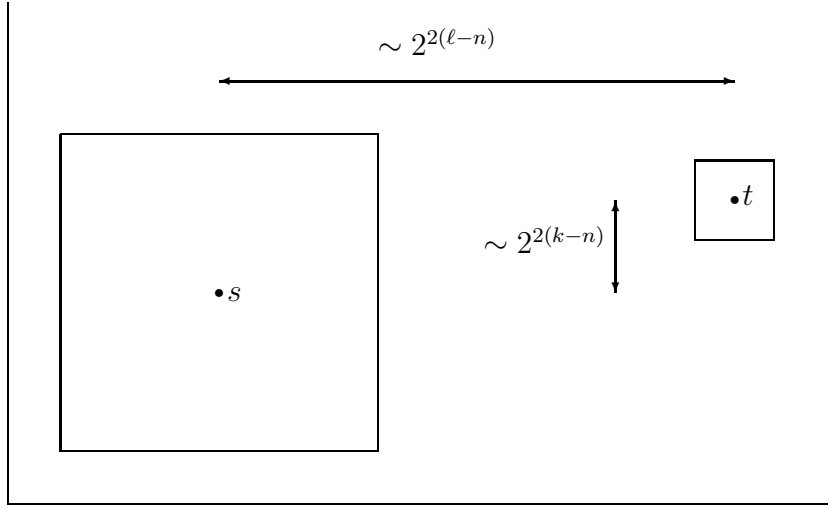
\begin{figure}
\begin{center}
\begin{picture}(320,220)
\put(10,10){\line(0,1){190}}
\put(10,10){\line(1,0){310}}

\put(150,180){$\sim 2^{2(\ell - n)}$}
\put(90,170){\vector(1,0){195}}
\put(90,170){\vector(-1,0){0}}

\put(90,90){\circle*{3}}
\put(93,87){$s$}

\put(30,30){\line(1,0){120}}
\put(30,30){\line(0,1){120}}
\put(30,150){\line(1,0){120}}
\put(150,30){\line(0,1){120}}

\put(285,125){\circle*{3}}
\put(288,123){$t$}

\put(270,110){\line(1,0){30}}
\put(270,110){\line(0,1){30}}
\put(270,140){\line(1,0){30}}
\put(300,110){\line(0,1){30}}

\put(240,90){\vector(0,1){35}}
\put(240,90){\vector(0,-1){0}}

\put(190,105){$\sim 2^{2(k - n)}$}

\end{picture}
\end{center}
\caption{The relative positions of $s$ and $t$. \label{figE1}} 
\end{figure}

   Now fix $(s,t) \in \ID_n(k, \ell)$ (see Figure \ref{figE1}). We assume without loss of generality that $s \leq t$ (for the Brownian sheet, we would have to treat separately this case and the case where neither $s \leq t$ nor $t \leq s$, but for ABM, this distinction is not necessary), and that $t_2-s_2 \leq t_1-s_1$. Notice that $t_1-\ulT_3^{t,k}-(s_1 + \ulT_1^{s,\ell}) >0$, and set
$$
   \zeta_1 = \sup_{0 \leq u_1 \leq t_1-\ulT_3^{t,k}-(s_1 
      + \ulT_1^{s,\ell})} (\tilde{Z}_1(s_1 + \ulT_1^{s,\ell} + u_1) - \tilde{Z}_1
         (s_1 + \ulT_1^{s,\ell})).
$$
We note that $\zeta_1$ is conditionally independent of 
${\cal{F}}_{\ulT^{s,\ell}}^s \vee {\F}^t_{\ulT^{t,k}}$ given $\ulT^{s,\ell}$ and $\ulT^{t,k}$, and 
is stochastically dominated by 
$\sup_{0\leq u_1 \leq 2^{2(\ell-n)}}(\tilde{Z}_1(s_1+u_1)-\tilde{Z}_1(s_1))$. 
For $m \in \IN$, define
$$
   L(m) = \{m2^{\ell-n} \leq \zeta_1 < (m+1)2^{\ell-n}\}.  
$$
We note that the event $\{s_2 + \ulT_2^{s,\ell} < t_2 - \ulT_4^{t,k}\}$ has positive 
probability, and we set
$$
   \zeta_2 = \displaystyle\sup_{0 \leq u_2 \leq t_2-\ulT_4^{t,k}-(s_2+ \ulT_2^{s,\ell})} 
     \left(\tilde{Z}_2(s_2+ \ulT_2^{s,\ell} + u_2) - \tilde{Z}_2(s_2 + \ulT_2^{s,\ell})\right)
$$
if $s_2 + \ulT_2^{s,\ell} < t_2-\ulT_4^{t,k}$, and $\zeta_2 = 0$ otherwise.
We note that $\zeta_2$ is conditionally independent of 
$\sigma(\zeta_1) \vee \F^t_{\ulT^{t,k}} \vee {\F}^s_{\ulT^{s,\ell}}$ given $\ulT^{s,\ell}$ and 
$\ulT^{t,k},$ and is stochastically dominated by $\sup_{0 \leq u_2 \leq 2^{2(k-n)}} 
(\tilde{Z}_2(s_2+u) - \tilde{Z}_2(s_2))$.

   Let $\tilde R(s,t,n)$ be the smallest rectangle that contains $\cR_s(\ulT^{s,\ell}) \cup 
\cR_t(\ulT^{t,k})$, and let $\ulS$ be the $\IR_+^4$-valued $(\F^s_{\ulu})$-stopping 
point such that $\cR_s(\ulS) = \tilde R(s,t,n)$. Clearly,
\begin{eqnarray*}
  2^{-n} + \sup_{r \in \cR_s(\ulS)} X^s(r) &\leq& 2^{-n} + \sup_{r \in \cR_s(\ulT^{s,\ell})} X^s(r) + \sup_{r \in \cR_t(\ulT^{t,k})} X^t(r)\\
  &&\qquad\qquad + \sup_{r \in \cR_t(\ulT^{t,k})} (-X^t(r)) + \zeta_1 + \zeta_2 \\
    \\
    & \leq & 2^{-n} + 2^{\ell-n} + 2^{1+k-n} + \zeta_1 + \zeta_2,
\end{eqnarray*}
where we have used Lemma \ref{lem5.3prime}. Let $(B_i^{s, \underline{S}})$,\index{$B_i^{s, \underline{S}}$} $i = 1,2,3,4$, be the Brownian motions defined 
in (\ref{startp6}), with $\ulT$ there replaced by $\ulS$. As in the lines following (\ref{startp6}), use these four Brownian 
motions to form an additive Brownian motion $X^{\ulS}$. Apply the 
DW-algorithm started at (0,0) with value $2^{1-n} + 2^{\ell-n} + 2^{1+k-n} + \zeta_1 
+ \zeta_2$ to $X^{\ulS}$, and let $G(s, t, n)$ be the event that this algorithm reaches level $1$ within the square $[-1,1]^2$.

   The remainder of the proof relies on the following central observation:
\begin{eqnarray}\nonumber
   &&F(s,2^{-n}) \cap F(t,2^{-n}) \\
   &&\qquad\qquad \subset F_1(t,2^{-n}) \cap F_1(s,2^{-n}) \cap 
      G(s,t,n) \cap \E(s,n,\ell) \cap \E(t,n,k). \label{5.3}
\end{eqnarray}
Indeed, when $F(s,2^{-n})\cap F(t,2^{-n})$ occurs, so do $F_1(t,2^{-n})$, 
$F_1(s,2^{-n})$, $\E(s,n,\ell)$ and $\E(t,n,k)$, by the definitions of these 
events, and $G(s,t,n)$ occurs by Lemma \ref{lem5.2prime}.

   The right-hand side of (\ref{5.3}) is equal to
\begin{equation}
   \cup_{m \in \IN} (F_1(t, 2^{-n}) \cap H_1(m)),
\end{equation}
where
$$
   H_1(m) = F_1(s,2^{-n}) \cap G(s,t,n) \cap L(m) \cap \E(s,n,\ell) 
            \cap \E(t,n,k).
$$
Consider the $\sigma$-field
$$
   \tilde{\G}_1 = \sigma(\ti X(s)) \vee {\F}^s_{\ulT^{s,\ell}} \vee {\F}^t_{\ulT^{t,k}} \vee 
   \sigma (\zeta_1) \vee \sigma (X^{\ulS}(u_1, u_2),\ (u_1, u_2) \in \IR^2).
$$
Then $H_1(m) \in \tilde{\G}_1$, and $\ti X(t) = \Delta Z_1 + \hat{X}(t)$, where $\Delta Z_1 = Z_1 (t_1-\ulT_3^{t,k}) - Z_1(s_1+\ulT_1^{s,\ell})$ and $\hat{X}(t)$ is $\tilde{\G}_1$-measurable. Further, $\Delta Z_1$ is conditionally independent of $\tilde{\G}_1$ given $\zeta_1$, $\ulT_1^{s,\ell}$ and $\ulT_3^{t,k}$. Let $g_{z_1, u_1, v_1}$ denote the conditional density of $\Delta Z_1$ given $\zeta_1 = z_1$, $\ulT_1^{s,\ell} = u_1$ and $\ulT_3^{t,k} = v_1$. Then
$$
   P(F_1(t,2^{-n}) \cap H_1(m)) = E(P\{-\hat{X}(t) - 2^{-n+1} 
      \leq \Delta Z_1 \leq \hat{X}(t) + 2^{-n+1} \vert \tilde{\G}_1 \}
         1_{H_1(m)})
$$
and the conditional probability is bounded above by
$$
   2^{-n+2} \sup_{b \in \IR} g_{\zeta_1, {\ulT_2^{s,\ell}}, \ulT_3^{t,k}}(b) 
   \leq 2^{-n+2} \sup_{b \in\IR,\, u_1 \leq 2^{2(\ell-n)},\, v_1 \leq 2^{2(k-n)}} 
        g_{\zeta_1, u_1, v_1}(b).
$$
Notice that $g_{z_1, u_1, v_1}(\cdot)$ is the conditional density of 
$B(t_1-v_1-(s_1+u_1))$ given $B^\ast(t_1-v_1 - (s_1+u_1)) = z_1$, where 
$B(\cdot)$ is a standard Brownian motion. By Lemma \ref{supBM} below and the fact 
that $t_1-v_1-(s_1+u_1) \geq 2^{2(\ell-n)} /3$, we see that
$$
   P(F_1(t, 2^{-n}) \vert \tilde{\G}) \leq \frac{2^{-n+3}}{2^{2(\ell-n)}/3} 
      \zeta_1 \leq 24\,(m+1)\, 2^{-\ell} \qquad\mbox{on } L(m).
$$
In particular,
\begin{equation}
   P(F_1(t,2^{-n}) \cap H_1(m)) \leq 24\, (m+1)\, 2^{-\ell}\, P(H_1(m)). 
\end{equation}

   We now set 
$$
   H_2(m) = G(s,t,n) \cap L(m) \cap \E(s,n,\ell) \cap \E(t,n,k).
$$
Let $\tilde{\G}_2 = {\F}_{(1,1,1,1)}^{(5/2, 5/2)}$. As in the proof of 
Proposition \ref{univariateup}, $H_2(m) \in \tilde{\G}_2$, and 
$\ti X(s) = \ti X(\frac{3}{2}, \frac{3}{2}) + 
\hat{X}(s)$, where $\hat{X}(s)$ is $\tilde{\G}_2$-measurable, and 
$\ti X(\frac{3}{2}, \frac{3}{2})$ is independent of $\tilde{\G}_2$. Therefore, 
reasoning as above, we see that for some universal constant $C$,
\begin{equation}\label{5.6}
  \begin{array}{lll}
     P(H_1(m)) &=& P(\{\vert \ti X(s)\vert \leq 2^{-n+1}\} \cap H_2(m))\\ 
               &\leq& C 2^{-n} P(H_2(m)).
  \end{array}
\end{equation}

   Let 
$$
      H_3(m) = L(m) \cap \E(s,n,\ell) \cap \E(t,n,k)
$$ 
and $\tilde{\G}_3 = \F^s_{\ulS}$. Then $G(s,t,n)$ is conditionally independent 
of $\tilde{\G}_3$ given $\zeta_1$ and $\zeta_2$, $H_3(m) \in \tilde{\G}_3$, and 
therefore by Theorem \ref{prop1},
\begin{equation}\label{5.7}
   P(H_2(m)) \leq C E((2^{1-n} + 2^{\ell-n} + 2^{1+k-n} + (m+1) 2^{\ell-n} + 
     \zeta_2)^{\lambda_1} 1_{H_3(m)}).
\end{equation}
Set $\hat{\zeta}_2 = 2^{-k+n} \zeta_2$. Then $\hat{\zeta}_2$ is conditionally independent of $\tilde{\G}_4 = \sigma(\zeta_1) \vee \F^s_{\ulT^{s,\ell}} \vee \F^t_{\ulT^{t,k}}$ given $\ulT^{s,\ell}$ and $\ulT^{t,k}$, and by the observation that follows 
the definition of $\zeta_2$, there are universal constants $c$ and $C$ such 
that
\begin{equation}\label{5.8}
   P(\hat{\zeta}_2 \geq z \vert \ulT^{s,\ell}, \ulT^{t,k}) \leq C e^{-cz^2}.
\end{equation}
Writing
\begin{eqnarray*}
  && 2^{1-n} + 2^{\ell-n} + 2^{1+k-n} + (m+1) 2^{\ell-n} + \zeta_2\\
  &&\qquad =
      (2^{1-n} + (m+2) 2^{\ell-n} + 2^{1+k-n}) \left(1 + \frac{2^{k-n}}{2^{1-n} + (m+2)2^{\ell-n}
         +2^{1+k-n}} \hat{\zeta}_2\right) \\
   &&\qquad \leq C (m+2) 2^{\ell-n} (1 + \hat{\zeta}_2),
\end{eqnarray*}
we see from (\ref{5.7}) and the fact that $H_3 \in \tilde{\G}_4$ that
\begin{eqnarray}\nonumber
  P(H_2(m)) &\leq& C(m+2)^{\lambda_1} 2^{(\ell-n)\lambda_1} 
  E(E((1+ \hat{\zeta}_2)^{\lambda_1} \mid \ulT^{s,\ell}, \ulT^{t,k}) 1_{H_3(m)})\\ 
  & \leq & C^\prime (m+2)^{\lambda_1} 2^{(\ell-n) \lambda_1} P(H_3(m)) 
  \label{5.8a}
\end{eqnarray}
by (\ref{5.8}).

   We now observe that $P(L(m)) \leq C e^{-cm^2}$, for some universal positive constants $c$ and $C$, by the observations that follow the definition of $\zeta_1$. Since $\E(s,n,\ell) \cap \E(t,n,k) \in \F^s_{\ulT^{s,\ell}} \vee \F^t_{\ulT^{t,k}}$,
\begin{equation}
   P(H_3(m)) \leq C e^{-cm^2} P(\E(s,n,\ell) \cap \E(t,n,k)).
\end{equation}
We now use a trivially extended form of Proposition \ref{prop1lem16} (see Remark \ref{remprop1lem16}) to see that
\begin{eqnarray}\nonumber
   P(\E(s,n,\ell) \cap \E(t,n,k)) &\leq& E\left(C 2^{-\ell\lambda_1} 
   \left(1 + \frac{2^{k-n}}{2^{k-n}/3}\right)^{\lambda_1} 1_{\E(t,n,k)}\right)\\
   \nonumber 
    &\leq& C^\prime 2^{-\ell \lambda_1} P(\E(t,n,k))\\   
&\leq& C^\prime 2^{-\ell \lambda_1} 2^{-k \lambda_1} 
\label{5.9a}
\end{eqnarray}
by Theorem \ref{thm3}.

   Putting together (\ref{5.3})--(\ref{5.6}) and (\ref{5.8a})--(\ref{5.9a}), we 
conclude that
\begin{eqnarray*}
   P(F(s, 2^{-n}) \cap F(t, 2^{-n})) &\leq& C 2^{-\ell} 2^{-n} 
   2^{(\ell-n)\lambda_1} 2^{-\ell \lambda_1} 2^{-k \lambda_1}
    \sum^\infty_{m=0} (m+1)(m+2)^{\lambda_1} e^{-cm^2}\\
&\leq& C^\prime 2^{-(1+ \lambda_1)n-\ell-k \lambda_1}
\end{eqnarray*}
since the series converges. This proves (\ref{bivupper}) in the case where 
$3 < k \leq \ell$.

   In the case where $k \leq 3 < \ell$, the set $\E(t,n,k)$ plays no role. Instead of (\ref{5.3}), we write
   \begin{eqnarray}\nonumber
   &&F(s,2^{-n}) \cap F(t,2^{-n}) \\
   &&\qquad\qquad \subset F_1(t,2^{-n}) \cap F_1(s,2^{-n}) \cap 
      G(s,t,n) \cap \E(s,n,\ell).
\end{eqnarray}
The remainder of the proof follows as above.
\hfill$\Box$
\vskip 16pt

The following lemma was used in the proof of Proposition \ref{lembivariate}.

\begin{lemma} Let $(B(u),\ u \geq 0)$ be a standard Brownian motion. For 
$u \geq 0$, set $B^\ast(u) = \sup_{0 \leq v \leq u} B(v)$. Then the conditional 
density of $B(u)$ given $B^\ast(u)$ is bounded above by $2 B^\ast(u)/u$.
\label{supBM}
\end{lemma}

\proof According to \cite[Chapter 2.8]{KSh}, the joint density 
of $(B(u), B^\ast(u))$ is
$$
   f_u(a,b) = \frac{2 (2b-a)}{\sqrt{2 \pi u^3}}
   \exp\left(-\frac{(2b-a)^2}{2u}\right), 
   \qquad a \leq b,\quad b \geq 0.
$$
From the reflection principle, the density of $B^\ast(u)$ is 
$2(2 \pi u)^{-1/2} \exp(-b^2/(2u))$, so the conditional density of $B(u)$ given 
$B^\ast(u) = b$ is
$$
   \frac{1}{u} (2b-a) \exp\left(- \frac{(3b-a)(b-a)}{2u}\right) \leq \frac{2b}{u}.
$$
\hfill$\Box$
\vskip 16pt
 
\end{section}
\eject

\begin{section}{ABM: Lower bounds on certain escape probabilities}\label{sec6}

   The objective of this section is to establish the counterpart to Proposition \ref{univariateup}, namely a lower bound on escape probabilities. This is the remaining important ingredient needed for the second-moment argument. However, an additional requirement is needed for the ``escaping path." Indeed, since the argument in Section \ref{sec7} will use a sequence of paths, we need to ensure that the value of the standard ABM viewed along the limiting path grows sufficiently quickly as ones moves along the path away from its starting point. We will do this by showing that we can require a uniform rate of growth of the ABM along the escaping path without significantly changing the escape probability (see Proposition \ref{rdprop21}). 
	
	The required lower bound is stated and proved in Proposition \ref{lowerlem1}(a), after a sequence of preliminary results. The first lemma is concerned with the probability of reaching a level before exiting a square.

\begin{lemma} There is $c_0 > 0$ such that for all $a \geq 1$ and 
$r \in \IR^2$,
$$
   P_1\{\ultau^{1,r,a} \leq \ulsigma^{1,r,\cR_r(a^2)}\} 
      \geq c_0\, a^{-\lambda_1}.
$$
\label{lem6.18}
\end{lemma}

\proof Notice that the event considered in the statement of the lemma is ``the DW-algorithm 
started at $r$ with value 1 reaches level $a$ within the square $\cR_r(a^2)$''.
By Theorem \ref{prop1}, there is $c_1 > 0$ such that for all $a \geq 1$ and 
$r \in \IR^2$, $P_1\{\ultau^{1,r,a} < \ulinfty\} \geq c_1 a^{-\lambda_1}$. By 
Theorem \ref{thm3}, there is $c_2 > 0$ such that for all $a \geq 1$ and $K \geq 
1$,
$$
   P_1\{\cR_r (\ultau^{1,r}_{(N)}) \not\subset \cR_r(K^2 a^2)\} \leq 
        c_2(Ka)^{-\lambda_1}.
$$

   Fix $K$ such that $K^{\lambda_1} > 2 c_2/c_1.$ Then the right-hand side above is $\leq a^{-\lambda_1} c_1/2$, and therefore,
\begin{eqnarray*}
   P_1\{\ultau^{1,r,a} \leq \ulsigma^{1,r,\cR_r(K^2a^2)} \}
     &\geq& P_1\{\ultau^{1,r,a} < \ulinfty,\ \cR_r (\ultau^{1,r}_{(N)}) 
        \subset \cR_r(K^2 a^2)\}\\ 
   &\geq& P_1\{\ultau^{1,r,a} < \ulinfty\} - P_1\{\cR_r(\ultau_{(N)}^{1,r}) 
       \not\subset \cR_r(K^2 a^2)\} \\ 
   &\geq& c_1 a^{-\lambda_1} - \frac{c_1}{2} a^{-\lambda_1}\\
   &=& \frac{c_1}{2} a^{-\lambda_1}.
\end{eqnarray*}
Replacing $a$ by $a/(2K)$ and writing $\ultau$ for $\ultau^{1,r,a/(2K)}$ and 
$\ulsigma$ for $\ulsigma^{1,r,\cR_r(a^2/4)}$, we see that for all $a \geq 2K$,
\begin{equation}\label{5.10}
   P_1\{\ultau \leq \ulsigma\} \geq \frac{c_1}{2} (2K)^{\lambda_1} a^{-\lambda_1}.
\end{equation}

   On the event $\{\ultau < \ulinfty\}$, which belongs to $\F_{\ultau}$, if level $a/(2K)$ is reached during an odd stage $2n+1$, then one of the events $F_1$, $F_2$ and $F_3$ occur, where
$$
  \begin{array}{ll}
    F_1 = \{1+ X^r(r_1+\ultau_1, T^n_2) = \frac{a}{2K},
     & 1+ X^r(r_1-\ultau_3, T^n_2) = \frac{a}{2K}\},\\
     \\
    F_2 = \{1+ X^r(r_1+\ultau_1, T^n_2) = \frac{a}{2K},
       & 1+ X^r(r_1-\ultau_3, T^n_2) = 0\},\\
     \\
    F_3 = \{1+ X^r(r_1+\ultau_1, T^n_2) = 0,
     & 1+ X^r(r_1-\ultau_3, T^n_2) = \frac{a}{2K}\}.
  \end{array}
$$
If level $a/(2K)$ is reached during an even stage, then one of $F_4$, $F_5$ or 
$F_6$ occurs, where these events are defined using the obvious analogy with $F_1$, $F_2$, $F_3$. For 
$i = 1,2,3, 4$, let $G_i$ be the event ``$\frac{a}{2K} + Z_i^{r,\ultau}$ hits level $a$ 
before level $0$ and before time $3a^2/4$.'' By Brownian scaling, 
$P_{a/(2K)} (G_i) \geq c_3 > 0$, where $c_3$ depends on $K$ but not on $a, r$ 
or $i$. Set
$$
  \tilde{F}_1 = F_1 \cap G_1 \cap G_3, \quad \tilde{F}_2 = F_2 \cap G_1, \quad
   \tilde{F}_3 = F_3 \cap G_3,
$$
and define $\tilde{F}_4$, $\tilde{F}_5$ and $\tilde{F}_6$ by analogy. 

   The key observation is that
\begin{equation}\label{5.11}
   P_1\{\ultau^{1,r,a} \leq \ulsigma^{1,r,\cR_r(a^2)}\} \geq 
   \sum_{i=1}^6 P(\{\ultau \leq \ulsigma\} \cap \tilde{F}_i).
\end{equation}
Indeed, the events on the right-hand side are disjoint, and each is contained in the event on the left-hand side. The idea behind \eqref{5.11} is that as soon as the DW-algorithm has reached level $a/(2K)$, it has probability at least $c_3$ of reaching level $a$ during the next step, so little is lost in inequality \eqref{5.11}.

   Since $G_1$ and $G_3$ are independent of $\F^r_{\ultau}$, it follows that the term in \eqref{5.11} with $i=1$ is 
bounded below by
$$
   c_3^2\, P_1(\{\ultau \leq \ulsigma\} \cap F_1),
$$
while the terms with $i \in \{2,3\}$ are bounded below by
$$
   c_3\, P_1(\{\ultau \leq \ulsigma\} \cap F_i),
$$
and similar inequalities hold for $i = 4,5,6$. Therefore, by (\ref{5.11}),
\begin{eqnarray*}
   P_1\{\ultau^{1,r,a} \leq \ulsigma^{1,r,\cR_r(a^2)}\} &\geq& 
      c_3^2\, \sum_{i=1}^6 P_1(\{\ultau \leq \ulsigma\} \cap F_i)\\
   &=& c_3^2\, P_1 \{\ultau \leq \ulsigma\}\\
   &\geq& c_3^2\, \frac{c_1}{2}(2K)^{\lambda_1}\, a^{-\lambda_1}
\end{eqnarray*}
by (\ref{5.10}).

   In order to handle the case where $1 \leq a \leq 2K$, we note that in this 
case,
$$
   P_1\{\ultau^{1,r,a} \leq \ulsigma^{1,r,\cR_r(a^2)}\} \geq 
     P_1\{\ultau^{1,r,a} \leq \ulsigma^{1,r,\cR_r(1)}\},
$$
and this probability is bounded below uniformly over $a \in [1,2K]$, since it 
is bounded below by the probability that level $2K$ is reached at the first stage 
of the DW-algorithm, within $1$ unit of $r$. This completes the proof. 
\hfill $\Box$
\vskip 16pt

   The next lemma is concerned with the probability of reaching a level before exiting a square, and at the same time, of reaching a geometric sequence of intermediate levels much later than is typical.

\begin{lemma} Let $c_0$ be the constant from Lemma \ref{lem6.18}. Then for $K$ 
sufficiently large, for all $r$ and $n\geq 2$,
\begin{equation}\label{eq6.3}
   P_{2^{-n}}\left(\{\ultau^{2^{-n},r,1} \leq \ulsigma^{2^{-n},r,\cR_r(1/2)}\} \cap
    \bigcup^{n-1}_{j=0} \{\Vert \ultau^{2^{-n},r,2^{j+1-n}} \Vert 
       > K(n-j)2^{2(j-n)}\}\right) \leq \frac{c_0}{10} 2^{-n \lambda_1}.
\end{equation}
\label{lem6.19}
\end{lemma}

\proof Because
$$
   \frac{1}{2} \sum_{\ell=0}^{j-1} (n-\ell)2^{2(\ell-n)} \leq (n-j)2^{2(j-n)}, 
   \qquad\mbox{ for all } 1 \leq j \leq n-1,
$$
the event on the left-hand side of (\ref{eq6.3}) is contained in
\begin{eqnarray*}
  && \{\ultau^{2^{-n},r,1} \leq \ulsigma^{2^{-n},r,\cR_r(1/2)} \} \cap 
     \bigcup^{n-1}_{j=0} \{ \Vert \ultau^{2^{-n},r,2^{j+1-n}} 
         - \ultau^{2^{-n}, r, 2^{j-n}} \Vert > \frac{K}{2} (n-j) 2^{2(j-n)}\}\\
   &&\qquad \subset \bigcup_{j=0}^{n-1} (A_1(j) \cap A_1(j+1) \cap A_2(j) \cap A_3),
\end{eqnarray*}
where
\begin{eqnarray*}
   A_1(j) &=& \{\ultau^{2^{-n},r,2^{j-n}} < \ulinfty\},\\
   A_2(j) &=& \{ \Vert \ultau^{2^{-n},r,2^{j+1-n}} - \ultau^{2^{-n}, r, 2^{j-n}} \Vert > \frac{K}{2} (n-j) 2^{2(j-n)}\},\\
  A_3 &=& \{\ultau^{2^{-n},r,1} < \ulinfty\}.
\end{eqnarray*}
By Lemma \ref{lem2lem16},
$$
   P_{2^{-n}}(A_3 \vert {\cal{F}}_{\ultau^{2^{j+1-n}}}^r) 
     \leq c (2^{j+1-n})^{\lambda_1}\qquad \mbox{on } A_1(j+1),
$$
and $A_1(j) \cap A_1(j+1) \cap A_2(j) \in \F_{\ultau^{2^{j+1-n}}}^r$. By Lemma \ref{rdlem10},
$$
   P_{2^{-n}}(A_1(j+1) \cap A_2(j) \vert {\cal{F}}_{\ultau^{2^{j-n}}}^r) \leq C e^{-cK(n-j)} \qquad \mbox{on } A_1(j),
$$
and therefore, by iterated conditioning and Theorem \ref{prop1},
\begin{align*}
   &P_{2^{-n}}(A_3 \cap A_1(j+1) \cap A_2(j) \cap A_1(j)) \\
   &\qquad \leq C (2^{j+1-n})^{\lambda_1} P(A_1(j+1) \cap A_2(j) \cap A_1(j)) \\
   &\qquad \leq C\,2^{(j+1-n) \lambda_1} e^{-cK(n-j)} 2^{-j \lambda_1}\\
   &\qquad = C\,2^{-n \lambda_1} e^{-cK(n-j)}.
\end{align*}
It follows that the left-hand side of (\ref{eq6.3}) is bounded above by
$$
   \sum_{j=0}^{n-1} C\,2^{-n \lambda_1} e^{-cK(n-j)} 
     \leq C\,2^{-n \lambda_1} \sum^\infty_{j=1} e^{-cKj}.
$$
By choosing $K$ large, the series can be made arbitrarily small, and this 
proves (\ref{eq6.3}). 
\hfill $\Box$
\vskip 16pt

   The lemma below is concerned with the probability that the DW-algorithm reaches a certain level and drops back far below an intermediate level after reaching this intermediate level, for a geometric sequence of intermediate levels.

\begin{lemma} Let $\Gamma^{\ast, 2^{-n},r}$\index{$\Gamma^{\ast, 2^{-n},r}$} be the path constructed by the DW-algorithm started at $r$ with value $2^{-n}$, and let 
$\gamma^{\ast, 2^{-n},r} : \IR_+ \to \IR^2$\index{$\gamma^{\ast, 2^{-n},r}$} be the one-to-one parametrization 
by arc-length of this path, such that $\gamma^{\ast,2^{-n},r}(0) = r$. Define 
$\alpha_n^{j,r} \in \IR_+$\index{$\alpha_n^{j,r}$} by $\alpha_n^{j,r} = \inf\{ u\geq 0: 2^{-n} + X^r(\gamma^{\ast,2^{-n},r}(u)) = 2^{j-n}\}$. 
Let $c_0$ be the constant from Lemma \ref{lem6.18}, 
and for $K > 0$ and $j \geq 1$, set $c_j = K^{-2}j^{-3}$. For $K$ sufficiently 
large, for all $r$ and $n$,
\begin{equation}\label{eq6.4}
   P_{2^{-n}}(\{\ultau^{2^{-n}, r, 1} < \ulinfty\} \cap 
     \bigcup^{n-1}_{j=0} \{2^{-n} + \inf_{\alpha^{j,r}_n \leq u < \alpha^{j+1,n}_n} 
       X^r(\gamma^{\ast, 2^{-n},r}(u)) \leq c_{n-j} 2^{j-n}\}) 
       \leq \frac{c_0}{10} 2^{-n\lambda_1}.
\end{equation}
\label{lem6.20}
\end{lemma}

\begin{remark} On the stage $k$ where the value $2^{j-n}$ is first achieved (i.e.~so that $H_{k-1} < 2^{j-n} \leq H_k$, it may well be the case that this value is attained in both possible directions.  The point $\gamma^{\ast,2^{-n},r}(\alpha_n^{j,r})$ is the relevant position in the direction that later leads to the highest maximum before hitting zero (this position is a.s.~unique). In particular, $\gamma^{\ast,2^{-n},r}(\alpha_n^{j,r}) \in \cR_r(\ultau^{2^{-n}, r, 2^{j-n}})$.
\end{remark}

\proof For $0 \leq j \leq n-1,$ suppose level $2^{j-n}$ is reached during stage 
$N_j$, that is, $H_{N_j-1} < 2^{j-n} \leq H_{N_j}$, and let $\nu_j$ be the 
(random) number of stages needed to pass from level $2^{j-n}$ to level 
$2^{j+1-n}$.

   For  $N_j$ of the form $2m-1$, we define 
\begin{eqnarray*}
   F_{j,1} &=& \{2^{-n} + \inf_{\alpha^{j,r}_n < u < \alpha^{j+1,r} \wedge
   (\gamma^{\ast, 2^{-n},r})^{-1}(T_1^m, T_2^{m-1})} 
   X^r(\gamma^{\ast, 2^{-n},r}(u)) \leq c_{n-j} 2^{j-n}\},\\
   \\
   F_{j,2} &=& \{H_{N_{j+1}} - H_{N_{j+1}-1} \leq c_{n-j} 2^{j-n}\},\\
   \\
    F_{j,3} &=& \{\nu_j \geq K(n-j)\},
\end{eqnarray*}
with an analogous definition when $N_j$ is even, and for $i \in \{4, 5\}$,
$$
F_{j,i} = \bigcup_{k=N_j}^{N_{j+1}-1} F_{j, i, k},
$$
where
$$
F_{j, 4, k} = \{H_k - H_{k-1} \leq c_{n-j} 2^{j-n}\},
$$
\begin{eqnarray*}
   F_{j, 5, 2k-1} &=& \{T_1^k < U_{k-1},\ 2^{-n} + \inf_{[T^k_1, U_{k-1}]} 
       X^r(\cdot, T^{k-1}_2) \leq c_{n-j} 2^{j-n}\} \\
       &&\qquad \cup\, \{ T_1^k > U^\prime_{k-1},\ 2^{-n} +  \inf_{[U^\prime_{k-1}, T^k_1\}} 
       X(\cdot, T^{k-1}_2) \leq c_{n-j} 2^{j-n}\},\\
   && \\
   F_{j, 5, 2k} &=& \{T_2^k < V_{k-1},\ 2^{-n} + \inf_{[T^k_2, V_{k-1}]} 
   X^r(T^k_1, \cdot) \leq c_{n-j} 2^{j-n}\} \\
   &&\qquad  \cup\, \{T_2^k > V^\prime _{k-1},\ 2^{-n} + \inf_{[V^\prime_{k-1}, T_2^k \cdot]}
   X^r(T^k_1,\cdot)  \leq c_{n-j} 2^{j-n}\}
\end{eqnarray*} 
(notice that on the event $F_{j, 5, 2k-1}$, the ABM $X^r$ reaches the low level $c_{n-j} 2^{j-n}$ during stage $2k-1$ before reaching level $2^{j-n}$). Use these events to define, for $i \in \{ 1, 2, 3\}$,
$$
   G_{j, i} = \{\ultau^{2^{-n}, r, 1} < \ulinfty\} 
   \cap F_{j, i},
$$
and for $ i \in \{ 4, 5\}$,
$$
  G_{j, i} = \{\ultau^{2^{-n}, r, 1} < \ulinfty\} 
     \cap \{\nu_j < K(n-j)\} \cap F_{j,i}.
$$

   Let $F$ be the event on the left-hand side of (\ref{eq6.4}). The key 
observation is that
\begin{equation}\label{eq6.5}
   F \subset \bigcup_{j=0}^{n-1} \bigcup_{i=1}^5 G_{j, i}.
\end{equation}
Indeed, suppose $\ultau^{2^{-n}, r, 1} < \ulinfty$ and 
\begin{equation}\label{eq6.6}
   2^{-n} + \inf_{\alpha_n^{j, r} \leq u < \alpha_n^{j+1,r}} 
      X^r(\gamma^{\ast, 2^{-n}, r}(u)) \leq c_{n-j} 2^{j-n}.
\end{equation}
If $N_{j+1} = N_j$, then the infimum in (\ref{eq6.6}) is attained already during stage $N_j$, and so $F_{j,1}$, hence $G_{j,1}$, occurs.
If $N_{j+1} > N_j$ and this infimum is attained during stage $2m$ and $N_{j+1} = 2m$, then either it is attained on the segment with extremities $\{T^m_1\} \times [V_{m-1}, V_{m-1}^\prime]$, in which case $F_{j,2}$, hence $G_{j,2}$, occurs by (\ref{X_0}) and (\ref{starrd1}), or it is attained outside this segment, in which case $F_{j, 5, 2m}$, hence $G_{j, 3}$ or $G_{j, 5}$ occurs. If $\nu_j \geq K(n-j)$, then $F_{j,3}$ and $G_{j,3}$ occur. 
If $\nu_j < K(n-j)$ and the infimum in (\ref{eq6.6}) is attained during an odd 
stage $2k-1 \in [N_j, N_{j+1} [$, then either this infimum is attained on the 
segment $[U_{k-1}, U^\prime_{k-1}] \times \{T^{k-1}_2\}$, in which case 
$F_{j, 4, 2k-1}$ occurs by (\ref{X_0}) and (\ref{starrd1}), or it is attained outside this segment, in which case $F_{j, 5, 2k-1}$ occurs ($F_{j, 4, 2k}$ and $F_{j, 5, 2k}$ occur respectively if odd is replaced by even). This proves (\ref{eq6.5}).

   We now bound the probability of each $G_{j,i}$. For fixed $j$, set 
$\ulT = \ultau^{2^{j + 1-n}}$, let $X^{\ulT}$ be defined below (\ref{startp6}), and let $E_j$ be the event ``the DW-algorithm applied to 
$X^{\ulT}$, started at 0 with value $2^{j+1-n}$, reaches level 1''. Then 
$E_j$ is independent of $\F_{\ulT}^r$, and $P_{2^{j+1-n}} (E_j) \leq c 2^{(j-n) \lambda_1}$ by Theorem \ref{prop1}.

   Observe by Lemma \ref{lem5.2prime} that
$$
   G_{j,1} \subset \{\ultau^{2^{-n}, r, 2^{j-n}} < \ulinfty\} \cap F_{j,1} 
     \cap E_j,
$$
the first two events on the right-hand side are $\F_{\ulT}^r$-measurable, and 
$P(F_{j,1} \vert \F^r_{\ultau^{2^{-n}, r, 2^{j-n}}})$ is bounded above by the 
probability that a Brownian motion started at $c_{n-j}\, 2^{j-n}$ hits $2^{j-n}$ 
before 0. By iterated conditioning, we see that
\begin{equation}\label{eq6.7}
   P_{2^{-n}}(G_{j,1}) \leq C (2^{-j})^{\lambda_1} c_{n-j}\, (2^{(j-n)})^{\lambda_1} 
     = C \ c_{n-j} 2^{-n \lambda_1}.
\end{equation}

   Similarly,
$$
   G_{j,2} \subset \{\ultau^{2^{-n}, r, 2^{j-n}} < \ulinfty\} \cap 
     \{H_{N_{j+1}} - 2^{j-n} \leq c_{n-j}\, 2^{j-n}\} \cap E_j,
$$
and since the probability of the second event is bounded above by the 
probability that a Brownian motion started at $2^{j-n}$ hits 0 before 
$2^{j-n}(1 + c_{n-j}),$ the same arguments as above show that
\begin{equation}
   P_{2^{-n}}(G_{j,2}) \leq C (2^{-j})^{\lambda_1} c_{n-j} (2^{j-n})^{\lambda_1} 
     = C c_{n-j} 2^{-n \lambda_1}.
\end{equation}

  Observe that
$$
   G_{j,3} \subset \{\ultau^{2^{-n}, r, 2^{j-n}} < \ulinfty\} \cap 
     \{\nu_j > K(n-j)\} \cap E_j,
$$
so using Lemma \ref{rdlem9}, one finds that
\begin{equation}
   P_{2^{-n}}(G_{j, 3}) \leq C(2^{-j})^{\lambda_1} \left(\frac{3}{4}\right)^{K(n-j)-2} 
     \cdot (2^{j-n})^{\lambda_1} \leq C^\prime \left(\frac{3}{4}\right)^{K(n-j)} 
     \cdot 2^{-n \lambda_1}.
\end{equation}

   Turning to $G_{j,4}$, we observe that
$$
  G_{j, 4} \subset \bigcup_{k = N_j}^{N_j + K(n-j)-1} G_{j, 4, k},
$$
where
$$
   G_{j, 4, k} = \{\ultau^{2^{-n}, r, 2^{j-n}} < \ulinfty\} \cap 
    \{N_j \leq k < N_{j+1}\} \cap F_{j, 4, k} \cap E_j.
$$
In order to evaluate $P_{2^{-n}}(F_{j, 4, k} \vert \F_{\ultau_{(k-1)}})$, let 
$f_{x,y}(z)$ denote the right-hand side of (\ref{d2}). An elementary 
calculation (see \eqref{fZn}) shows that $0\leq -f^\prime_{x, y} (z) \leq 2/z$ (since $0<x<y<z$), and therefore, by 
Proposition \ref{noSTOP}, on $\{N_j \leq k < N_{j+1}\}$,
$$
   P_{2^{-n}}(F_{j, 4, k} \vert \F_{\ultau_{(k-1)}}) \leq 
   \frac{2}{2^{j-n}}\, c_{n-j}\, 2^{j-n} = 2\, c_{n-j}.
$$
It follows that
$$
   P_{2^{-n}}(G_{j, 4, k}) \leq C (2^{-j})^{\lambda_1} \cdot 2\,c_{n-j} \cdot 
     (2^{j-n})^{\lambda_1},
$$
and so
\begin{equation}
   P_{2^{-n}}(G_{j, 4}) \leq C^\prime K\,(n-j)\, c_{n-j}\, 2^{-n \lambda_1}.
\end{equation}

   Finally, we observe that
$$
   G_{j, 5} \subset \bigcup_{k = N_j}^{N_j + K(n-j)-1} G_{j, 5, k},
$$
where
$$
   G_{j, 5, k} = \{ \ultau^{2^{-n}, r, 2^{j-n}} < \ulinfty\} \cap 
     \{ N_j \leq k < N_{j+1}\} \cap F_{j, 5, k} \cap E_j,
$$
\\
and on $\{N_j \leq k < N_{j+1}\}$, $P_{2^{-n}}(F_{j, 5, k} \vert \F_{\ultau_{(k-1)}})$ 
is no greater than twice the probability that a Brownian motion started at 
$c_{n-j}\, 2^{j-n}$ hits $2^{j-n}$ before 0. Therefore,
$$
   P_{2^{-n}}(G_{j, 5, k}) \leq C (2^{-j})^{\lambda_1} \cdot 2 c_{n-j} \cdot (2^{j-n})^{\lambda_1},
$$
and so
\begin{equation}\label{eq6.11}
   P_{2^{-n}}(G_{j, 5}) \leq C^\prime K(n-j)\, c_{n-j}\, 2^{-n \lambda_1}.
\end{equation}
It now follows from (\ref{eq6.5}) and (\ref{eq6.7})--(\ref{eq6.11}) that
$$
   P_{2^{-n}}(F) \leq C \, 2^{-n \lambda_1} \sum_{j=0}^{n-1} 
   \left(2 c_{n-j} + \left( \frac{3}{4} \right)^{K(n-j)} + 2 K(n-j)
   c_{n-j}\right).
$$
The sum is bounded above by
$$
   \frac{2}{K^2} \sum^\infty_{j=1} \frac{1}{j^3} + 4 \left( \frac{3}{4} \right)^K + \frac{2}{K} \sum^\infty_{j=1} \frac{1}{j^2},
$$
which can be made as small as desired by choosing $K$ sufficiently large. This 
proves the lemma. 
\hfill $\Box$
\vskip 16pt

   The following proposition is concerned with the probability that the DW-algorithm reaches level $1$ within a fixed square, but drops below a low level after having moved significantly away from its starting position.

\begin{prop} Let $\gamma^{\ast, 2^{-n}, r}$, $\alpha_n^{j,r}$ and $c_j$ be as in Lemma \ref{lem6.20}. For $t=(t_1,t_2) \in \IR^2$, set $\vert t \vert = \vert t_1 \vert + \vert t_2\vert$. For $ K > 0$ and $0 \leq j < n$, set
$$
   I^r_{n,j} = \{u \in \IR_+ : u \leq \alpha_n^{n,r} \mbox{ and } 
   \vert \gamma^{\ast, 2^{-n}, r}(u) -r \vert \geq K\,(n-j)\, 2^{2(j-n)}\}.
$$
Let $c_0$ be the constant from Lemma \ref{lem6.18}. For $K$ sufficiently large, 
for all $r$ and $n$,
\begin{eqnarray*}
   &&P_{2^{-n}}\left(\{\ultau^{2^{-n}, r, 1} \leq \ulsigma^{2^{-n}, r, \cR_r(1/2)}\} \cap 
   \bigcap^{n-1}_{j=0}\left\{2^{-n}+ \inf_{u \in I^r_{n,j}} 
   X^r(\gamma^{\ast, 2^{-n}, r}(u)) \geq c_{n-j} 2^{j-n}\right\}\right)\\
   &&\qquad \geq \frac{4}{5} c_0 2^{-n \lambda_1}
\end{eqnarray*}
(we use the convention $\inf \emptyset = + \infty$, for the $n$ and $j$ such 
that $I^r_{n,j} = \emptyset$).
\label{rdprop21}
\end{prop}

\proof Fix $K$ large enough so that the inequalities of Lemmas \ref{lem6.19} and
\ref{lem6.20} hold. Then
$$
   P_{2^{-n}}(\{\ultau^{2^{-n},r,1} \leq \ulsigma^{2^{-n}, r, \cR_r(1/2)}\} 
   \cap F^c_n) \leq \frac{2}{10} c_0 2^{-n \lambda_1},
$$
where
\begin{eqnarray*}
  F_n &=& \left(\bigcap^{n-1}_{j=0} \{ \Vert \ultau^{2^{-n}, r, 2^{j+1-n}} \Vert 
    \leq K(n-j) 2^{2(j-n)}\}\right) \\
    &&\qquad \cap\, \left(\bigcap^{n-1}_{j=0} \left\{2^{-n}+ \inf_{\alpha^{j, r}_{n} \leq u < 
    \alpha^{j+1,r}_n} X^r(\gamma^{\ast, 2^{-n}, r}(u)) > c_{n-j}\,
    2^{j-n}\right\}\right).
\end{eqnarray*}
Therefore, by Lemma \ref{lem6.18},
\begin{equation}\label{en6.12}
   P_{2^{-n}}(\{\ultau^{2^{-n}, r, 1} \leq \ulsigma^{2^{-n}, r, \cR_r(1/2)}\} \cap F_n) 
   \geq \frac{4}{5} c_0 2^{-n \lambda_1}.
\end{equation}
On this event, $\Vert \ultau^{2^{-n}, r, 1} \Vert \leq 4 \cdot \frac{1}{2} = 2$,
so there is $n_0 > 1$ such that $I^r_{n,j} = \emptyset$ if
$n-n_0 \leq j < n$. For $j < n-n_0$, on the event on the left-hand side of \eqref{en6.12}, $\Vert \ultau^{2^{-n}, r, 2^{j+1-n}}\Vert \leq 
K (n-j) 2^{2(j-n)}$, so $I^r_{n,j} \subset [\alpha^{j,r}_n, \infty[$, and 
therefore
$$
  2^{-n}+ \inf_{u \in I^r_{n,j}} X^r(\gamma^{\ast, 2^{-n}, r}(u)) \geq c_{n-j}\, 2^{j-n}.
$$
This proves the proposition. 
\hfill $\Box$
\vskip 16pt

   We now introduce the notation needed for Proposition \ref{lowerlem1} below. This proposition contains all the ingredients needed for the second-moment argument that we will implement in Section \ref{sec7} (see Lemma \ref{lowerlem2}).

   For $r \in \IR^2$ and $q \in \IR$, let $\cC_r(q)$\index{$\cC_r(q)$} denote the connected component of $\{s \in \IR^2: \ti X(s) > q\}$ that contains $r$. Let $\cC_r^x$\index{$\cC_r^x$} denote the connected component of $\{s\in \IR^2: x + X^r(s) >0\}$ that contains $r$, and let $\partial \cC_r^{x,\alpha}$\index{$\partial \cC_r^{x,\alpha}$} denote the subset of points in $\partial \cC_r^x$ to which one can get arbitrarily close by following a curve starting at $r$ and contained in $\cC_r^x \cap ([r_1,r_1 + \alpha]\times [r_2,r_2 + \alpha])$.
  
   For $K > 0$ and $t \in \IR^2$, using the notation of Proposition \ref{rdprop21}, let
\begin{eqnarray*}
   A_0 &=& \{ 1+ \ti X(u_1,u_1) >0, \mbox{ for all } u_1 \in [0,1]\},\\
   A_1(t, n) &=& \{1+ \ti X(t) \in [-2^{1-n}, -2^{-n}]\}, \\
   A_2(t,n) &=& \left\{\ultau^{2^{-n},t,1} \leq \ulsigma^{2^{-n},t,\cR(1/2)}
   \right\}, \\
   A_3(K,t,n) &=& \bigcap_{j=0}^{n-1} \left\{2^{-n} +\inf_{u \in I^t_{n,j}}
   X^t(\gamma^{*,2^{-n},t}(u)) \geq c_{n-j}\, 2^{j-n} \right\},
\end{eqnarray*}
and let $A_4(t,n)$ be the event ``there is a path with extremities $(1,1)$ and $\alpha_n^{n,t}$ contained in $[1,4]^2$ along which $X^t(\cdot) \in [\frac{1}{2}, 10]$." Finally, set
$$
  A(K, t, n) = A_0 \cap A_1(t, n) \cap A_2(t, n) \cap A_3(K, t, n) \cap A_4(t, n).
$$
Observe from the definition of $I^t_{n,j}$ and $c_{n-j}$ that if $A(K, t, n)$ occurs, then $t$ is no more than 
$K^5n^72^{-2n}$ units away from a point in $\partial{\cC}_{(0,0)}^{1}$, and even from a point in $\partial \cC_{(0,0)}^{1,4}$.

\begin{prop} There are $K > 0$, $c > 0,$ and $C > 0$ such that:

  (a) for all large $n \in \IN$ and $t \in [2, 3]^2,$
$$
   c 2^{-(1+\lambda_1)n} \leq P(A(K,t,n)) \leq 
      \frac{1}{c} 2^{-(1+\lambda_1)n};
$$

  (b) for all large $n \in \IN$, $1 \leq k \leq \ell \leq n$ and 
$(s, t) \in \ID_n(k, \ell),$
$$
   P(A(K,t,n) \cap A(K,s,n)) \leq C\,2^{-(1+\lambda_1)n-\ell-k \lambda_1}.
$$
\label{lowerlem1}
\end{prop}

\proof Part (b) and the upper bound in (a) are respectively a consequence of 
Propositions \ref{lembivariate} and \ref{univariateup}. We therefore proceed 
to prove the lower bound in part (a). 

   Let $\G = \F^t_{(1,1,t_1-1,t_2-1)}$. Then $A_2(t,n)$, $A_3(K, t, n)$ and 
$A_4(t,n)$ belong to $\G$, and $\ti X(t) = \ti X(1,1) + Y$, where $Y= X^t(1,1)$ is $\G$-measurable and $\ti X(1,1)$ is independent of $\G$. On $A_1(t,n) \cap A_4(t,n)$, $1+\ti X(1,1) \geq 1 + \ti X(t) + X^t(1,1) \geq -2^{1-n} + \half \geq \frac{1}{4}$ for large $n$. Notice that $A_0$ is conditionally independent of $\sigma(\ti X(1,1)) \vee\G$ given $\ti X(1,1)$, and there is $\ti c_0 >0$ such that $P(A_0 \mid \ti X(1,1)) \geq \ti c_0$ on $\{1+\ti X(1,1) \geq \frac{1}{4}\}$. In addition, on $A_4(t,n)$, $\vert Y \vert \leq 10,$ so there is $c_1 > 0$ such that on $A_4(t,n)$,
$$
  P(A_1(t,n) \mid \G) = P\{1+ \ti X(1,1) + Y \in [-2^{-n+1}, - 2^{-n}]\mid \G\} \geq c_1 2^{-n},
$$
and therefore
\begin{eqnarray}\nonumber
   P(A(K,t,n)) &\geq& E(P(A_0\mid \sigma(\ti X(1,1)) \vee \G) 1_{A_1(t,n) \cap A_2(t,n) \cap A_3(K,t,n)\cap A_4(t,n)}) \\
   &\geq& \ti c_0 c_1 2^{-n} P(A_4(t,n) \cap A_3(K,t,n) \cap A_2(t,n)).
\label{eq6.12}
\end{eqnarray}
Set $\ultau = \ultau^{2^{-n},t,1}$, and let $H_1$ (resp. $H_3$) be the event ``the DW-algorithm started at $t$ with value $2^{-n}$ hits 1 during some odd stage $2n+1,$ and $2^{-n}+ X^t(t_1 + \ultau_1, T^n_2) = 1$ (resp.~$2^{-n}+ X^t(t_1 - \ultau_3, T^n_2) = 1 > 2^{-n}+ X^t(t_1 + \ultau_1, T^n_2))$.'' Similarly, let $H_2$ (resp.~$H_4$) be the event ``the DW-algorithm started at $t$ with value $2^{-n}$ hits 1 during some even stage $2n,$ and $2^{-n}+ X^t(T_1^n, t_2 + \ultau_2) = 1$ (resp.~$2^{-n}+ X^t(T_1^n, t_2 - \ultau_4) = 1 > 2^{-n}+ X^t(T^n_1,t_2 +  \ultau_2))$.'' 

   Set $\cH = \F_{\ultau}^t$. Then $H_i$, $A_2(t,n)$ and $A_3(K,t,n)$ belong to $\cH$. By (\ref{eq6.12}), 
\begin{equation}\label{eq6.13}
   P(A(K,t,n)) \geq \ti c_0 c_1 2^{-n} \sum_{i=1}^4 E(P(A_4(t,n)\mid \cH)\,
      1_{H_i \cap A_2(t,n) \cap A_3(K,t,n)}).
\end{equation}
We claim that there is $c_2 > 0$ such that for all $n$ and $t \in [2,3]^2$,
\begin{equation}\label{eq6.14}
   P(A_4(t,n) \vert {\cal{H}}) \geq c_2\qquad \mbox{ on } H_i \cap A_2(t,n).
\end{equation}
Assuming (\ref{eq6.14}) for the moment, we complete the proof of the lower bound 
in (a). By (\ref{eq6.13}) and (\ref{eq6.14}),
\begin{eqnarray*}
   P(A(K,t,n)) &\geq& \ti c_0 c_1c_2 2^{-n} \sum_{i=1}^4 P(H_i \cap A_2(t,n) 
        \cap A_3(K,t,n))\\
   &=& \ti c_0 c_1c_2 2^{-n} P(A_2(t,n) \cap A_3(K,t,n)).
\end{eqnarray*}
By Proposition \ref{rdprop21}, the right-hand side is $\geq \frac{4}{5} c_0 
2^{-n \lambda_1}$, which establishes the lower bound in (a).

   We now prove (\ref{eq6.14}). We only consider the case where $i=1$, $\ultau$ occurs on a horizontal stage and $2^{-n} + X^t(t_1 + \ultau_1,T^n_2)=1$, since the other cases are similar (but simpler). Consider the events
\begin{eqnarray*}
   G_1 &=& \left\{X^t(t_1 + \ultau_1 + u_1, T^n_2) \in [\frac{1}{2}, \frac{7}{2}],\quad 
     \forall u_1 \in [0,1]\right\} \cap \{X^t(t_1 + \ultau_1+1, T_2^n) \geq 3\},\\
  G_2 &=& \left\{X^t(t_1 + \ultau_1 + 1, u_2) \in [\frac{1}{2}, \frac{13}{2}],\quad 
    \forall u_2 \in [1, V_n]\right\} \cap \{X^t(t_1 + \ultau_1+1,1) \geq 6\}, \\
  G_3 &=& \left\{X^t(1+u_1, 1) \in [\frac{1}{2}, 10],\quad \forall u_1 \in 
        [0, t_1 - \ultau_3 - 1]\right\}.
\end{eqnarray*}
We claim that
\begin{equation}\label{eq6.15}
   G_1 \cap G_2 \cap G_3 \cap H_1 \cap A_2(t,n) 
       \subset A_4(t,n) \cap H_1 \cap A_2(t,n).
\end{equation}
Indeed, this is a consequence of the fact that on $G=G_1 \cap G_2 \cap G_3 \cap 
H_1 \cap A_2(t,n)$, $X^t(\cdot) \in [\frac{1}{2}, 10]$ on the path
$$
  ([t_1 + \ultau_1, t_1 + \ultau_1 + 1] \times \{T^n_2\}) \cup (\{t_1 + \ultau_1 + 1\} 
   \times [1, T^n_2]) \cup ([1, t_1 + \ultau_1 + 1] \times \{1\}),
$$
(which has $(1,1)$ and $\alpha_n^{n,t}$ as extremities; see Figure \ref{figI_08}), as we now check.
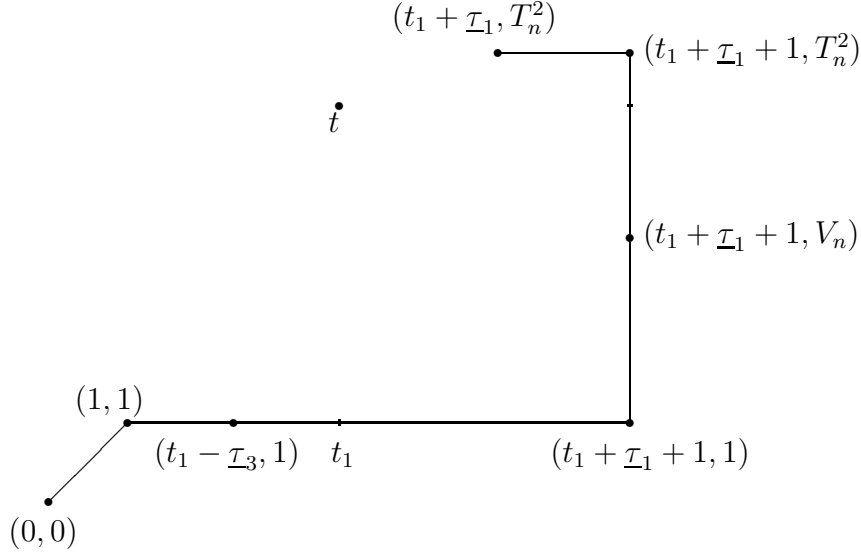
\begin{figure}
\begin{center}
\begin{picture}(350,200)
\put(20,20){\line(1,1){30}}
\put(20,20){\circle*{3}}
\put(5,5){$(0,0)$}

\put(50,50){\circle*{3}}
\put(30,55){$(1,1)$}

\put(50,50){\line(1,0){190}}

\put(90,50){\circle*{3}}
\put(60,35){$(t_1 - \ultau_3,1)$}

\put(130,49){\line(0,1){2}}
\put(127,35){$t_1$}

\put(240,50){\circle*{3}}
\put(210,35){$(t_1 + \ultau_1+1,1)$}

\put(240,50){\line(0,1){140}}
\put(240,120){\circle*{3}} \put(245,118){$(t_1 + \ultau_1+1, V_n)$}

\put(240,190){\circle*{3}} \put(245,188){$(t_1 + \ultau_1+1, T_n^2)$}

\put(240,190){\line(-1,0){50}}

\put(150,200){$(t_1 + \ultau_1, T_n^2)$}
\put(190,190){\circle*{3}}

\put(130,170){\circle*{3}}
\put(126,160){$t$}

\put(239,170){\line(1,0){2}}

\end{picture}
\end{center}

\caption{The path with extremities $(1,1)$ and $\alpha_n^{n,t}$.\label{figI_08}} 
\end{figure}

   Suppose $G$ occurs. Along the segment $[t_1 + \ultau_1, t_1 + \ultau_1 + 1] \times 
\{T^n_2\}$, $X^t(\cdot) \in [\frac{1}{2}, \frac{7}{2}] \subset [\half,10]$ by the definition of $G_1$. Along the segment $\{t_1 + \ultau_1+1\} \times [V_n, T^n_2]$,
$$
   X^t(\cdot) = X^t(\cdot) - X^t(t_1 + \ultau_1 + 1, T^n_2) 
        + X^t(t_1 + \ultau_1 + 1, T^n_2),
$$
and on $A_2(t,n)$, along this segment,
$$
   \vert X^t(\cdot) - X^t(t_1 + \ultau_1 + 1, T^n_2)\vert 
     = \vert X^t(\cdot-(1,0)) - X^t(t_1 + \ultau_1, T^n_2)\vert \leq 2
$$
by Lemma \ref{lem5.3prime}, so $X^t(\cdot) \in [1, \frac{11}{2}] \subset [\frac{1}{2},10],$ by the definition of $G_1.$

   Along the segment $\{t_1 + \ultau_1 + 1\} \times [1, V_n]$, $X^t(\cdot) \in 
[\frac{1}{2}, \frac{13}{2}] \subset [\frac{1}{2}, 10],$ by the definition of $G_2$. Along the segment $[t_1-\ultau_3, t_1 + \ultau_1 + 1] \times \{1\}$, $X^t(\cdot) = Y(\cdot) + Y_1 + Y_2$, where
$$
   Y(\cdot) = X^t(\cdot) - X^t(t_1 + \ultau_1, 1),
$$
$$
   Y_1 = X^t(t_1 + \ultau_1, 1) - X^t(t_1 + \ultau_1 + 1,1), \qquad 
   Y_2 = X^t(t_1 + \ultau_1 + 1,1).
$$
By Lemma \ref{lem5.3prime}, $\vert Y(\cdot)\vert \leq 2,$ and by definition of 
$G_1$ (resp.~$G_2$) and because $X^t$ is an ABM, $Y_1 \in [-\frac{5}{2}, -2]$ (resp.~$Y_2 \in 
[6, \frac{13}{2}]$). Therefore, $X^t(\cdot)\in [\frac{3}{2}, \frac{13}{2}] \subset 
[\frac{1}{2}, 10]$.

   Along the segment $[1, t_1 - \ultau_3] \times \{1\}$, $X^t(\cdot) \in [\frac{1}{2}, 10]$ 
by definition of $G_3$. This proves (\ref{eq6.15}).

   In order to prove (\ref{eq6.14}) with $i=1$, it suffices by (\ref{eq6.15}) 
to show that
\begin{equation}\label{eq6.16}
   P(G_3 \cap G_2 \cap G_1 \mid \cH ) \geq c_1\qquad \mbox{ on } 
      H_1 \cap A_2(t,n).
\end{equation}
Observe that $G_2 \cap G_1 \cap H_1 \cap A_2(t,n) \in \F^t_{(\ultau_1 +1,\ultau_2,\ultau_3, \ultau_4 \vee (t_2-1))}$, and on $A_2(t,n)$, the process $(X^t(t_1-\ultau_3-v_1, 1),\ v_1 \in [0, t_1-\ultau_3-1])$ is conditionally independent of this $\sigma$-field given $X^t(t_1-\ultau_3, 1),$ with conditional distribution equal to that of a Brownian motion started at 
$X^t(t_1-\ultau_3, 1) \in [\frac{3}{2}, \frac{13}{2}]$ on 
$G_2 \cap G_1 \cap H_1 \cap A_2(t,n).$ Therefore,
\begin{equation}\label{eq6.17}
   P(G_3 \cap G_2 \cap G_1 \cap H_1 \cap A_2(t,n) \mid \cH) \geq 
      c_2 P(G_2 \cap G_1 \cap H_1 \cap A_2 (t,n) \mid \cH).
\end{equation}
The process $(X^t(t_1 + \underline{\tau}_1 + 1, V_n -u_2),\ 0 \leq u_2 \leq V_n-1)$ 
is conditionally independent of $\F^t_{(\ultau_1+1,\ultau_2,\ultau_3, \ultau_4)}$ given $X^t(t_1 + \ultau_1 + 1, 
V_n)$ and $V_n$, and its conditional distribution is that of a Brownian motion 
started at $X^t(t_1 + \ultau_1 + 1, V_n) \in [1, \frac{11}{2}]$. Therefore,
\begin{equation}\label{eq6.18}
   P(G_2 \cap G_1 \cap H_1 \cap A_2(t,n) \mid \cH) \geq 
      c_3 P(G_1 \cap H_1 \cap A_2(t,n) \mid \cH).
\end{equation}
Finally, the process $(X^t(t_1 + \ultau_1 + u_1, T^n_2),\ 0 \leq u_1 \leq 1)$ is 
conditionally independent of $\cH$ given $H_1 \cap A_2(t,n),$ and its 
distribution is that of a Brownian motion started at 1. Therefore,
\begin{equation}\label{eq6.19}
   P(G_1 \cap H_1 \cap A_2(t,n) \mid \cH) \geq c_4\, 1_{H_1 \cap A_2(t,n)}.
\end{equation}
Inequalities (\ref{eq6.17}), (\ref{eq6.18}) and (\ref{eq6.19}) establish 
(\ref{eq6.16}). Together with (\ref{eq6.15}), this proves (\ref{eq6.14}) and 
completes the proof of the lemma.
\hfill $\Box$
\vskip 16pt

\end{section}
\eject

\begin{section}{ABM: Lower bound on the Hausdorff dimension}\label{sec7}

   In this section, we show that the Hausdorff dimension of the boundary of every $q$-bubble is $\geq (3 - \lambda_1)/2$ (the converse inequality was proved in Section \ref{sec4}). The proof requires several steps. The idea is to study first $\partial\cC^1_{(0,0)}$, by defining discrete measures whose supports are ``close'' to $\partial{\cal{C}}^1_{(0,0)}$, and then to pass to the limit. In Lemma \ref{lowerlem2} below, we use the second-moment argument and the estimates of Proposition \ref{lowerlem1} to get a lower bound on the probability that these discrete measures have bounded $\beta$-dimensional energy (in the sense of potential theory, see \cite[Appendix D]{davar}), for all $\beta \in \,]0,(3 - \lambda_1)/2[$. Then we study the Hausdorff dimension of $\partial \cC^{x,4x^2}_{(0,0)}$, obtaining in Proposition \ref{lbprop14} an estimate that is uniform in $x$. Finally, we establish the desired lower bound in Theorem \ref{thm7.3}, by using a convergent sequence of random points near the boundary of a given $q$-bubble of the ABM $\ti X$, and relating this boundary to the sets $\partial\cC^{x,4x^2}_{(0,0)}$ for a sequence of ABM's related to $\ti X$.

\begin{lemma} Let $A(K,t,n)$ (respectively $\ID_n$) be as defined just before Proposition \ref{lowerlem1} (respectively \ref{lembivariate}). Fix $K > 0$, $c > 0$ and $C > 0$ such that the conclusions of Proposition \ref{lowerlem1} hold. Let $\mu_n$ be the random measure on $[2, 3]^2$ defined by
$$
   \mu_n(E) = 2^{-(3-\lambda_1)n} \sum_{t \in \ID_n} \delta_t(E)\, 1_{A(K,t,n)},
$$
where $\delta_t(E) = 1$ if $t \in E$ and $\delta_t(E) = 0$ otherwise. 
For $0 < \beta < (3-\lambda_1)/2,$ there is $K_\beta < \infty$ such that for all large $n,$
$$
   P\left\{\mu_n([2, 3]^2) \in \left[\frac{c}{4}, \frac{2 C}{c}\right],\ Z_n \leq K_\beta\right\} 
        \geq \frac{c^2}{8C},
$$
where
$$
   Z_n = \int_{[2,3]^2} \int_{[2,3]^2} 
   \frac{1}{(\vert t-s \vert \vee 2^{-2n})^{\beta}}\, \mu_n(dt) \mu_n (ds).
$$
\label{lowerlem2}
\end{lemma}

\proof Set $X_n = \mu_n([2,3]^2)$. By Proposition \ref{lowerlem1}(a),
$$
   E(X_n) = 2^{-(3-\lambda_1)n} \sum_{t \in \ID_n} P(A(K,t,n)) 
     \geq c 2^{-(3-\lambda_1)n}2^{4n}2^{-(1+\lambda_1)n} = c.
$$
By Proposition \ref{lowerlem1}(b),
\begin{eqnarray*}
   E(X_n^2) &=& 2^{-2(3-\lambda_1)n} \sum_{s \in \ID_n} \sum_{t \in \ID_n} 
        P(A(K,t,n) \cap A(K,s,n)) \\
   &\leq& C\,2^{-2(3-\lambda_1)n} \sum^n_{\ell=1} \sum^\ell_{k=1} 
       \sum_{(s,t)\in \ID_n(k, \ell)} 2^{-(1+\lambda_1)n-\ell-k\lambda_1}.
\end{eqnarray*}
Use the bound card~${\ID_n}(k,\ell) \leq \frac{1}{4} 2^{4n+2\ell+ 2k}$ to see that this 
is bounded by
$$
  \frac{1}{4} C 2^{(-3+\lambda_1)n} \sum^n_{\ell=1} 2^\ell \sum^\ell_{k=1} 2^{(2-\lambda_1)k}.
$$
The sum over $k$ is bounded by $2^{(2-\lambda_1) \ell +1}$, and so 
$E(X_n^2) \leq C.$

   From the Paley-Zygmund Inequality $P\{X > \lambda E(X)\} \geq 
(1-\lambda)^2 E(X)^2/E(X^2),$ valid for non-negative random variables $X$ and 
$0 < \lambda < 1$ \cite[Section 1.6]{kahane}, we see that
$$
   P\left\{X_n > \frac{c}{4}\right\} \geq \frac{9}{16} \frac{c^2}{C} 
       \geq \frac{c^2}{2C}.
$$
By Markov's inequality,
$$
   P\left\{X_n > \frac{2C}{c}\right\} \leq \frac{E(X_n^2)}{(\frac{2C}{c})^2} 
      \leq \frac{c^2}{4C},
$$
so
$$
   P\left\{\frac{c}{4} < X_n \leq \frac{2C}{c}\right\} 
      = P\left\{X_n > \frac{c}{4}\right\}-P\left\{X_n > \frac{2C}{c}\right\} \geq \frac{c^2}{4C}.
$$
We now compute $E(Z_n)$, by proceeding as in the estimate for $E(X_n^2)$ above. Using Proposition \ref{lowerlem1}(b), this yields, for all $n$,
\begin{align*}
   E(Z_n) & \leq C \,2^{-2(3-\lambda_1)n} \sum^n_{\ell=1} \sum^\ell_{k=1} 
       \sum_{(s,t)\in \ID_n(k, \ell)} \frac{1}{2^{2(\ell - n)\beta}} 2^{-(1+\lambda_1)n-\ell-k\lambda_1} \\
       & = C 2^{(-3 + 2\beta+\lambda_1)n} \sum^n_{\ell=1} 2^{(1-2\beta)\ell} \sum^\ell_{k=1} 2^{(2-\lambda_1)k}\\
   &\leq c_\beta,
\end{align*}
for some constant $c_\beta < \infty$. 

   Take $K_\beta$ large 
enough so that $c_\beta/K_\beta < c^2/(8C).$ Then by Markov's inequality,
$P\{Z_n > K_\beta\} \leq c_\beta/K_\beta < c^2/(8C),$ and, as above,
$$
   P\left\{\frac{c}{4} < X_n < \frac{2C}{c},\ Z_n \leq K_\beta\right\} \geq P\left\{\frac{c}{2} < X_n < \frac{2C}{c}\right\} - P\{Z_n > K_\beta\} \geq \frac{c^2}{8C}.
$$
Lemma \ref{lowerlem2} is proved.
\hfill $\Box$
\vskip 16pt

   Recall the notation  $\partial{\cal{C}}_r^{x,\alpha}$ introduced before Proposition \ref{lowerlem1}.

\begin{prop} For $0 < \beta < (3-\lambda_1)/2,$ there is $c_0 > 0$ such that 
for all $x > 0,$
$$
   P\{\dim(\partial {\cal{C}}_{(0,0)}^{x,4x^2}) \geq \beta\} 
      \geq c_0.
$$
\label{lbprop14}
\end{prop}

\proof By the scaling property of Brownian motion, it suffices to set
$$
   c_0 = P\{\dim(\partial {\cal{C}}_{(0,0)}^{1,4}) \geq \beta\}
$$
and to show that $c_0 > 0$. Fix $K > 0$, $c > 0$ and $C > 0$ such that the conclusion of Proposition \ref{lowerlem1} holds, and fix 
$0 < \beta < (3-\lambda_1)/2$. Let $\mu_n$, $Z_n$, $K_\beta$ and $X_n$ be as in Lemma \ref{lowerlem2} and its proof, and set
$$
   F_n = \left\{X_n \in \left[\frac{c}{4}, \frac{2C}{c}\right],\ Z_n \leq K_\beta\right\},
$$
and $F = \limsup_{n\to\infty} F_n$. By Fatou's lemma and Lemma \ref{lowerlem2}, $P(F) \geq c^2/(8C)$, so it suffices to show that on $F$, $\dim(\partial {\cal{C}}_{(0,0)}^{1,4}) \geq \beta$.

   Fix $\omega \in F.$ Let $(n_k)$ be such that $\omega \in F_{n_k}$, for all $k \in \IN$. Because the set of measures with support in $[2,3]^2$ and with total mass in $[c/4, 2C/c]$ is weakly compact, there is a subsequence of $(n_k)$, which we again denote $(n_k)$, that converges weakly to a measure $\mu$. In view of the definition of $\mu_n$, for any $t$ in the support of $\mu_n$, the event $A(K,t,n)$ occurs, and we observed just above Proposition \ref{lowerlem1} that the definition of this event implies that $t$ is no more than $K^5 n^7 2^{-2n}$ units away from a point in $\partial {\cal{C}}_{(0,0)}^{1,4}$. Therefore, for any $\varepsilon > 0$ and for all large $n$, the support of $\mu_n$ is contained in the $\varepsilon$-enlargement of $\partial {\cal{C}}_{(0,0)}^{1,4}$. Therefore, the support of $\mu$ is 
contained in $\partial {\cal{C}}_{(0,0)}^{1,4}$.

   Fix $M > 0.$ For large $k$, on $F_{n_k}$,
$$
   \int_{[2,3]^2} \int_{[2,3]^2} (\frac{1}{\vert t-s \vert^\beta} \wedge M)\,
       \mu_{n_k} (dt) \mu_{n_k}(ds) \leq K_\beta,
$$
so the same inequality holds if $\mu_{n_k}$ is replaced by $\mu$. Now let 
$M \uparrow \infty$ and use the monotone convergence theorem to see that
$$
\int_{[2,3]^2} \int_{[2,3]^2} \frac{1}{\vert t-s \vert^\beta}\, \mu (dt) \mu(ds) \leq K_\beta < \infty.
$$
Because $\mu([2,3]^2) \in [c/4, 2C/c]$, this shows that 
$\partial{\cal{C}}_{(0,0)}^{1,4}$ has positive $\beta$-capacity. 
By Frostman's theorem \cite{F}, \cite{landkof}, 
$\dim(\partial {\cal{C}}_{(0,0)}^{1,4}) \geq \beta$. This proves the 
proposition. 
\hfill $\Box$
\vskip 16pt

   Proposition \ref{lbprop14} shows that with positive probability, $\dim(\partial {\cal{C}}_{(0,0)}^{x,4x^2}) \geq \beta$, for $\beta \in\,]0,(3-\lambda_1)/2[$. The next theorem transforms this into a statement valid with probability one, by considering a convergent sequence of distinct locations on the boundary of a bubble.

\begin{thm} Fix $q \in \IR$. A.s., the Hausdorff dimension of the boundary of every $q$-bubble of the ABM $\ti X$ is $\geq (3-\lambda_1)/2$.
\label{thm7.3}
\end{thm}

\proof Fix $q \in \IR$. It suffices to consider upwards 
$q$-bubbles. Since each such bubble contains a point $r$ with rational 
coordinates, it suffices to fix $r = (r_1, r_2)$, assume $\ti X(r) > q$ and show that $P_x$-a.s., $\dim \partial{\cal{C}}_r(q) \geq (3-\lambda_1)/2$, where ${\cal{C}}_r(q)$ denotes the component of $\{s \in \IR^2: \ti X(s) > q \}$ that contains $r$.

   Set $T_1 = \inf \{t_1 \geq r_1: \ti X(t_1, r_2) = q\}$. Then 
$T_1 < \infty$ a.s. For $\varepsilon > 0$, set
$$
   \tau_\varepsilon = \inf \{t_2 \geq r_2: \ti X(T_1, t_2) =q + \varepsilon\},
$$
and $S^\varepsilon = (S_1^\varepsilon, S^\varepsilon_2)$, where 
$S_1^\varepsilon = T_1$ and $S^\varepsilon_2 = \tau_\varepsilon.$ Because planar increments of $\ti X$ vanish, for $t_1 \in [r_1,  S_1^\varepsilon]$,
$$
   \ti X(t_1,S_2^\varepsilon) = \ti X(t_1,r_2) + \ti X(S_1^\varepsilon, S_2^\varepsilon) - \ti X(S_1^\varepsilon,r_2) = \ti X(t_1,r_2) + \varepsilon > q + \varepsilon,
$$
therefore $\ti X > q$ on $[r_1, S_1^\varepsilon] \times 
\{S_2^\varepsilon\}$, and for small $\varepsilon > 0$, $\ti X > q$ on 
$\{r_1\} \times [r_2, S^\varepsilon_2]$ by continuity of $\ti X$, so 
$S^\varepsilon \in {\cal{C}}_r(q).$

   Using the notation introduced below (\ref{startp6}), let $Y^\varepsilon =(Y^\varepsilon(t) = X^{r,(S_1^\varepsilon, S_2^\varepsilon,0,0)}(t),\ t \in \IR^2)$: this is a standard ABM that is independent of ${\cal{F}}^r_{(S_1^\varepsilon, S_2^\varepsilon,0,0)}$.  Further, a path in $\rtoo$ starting at the origin along which $\varepsilon + Y^\varepsilon > 0$ corresponds to a path starting at $S^\varepsilon$ along which $\ti X >  q$. 
Therefore, $\partial\cC_{(0,0)}^{\varepsilon,4\varepsilon^2}(Y^\varepsilon)$ (component for the process $Y^\varepsilon$) corresponds to a subset of $\partial \cC_r(q,\ti X)$ (component for the process $\ti X$).

   Fix $0 < \beta < (3-\lambda_1)/2,$ and for $\varepsilon > 0,$ set 
$G_\varepsilon =
\{ \dim \partial {\cal{C}}_{(0,0)}^{\varepsilon,4\varepsilon^2}(Y^\varepsilon) \geq \beta\}$. Let $\G^\varepsilon$ be the sigma-field generated by $(X^{r,(S_1^\varepsilon,0,0,0)}(s_1,s_2),\ 0 \leq s_1 \leq 4 \varepsilon^2,\, 0 \leq s_2 \leq S_2^\varepsilon + 4 \varepsilon^2)$, 
and observe that $G_\varepsilon \in {\cal{G}}^\varepsilon$. By the 0-1 law for the additive Brownian motion $X^{r,(S_1^\varepsilon,0,0,0)}$, $P(\limsup_{n \to \infty} G_{1/n}) \in \{0, 1\}$. Further, by Fatou's lemma,
$$
   P(\limsup_{n \to \infty} G_{1/n}) \geq \limsup_{n \to \infty} P(G_{1/n}),
$$
and by Proposition \ref{lbprop14}, $P(G_{1/n}) \geq c_0$.
Therefore, $P(\limsup_{n \to \infty} G_{1/n}) = 1$, 
and on this event, $\dim \partial {\cal{C}}_r(q) \geq \beta$. The theorem is 
proved. 
\hfill $\Box$
\vskip 16pt

\noindent{\em Proof of Theorem \ref{thm2}.} This statement is an immediate consequence of Propositions \ref{thm4} and \ref{thm7.3}.
\hfill $\Box$
\vskip 16pt


\end{section} 
\eject

\begin{section}{Upper bound on the Hausdorff dimension of the boundaries of bubbles of the Brownian sheet}\label{sec8}

   The objective of this section is to prove the following result.

\begin{prop} Fix $q \in \IR$. With probability one, the Hausdorff dimension of the boundary of each $q$-bubble of the Brownian sheet is $\leq (3-\lambda_1)/2$.
\label{bs_ubprop1}
\end{prop}

   The fact that, locally in the neighborhood of a point $t \in \IR^2$, the Brownian sheet is well-approximated by an additive Brownian motion \cite{DW0} is the basis for having the same upper bound in Propositions \ref{bs_ubprop1} and \ref{thm4}. However, in order to handle the error in this approximation, we need a variant on the DW-algorithm which terminates only upon constructing a contour on which the ABM is ``significantly negative''.
\vskip 12pt

\noindent{\em The $\delta$-DW-algorithm started at the origin with value $x_0$}
\vskip 12pt

   Fix $\delta>0$, $x_0 > 0$ and let $(\tilde{X}(t),\ t \in \IR^2)$ be a standard ABM. Set $X(t) = x_0 + \tilde{X}(t)$, $M_0 = x_0$ and $\ulT^{(0)} = (0,0,0,0).$ Begin the algorithm at Stage $i=1$.
\vskip 12pt

\noindent{\em Stage $i$.} Run the DW-algorithm for the ABM $X^{\underline{T}^{(i-1)}}$ started at (0,0) with value $M_{i-1}$, until this algorithm terminates, at stage $N^{(i)}$, after having explored $\cR(\underline{\tau}^{(i)}_{(N^{(i)})})$ and $M_{i-1} + X^{\ulT^{(i-1)}}$ has reached the maximum level $M_i$.

   Set $\underline{S}^{(i)} = \underline{T}^{(i-1)} + \underline{\tau}^{(i)}_{(N^{(i)})}.$ For $j=1, \ldots, 4,$ set $\sigma_j^{(i),0} = 0$ and, using the notation from \eqref{startp6}, for $k = 1, 2, \ldots,$
$$
\sigma_j^{(i),k} = \inf\{u > \sigma_j^{(i),k-1}: B_j^{\underline{S}^{(i)}}(u) - B_j^{\underline{S}^{(i)}}(0) \in \{-2^{k} \delta M_i, 2^{k-1} \delta M_i\} \cup [\frac{M_i}{4}, \infty[\}.
$$ 
Let $k_0$ be the smallest integer $k \geq 1$ such that either 
\begin{equation}\label{bs_ub1}
   \mbox{for all } j \in \{1, \ldots, 4\},\  B_j^{\underline{S}^{(i)}}(\sigma_j^{(i), k}) -
      B_j^{\underline{S}^{(i)}}(0) = -2^{k} \delta M_i,
\end{equation}
or
\begin{equation}\label{bs_ub2}
   \mbox{for some } j \in \{1, \ldots, 4\},\ B_j^{\underline{S}^{(i)}} (\sigma_j^{(i), k}) -
       B_j^{\underline{S}^{(i)}}(0) \geq M_i/4.
\end{equation}
Set $\underline{\sigma}^{(i)} = (\sigma_1^{(i), k_0}, \sigma_2^{(i),k_0}, \sigma_3^{(i), k_0}, \sigma_4^{(i),k_0})$ and $\underline{T}^{(i)} = \underline{S}^{(i)} + \underline{\sigma}^{(i)}$. If (\ref{bs_ub1}) occurs, set $I=i$ and the $\delta$-DW-algorithm terminates. Otherwise, it proceeds to Stage $i+1$.
\vskip 12pt

   The next lemma shows in particular that the $\delta$-DW-algorithm terminates after a random finite number $I$ of stages.

\begin{lemma} (a) On $\partial {\cal{R}}(\underline{T}^{(I)})$, $x_0 + \tilde{X}(\cdot) \leq - \delta \sup_{t \in {\cal{R}}(\underline{T}^{(I)})} (x_0 + \tilde{X}(t))$.

   (b) The conditional probability $P\{I = i \mid {\cal{F}}_{\underline{T}^{(i-1)}}\}$ is equal to $c(\delta) \, 1_{\{I \geq i \}}$, where $c(\delta)$ is deterministic, does not depend on $i \geq 1$, and
$\lim_{\delta \downarrow 0} c(\delta) = 1.$
\label{bs_ub_lem2}
\end{lemma}

\begin{remark} {\rm Lemma \ref{bs_ub_lem2}(a) states that when the $\delta$-DW-algorithm terminates, it has constructed a rectangle on the boundary of which $x_0 + \tilde{X}$ is less than $(-\delta)$ times the maximum value of $x_0 + \tilde{X}$ in this rectangle.

   Lemma \ref{bs_ub_lem2}(b) states that the conditional probability that the $\delta$-DW-algorithm terminates during a particular stage, given that it has not previously terminated, does not depend on the stage and increases to 1 as $\delta \downarrow 0$. In other words, for $\delta$ small, the $\delta$-DW-algorithm is unlikely to run for more than one stage and hence is an approximation of the DW-algorithm.
}
\end{remark}

\noindent{\em Proof of Lemma \ref{bs_ub_lem2}.} (a) Fix $i \geq 1$. By construction, when the $\delta$-DW-algorithm terminates at stage $I=i$, (\ref{bs_ub1}) occurs for some $k_0 \geq 1$, and therefore for $j=1, \ldots, 4,$
$$
   B_j^{\underline{S}^{(i)}}(\sigma_j^{(i),k_0}) - B_j^{\underline{S}^{(i)}}(0) \leq - \delta M_i.
$$
By Proposition \ref{prop2}(a), 
$$
   M_{i-1} + X^{\ulT^{(i-1)}}(\cdot) \leq 0 \qquad\mbox{ \ on \ } \partial
       {\cal{R}}(\underline{\tau}^{(i)}_{(N^{(i)})} ).
$$
We want to show that
\begin{equation}\label{e8.3}
   M_{i-1} + X^{\ulT^{(i-1)}}(\cdot) \leq - \delta M_i \qquad\mbox{ \ on \ } \partial
       {\cal{R}}(\underline{\tau}^{(i)}_{(N^{(i)})} + \underline{\sigma}^{(i)}).
\end{equation}
Assume that $(\ti S_1^{(i)}, \ti S_2^{(i)})$ is the unique point in ${\cal{R}}(\ultau^{(i)}_{(N^{(i)})})$ at which $M_{i-1} + X^{\ulT^{(i-1)}}(\cdot)$ is equal to $M_i$, and suppose that ${\cal{R}}(\ultau^{(i)}_{(N^{(i)})}) = [U^{(i)}, U'^{(i)}] \times [V^{(i)}, V'^{(i)}]$. As explained in \eqref{rectincr}, $M_{i-1} + X^{\ulT^{(i-1)}}(\cdot)$ is positive on the union of the two segments $]U^{(i)}, U'^{(i)}[\times \{\ti S_2^{(i)}\}$ and $\{\ti S_1^{(i)}\} \times \, ]V^{(i)}, V'^{(i)}[$, and
\begin{align*}
  X^{\ulT^{(i-1)}}(U^{(i)},\ti S_2^{(i)}) &= X^{\ulT^{(i-1)}}(U'^{(i)},\ti S_2^{(i)}) \\
	&= X^{\ulT^{(i-1)}}(\ti S_1^{(i)},V^{(i)}) = X^{\ulT^{(i-1)}}(\ti S_1^{(i)},V'^{(i)}) = - M_{i-1}.
\end{align*}
The set $\partial{\cal{R}}(\underline{\tau}^{(i)}_{(N^{(i)})} + \underline{\sigma}^{(i)})$ contains the two segments  $\{U'^{(i)} + \ulsigma_1^{(i)} \} \times [\ti S_2^{(i)},V'^{(i)}]$ and $\{U'^{(i)} + \sigma_1^{(i)} \} \times [V'^{(i)}, V'^{(i)} + \ulsigma_2^{(i)}]$, as well as 14 other segments, each of which is handled similarly to one of these two. For $u_2 \in [\ti S_2^{(i)},V'^{(i)}]$,
\begin{align*}
   M_{i-1} + X^{\ulT^{(i-1)}}(U'^{(i)} + \ulsigma_1^{(i)}, u_2) &=  M_{i-1} + X^{\ulT^{(i-1)}}(U'^{(i)}, u_2) + B_1^{\ulS^{(i)}}(\ulsigma_1^{(i)}) - B_1^{\ulS^{(i)}}(0) \\
	&\leq 0 - 2^{k_0} \delta M_i \leq - \delta M_i.
\end{align*}
For $u_2 \in [V'^{(i)}, V'^{(i)} + \ulsigma_2^{(i)}]$,
\begin{align*}
   M_{i-1} + X^{\ulT^{(i-1)}}(U'^{(i)} + \ulsigma_1^{(i)}, u_2) &=  M_{i-1} + X^{\ulT^{(i-1)}}(U'^{(i)}, V'^{(i)}) \\
	  &\qquad + B_1^{\ulS^{(i)}}(\ulsigma_1^{(i)}) - B_1^{\ulS^{(i)}}(0) + B_2^{\ulS^{(i)}}(u_2 - V'^{(i)}) - B_2^{\ulS^{(i)}}(0) \\
		&\leq M_{i-1} + X^{\ulT^{(i-1)}}(U'^{(i)}, V'^{(i)}) - 2^{k_0} \delta M_i + 2^{k_0-1} \delta M_i \\
		&= M_{i-1} + X^{\ulT^{(i-1)}}(U'^{(i)}, \ti S_2^{(i)}) + X^{\ulT^{(i-1)}}(\ti S_1^{(i)} , V'^{(i)}) \\
		&\qquad - X^{\ulT^{(i-1)}}(\ti S_1^{(i)} ,\ti S_2^{(i)}) - 2^{k_0-1} \delta M_i \\
		&= M_{i-1} - M_{i-1} - M_{i-1} -(M_i - M_{i-1}) - 2^{k_0-1} \delta M_i \\
		&= - M_i - 2^{k_0-1} \delta M_i \leq - \delta M_i.
\end{align*}
This establishes \eqref{e8.3}.

We now check that
\begin{equation}\label{rd_ub3}
   M_i = \sup_{t \in {\cal{R}}(\underline{\tau}^{(i)}_{(N^{(i)})})} (M_{i-1} 
   + X^{\underline{T}^{(i-1)}}(t)) = M_{i-1} + \sup_{t \in {\cal{R}}(\tau^{(i)}_{(N^{(i)})} 
   + \underline{\sigma}^{(i)})} X^{\underline{T}^{(i-1)}}(t). 
\end{equation}
Indeed, for $ t \in {\cal{R}}(\tau^{(i)}_{(N^{(i)})} + \underline{\sigma}^{(i)}) \setminus {\cal{R}}(\tau^{(i)}_{(N^{(i)})})$, suppose that $t_1 \in [U'^{(i)}, U'^{(i)} + \ulsigma_1^{(i)}]$ and $t_2 \in [\ti S_2^{(i)}, V'^{(i)}]$. Then
\begin{align*}
  M_{i-1} + X^{\ulT^{(i-1)}}(t) &= M_{i-1} + X^{\ulT^{(i-1)}}(U'^{(i)}, t_2) + B_1^{\ulS^{(i)}}(t_1 - U'^{(i)}) - B_1^{\ulS^{(i)}}(0) \\
	&\leq M_{i-1} + X^{\ulT^{(i-1)}}(U'^{(i)}, t_2) + \frac{M_i}{4}\\
	&=  M_{i-1} + X^{\ulT^{(i-1)}}(U'^{(i)}, \ti S_2^{(i)}) \\
	&\qquad + X^{\ulT^{(i-1)}}(\ti S_1^{(i)}, t_2) - X^{\ulT^{(i-1)}}(\ti S_1^{(i)}, \ti S_2^{(i)}) + \frac{M_i}{4}\\
	&\leq M_{i-1} - M_{i-1} + 0 + \frac{M_i}{4}\\
	&= \frac{M_i}{4},
\end{align*}
and for $t_1 \in [U'^{(i)}, U'^{(i)} + \ulsigma_1^{(i)}]$ and $t_2 \in [V'^{(i)}, V'^{(i)} + \ulsigma_2^{(i)}]$,
\begin{align*}
  M_{i-1} + X^{\ulT^{(i-1)}}(t) &= M_{i-1} + X^{\ulT^{(i-1)}}(U'^{(i)}, V'^{(i)}) \\
	 & \qquad + B_1^{\ulS^{(i)}}(t_1 - U'^{(i)}) - B_1^{\ulS^{(i)}}(0) + B_2^{\ulS^{(i)}}(t_2 - V'^{(i)}) - B_2^{\ulS^{(i)}}(0) \\
	&\leq M_{i-1} + X^{\ulT^{(i-1)}}(U'^{(i)}, V'^{(i)}) + \frac{M_i}{4} + \frac{M_i}{4} \\
	&= - M_i + \frac{M_i}{2} = - \frac{M_i}{2}.
\end{align*}
All other possibilities for $t \in {\cal{R}}(\tau^{(i)}_{(N^{(i)})} + \underline{\sigma}^{(i)}) \setminus {\cal{R}}(\tau^{(i)}_{(N^{(i)})})$ are treated similarly. This proves \eqref{rd_ub3}.

   For $i \geq 1$, assume by induction that $M_{i-1} \geq x_0 + \sup_{t\in{\cal{R}}(\underline{T}^{(i-1)})} \tilde{X}(t)$ (this clearly holds for $i=1$). Then by (\ref{rd_ub3}),
\begin{eqnarray*} 
   M_i &\geq& x_0 + \sup_{t \in {\cal{R}}(\ulT^{(i-1)})} \tilde{X}(t) 
      + \sup_{t \in {\cal{R}}(\underline{\tau}_{(N^{(i)})}^{(i)}
      + \underline{\sigma}^{(i)})} X^{\underline{T}^{(i-1)}}(t)\\
    &\geq& x_0 + \sup_{t \in {\cal{R}}(\underline{T}^{(i)})} \tilde{X}(t).
\end{eqnarray*}
Together with \eqref{e8.3} above, this proves (a).

   (b) Notice that $2^{k-1} \delta M_i < \frac{M_i}{4}$ if and only if $k < \log_2(1/\delta) - 1$. Given that the $\delta$-DW-algorithm has not terminated before beginning Stage $i$ and given ${\cal{F}}_{\underline{T}^{(i-1)}}$, the conditional probability that it does not terminate during Stage $i$ is simply the probability that (\ref{bs_ub2}) occurs before (\ref{bs_ub1}), that is, for $k=1, \ldots, [\log_2(1/\delta)],$ (\ref{bs_ub1}) does not occur.
	
	For a standard Brownian motion $B$, the probability, starting from 0, of hitting $-2 \delta M_i$ before $\delta M_i$ is $\frac{1}{3} \geq \frac{1}{6}$. For $k \geq 2$, 
\begin{align*}
  & P_{- 2^{k-1} \delta M_i}\{B\mbox{ hits } -2^{k} \delta M_i \mbox{ before } 2^{k-1} \delta M_i \} \\
	&\qquad \geq P_{2^{k-2} \delta M_i}\{B\mbox{ hits } -2^{k} \delta M_i \mbox{ before } 2^{k-1} \delta M_i \} \\
	&\qquad = \frac{2^{k-1} - 2^{k-2}}{2^{k-1} + 2^{k}} = \frac{1}{6}.
\end{align*}
Therefore, given that (\ref{bs_ub1}) has not occured for $1, \ldots, k-1,$ the probability that it does not occur for $k$ is bounded above by $1-6^{-4}$, so (\ref{bs_ub2}) occurs before (\ref{bs_ub1}) with probability $\leq c (1-6^{-4})^{[\log_2(1/\delta)]}$, which converges to 0 as $\delta \downarrow 0$. 
\hfill $\Box$
\vskip 16pt

   In the next lemma, we examine the probability that the maximum level reached before the $\delta$-DW-algorithm terminates exceeds a given level.
	
\begin{lemma} For the $\delta$-DW-algorithm started at the origin with value $x_0 = 1$, set $M(\delta) = 1 + \sup_{t \in {\cal{R}}(\ulT^{(I)})} \tilde{X}(t)$. For all $\varepsilon > 0,$ there are $K < \infty$ and $\delta_0 > 0$ such that for $0 < \delta < \delta_0$ and $ x \geq 1$, $P\{M(\delta) > x\} \leq K x^{-(\lambda_1-\varepsilon)}.$
\label{bs_ublem3}
\end{lemma}

\proof Fix $\varepsilon > 0.$ It suffices to find $K < \infty$ and $\delta_0 > 0$ such that for $0 < \delta < \delta_0$ and all integers $n \geq 1,$
$$
   P\{M(\delta) > 2^n\} \leq K 2^{-n(\lambda_1-\varepsilon)}.
$$
Set
$$
  K_n = \sup_{1 \leq m \leq n} 2^{m(\lambda_1-\varepsilon)} P\{M(\delta) \geq 2^m\} \qquad (< + \infty).
$$
We shall show that there is $\delta_0 > 0$ such that for $n$ sufficiently large and for $0 < \delta < \delta_0$, $K_n \leq \max(1, K_{n-1}),$ which implies that $K = \sup_n K_n < \infty,$ and will prove the lemma.

   Let $M_1(\delta)$ be the maximum value attained by $X(\cdot)$ during Stage 1 of the $\delta$-DW-algorithm, which is simply the maximum value achieved by the DW-algorithm for $X$, started with value $1$, upon termination. Observe that by the scaling property of Brownian motion, for $n \geq m$,
$$
   P\{M(\delta) \geq 2^n \mid M_1(\delta) \in [2^{m-1}, 2^m[\} \leq (1-c(\delta)) P\{M(\delta) 
      \geq 2^{n-m}\},
$$
where $c(\delta)$ is the constant of Lemma \ref{bs_ub_lem2}(b). Therefore,
\begin{eqnarray*}
   P\{M(\delta) \geq 2^n\} &\leq& \sum_{m=1}^{n-1} P\{M_1(\delta) \in [2^{m-1}, 2^m[\}\,(1-c(\delta))\,
     P\{M(\delta) \geq 2^{n-m}\}\\
     &&\qquad +\, P\{M_1(\delta) \geq 2^{n-1}\}.
\end{eqnarray*}
By Theorem \ref{prop1}, this is bounded above by
$$
    \sum_{m=1}^{n-1}C \ 2^{-\lambda_1(m-1)} (1-c(\delta)) P\{M(\delta) \geq 2^{n-m}\} 
      + C \ 2^{-\lambda_1(n-1)}
$$
By the definition of $K_{n-1}$, this is bounded above by
\begin{eqnarray*}   
   && C(1-c(\delta)) K_{n-1} \displaystyle\sum_{m=1}^{n-1} 2^{-\lambda_1(m-1)}
      2^{-(n-m)(\lambda_1-\varepsilon)} + C \ 2^{-\lambda_1(n-1)}\\
   &&\qquad \leq C \ 2^{\lambda_1} 2^{-n(\lambda_1-\varepsilon)} \left((1-c(\delta)) K_{n-1}
       \sum_{m=1}^{n-1} 2^{-m \varepsilon} + 2^{-n \varepsilon}\right)\\
   &&\qquad \leq C \ 2^{\lambda_1} 2^{-n(\lambda_1-\varepsilon)}
       \left(\frac{1-c(\delta)}{1-2^{-\varepsilon}} K_{n-1} + 2^{-n \varepsilon}\right).
\end{eqnarray*}
This is bounded above by $2^{-n(\lambda_1-\varepsilon)} \max(K_{n-1}, 1)$ provided $\delta$ is sufficiently small and $n$ is large enough so that
$$
   C \ 2^{\lambda_1} \frac{1-c(\delta)}{1-2^{-\varepsilon}} \leq \frac{1}{2} 
   \qquad\mbox{ and } \qquad C \ 2^{\lambda_1} 2^{-n \varepsilon} \leq \frac{1}{2}.
$$
This proves the lemma. 
\hfill $\Box$
\vskip 16pt

   We now want to obtain bounds on escape probabilities for the $\delta$-DW-algorithm. We begin with the following lemma.

\begin{lemma} For the $\delta$-DW-algorithm started at $(0,0)$ with value $x_0 = 1$, let $\cR(\ulS^{(1)})$ be the rectangle explored by the DW-algorithm during Stage 1 of the $\delta$-DW-algorithm and let $M_1$ be the maximum level reached by $1 + X^{\ulT^{(0)}}$ in this rectangle. Define $\ulsigma^{(1)}$ as in \eqref{bs_ub1} and \eqref{bs_ub2}, with $i=1$. There are $c>0$ and $C < \infty$ such that, for all $m \geq 0$ and  $y\geq 0$, on the event $\{ M_1 \leq 2^{m+1}\}$,
$$
   P\{\Vert \ulsigma^{(1)} \Vert \geq y \mid \F_{\ulS^{(1)}} \} \leq C \exp(-cy2^{-2m}).
$$
\label{lem8.5}
\end{lemma}

\proof Set $k_1 = [\log_2(1/\delta)] -1$, and for $j=1,\dots,4$, set
$$
   R_j = \inf\{u>0: B_j^{\underline{S}^{(1)}}(u) - B_j^{\underline{S}^{(1)}}(0) \not\in [-2^{k_1+1}\delta M_1, M_1/4] \}.
$$
We observed in the proof of Lemma \ref{bs_ub_lem2}(b) that $2^{k_1-1} \delta M_1 < M_1 /4 \leq 2^{k_1} \delta M_1$, so $M_1 / 2 \leq 2^{k_1+1}\delta M_1  < M_1$. Let $k_0$ be defined as in \eqref{bs_ub1} and \eqref{bs_ub2} (with $i=1$). 

   We claim that for $j=1,\dots,4$, $\sigma_j^{(1),k_0} \leq R_j$. Indeed, if there is $j \in \{ 1,\dots,4\}$ such that $B_j^{\underline{S}^{(1)}}(u) - B_j^{\underline{S}^{(1)}}(0) = M_1 /4$, then $\sigma_j^{(1),k_0} \leq R_j$ by the definition of $\sigma_j^{(1),k_0}$. If, for all $j \in \{ 1,\dots,4\}$, $B_j^{\underline{S}^{(1)}}(u) - B_j^{\underline{S}^{(1)}}(0) = -2^{k_1+1}\delta M_1$, then \eqref{bs_ub1} occurs with $k= k_1 +1$, so $k_0 \leq k_1+1$ and $\sigma_j^{(1),k_0} \leq R_j$.
	
	Let
$$
   \ti R_j = \inf\{u>0: B_j^{\underline{S}^{(1)}}(u) - B_j^{\underline{S}^{(1)}}(0) \not\in [- M_1, M_1/4] \}.
$$
Then $\ti R_j \geq R_j \geq \sigma_j^{(1),k_0}$, so
$$
   P\{ \Vert \ulsigma^{(1)} \Vert \geq y \mid \F_{\ulS^{(1)}} \} \leq P\{\ti R_1 + \cdots + \ti R_4 \geq y \mid \F_{\ulS^{(1)}} \}.
$$
By the scaling property of Brownian motion, given $M_1 = x$, the conditional law of $\ti R_j$ is the same as that of $x^2 R_j'$, where $R_j'$ is independent of $\F_{\ulS^{(1)}}$ and has the same law as the first exit time of $[-1,1/4]$ by a standard Brownian motion. Therefore, on $\{M_1 \leq 2^{m+1} \}$,
$$
   P\{ \Vert \ulsigma^{(1)} \Vert \geq y \mid \F_{\ulS^{(1)}} \} \leq P\{\ti R_1' + \cdots \ti R_4' \geq y 2^{-2m-2}\}.
$$
By standard results on Brownian motion (see e.g.~\cite[Section 1]{LS}), this is no greater than $C \exp(-cy2^{-2m-2})$, and Lemma \ref{lem8.5} is proved.
\hfill $\Box$
\vskip 16pt

\begin{lemma} Under the same assumption as in Lemma \ref{lem8.5}, let $\ulT^{(1)} = \ulS^{(1)} + \ulsigma^{(1)}$. Then there exist $c>0$ and $C < \infty$ such that, for all non-negative integers $m \leq k$,
$$
   P\{M_1 \in \,]2^m,2^{m+1}], \ \cR(\ulT^{(1)}) \not\subset \cR(2^{2k})  \} \leq C 2^{-m\lambda_1} \exp(-c 2^{k-m}).
$$
\label{lem8.6}
\end{lemma}

\proof For $m=k$, the conclusion follows from Theorem \ref{prop1}, so we assume that $m < k$. We note that $M_1$ is $\F_{\ulS^{(1)}}$-measurable and is also the maximum of $1 + X^{\ulT^{(0)}}$ over $\cR(\ulT^{(1)})$. Observe that the event in the statement of the lemma is contained in 
$$
   \cup_{\ell = m}^{k-1} G_{1,\ell} \cup G_2 ,
$$
where
$$
   G_{1,m} = \{M_1 \in \,]2^m,2^{m+1}], \  \cR(\ulS^{(1)}) \subset \cR(2^{2m}),\ \cR(\ulT^{(1)}) \not\subset \cR(2^{2k}) \},
$$
for $m+1 \leq \ell \leq k-1$,
$$
   G_{1,\ell} = \{M_1 \in \,]2^m,2^{m+1}], \  \cR(\ulS^{(1)}) \not\subset \cR(2^{2(\ell - 1)}),\ \cR(\ulS^{(1)}) \subset \cR(2^{2\ell}),\ \cR(\ulT^{(1)}) \not\subset \cR(2^{2k}) \},
$$
and
$$
  G_2 = \{M_1 \in \,]2^m,2^{m+1}], \  \cR(\ulS^{(1)}) \not\subset \cR(2^{2(k - 1)}) \}.
$$
On $G_{1,\ell}$ ($m\leq \ell \leq k-1$), $\Vert \ulsigma^{(1)} \Vert \geq 2^{2k} - 2^{2\ell} \geq 2^{2k-1}$, so by Lemma \ref{lem8.5},
\begin{align*}
  & P(\cup_{\ell = m}^{k-1} G_{1,\ell} \cup G_2) \\
	&\qquad \leq P\{M_1 \in \,]2^m,2^{m+1}] \}\, C \exp(-c2^{2k} 2^{-2m}) \\
	&\qquad\qquad + \sum_{\ell = m+1}^{k-1} P\{M_1 \in \,]2^m,2^{m+1}],\ \cR(\ulS^{(1)}) \not\subset \cR(2^{2(\ell - 1)})\}\, C \exp(-c2^{2k} 2^{-2m}) \\
	&\qquad\qquad + P(G_2).
\end{align*}
By Theorem \ref{prop1}, $P\{M_1 \in \,]2^m,2^{m+1}] \} \leq c 2^{-m\lambda_1}$, by Proposition \ref{bs_ubprop6}, 
$$
   P(G_2) \leq K 2^{-(k-1)\lambda_1} \exp(-c2^{k-1-(m+1)}), 
$$
and for $m+1 \leq \ell \leq k-1$,
$$
   P\{M_1 \in \,]2^m,2^{m+1}],\ \cR(\ulS^{(1)}) \not\subset \cR(2^{2(\ell - 1)})\} \leq K 2^{-(\ell-1)\lambda_1} \exp(-c 2^{\ell-1-(m+1)}) .
$$
Therefore,
\begin{align*}
   & P\{M_1 \in \,]2^m,2^{m+1}], \ \cR(\ulT^{(1)}) \not\subset \cR(2^{2k})  \} \\
	& \qquad \leq C 2^{-m\lambda_1} \exp(-c2^{2k-2m}) + \sum_{\ell=m+1}^{k-1} C 2^{-m\lambda_1} \exp(-c 2^{\ell-m}) \exp(-c2^{2k-2m})\\
	& \qquad \qquad + K 2^{-m\lambda_1} \exp(-c2^{k-m}).
\end{align*}
Since $\sum_{\ell=m+1}^\infty \exp(-c 2^{\ell-m}) = \sum_{\ell=1}^\infty \exp(-c 2^{\ell}) < \infty$, Lemma \ref{lem8.6} is proved.
\hfill $\Box$
\vskip 16pt

The next lemma contains the results on escape probabilities of the $\delta$-DW-algorithm that we have been aiming for.

\begin{lemma} For the $\delta$-DW-algorithm started at $(0,0)$ with value $x_0 = 1$, let ${\cal{R}}(\ulT^{(I)})$ be the rectangle constructed during the terminal stage of the algorithm. For all $\varepsilon > 0,$ there are $K < \infty$ and $\delta_0 > 0$ such that for $ 0 < \delta \leq \delta_0$ and $x \geq 1,$
$$
   P\{{\cal{R}}(\underline{T}^{(I)}) \not\subset {\cal{R}}(x)\} 
      \leq K \ x^{-(\lambda_1-\varepsilon)/2}.
$$ 
\label{bs_ub_lem4}
\end{lemma}

\proof Fix $\varepsilon > 0$. As in the proof of Lemma \ref{bs_ublem3}, it suffices, by monotonicity in $x$, to find $K < \infty$ and $\delta_0 > 0$ such that for $0 < \delta < \delta_0$ and all integers $n \geq 1,$
$$
   P\{{\cal{R}}(\underline{T}^{(I)}) \not\subset {\cal{R}}(2^{2n})\} 
      \leq K \ 2^{-n(\lambda_1-\varepsilon)}.
$$
Set
$$
   K_n = \sup_{1 \leq m \leq n} 2^{m(\lambda_1-\varepsilon)} P\{{\cal{R}}(\underline{T}^{(I)}) 
      \not\subset {\cal{R}}(2^{2m})\} .
$$
As in the proof of Lemma \ref{bs_ublem3}, we shall show that there is $\delta_0 > 0$ such that for $n$ sufficiently large and for $0 < \delta < \delta_0$, $K_n \leq \max(1, K_{n-1}).$

   Let $M_1(\delta)$ be the maximum value attained by $1+X(\cdot)$ during stage 1 of the $\delta$-DW-algorithm. We decompose the event $\{{\cal{R}}(\underline{T}^{(I)}) \not\subset {\cal{R}}(2^{2m})\}$ according to the values of $M_1(\delta)$ and the position where this value is attained:
\begin{eqnarray*}
\lefteqn
   {\{{\cal{R}}(\underline{T}^{(I)}) \not\subset {\cal{R}}(2^{2n})\}}\\
   &&\subset \bigcup^{n-1}_{m=1} \left(\{M_1(\delta) \in\,]2^{m-1}, 2^{m}],
        \ \cR(\ulT^{(1)}) \subset \cR(2^{2m}),
       \  {\cal{R}}(\underline{T}^{(I)}) \not\subset {\cal{R}}(2^{2n})\}\right.\\
  \\
  && \qquad\qquad\qquad \cup \ \bigcup_{k=m}^{n-2}  \{M_1(\delta) \in\,]2^{m-1}, 2^{m}],
     \ \cR(\ulT^{(1)}) \not\subset \cR(2^{2k}),\\
  && \qquad\qquad\qquad\qquad\qquad
      \cR(\ulT^{(1)}) \subset {\cal{R}}(2^{2(k+1)}),
     \ {\cal{R}}(\underline{T}^{(I)}) \not\subset {\cal{R}}(2^{2n})\} \\
     \\
  && \qquad\qquad\qquad \cup\ \{M_1(\delta) \in\,]2^{m-1}, 2^{m}],
     \ \cR(\ulT^{(1)}) \not\subset {\cal{R}}(2^{2(n-1)}), \\
     && \qquad\qquad\qquad\qquad\qquad\qquad\qquad\qquad\qquad \qquad\qquad \ \left. \ {\cal{R}}(\underline{T}^{(I)}) \not\subset {\cal{R}}(2^{2n}) \}  \right)\\
   \\
   && \qquad \cup\ \{M_1(\delta) \geq 2^{n-1},\ {\cal{R}}(\underline{T}^{(I)}) 
       \not\subset {\cal{R}}(2^{2n})\}.
\end{eqnarray*}
By Lemma \ref{lem8.6}, Theorem \ref{prop1} and the definition of $K_{n-1}$, the probability of this event is bounded above by
\begin{eqnarray*}
  \lefteqn
     {\sum_{m=1}^{n-1} \Big(C \ 2^{-m\lambda_1}(1-c(\delta)) K_{n-1}
         \left(\frac{2^{2n}-2^{2m}}{2^{2m}}\right)^{-(\lambda_1-\varepsilon)/2}} \\
   \\
   &&\quad + \sum_{k=m}^{n-2} \left[C \ 2^{-m\lambda_1} e^{-c 2^{k-m}} (1-c(\delta)) K_{n-1}
       \left(\frac{2^{2n}-2^{2(k+1)}}{2^{2m}}\right)^{-(\lambda_1-\varepsilon)/2}\right] \\
   &&\quad +\, C \ 2^{-m\lambda_1} e^{-c2^{n-m}}\Big) +\, C \ 2^{-n\lambda_1}\\
   \\
   &\leq& 2^{-n(\lambda_1-\varepsilon)} \Big[K_{n-1} (1-c(\delta)) C 
      \sum_{m=1}^{n-1}\Big( 2^{-m \varepsilon}
       \left(1-2^{2(m-n)}\right)^{-(\lambda_1-\varepsilon)/2}  \\
   \\
   &&\qquad \qquad\qquad \qquad +\, 2^{-m \varepsilon} \sum^{n-2}_{k=m}   
       e^{-c 2^{k-m}}(1-2^{2(k+1-n)})^{-(\lambda_1-\varepsilon)/2} \Big) \\
	 && \qquad \qquad\qquad +\, C\, 2^{-n \varepsilon} \sum_{m=1}^{n-1} 2^{(n-m)\lambda_1} e^{-c2^{n-m}}  
	     + C \ 2^{-n \varepsilon}\Big].
\end{eqnarray*}
Notice that for $m \leq n-1$,
$$
  (1- 2^{2(m-n)})^{-(\lambda_1-\varepsilon)/2} \leq (3/4)^{-(\lambda_1-\varepsilon)/2}
$$
and for $ k \leq n-2$,
$$
   (1-2^{2(k+1-n)})^{-(\lambda_1-\varepsilon)/2} \leq (3/4)^{-(\lambda_1-\varepsilon)/2}.
$$ 
Since $\sum_{m=1}^\infty 2^{-m \varepsilon} = C_\varepsilon < \infty$, this expression is bounded above by
$$
   2^{-n(\lambda_1-\varepsilon)}\left(K_{n-1} (1-c(\delta)) \, \tilde C_\varepsilon
  + \tilde C\, 2^{-n \varepsilon}\right),
$$
and this is bounded above by $2^{-n(\lambda_1-\varepsilon)} \max(K_{n-1}, 1)$ provided $\delta$ is sufficiently small and $n$ is large enough so that 
$$
   \tilde C_\varepsilon (1-c(\delta))  \leq \frac{1}{2}
    \qquad \mbox{ and }\qquad
    \tilde C \ 2^{-n\varepsilon}  \leq \frac{1}{2}.
$$ 
This proves the lemma. 
\hfill $\Box$
\vskip 16pt

\begin{lemma} \label{bs_ub_lem5} Let $\cR(\underline{T}^{(I)})$ be as in Lemma \ref{bs_ub_lem4} and let
$M(\delta)$ be as in Lemma \ref{bs_ublem3}. Define
$$
  G_n = \{\cR(\underline{T}^{(I)}) \subset \cR(2^{-n \varepsilon/10})\}, \qquad
  D_n = \left\{M(\delta) > \frac{1}{n} \max_{1 \leq j \leq 4} \sqrt{\underline{T}_j^{(I)}}\right\}.
$$
Then there exist $K < \infty$ and $c > 0$ such that for large $n$,
$$
  P_{n^2 2^{-n}} (G_n \cap D_n^c ) \leq K \ n e^{-cn^2}
$$
  (in other words, for an ABM started at value $n^2 2^{-n},$ the probability that the $\delta$-DW-algorithm does not escape $\cR(2^{-n \varepsilon/10})$ and the maximum value attained during the $\delta$-DW-algorithm is small relative to the size of the rectangle explored is exponentially small).
\end{lemma}
  
\proof The event $G_n \cap D_n^c$ is contained in 
$$
   \bigcup_{\log_2 n + n \epsilon/20 \leq k \leq n - 2 \log_2 n} V(k,n),
$$
where
$$
   V(k,n) = \{ 2^{-k} \leq M(\delta) < 2^{1-k}\} \cap \{ \cR(\ulT^{(I)}) \not\subset [-n^2 2^{-2k}, n^2 2^{-2k}]^2\}
$$
(note that for $k= \log_2 n + n \epsilon/20$, we have $n^2 2^{-2k} = 2^{-n \varepsilon/10}$, and for $k= n - 2 \log_2 n$, we have $2^{-k} = n^2 2^{-n}$). Accordingly, it suffices to show that for each such $k$, $P_{n^2 2^{-n}}(V(k,n)) \leq C e^{-cn^2}$ for constants $c$ and $C$ not depending on $n$ or $k$. Let $R_{i,k} = [-i2^{-2k}, i2^{-2k}]^2$ and $\sigma_{i,k} = \ulsigma^{n^2 2^{-n},0,R_{i,k}}$. The desired inequality follows directly from the easily established consequence of Lemma \ref{lem5.2prime}:
$$
   P(\cR(\ulT^{(I)}) \subset R_{i+1,k} \mid \F_{\sigma_{i,k}}) > \tilde c >0,
$$
on the set $\{\sigma_{i,k} < \ultau^{n^2 2^{-n},0,2^{1-k}}\}$, for some universal constant $c$.
%
\hfill $\Box$
\vskip 16pt

   We are going to describe a local decomposition of the Brownian sheet in terms of a standard ABM and an error term, following \cite{DW0}. Let
\begin{equation}\label{08_08_14_1}
  \begin{array}{rlll}
  B_1(u_1) &=& W(1+u_1, 1) - W(1,1), & u_1 \geq 0,\\
  B_2(u_2) &=& W(1,1+u_2) - W(1,1),  & u_2 \geq 0,\\
  B_3(-u_1) &=& (1-u_1) W\left( \frac{1}{1-u_1}, 1\right) - W(1,1),  & u_1 \leq 0\\
  B_4(-u_2) &=& (1-u_2) W\left(1, \frac{1}{1-u_2}\right) - W(1,1), & u_2 \leq 0.
  \end{array}
\end{equation}
  Then $B_i(\cdot), i = 1, 2, 3, 4,$ are independent standard Brownian motions that are independent of $W(1,1)$. Let
$$
  Q_1 = \IR_+ \times \IR_+, \quad Q_2 = \IR_- \times \IR_+, \quad Q_3 = \IR_- \times \IR_-, \quad Q_4 = \IR_+ \times \IR_-
$$
be the four quadrants in $\IR^2$, and, using the notation for rectangular increments introduced in \eqref{defdelta}, let
\begin{eqnarray*}
 \E_1(u_1, u_2) &=& \Delta_{]1, 1+u_1]\times ]1,1+u_2]} W, {\hskip 6.5 cm} (u_1, u_2) \in Q_1,\\
 \E_2(u_1, u_2) &=& \Delta_{](1-u_1)^{-1},1] \times ]1, 1+u_2]} W + u_1 W\left( \frac{1}{1-u_1}, 1\right), {\hskip 2.2 cm} (u_1, u_2) \in Q_2,\\
 \E_3(u_1, u_2) &=& \Delta_{](1-u_1)^{-1}, 1] \times ](1-u_2)^{-1}, 1]} W + u_1 W\left(\frac{1}{1-u_1}, 1\right) +  u_2 W\left(1, \frac{1}{1-u_2}\right),\\
    && {\hskip 9.9 cm} (u_1, u_2) \in Q_3,\\
 \E_4(u_1, u_2) &=& \Delta_{]1, 1+u_1] \times ](1-u_2)^{-1}, 1]} W + u_2 W\left(1, \frac{1}{1-u_2}\right),{\hskip 2.2 cm} (u_1, u_2) \in Q_4,
\end{eqnarray*}
and for $(u_1, u_2) \in \IR^2,$ set
$$
 \E(u_1, u_2) = \E_i(u_1, u_2) \quad \mbox{ if } \ (u_1, u_2) \in Q_i, \ i = 1, 2, 3, 4.
$$
Consider the transformation $S : \IR^2 \to \IR^2_+$ defined by
$$
 S(u_1, u_2) = \left\{\begin{array}{lll}
 (1+u_1, 1+u_2) & \mbox{if} & (u_1,u_2) \in Q_1,\\
 ((1-u_1)^{-1}, 1+u_2) & \mbox{if} & (u_1, u_2) \in Q_2,\\
 ((1-u_1)^{-1}, (1-u_2)^{-1}) & \mbox{if} & (u_1, u_2) \in Q_3,\\
 (1+u_1, (1-u_2)^{-1}) & \mbox{if} & (u_1, u_2) \in Q_4.
 \end{array}\right.
$$
 Let $(\tilde X(u_1, u_2),\ (u_1, u_2) \in \IR^2)$ be the additive Brownian motion derived from $B_1, \ldots, B_4$ above. Then the following local decomposition of $W$ is easily checked:
$$
 W(S(u_1, u_2)) = W(1,1) + \tilde X(u_1, u_2) + \E(u_1, u_2), \qquad (u_1, u_2) \in \IR^2.
$$
 Note that $X(u_1, u_2)$ is of order $\sqrt{\vert u_1\vert} + \sqrt{\vert u_2\vert},$ whereas $\E(u_1, u_2)$ is of order $\vert u_1\vert + \vert u_2\vert.$ Observe that for $u_i \leq 0$, $v_i \geq 0$, $i = 1,2,$
$$
 S(\partial([u_1, v_1] \times [u_2, v_2])) = \partial([(1-u_1)^{-1}, 1+v_1] \times[(1-u_2)^{-1}, 1+v_1]),
$$
 so a behavior of $X$ on the boundary of a rectangle containing (0,0) translates into an approximate behavior of $W$ on a rectangle containing (1,1), and vice-versa.
\vskip 12pt

\noindent{\em Proof of Proposition \ref{bs_ubprop1}.}  Fix $c>0$ and $\varepsilon >0.$ Let $H_c$ be the set of points in $\IR^2_+$ which are in the boundary of an upwards $q$-bubble of diameter $\geq c.$ We will show that $\dim(H_c \cap [1, 2]^2) \leq (3-\lambda_1)/2,$ and in fact, the same proof will show that $\dim(H_c \cap ([k,k+1]\times[\ell,\ell+1])) \leq (3-\lambda_1)/2$, for all $k,\ell \in \IN \setminus\{0\}$, which implies that a.s., $\dim(H_c \cap [1, \infty[^2) \leq (3-\lambda_1)/2$. Using the scaling property of the Brownian sheet, we deduce that $\dim(H_c \cap\, ]0, \infty[^2) \leq (3-\lambda_1)/2$.

  Let $D^n_{i,j}$ be as defined in (\ref{rdsquares}). It is sufficient to show that for all $\varepsilon > 0,$
$$
  E\left(\sum^{2^{2n}-1}_{i,j=1} (2^{-2n})^{(3-\lambda_1+\varepsilon)/2}1_{\{H_c \cap D^n_{i,j} \not= \emptyset\}}\right) \to 0 \qquad\mbox{as } {n\to\infty}.
$$
The expectation is bounded by
$$
  2^{n(1+\lambda_1-\varepsilon)} \sup_{i,j} P(H_c \cap D^n_{i,j} \not= \emptyset),
$$ 
so we need to bound $P(H_c \cap D^n_{i,j} \not= \emptyset)$. It turns out that the bound does not depend on $(i,j)$ ($1 \leq i,j \leq 2^{2n}-1)$, so  in order to simplify the notation, we only consider the case $i = j = 0,$ and we set
$$
  A_n = \{H_c \cap D^n_{0,0}\neq \emptyset\}, \qquad
  F_n = \left\{\sup_{s \in D^n_{0,0}} \vert W(s)-W(1,1)\vert < \frac{n}{2} 2^{-n}\right\}.
$$
Fix $\varepsilon > 0$, so that Lemma \ref{bs_ub_lem4} applies to $\varepsilon/2$, and set $\delta=\delta_0.$ For the additive Brownian motion $n^22^{-n} + \tilde X$, let $G_n$ be the event described in Lemma \ref{bs_ub_lem5}. Then
$$
  P(A_n) = P(A_n \cap F_n \cap G_n) + P(A_n \cap F_n \cap G^c_n) + P(A_n \cap F^c_n).
$$
Clearly,
$$
  P(A_n \cap F_n^c) \leq P(F_n^c) \leq K \ e^{-cn^2},
$$
and since, when $A_n$ occurs, there is $s \in D^n_{0,0}$ for which $W(s) = 0,$
\begin{eqnarray*}
  P(A_n \cap F_n \cap G^c_n) &\leq& P(G_n^c \cap \{\vert W(1,1)\vert \leq \frac{n}{2} 2^{-n}\})\\
  &=& P(G_n^c) P\{\vert W(1,1) \vert \leq \frac{n}{2} 2^{-n}\}
\end{eqnarray*}
because $\tilde X$ and $W(1,1)$ are independent. From Lemma \ref{bs_ub_lem4}, we conclude that
\begin{eqnarray*}
  P(A_n \cap F_n \cap G^c_n) &\leq& K \left(\frac{2^{-n \varepsilon/10}}{(n^2 2^{-n})^2} \right)^{-(\lambda_1-\varepsilon)/2} \cdot \frac{n}{2} 2^{-n}\\
  && \\
  &=& \tilde K n^{1+2(\lambda_1-\varepsilon)}\,  
      2^{-n(1+\lambda_1-\varepsilon+\varepsilon(\lambda_1-\varepsilon)/20)}.
\end{eqnarray*}
Therefore,
$$
  2^{n(1+\lambda_1-\varepsilon)} (P(A_n \cap F_n \cap G^c_n) + P(A_n \cap F^c_n)) \to 0 \quad \mbox{as } n \to \infty,
$$
and it remains to show that
\begin{equation}\label{bs_ub_eq5}
2^{n(1+\lambda_1-\varepsilon)} P(A_n \cap F_n \cap G_n) \to 0 \quad \mbox{as } n \to \infty.
\end{equation} 
 
    Let $D_n$ be as defined in Lemma \ref{bs_ub_lem5}, for the additive Brownian motion $n^2 2^{-n} + \tilde X,$ and set
\begin{equation}\label{bs_ub_eq6}
 E_n = \left\{\sup_{\vert u_1 \vert \leq 2^{-2n}} \vert \tilde X(u_1, 0)\vert \leq n^2 2^{-n-1},\ \sup_{\vert u_2 \vert \leq 2^{-2n}} \vert \tilde X (0, u_2) \vert \leq n^2 2^{-n-1}\right\}.
\end{equation}
Define
\begin{equation}\label{bs_ub_eq7}
 J_n = \{\forall \ h \in [2^{-2n}, 2^{-2\varepsilon n/10}] : \sup_{\vert u_1\vert \leq h,\ \vert u_2 \vert \leq h} \vert \E(u_1, u_2) \vert \leq h n^2 \}.
\end{equation}
Clearly,
$$
 P(A_n \cap F_n \cap G_n) \leq P(A_n \cap F_n \cap G_n \cap D_n) + P(G_n \cap D_n^c).
$$
 By Lemma \ref{bs_ub_lem5} and Lemma \ref{bs_ub_lem6} below, (\ref{bs_ub_eq5}) will be proved provided we show that for large $n$,
\begin{equation}\label{bs_ub_eq8}
 U_n \ \stackrel{\mbox{\footnotesize def}}{=} \ A_n \cap F_n \cap G_n \cap D_n \ \subset\  E^c_n \cup J_n^c,
\end{equation}
which we now proceed to do.
 
    Let $\cR(\underline{T}^{(I)})$ be the rectangle explored by the $\delta$-DW-algorithm applied to $n^22^{-n}+ \tilde X.$ For $n$ large enough so that $2^{-n\varepsilon/10} < c$, on $U_n$, there is $(u_1, u_2) \in \cR(2^{-n \varepsilon/10}) \cap \partial R(\underline{T}^{(I)})$ ($(u_1,u_2)$ is on a positive path starting near a point in $H_c$) such that $W(S(u_1, u_2)) > 0$. Because $\{\vert W(1,1)\vert \leq \frac{n}{2} 2^{-n}\}$ on $F_n$, we see from Lemma \ref{bs_ub_lem2}(a) that on $U_n$,
\begin{eqnarray*}
 0 < W(S(u_1, u_2)) &=& W(1,1) + X(u_1, u_2) + \E(u_1, u_2)\\
 &\leq& \frac{n}{2} 2^{-n} - \delta M_I + \E(u_1, u_2).
\end{eqnarray*}
For $n$ large, $n \delta \geq 1,$ so $\frac{n}{2} 2^{-n} \leq \delta \frac{n^2}{2} 2^{-n} \leq \frac{\delta}{2} M_I$ since the $\delta$-DW-algorithm starts with value $n^2 2^{-n}$. Therefore,
$$
 \E(u_1, u_2) > \frac{\delta}{2} M_I.
$$
Observe that if $J_n$ occurs and $2^{-2 \varepsilon n/10} \geq \vert \underline{T}_j^{(I)} \vert \geq 2^{-2n}$ for $j=1,\dots,4$, then because $D_n$ occurs,
$$
   n^2 \max_j (\underline{T}_j^{(I)}) \geq \E(u_1, u_2) > \frac{\delta}{2} M_I \geq \frac{\delta}{2n} \max_j \sqrt{\underline{T}^{(I)}j},
$$ 
therefore
$$
   \delta \leq 2 \ n^3 \max_j \sqrt{\underline{T}_j^{(I)}} \leq 2 \ n^3 2^{-\varepsilon n/20},
$$
and this cannot hold for large $n$. Therefore, on $U_n$, either $J^c_n$ occurs, or $\vert \underline{T}_j^{(I)} \vert \leq 2^{-2n}$ for some $1 \leq j \leq 4,$ in which case $E^c_n$ occurs (since the increments of $\tilde X$ would have to be sufficiently negative to compensate the starting value $n^2 2^{-n}$). This completes the proof of (\ref{bs_ub_eq8}), and therefore the proof of Proposition \ref{bs_ubprop1}.
\hfill $\Box$
\vskip 16pt

   The following lemma was used in the proof above.

\begin{lemma} \label{bs_ub_lem6} Let $E_n$ and $J_n$ be as defined in (\ref{bs_ub_eq6}) and (\ref{bs_ub_eq7}). Then there are $C < \infty$ and $c>0$ such that for all large $n$,
$$
  P(E^c_n \cup J_n^c) \leq C \ e^{-cn^4}.
$$
\end{lemma}

\proof
That $P(E^c_n  ) \leq C \ e^{-cn^4}$ is a simple consequence of basic properties of Brownian motion.  The probability $P( J_n^c)$ is bounded by
$$
\sum_{k= 1+ 2 \epsilon n /10} ^{2n} 
P\left\{\sup_{|u_1|,|u_2| \leq 2^{-k+1}} |\E(u_1,u_2)| \geq \frac{2^{-k}n^2}{2}\right\},
$$
so it will suffice to show that each of the terms in the sum can be bounded
by $Ce^{-cn^4}$ for universal $c,C$.  We fix a $k \in 1+[2\epsilon n/10, 2n]$. Then
$$
   P\left\{\sup_{|u_1|,|u_2| \leq 2^{-k+1}} |\E(u_1,u_2)| \geq \frac{2^{-k}n^2}{2}\right\} 
$$
is bounded by the sum of
$$
   P\left\{\sup_{0 \leq -u_1,u_2 \leq 2^{-k+1}} |\E(u_1,u_2)| \geq \frac{2^{-k}n^2}{2}\right\}  
$$
and three other similar terms.  We will explicitly bound the first term since similar arguments apply to the three remaining terms. Using the definition of $\E(u_1,u_2)$,
we see that the probability in question is bounded by 
\begin{eqnarray*}
  && P\left\{\sup _{0\leq u_1,u_2 \leq 2^{-k+1}}|\Delta_{](1-u_1)^{-1},1] \times ]1, 1+u_2]} W| \geq \frac{2^{-k}n^2}{6}\right\}\\
  &&\qquad\qquad +\, 2 P\left\{\sup_{0\leq u_1\leq 2^{-k+1}}\left|u_1 W\left( \frac{1}{1-u_1}, 1\right)\right| \geq \frac{2^{-k}n^2}{6}\right\}.
\end{eqnarray*}
By \cite[Lemma 1.2]{OP}, the first term is bounded by 
$$
   4P\left\{|\Delta_{](1-2^{-k+1})^{-1},1] \times ]1, 1+2^{-k+1}]} W| \geq \frac{2^{-k}n^2}{6}\right\}
$$ 
which by standard Gaussian tail estimates satisfies the desired bound.  For the second term, simply note that it is bounded by 
$$
   2 P\left\{\sup_{0\leq u_1\leq 2^{-k+1}}\left|W\left( \frac{1}{1-u_1}, 1\right)\right| \geq \frac{n^2}{6}\right\}.
$$  
Again the reflection principle (this time applied to standard Brownian motion) yields the desired bound. 
\hfill $\Box$
\vskip 12pt

\end{section}
\eject

\begin{section}{Robustness of the DW-algorithm}\label{sec9}
\vskip 12pt

   The remainder of this paper is devoted to proving that $(3-\lambda_1)/2$ is a lower bound for the Hausdorff dimension of the boundary of any $q$-bubble of the Brownian sheet. Together with Proposition \ref{bs_ubprop1}, this will complete the proof of Theorem \ref{thm3a1}.
   
   Since we will use the fact that the Brownian sheet can be approximated by an ABM (see \eqref{rd9.11a}) and this ABM can in turn be approximated by a standard ABM, we need to develop a notion of continuity, or {\em robustness,} of the DW-algorithm. Indeed, if an ABM $X$ is replaced by the ABM $X + \varepsilon,$ for small $\varepsilon > 0,$ it is possible that the DW-algorithm applied to $X$ and to $X+\varepsilon$ will produce substantially different numbers of stages before termination and will explore rectangles of substantially different sizes. However, this is not likely, and we want to quantify this statement, by imposing, among other conditions, that when the DW-algorithm for $X$ terminates, it not only constructs a rectangle along which $X \leq 0$, but on which $X$ is significantly negative, so that the DW-algorithm for ``small perturbations'' of $X$ also terminates. 
   
   We begin by introducing the notion of episodes.
\vskip 12pt

\noindent{\em Episodes}
\vskip 12pt

  Consider an ABM $\ti X$. When the DW-algorithm started at $r$ with value $x_0 \in\, ]0,1[$ terminates, say at an even stage $N = 2m$, it has explored a rectangle $\cR(\underline{\tau}_{(N)}) = [U_m, U^\prime_m] \times [V_m, V^\prime_m]$. The interval $[U_m, U^\prime_m[$ (resp. $[V_m, V^\prime_m[)$ is the disjoint union of the intervals $[U_{\ell}, U_{\ell-1}[$, $\ell = m, \ldots, 1,$ and $[U^\prime_{\ell-1}, U^\prime_\ell[$, $\ell = 1, \ldots, m$ (resp. $[V_{\ell}, V_{\ell-1}[$, $\ell = m, \ldots, 1$ and $[V^\prime_{\ell-1}, V^\prime_{\ell}[$, $\ell = 1, \ldots, m$). We are going to further refine this partition of $[U_m, U^\prime_m[$ (resp. $[V_m, V^\prime_m[)$ in order to take into account the magnitude of the ABM $X$ during each stage, using intervals that we will call {\em episodes} and that we now define.
  
  We first define episodes produced during an odd stage $2n-1$, for $n \in \{1, \ldots, m\}$. If $H_{2n-2} \in [2^{-k}, 2^{1-k}[$ for some $k \geq 1$, and $x_0 + \sup_{U^\prime_{n-1} < u < U^\prime_n} X^r(u, T^{n-1}_2) < 2^{1-k}$, then $[U_{n-1}',U_n']$ is a single {\em episode of order $k$}. If $H_{2n-2} \in [2^{-k}, 2^{1-k}[$ and 
$$
   x_0 + \sup_{U^\prime_{n-1} < u < U^\prime_n}
   X^r(u, T^{n-1}_2) \in [2^{j^\prime_n-k}, 2^{j^\prime_n+1-k}[,
$$ 
for some $j^\prime_n \geq 1$, then the interval $[U^\prime_{n-1}, U^\prime_n]$ will produce $j^\prime_n + 1$ episodes, defined as follows. Set $U^{\prime(0)}_n = U^\prime_{n-1},$ and for $\ell = 1,\ldots, j^\prime_n$, let
$$
  U^{\prime(\ell)}_n = \inf\{u > U^{\prime(\ell-1)}_n : x_0 + X^r(u, T^{n-1}_2) \geq 2^{\ell-k}\}.
$$
and\index{$U^{\prime(\ell)}_n$} $U^{\prime(j^\prime_n+1)}_n = U^\prime_n.$ For $\ell = 0, \ldots, j^\prime_n$, the interval $[U^{\prime(\ell)}_n, U^{\prime(\ell+1)}_n]$ is an {\em episode of order} $k-\ell.$ Note that these episodes form a partition of $[U^\prime_{n-1}, U^\prime_n]$, and the maximum of $X^r( \cdot, T^{n-1}_2)$ over an episode of order $k-\ell$ belongs to the interval $[2^{\ell-k},2^{1+\ell-k}].$ If $\ell < j^\prime_n$, then the episode is termed an {\em interior episode,} and otherwise an {\em extremity episode.}

   Similarly, if $H_{2n-2} \in [2^{-k}, 2^{1-k}[$ and $x_0 + \sup_{U_n < u <U_{n-1}} \ X^r(u, T^{n-1}_2) < 2^{1-k}$, then $[U_n, U_{n-1}]$ is a single (extremity) episode of order $k$. If
$$
   x_0 + \sup_{U_n < u <U_{n-1}} \ X^r(u, T^{n-1}_2) \in [2^{j_n-k}, 2^{j_n+1-k}[,
$$
for some $j_n \geq 1$, then $[U_n, U_{n-1}]$ will produce $j_n+1$ episodes, defined as follows. Set $U^{(0)}_n = U_{n-1}$, and for $\ell = 1, \ldots, j_n$, 
$$
   U_n^{(\ell)} = \sup \{u < U^{(\ell-1)}_n : x_0 + X^r(u, T_2^{u-1}) \geq 2^{\ell-k} \},
$$
and\index{$U^{(\ell)}_n$} $U_n^{(j_n+1)} = U_n$. For $\ell = 0, \ldots, j_n$, $[U^{(\ell+1)}_n, U^{(\ell)}_n]$ is an {\em episode of order} $k-\ell$. Note that these episodes partition $[U_n, U_{n-1}]$. 

   Episodes produced during an even stage are defined in a similar manner. In this case, the episodes are of the form $[V_n^{(\ell+1)}, V_n^{(\ell)}]$ and $[V_n^{\prime(\ell)}, V_n^{\prime(\ell+1)}],$ and form a partition of $[V_n, V_{n-1}] \cup [V^\prime_{n-1}, V^\prime_n]$. In addition, for $n$ even or odd and $k \in \ZZ$, stage $n$ produces an episode of order $k$ if and only if $H_{n-1} < 2^{1-k}$ and $H_n \geq 2^{-k},$ and stage $n$ produces at most two episodes of order $k$.
	
	There are four kinds of episodes of order $k$ in which the DW-algorithm does not STOP:
	
	{\em Type 1.} The episode arises during a stage $n$ for which $H_{n-1} \in [2^{-k}, 2^{1-k}[$, at the beginning of the episode, the ABM starts at the value $H_{n-1} - H_{n-2}$, reaches level $2^{-k}$ but does not reach level $2^{1-k}$ and the ABM has value $0$ at the end of the episode. For this type, $2^{-k} \leq H_{n-1} < H_n < 2^{1-k}$.
	
	{\em Type 2.} The episode arises during a stage $n$ for which $H_{n-1} \in [2^{-k}, 2^{1-k}[$, at the beginning of the episode, the ABM starts at the value $H_{n-1} - H_{n-2}$, and it reaches level $2^{1-k}$ at the end of the episode. For this type, $H_{n-1} - H_{n-2} < 2^{1-k} \leq H_n$.
	
	{\em Type 3.} The ABM has value $2^{-k}$ at the beginning of the episode and value $2^{1-k}$ at the end of the episode. For this type, $H_{n-1} - H_{n-2} < 2^{-k} < 2^{1-k} \leq H_n$.
	
	{\em Type 4.} The ABM has value $2^{-k}$ at the beginning of the episode, it does not reach level $2^{1-k}$ during the episode and has value $0$ at the end of the episode. For this type, $H_{n-1} - H_{n-2} < 2^{-k}< H_n < 2^{1-k}$.
	
\vskip 12pt
 
\noindent{\em Robustness}
\vskip 12pt
  
   We begin by defining a property of Brownian motion. Consider the functions\index{$f_n(\ell)$}
$$
  f_n(\ell) = \ell^{-n}\, 2^{-\ell}, \qquad n \in \IR,\ \ell > 0.
$$
Fix $v \geq 1$, $\ell > 0.$ Consider a Brownian motion $B = (B_u,\ u \geq 0)$ such that $B_0 = x > v^{-2} f_8(\ell)$. Define
  \begin{eqnarray*}
  \tau^1 &=& \inf\{u \geq 0: B_u = v^{-2} f_8(\ell)\},\\
  \tau^0 &=& \inf \{u \geq 0: B_u = 0\},\\
  \tau^2 &=& \inf\{u \geq 0: B_u = - v^{-2} f_8(\ell)\}.
  \end{eqnarray*}
Then we say that $B$ hits 0 $v$-{\em robustly for order $\ell$} if the following properties hold.
$$
  \tau^0 - v^{-1}\, \frac{2^{-2\ell}}{\ell^{10}} \leq \tau^1 \leq \tau^0 \leq \tau^2 \leq \tau^0 + v^{-1}\,  \frac{2^{-2\ell}}{\ell^{10}}
$$
and
$$
  \sup_{\tau^1 \leq u \leq \tau^2} B_u \leq v^{-1/4} \ f_4(\ell).
$$
More generally, for a Brownian motion $B = (B_u,\ u \geq 0)$ such that $B_0 = x <y- v^{-2} f_8(\ell)$, we say that {\em $B$ hits $y$ $v$-robustly for order $\ell$} if the Brownian motion $y-B$ hits 0 $v$-robustly for order $\ell$.

\begin{remark} This property states that as soon as $B$ gets near 0 (within $v^{-2}f_8(\ell)$), it becomes sufficiently negative (with value $- v^{-2}f_8(\ell)$) fairly quickly (taking no more than $v^{-1} \ell^{-10} 2^{1-2\ell}$ units of time) and before becoming too positive (it stays below $v^{-1/4}f_4(\ell)$). Therefore, if some other process $\bar B$ is very close to $B$ (within $v^{-2} f_8(\ell)$), then $\bar B$ will hit $y$ at about the same time as $B$.
\end{remark}

   The next lemma shows that for large $v$ and $\ell$, it is highly probable that a Brownian motion hits $0$ $v$-robustly for order $\ell$.

\begin{lemma} For $x > v^{-2} f_8(\ell)$, the probability that a Brownian motion $B$ starting at $x$ does not hit $0$ $v$-robustly for order $\ell$ is bounded above by $3 v^{-3/2} \ell^{-3}$.
\label{lem9.2}
\end{lemma}

\proof The event ``$B$ does not hit $0$ $v$-robustly for order $\ell$" is contained in the union of the three events
\begin{align*}
   G_1 &= \left\{ \tau^0 > \tau^1 + v^{-1} \frac{2^{-2\ell}}{\ell^{10}}\right\}, \qquad
	 G_2 =  \left\{ \tau^2 > \tau^0 + v^{-1} \frac{2^{-2\ell}}{\ell^{10}}\right\}, \\
	 G_3 &=  \left\{\sup_{\tau^1 < u < \tau^2} B_u > v^{-1/4} f_4(\ell),\ \tau^2 - \tau^1 \leq 2 v^{-1}  \frac{2^{-2\ell}}{\ell^{10}} \right\}.
\end{align*}
The first two events have the same probability. In addition,
\begin{align*}
   P_x(G_1) &= P_0 \left\{\max_{0 \leq u \leq v^{-1} \ell^{-10} 2^{-2\ell}} B_u < v^{-2} f_8(\ell) \right\} \\
	   &= P_0 \left\{ \max_{0 \leq u \leq 1} B_u < v^{-3/2} \ell^{-3} \right\} \\
		 &\leq v^{-3/2} \ell^{-3}.
\end{align*}
Further,
\begin{align*}
   P_x(G_3) & \leq P_0 \left\{ \max_{0 \leq u \leq 2 v^{-1} \ell^{-10} 2^{-2\ell}} B_u > v^{-1/4} f_4(\ell) \right\} \\
	   &= P_0 \left\{ \max_{0 \leq u \leq 1} B_u > 2^{-1/2} v^{1/4} \ell \right\} \\
		 &\leq \exp\left[-\frac{1}{4} v^{1/2} \ell^2 \right].
\end{align*}
This proves the lemma.
\hfill $\Box$
\vskip 16pt

  The functions $f_n(\ell)$ have the following properties.
  
\begin{lemma} (a) For all $V > 0$ and $M > 0$, there exists $\ell_0>0$ such that for all $v \in [V^{-1}, V]$, for all $n,m \in [-M, M]$, for all $\ell \geq \ell_0$,
$$
  f_n (f_0^{-1}(v f_m(\ell))) \leq f_{n+m-1} (\ell).
$$


   (b) Fix $c>0$ and $\beta >0$. There is $C>0$ such that for all $k \geq 2$, 
$$
   \sum_{j=1}^{k-1} e^{-cj}\, (k-j)^{-\beta} \leq C k^{-\beta}.
$$

   (c) For $k$ sufficiently large, 
$$
   \sum_{\ell = k}^\infty \ell^2\, 2^{-2\ell} \leq 2\, k^2\, 2^{-2k}.
$$
\label{rdlem35}
\end{lemma}

\begin{proof} (a) One easily checks that
\begin{equation}\label{e9.0}
   f_n(f_0^{-1}(v f_m(\ell))) = \left(1 - \frac{\ln_2(v)}{\ell} + m \frac{\ln_2(\ell)}{\ell}\right)^{-n} \frac{v}{\ell} f_{n+m-1}(\ell),
\end{equation}
and (a) follows


   (b) Let $\ell = k-j$, to see that the sum is equal to
$$
   e^{-ck} \sum_{\ell=1}^{k-1} e^{c\ell}\, \ell^{-\beta} \leq e^{-ck} \left(e^c + \int_2^k e^{cx}\, (x-1)^{-\beta}\, dx \right).
$$
Using l'H\^opital's rule, one easily checks that the integral is bounded by $C e^{ck} k^{-\beta}$, and (b) is proved.
   
   (c) The proof of this elementary inequality is left to the reader.
\end{proof}
\vskip 16pt

  For $v \geq 1$, we say that the DW-algorithm started at $(0,0)$ with value $x_0 >0$ behaves {\em robustly with tolerance} $v$ {\em for orders greater than} $k_0$, or simply $v$-{\em robustly above order} $k_0$, if the following properties (R1) to (R7) hold.
\begin{itemize}
\item[(R1)] For $k \geq k_0$, let $V(k)$\index{$V(k)$} be the number of stages that produce an episode of order $k$, that is,
$$
  V(k) = \sum_{m \geq 1} 1_{\{H_{m-1} < 2^{1-k},\, H_m \geq 2^{-k}}\}.
$$
Then $V(k) \leq k \sqrt{v}.$
  
\item[(R2)] For $k \geq k_0,$ every episode of order $k$ has length $\leq v k 2^{-2k}.$
  
\item[(R3)] For $k \geq k_0$, every episode of order $k$ has length $\geq v^{-2} k^{-7} 2^{-2k}$. 
  
\item[(R4)] For $k \geq k_0$ and each stage $m$ such that $H_m \in [2^{-k}, 2^{1-k}[,$
$$
  H_m - H_{m-1} \geq \frac{H_m}{vk^3}.
$$
  
\item[(R5)](a) We use the notation introduced while defining episodes. For $k \geq k_0$, for each odd stage $2n-1$ such that $H_{2n-2} \in [2^{-k}, 2^{1-k}[$, for $1 \leq \ell \leq j_n$ such that $k-\ell \geq k_0$ (resp. $1 \leq \ell \leq j_n^\prime $ such that $k-\ell \geq k_0)$, the Brownian motion $B^{(\ell)} = (B^{(\ell)}_u = x_0 + X^r(U^{(\ell-1)}_n -u, T^{n-1}_2),\ u \geq 0)$
(resp. $B^{(\ell)} = (B^{(\ell)}_u = x_0 + X^r(U_n^{\prime(\ell-1)} + u, T^{n-1}_2),\ u \geq 0)$
does not hit $v^{-1} f_8(k-\ell)$ and hits $2^{\ell-k}$ $v$-robustly for order $k-\ell$.  For
$\ell =j_n +1$ (resp. $\ell = j^\prime_n +1)$), $B^{(\ell)}$ hits 0 $v$-robustly for order $k-j_n$ (resp. for order $k-j_n^\prime$).
  
\item[(R5)] (b) For the same values of parameters $k$, $n$, $j_n$ $j^\prime_n$ as in (R5)(a),
$$
  \left\vert\, \sup_{u \in [U_n, U_{n-1]}} X(u, T_2^{n-1}) - \sup_{u \in [U^\prime_{n-1}, U^\prime_n]} X(u, T^{n-1}_2) \right\vert > v^{-2} f_7(k-(j_n \vee j^\prime_n))
$$
  
\item[(R6)] (a) Similar to (R5)(a), but for even stages.
  
\item[(R6)] (b) Similar to (R5)(b), but for even stages.

\item[(R7)] For $k \geq k_0$ and each stage $m$ such that $H_m \in [2^{-k}, 2^{1-k}[, \ H_{m} \leq H_{m-1} vk^3$.  
\end{itemize}

  We now show that when the DW-algorithm behaves $v$-robustly for an ABM $X$, it also behaves $\tilde v$-robustly for a small perturbation $\tilde X$ of $X$, with $\tilde v = 9 v$. In fact, this proposition applies to additive processes which are not necessarily ABM's. Here, an {\em additive process} is a process $(X(s),\ s \in \IR_+^2)$ such that for all $s=(s_1,s_2)$,
$$
   X(s_1,s_2) = X(s_1,0) + X(0,s_2) - X(0,0).
$$

\begin{prop} Fix $v \geq 1.$ There is $L = L(v) > 0$ such that for all $n > \ell_0 \geq L$, the following property holds: let $X$ and $\tilde X$ be two additive processes (not necessarily ABM's) starting at $(0,0)$ with value $2^{-n}$ such that:

   (a) the DW-algorithm applied to $X$ reaches level $2^{-\ell_0}$ within $[-2^{-2\ell_0},2^{-2\ell_0}]^2$ or escapes $[-2^{-2 \ell_0},2^{-2 \ell_0}]^2$, and behaves $v$-robustly above order $\ell_0$,
   
   (b) letting $R_{2\ell_0} = \cR(2^{-2\ell_0})$ and $\ultau = \ultau^{2^{-\ell_0}} \wedge \ulsigma^{R_{2\ell_0}}$ (defined relative to $X$), for $\ell_0 \leq \ell \leq n$,
\begin{align} \nonumber
  &\max\left[ \sup_{-(\ultau_3 \wedge 2^{-2\ell})\leq u \leq \ultau_1 \wedge 2^{-2\ell}} \vert X(u, 0) - \tilde X(u, 0) \vert, \right.\\
  &\qquad \qquad
  \left. \sup_{-(\ultau_4 \wedge 2^{-2\ell})\leq u \leq \ultau_2 \wedge 2^{-2\ell}}
     \vert X(0, u) - \tilde X(0,u)\vert\right] \leq f_{15}(\ell).
\label{rd9.1}
\end{align}
Then the DW-algorithm applied to $\tilde X$ achieves value $2^{-\ell_0-1}$ within $[-2^{-2\ell_0-1}, 2^{-2\ell_0-1}]^2$ or escapes $[-2^{-2\ell_0-1}, 2^{-2\ell_0-1}]^2$,
and behaves $\tilde v$-robustly above order $\ell_0+1$, with $\tilde v = 9v$.
\label{rdprop37}
\end{prop}

\begin{proof} For every odd stage $2m-1$, let $j$ (resp. $j^\prime$) be the order of the episode with left endpoint $U_m$ (resp. right endpoint $U^\prime_m)$. Define
\begin{eqnarray*}
  \gamma^o_m &=& \sup\{u < U_{m-1}: x_0 + X(u, T^{m-1}_2) \leq v^{-2} f_8(j)\},\\
  \lambda_m^o &=& \sup\{ u < U_m: x_0 + X(u, T^{m-1}_2) \leq - v^{-2} f_8(j)\}\\
  \gamma^{o\prime}_m &=& \inf\{u > U^\prime_{m-1}: x_0 + X(u, T^{m-1}_2) \leq v^{-2} f_8(j^\prime)\},\\
  \lambda^{o\prime}_m &=& \inf\{u > U^\prime_m: x_0 + X(u, T^{m-1}_2) \leq - v^{-2} f_8(j^\prime)\}.
\end{eqnarray*}
\index{$\gamma^o_m$}\index{$\lambda_m^o$}\index{$\gamma^{o\prime}_m$}\index{$\lambda^{o\prime}_m$}For even stages $2m$, we define the corresponding random variables $\gamma^e_m$, $\lambda_m^e$, $\gamma^{e\prime}_m$, $\lambda^{e\prime}_m$.\index{$\gamma^e_m$}\index{$\lambda_m^e$}\index{$\gamma^{e\prime}_m$}\index{$\lambda^{e\prime}_m$}

  Objects related to the DW-algorithm for $\tilde X$ will be denoted with a $\tilde{~~}$, e.g. $\tilde U_m$, $\tilde U^\prime_m$, etc.
  
  Assume that $X$ and $\tilde X$ satisfy (\ref{rd9.1}) for $\ell_0 \leq \ell \leq n,$ and that the DW-algorithm applied to $X$ reaches level $2^{-\ell_0}$ before exiting $\cR(2^{-2 \ell_0})$ or escapes $\cR(2^{-2 \ell_0})$, and behaves $v$-robustly above order $\ell_0$. We are going to show first that the DW-algorithm applied to $\tilde X$ is compatible with that for $X$, in the following sense.
  
  We say that an odd stage $2m-1$ for $\tilde X$ is {\em compatible} with stage $2m-1$ for $X$ if the following three conditions hold:
\begin{itemize}
\item[(c1)] $\tilde U_m \in [\lambda^o_m, \gamma_m^o]$ and $\tilde U_m^\prime \in [\gamma_m^{o\prime}, \lambda^{o\prime}_m],$
  
\item[(c2)] $\vert X(T^m_1, 0) - X(\tilde T^m_1, 0) \vert
\leq f_{12}(k),$ where $k$ is such that $H_{2m-1} \in [2^{-k}, 2^{1-k}[.$ 
  
\item[(c3)] When $(T^m_1, T^{m-1}_2)$ belongs to one of the two segments $[U_m, U_{m-1}] \times \{T^{m-1}_2\}$ or $[U^\prime_{m-1}, U^\prime_m] \times \{T^{m-1}_2\},$ then $(\tilde T^m_1, \tilde T^{m-1}_2)$ belongs to the corresponding segment for $\tilde X$.
\end{itemize}
For even stages, the notion of compatibility is defined analogously. The underlying idea is that as long as stages for $\tilde X$ are compatible with those of $X$, the DW-algorithms for $\tilde X$ and $X$ explore essentially the same rectangles and construct parallel paths, with a control over the discrepancies because of $v$-robustness. This is illustrated in Figure \ref{figAa}.
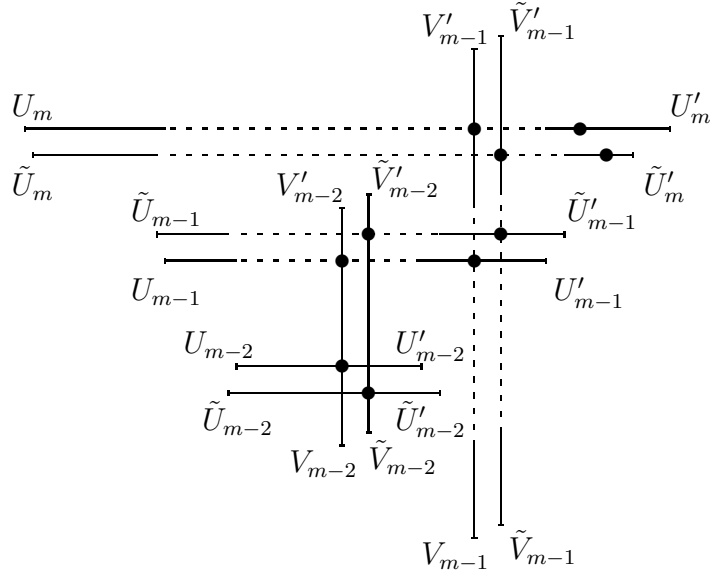
\begin{figure}
\begin{center}
\begin{picture}(280,200)

\put(20,170){\line(1,0){53}} \multiput(75,170)(6,0){24}{\line(1,0){2}} \put(217,170){\line(1,0){47}} 
\put(23,160){\line(1,0){47}} \multiput(75,160)(6,0){25}{\line(1,0){2}} \put(224,160){\line(1,0){26}} 

\put(70,130){\line(1,0){27}} \multiput(102,130)(6,0){13}{\line(1,0){2}} \put(177,130){\line(1,0){47}} 
\put(73,120){\line(1,0){27}} \multiput(102,120)(6,0){12}{\line(1,0){2}} \put(170,120){\line(1,0){47}} 

\put(100,80){\line(1,0){70}}
\put(97,70){\line(1,0){80}}

\put(140,50){\line(0,1){90}}
\put(150,55){\line(0,1){90}}

\put(190,15){\line(0,1){35}} \multiput(190,50)(0,6){15}{\line(0,1){2}} \put(190,140){\line(0,1){60}} 
\put(200,20){\line(0,1){35}} \multiput(200,55)(0,6){15}{\line(0,1){2}} \put(200,145){\line(0,1){60}} 

\put(140,80){\circle*{5}}
\put(150,70){\circle*{5}}

\put(140,120){\circle*{5}}
\put(150,130){\circle*{5}}

\put(200,130){\circle*{5}}
\put(190,120){\circle*{5}}

\put(200,160){\circle*{5}}
\put(190,170){\circle*{5}}

\put(240,160){\circle*{5}}
\put(230,170){\circle*{5}}

\put(20,169){\line(0,1){2}} \put(264,169){\line(0,1){2}} 
\put(23,159){\line(0,1){2}} \put(250,159){\line(0,1){2}}

\put(70,129){\line(0,1){2}} \put(224,129){\line(0,1){2}} 
\put(73,119){\line(0,1){2}} \put(217,119){\line(0,1){2}}

\put(100,79){\line(0,1){2}} \put(170,79){\line(0,1){2}}
\put(97,69){\line(0,1){2}} \put(177,69){\line(0,1){2}}

\put(139,50){\line(1,0){2}} \put(139,140){\line(1,0){2}} 
\put(149,55){\line(1,0){2}} \put(149,145){\line(1,0){2}}

\put(189,15){\line(1,0){2}} \put(189,200){\line(1,0){2}}
\put(199,20){\line(1,0){2}} \put(199,205){\line(1,0){2}}

\put(15,175){$U_m$} \put(264,175){$U_m'$}
\put(15,145){$\ti U_m$} \put(254,145){$\ti U_m'$}

\put(60,135){$\ti U_{m-1}$} \put(225,135){$\ti U_{m-1}'$}
\put(60,105){$U_{m-1}$} \put(220,105){$U_{m-1}'$}

\put(80,85){$U_{m-2}$} \put(160,85){$U_{m-2}'$}
\put(87,56){$\ti U_{m-2}$} \put(160,56){$\ti U_{m-2}'$}

\put(120,40){$V_{m-2}$}
\put(150,41){$\ti V_{m-2}$}

\put(115,145){$V_{m-2}'$}
\put(151,147){$\ti V_{m-2}'$}

\put(170,5){$V_{m-1}$}
\put(203,7){$\ti V_{m-1}$}

\put(170,205){$V_{m-1}'$}
\put(203,207){$\ti V_{m-1}'$}

\end{picture}
\end{center}
\caption{Examples of compatible stages. \label{figAa}} 
\end{figure}
 
 We now prove that under the previous assumptions, each stage for $\tilde X$, say an odd stage $2m-1$, is compatible with stage $2m-1$ for $X$, provided 
\begin{equation}\label{rd9.2}
   R_{2m-1} = [U_m, U^\prime_m] \times [V_{m-1}, V^\prime_{m-1}] \subset \cR(2^{-2\ell_0}) \qquad\mbox{and}\qquad \sup_{R_{2m-1}} X < 2^{-\ell_0}.
\end{equation}
 
 We do this by induction on $m$. For $m \geq 1,$ suppose compatibility holds up to stage $2m-2$, and stage $2m-1$ for $X$ satisfies (\ref{rd9.2}). We show that stage $2m-1$ for $\tilde X$ is compatible with stage $2m-1$ for $X$.
 
 Let $k$ be such that $H_{2m-1} \in [2^{-k}, 2^{1-k}[$. The first observation is that 
$$
   \max(\vert U_m\vert, U^\prime_m, \vert V_{m-1}\vert, V^\prime_{m-1}) \leq 2\,v^{3/2} k^2 2^{-2k}.
$$
Indeed, summing the length of all episodes of order greater than $k$, multiplied by the number of episodes of each order (properties (R2) and (R1) of $v$-robustness) gives the upper bound 
\begin{equation}\label{sum_ep}
 \sum^\infty_{\ell=k} (v \ell 2^{-2\ell}) (\ell \sqrt{v}) \leq 2\, v^{3/2} k^2 2^{-2k}
\end{equation}
by Lemma \ref{rdlem35}(c).
 Using property (R5)(a), we see that if $k$ is sufficiently large, then 
$$
  \max(\vert \lambda^o_m \vert, \lambda^{o\prime}_m, \vert\lambda_{m-1}^e\vert, \lambda^{e\prime}_{m-1}) \leq 3\, v^{3/2} k^2 2^{-2k}.
$$
Selecting the largest $\ell$ for which $2^{-2\ell} \geq 3\, v^{3/2} k^2 2^{-2k}$ which is  $f^{-1}_0(\sqrt{3}\, v^{3/4} f_{-1}(k))$, we use Lemma \ref{rdlem35}(a) and (\ref{rd9.1}) to conclude that
\begin{equation}\label{rd9.3}
  \sup_{s \in [\lambda^o_m, \lambda_m^{o\prime}] \times [\lambda^e_{m-1}, \lambda_{m-1}^{e\prime}]} \vert X(s) - \tilde X(s) \vert \leq f_{13} (k),
\end{equation}
provided $k$ is large enough.
  
    The next step in proving compatibility is to show that
\begin{equation}\label{rd9.4}
   U_m < \tilde U_{m-1},\quad             U^\prime_m > \tilde U_{m-1}^\prime,
   \quad    \tilde U_m < U_{m-1},\quad    \tilde U^\prime_m > U^\prime_{m-1},
\end{equation}
(that is, the horizontal projection of the rectangle explored during stage $2m-1$ for $X$ (resp. $\tilde X$) encompasses the horizontal projection of the rectangle explored up to stage $2m-2$ for $\tilde X$ (resp. $X$)). We only check the second inequality, since the others are checked similarly.
  
  By property (R4) and (\ref{starrd1}), 
\begin{equation}\label{rd9.5}
   x_0 + X(U^\prime_{m-1}, T^{m-1}_2) \geq v^{-1} f_3(k).
\end{equation}
There are two cases to distinguish, according as $U^\prime_{m-1} \leq \tilde U_{m-1}^\prime$ or $U^\prime_{m-1} > \tilde U^\prime_{m-1}$. We only consider the first case, since the other immediately gives $U^\prime_m >  U^\prime_{m-1} > \tilde U^\prime_{m-1}$.
  
  By property (c1) applied to stage $2m-3$ and (R5), absolute values of increments of $X(\cdot, T^{m-1}_2)$ over $[U^\prime_{m-1}, \tilde U_{m-1}^\prime]$ are bounded above by $2 v^{-1/4} f_4(k)$, so by (\ref{rd9.5}), $x_0 + X(\cdot, T_2^{m-1}) > 0$ over this interval provided $k$ is large enough so that $k > 2 v^{3/4}$. This proves (\ref{rd9.4}).
  
  We now check property (c1) for stage $2m-1$. In fact, we only check that $\tilde U^\prime_m > \gamma^{o\prime}_m$, since the other inequalities needed to establish (c1) are checked similarly. For this, it suffices to show that $x_0 + \tilde X(u, \tilde T^{m-1}_2) > 0$ for $u \in [\tilde U^\prime_{m-1}, \gamma^{o\prime}_m].$ To check this, note using (\ref{rd9.3}), the fact that rectangular increments of $X$ vanish and (R5), that for such $u$,
\begin{eqnarray*}
  x_0 + \tilde X(u, \tilde T^{m-1}_2) &\geq& x_0 + X(u, \tilde T^{m-1}_2) - f_{13}(k)\\
  &=& x_0 + X(u, T^{m-1}_2) + X(\tilde T^{m-1}_1, \tilde T^{m-1}_2) - X(\tilde T^{m-1}_1, T^{m-1}_2) - f_{13}(k)\\
  &\geq& v^{-2} f_8(k) + X(0, \tilde T^{m-1}_2) - X(0, T^{m-1}_2) - f_{13}(k).
\end{eqnarray*}
  We bound the remaining $X$-increment by using the bound from property (c2) for stage $2m-2$
to see that
\begin{equation}\label{rd9.6}
  \vert X(0, \tilde T^{m-1}_2) - X(0, T^{m-1}_2) \vert \leq f_{12}(k).
\end{equation}
Therefore,
$$
  x_0 + \tilde X(u, \tilde T^{m-1}_2) \geq v^{-2} f_8(k) - f_{12} (k)-f_{13}(k) > 0
$$
provided $k$ is large enough. Property (c1) for stage $2m-1$ is proved.
  
   We now prove property (c2) for stage $2m-1$. By the definition of an additive process, of $T^m_1$ and $\tilde T^m_1$, and property (c1) for stage $2m-1$,
$$
   X(T^m_1, 0) - X(\tilde T^m_1, 0) \geq 0
    \qquad\mbox{and}\qquad 
   \tilde X(T^m_1, 0) - \tilde X(\tilde T^m_1, 0) \leq 0.
$$
Using \eqref{rd9.3}, we deduce that $\vert X(T^m_1, 0) - X(\tilde T^m_1, 0) \vert \leq 2 f_{13}(k) \leq f_{12}(k)$, provided $k \geq 2$. This proves property (c2) for stage $2m-1$.
  
  We now prove property (c3) for stage $2m-1$. Suppose for instance that $(T_1^m, T_2^{m-1}) \in [U^\prime_{m-1}, U^\prime_m] \times \{T^{m-1}_2\}$. By property (R5)(b),
\begin{equation}\label{rd9.7}
  x_0 +   \sup_{u \in [U_m, U_{m-1}]} X(u, T^{m-1}_2) < H_{2m-1} - v^{-2} f_7(k).
\end{equation}
By (\ref{rd9.3}), then (\ref{rd9.6}), (c1) and (\ref{rd9.7}),
\begin{eqnarray*}
  x_0 + \sup_{u \in [\tilde U_m, \tilde U_{m-1}]} \tilde X(u, \tilde T^{m-1}_2) &\leq&  x_0 + \sup_{u \in [\tilde U_m, \tilde U_{m-1}]} X(u, \tilde T^{m-1}_2) + f_{13}(k)\\
  \\
  &\leq& x_0 + \sup_{u \in [\tilde U_m, \tilde U_{m-1}]} X(u, T^{m-1}_2) + f_{12}(k) + f_{13}(k)\\
  \\
  &=& x_0 + \sup_{u \in [U_m, U_{m-1}]} X(u, T^{m-1}_2) + f_{12} (k) + f_{13}(k)\\
  \\
  &<& H_{2m-1} - v^{-2} f_7(k) + 2 f_{12}(k) + f_{13} (k).
\end{eqnarray*}
Property (\ref{rd9.3}) and the definition of $H_{2m-1}$ and $\tilde H_{2m-1}$ imply that 
\begin{equation}\label{rd9.8a}
   \vert \tilde H_{2m-1} - H_{2m-1} \vert \leq f_{13}(k),
\end{equation}
so the last right-hand side is bounded above by
$$
 \tilde H_{2m-1} - v^{-2} f_7(k) + 2 f_{12}(k) + 2 f_{13}(k) < \tilde H_{2m-1}
$$
provided $k \geq \ell_0$ and $\ell_0$ is chosen large enough. 
   
  This completes the proof by induction of compatibility between stages of $\tilde X$ and $X$.

  We now check that the DW-algorithm applied to $\tilde X$ achieves value $2^{-\ell_0-1}$ within $\cR(2^{-2\ell_0-1})$ or escapes $\cR(2^{-2\ell_0-1})$, and behaves $9v$-robustly above order $\ell_0+1.$
  
  Observe first that if $\ti U^{(\ell)}$, $\ti U^{\prime(\ell)}$ (resp. $\ti V^{(\ell)}$, $\ti V^{\prime(\ell)}$) are endpoints of horizontal (resp. vertical) episodes of $\tilde X$ of order $k \geq \ell_0 +1$, such that $\ti U^{(\ell)} \leq 0 \leq \ti U^{\prime(\ell)}$ and $\ti V^{(\ell)} \leq 0 \leq \ti V^{\prime(\ell)}$, then just as in (\ref{rd9.3}), we have: 
\begin{equation}\label{rd9.8}
  \sup_{s \in [\ti U^{(\ell)}, \ti U^{\prime(\ell)}] \times [\ti V^{(\ell)}, \ti V^{\prime(\ell)}]} 
     \vert X(s) - \tilde X(s) \vert \leq f_{13}(k),
\end{equation}
even if the stage that produces any one of these episodes contains other episodes of order $\leq \ell_0$ or exits $\cR(2^{-2 \ell_0-1}).$
  
  We now check property (R1) for $\tilde X$. An order $k$ episode for $\tilde X$ occurs if for some stage $m$ of the DW-algorithm applied to $\tilde X$, the inequalities $\tilde H_{m-1} < 2^{1-k}$ and $\tilde H_m \geq 2^{-k}$ hold. If $k > \ell_0$, then by (\ref{rd9.3}), stage $m$ of the DW-algorithm applied to $X$ satisfies $H_{m-1} < 2^{2-k}$ and $H_m \geq 2^{-1-k}$, and therefore produces an episode of order $k-1$, $k$ or $k+1$. By (R1) for $X$, the number of such stages is bounded above by
$$
  ((k-1) + k + (k+1)) \sqrt{v} = k \sqrt{9v}.
$$
This establishes (R1) for $\tilde X$ with $\tilde v = 9v$.
  
  We now check property (R2) for $\tilde X$. By the closeness condition (\ref{rd9.8}), an interior episode for $\tilde X$ of order $k$ is no longer than three episodes for $X$, which are of order $k-1$, $k$ and $k+1$, whose cumulated length is, by (R2), no greater than
$$
  v(k-1)2^{-2(k-1)} + vk2^{-2k} + v (k+1)2^{-2(k+1)} \leq 6 v k 2^{-2k}.
$$
On the other hand, an extremity episode for $\tilde X$ of order $k$ is no longer than the union of one interior episode for $X$ and one extremity episode for $X$, plus $v^{-1} 2^{-2(k-1)}/(k-1)^{10}$, giving the upper bound
$$
  2vk 2^{-2k} + \frac{4}{v} \frac{2^{-2k}}{(k-1)^{10}} \leq 3 v k 2^{-2k}.
$$
This proves (R2) for $\tilde X$ with $\tilde v = 9v$.
  
  We now check property (R3) for $\tilde X$. We consider an interior episode for $\tilde X$ of order $k\geq \ell_0 +1$ occurring in $[U_{m-1}^\prime,U_{m}^\prime ]$ during the odd stage $2m-1$. Then $H_{2m-2} \leq 2^{1-k}$ and $H_{2m-1} \geq 2^{-k}$. By (R7), $H_{2m-1} \leq H_{2m-2} v k^3 \leq 2v f_{-3}(k)$. By (c2) and \eqref{e9.0},
\begin{equation}\label{e9.11}
   \vert X(0,T^{m-1}_2) - X(0,\tilde T^{m-1}_2) \vert \leq f_{12}(f_0^{-1}(2vf_{-3}(k))) \leq 4 v f_9(k).
\end{equation}
By the reasoning leading to \eqref{rd9.3}, setting 
$$
   G_{m,k,v}= \inf\left\{s > U_{m-1}^\prime: x_0 + X(s,T^{m-1}_2 ) = 2^{1-k}- v^{-2}\, f_8(k) \right\}
$$  
and 
$$
   G_{m,k,v}^\prime = \inf\left\{s > U_{m-1}^\prime: x_0 + X(s,T^{m-1}_2 ) = 2^{1-k} + v^{-2}\, f_8(k)\right\},
$$ 
we have, by \eqref{e9.11} and \eqref{rd9.3}, for $s \in [U_{m-1}^\prime, G_{m,k,v}^\prime]$, 
$$
 |X(s,T^{m-1}_2 )- \tilde X(s,\tilde T^{m-1}_2 ) | \leq 4v f_{9}(k) + f_{13}(k) \leq v^{-3} f_{8.5}(k).
$$
Consequently, for $\ell_0$ fixed sufficiently large, it follows that (quantities defined with $\tilde X$ instead of $X$ are denoted with a $ \tilde \ $) that 
$$
\tilde G_{m,k,9v} ^\prime \leq G_{m,k,v}^\prime 
$$
and
$$
\tilde G_{m,k,9v} \geq G_{m,k,v},
$$
hence the length of the associated interior episode of order $k$ for $\tilde X$ will be greater than
$ \tilde G_{m,k,9v} - \tilde G_{m,k+1,9v}^\prime \geq G_{m,k,v}-G_{m,k+1,v}^\prime$. 
By (R3) and the definition of $v$-robustness, this is $\geq 2^{-2k}/(v^2k^7)- 2^{1-2k}/(vk^{10})
\geq 2^{-2k}/((9v)^2k^7)$ for $ k\geq \ell_0$ provided $\ell_0 $ is sufficiently large.  Extremity episodes are treated similarly. 
  
  We now check property (R4) for $\tilde X$. Observe that if $\tilde H_m \in [2^{-k}, 2^{1-k}[$, then
$$
  \tilde H_m - \tilde H_{m-1} \geq H_m - f_{13}(k) - (H_{m-1} + f_{13}(k)) = H_m-H_{m-1} - 2f_{13}(k).
$$
By (R4) for $X$, this is greater than or equal to
$$
  \frac{H_m}{v k^3} - 2f_{13}(k) \geq \left( \tilde H_m - f_{13}(k)\right) \frac{1}{vk^3} - 2f_{13}(k).
$$
Provided $k \geq \ell_0$ and $\ell_0$ is sufficiently large, this is $\geq \tilde H_m/(9v k^3).$ This proves (R4) for $\tilde X$.
  
  We now check (R5)(a) for $\tilde X$. Consider an order $k$ episode for $\tilde X$. At worst, this corresponds to an order $k+1$ episode for $X$, so the minimum value of $\tilde X$ over the episode is bounded below by
$$
  \frac{1}{v^2} f_8(k+1) - f_{13}(k) \geq \frac{1}{(9v)^2} f_8(k)
$$
provided $k \geq \ell_0$ and $\ell_0$ is large enough. A similar argument checks (R5)(b) as well as (R6)(a) and (R6)(b) for $\tilde X$. 

   We now check (R7) for $\tilde X$. For $k \geq \ell_0+1$, suppose that for some $m$, $\tilde H_m \in [2^{-k}, 2^{1-k}[$. By \eqref{rd9.8a}, $H_m \in [2^{-(k+1)}, 2^{1-(k-1)}]$, so by property (R7) for $X$, $H_{m} \leq H_{m-1}\,  v\, (k+1)^3$. Therefore,
\begin{align*}
   \tilde H_{m} &\leq H_{m} + f_{13}(k-1) \leq H_{m-1}\,  v\,  (k+1)^3 + f_{13}(k-1) \\
   & \leq (\tilde H_{m-1} + f_{13}(k))\,  v\,  (k+1)^3 + f_{13}(k-1).
\end{align*}
Since $\tilde H_m \in [2^{-k}, 2^{1-k}[$, the right-hand side is easily seen to be $\leq \tilde H_{m-1}\,  9v\,  k^3$, provided $\ell_0$ is sufficiently large.

   Now suppose that on a stage $2m-1$, the DW-algorithm for $X$ reaches the level $2^{-\ell_0}$ within $[-2^{-2\ell_0},2^{-2\ell_0}]^2$. Then, by \eqref{rd9.3}, we have that on this stage, the algorithm for $\tilde X $ must attain at least 
$2^ {- \ell_0 } -f_{13}( \ell_0) $, which (providing $\ell_0$ has been fixed sufficiently large) will exceed
$2^{-\ell_0-1}$.  Similarly, suppose (without loss of generality) that the DW-algorithm for $X$ escapes the square
$[-2^{-2\ell_0}, 2^{-2\ell_0}]^2$ during  stage $2m-1$ without having reached level $2^{-\ell_0}$. Then, using (c1), we must have that
the DW-algorithm for $\tilde X$ escapes $[-2^{-2\ell_0}+ \frac{2^{-2k}}{vk^{10}}, 2^{-2\ell_0}- \frac{2^{-2k}}{vk^{10}} ]^2$ where $k\geq \ell_0$.
If $\ell_0$ is sufficiently large, then this square will contain $[-2^{-2\ell_0-1}, 2^{-2\ell_0-1}]^2$.
  
  The proof of Proposition \ref{rdprop37} is complete. 
\end{proof}

   We now show that for large $v$, the requirement that the DW-algorithm behave $v$-robustly does not significally change the gambler's ruin probabilities.
  
\begin{prop} For $v \geq 2$, there exists $c(v) < \infty$ with the following properties:

  (a) Let $X$ be a standard additive Brownian motion. For $n \geq 1$, the probability that the DW-algorithm for $X$, started at (0,0) with value $2^{-n}$, reaches level 1 and does not behave $v$-robustly above order 1, is $\leq c(v) 2^{-n \lambda_1}$.
  
  (b) $\lim_{v \to \infty} c(v) = 0.$
\label{rdprop38}
\end{prop}
  
\begin{proof} For $i = 1, \ldots, 7,$ let $F_i(v)$ be the event ``property (Ri) holds with $k_0 = 0$''. The event considered in part (a) is
$$
  \{\underline \tau^1 < \ulinfty \} \cap (F^c_1(v) \cup \cdots \cup F^c_7(v)),
$$
whose probability is bounded above by
\begin{align*}
  &P_{2^{-n}}(\{\underline \tau^1 < \ulinfty\} \cap F^c_1(v))
   + P_{2^{-n}} (\{\underline \tau^1 < \ulinfty\} \cap F_1(v) \cap F_2^c(v)) \\
	&\qquad + P_{2^{-n}} (\{\underline \tau^1 < \ulinfty\} \cap F_1(v) \cap F_2(v) \cap F_3^c(v)) \\
  &\qquad\qquad + \sum^7_{i=4} P_{2^{-n}} (\{\underline \tau^1 < \ulinfty\} \cap F_1(v) \cap F_2(v) \cap F_3(v) \cap F_i^c(v)).
\end{align*}
We treat each term separately, beginning with the first term.
  
  Observe that $V(\ell) \leq \nu_\ell + 2,$ where $\nu_\ell$ is defined in Lemma \ref{rdlem9}. Using the Markov property of the DW-algorithm, Lemma \ref{rdlem9} and Theorem \ref{prop1}, we see that for $v \geq 9$,
\begin{eqnarray}
  P_{2^{-n}}(\{\underline \tau^1 < \ulinfty\} \cap F_1^c(v)) &\leq& \sum^n_{\ell=1} P_{2^{-n}} \{\underline \tau^{2^{-\ell}} < \ulinfty,\, \nu_\ell \geq \ell \sqrt{v} - 2,\, \underline \tau^1 < \ulinfty\} \nonumber \\ 
  &\leq& \sum^n_{\ell=1} P_{2^{-n}} \{\underline \tau^{2^{-\ell}} < \ulinfty \} \left(\frac{3}{4}\right)^{\ell \sqrt{v}-4} P_{2^{1-\ell}}\{ \underline \tau^1 < \ulinfty\} \nonumber \\
  &\leq& C  \sum^\infty_{\ell=1} 2^{\lambda_1(\ell-n)} \left(\frac{3}{4}\right)^{\ell \sqrt{v}} 2^{-\lambda_1 (\ell-1)} \nonumber \\
  &=& C \ 2^{-n \lambda_1} \frac{(3/4)^{ \sqrt{v}}}{1-(3/4)^{\sqrt{v}}}.
\end{eqnarray} 
Clearly, the fraction tends to 0 as $v \to + \infty.$
 
   Before considering (R2) and (R3), we consider $i=4$. For $i = 4$, $P_{2^{-n }}(\{\underline \tau^1 < \ulinfty\} \cap F_1(v) \cap F_4^c(v))$ is bounded above by
\begin{align}\nonumber
 & P_{2^{-n}}\Big( \bigcup^n_{k=1} \Big[ \{ \underline \tau^{2^{-k}} < \ulinfty,\, V(k) \leq k \sqrt{v}\} \\ \nonumber
 &\qquad\qquad\qquad \cap \bigcup_{m \geq 1} \{ H_m \in [2^{-k}, 2^{1-k}[,\, H_{m-1} \leq H_m < H_{m-1} (1- \frac{1}{v k^3})^{-1}\}\Big]\\
 &\qquad\qquad  \cap \{\underline \tau^1 < \ulinfty\}\Big).
 \label{rd9.10a}
\end{align}

   Each stage $m$ such that $H_m \in [2^{-k}, 2^{1-k}[$ produces an episode of order $k$, so there are no more than $k \sqrt{v}$ such stages on $\{V(k) \leq k \sqrt{v}\}$. Let
$$
   U(k,i,\ell) = \{H_{i} < 2^{-k},\ 2^{-k} \leq H_{i+1} < \cdots < H_{i+\ell} < 2^{1-k}\}
$$
\index{$U(k,i,\ell)$}be the event ``the last stage before reaching level $2^{-k}$ is stage $i$, level $2^{-k}$ is reached during stage $i+1$ and the maximum remains in $[2^{-k}, 2^{1-k}[$ for at least the next $\ell$ stages." Then the event in \eqref{rd9.10a} can be written
\begin{align*}
    & \bigcup^n_{k=1} \Big [\{ \underline \tau^{2^{-k}} < \ulinfty,\ V(k) \leq k \sqrt{v}\}  \cap \Big( \bigcup_{i = 1}^{\infty} \bigcup_{\ell=1}^{k \sqrt{v}} \Big( U(k,i,\ell) \\
   &\qquad \cap\, \left\{
H_{i+\ell-1} \leq H_{i+\ell} < H_{i+\ell-1} (1- \frac{1}{v k^3})^{-1},\ \underline \tau^1 < \ulinfty\right\}\Big)\Big) \Big].
\end{align*}
Given that the DW-algorithm does not stop at stage $i+\ell$, the conditional probability that $H_{i+\ell} < H_{i+\ell-1} (1- \frac{1}{v k^3})^{-1}$ is bounded by the probability that a Brownian motion started at $H_{i+\ell-1}$ will hit $0$ before $H_{i+\ell-1}(1- \frac{1}{v k^3})^{-1}$. Using gambler's ruin probabilities for standard Brownian motion and the Markov property of the DW-algorithm, we see that for fixed $v \geq 2$ and $k,i,\ell$,
\begin{align}\nonumber
  & P\Big( \{ \underline \tau^{2^{-k}} < \ulinfty,\ V(k) \leq k \sqrt{v}\} \cap U(k,i,\ell)\\ \nonumber
  &\qquad\qquad  \cap \left\{H_{i+\ell-1} \leq H_{i+\ell} < H_{i+\ell-1} (1- \frac{1}{v k^3})^{-1}\right\}\cap \{ \underline \tau^1 < \ulinfty\}\Big) \\
    &\qquad \leq  P( \{ \underline \tau^{2^{-k}} < \ulinfty,\ V(k) \leq k \sqrt{v}\} \cap U(k,i,\ell -1))\, k^{-3} v^{-1}\, 2^{-\lambda_1(k-1)}.
    \label{rd9.10b}
\end{align}
For fixed $k$ and $\ell$, the events $U(k,i,\ell-1)$, $i \geq 1$, are disjoint, so the probability in \eqref{rd9.10a} is bounded above by
\begin{align} \nonumber
   & \sum_{k=1}^n \sum_{\ell=1}^{k \sqrt{v}} P\{\underline \tau^{2^{-k}} < \ulinfty,\ V(k) \leq k \sqrt{v}\} \, k^{-3} v^{-1}\, 2^{-\lambda_1(k-1)} \\ \nonumber
   &\qquad \leq \sum_{k=1}^\infty 2^{-\lambda_1 (n-k)}\, k^{-2} v^{-\half}\, 2^{-\lambda_1(k-1)}\\
   &\qquad \leq 2^{(1-n)\lambda_1} \left(v^{-\half}\sum_{k=1}^\infty k^{-2} \right).
\label{e9.14}
\end{align}
Clearly, the factor in parentheses tends to $0$ as $v \to \infty$.

   For (R2), one easily checks that the probability that a given episode of order $k$ fails this property is $\leq e^{-cvk}$.
Therefore, using a decomposition similar to the one used in \eqref{rd9.10a}, but with the event $\{H_{m-1} \leq H_m < H_{m-1} (1- \frac{1}{v k^3})^{-1}\}$ replaced by $\ti G_{m,k} =$ ``Stage $m$ produces an episode of order $k$ that fails (R2) or (R3)," one then proceeds as in the lines leading to \eqref{rd9.10b}, and one bounds the probability of $\ti G_{i+\ell,k} \cap \{ \ultau^1 < \ulinfty\}$ by $3 e^{-cvk}\, 2^{-\lambda_1 (k-1)}$, so
\begin{align*}
   & P(\{ \ultau^{2^{-k}} < \ulinfty,\ V(k) \leq k\sqrt{v}\} \cap U(k,i,\ell) \cap \ti G_{i+\ell,k} \cap \{\ultau^1 < \ulinfty\} \\
	 & \qquad \leq P(\{ \ultau^{2^{-k}} < \ulinfty,\ V(k) \leq k\sqrt{v}\} \cap U(k,i,\ell)) \, 3 e^{-cvk}\, 2^{-\lambda_1 (k-1)}.
\end{align*}
One then completes the proof as in the case of (R4).

   Before discussing (R3), we handle (R5) to (R7). For (R5)(a), we again use a decomposition similar to the one in \eqref{rd9.10a}, but with the event $\{H_{m-1} \leq H_m < H_{m-1} (1- \frac{1}{v k^3})^{-1}\}$ replaced by $G_{m,k}' = $ ``Stage $m$ produces an episode of order $k$ which fails the requirements of property (R5a)." One then proceeds as in the lines leading to \eqref{rd9.10b}. The event $G_{m,k}'$ can occur because $v^{-1} f_8(k)$ is hit when it should not have been, or because $2^{1-k}$ is not hit $v$-robustly for order $k$. For a given episode of order $k$, the first situation occurs with probability $\leq v^{-1} k^{-8}$, and by Lemma \ref{lem9.2}, the second situation occurs with probability $\leq 3 v^{-3/2} k^{-3}$. This leads to the bound $4 v^{-1} k^{-3} 2^{-\lambda_1(k-1)}$ for the probability of $G_{i+\ell,k}' \cap \{ \ultau^1 < \ulinfty\}$. One then completes the proof as in the case of (R4), obtaining the same bound as in \eqref{e9.14} (multiplied by $4$).
	
	For (R5b), we again use a decomposition similar to the one in \eqref{rd9.10a}, but with the event $\{H_{m-1} \leq H_m < H_{m-1} (1- \frac{1}{v k^3})^{-1}\}$ replaced by $\ti G_{m,k}' = $ ``the highest order of an episode produced by Stage $m$ is $k$, and the requirement of (R5b) fails," that is, the difference of two independent Brownian motions over intervals of length $\geq v^{-1} k^{-7} 2^{-2k}$ is no more that $v^{-2} k^{-7} 2^{-k}$. This probability is bounded above by $(v^{-1} k^{-7} 2^{-2k})^{-1/2} v^{-2} k^{-7} 2^{-k} = v^{-3/2} k^{-7/2}$. Proceeding as in the lines that lead to \eqref{rd9.10b}, we find that the probability of $\ti G_{i+\ell,k}' \cap \{ \ultau^1 < \ulinfty\}$ is bounded above by $v^{-3/2} k^{-7/2} 2^{-\lambda_1(k-1)}$, which leads to the bound \eqref{e9.14} replaced by
$$
    2^{(1-n)\lambda_1} \left(v^{-1} \sum_{k=1}^\infty k^{-5/2} \right).
$$
Clearly, the factor in parentheses tends to $0$ as $v \to \infty$.
   
   Property (R6) is handled in the same way as (R5). As for (R7), we use a decomposition similar to \eqref{rd9.10a}, but with the event $\{H_{m-1} \leq H_m < H_{m-1} (1- \frac{1}{v k^3})^{-1}\}$ replaced by $\{H_m > H_{m-1} v k^3\}$. The left-hand side of \eqref{rd9.10b} becomes
\begin{align*}
	 & P\Big( \{ \underline \tau^{2^{-k}} < \ulinfty,\ V(k) \leq k \sqrt{v}\} \cap U(k,i,\ell)
    \cap \left\{H_{i+\ell} \geq H_{i+\ell-1} vk^3\right\}\cap \{ \underline \tau^1 < \ulinfty\}\Big) \\
    &\qquad \leq  P\big( \{ \underline \tau^{2^{-k}} < \ulinfty,\ V(k) \leq k \sqrt{v}\} \cap U(k,i,\ell )) \\
		&\qquad\qquad\qquad\qquad \cap \left\{H_{i+\ell} \geq H_{i+\ell-1} vk^3\right\}\big)\, 2^{-\lambda_1(k-1)} \\
		&\qquad \leq P\big( \{ \underline \tau^{2^{-k}} < \ulinfty,\ V(k) \leq k \sqrt{v}\} \cap U(k,i,\ell -1)) \\
		&\qquad\qquad\qquad\qquad \cap \left\{H_{i+\ell} \geq H_{i+\ell-1} vk^3\right\}\big)\, 2^{-\lambda_1(k-1)}.
\end{align*}
Conditional on the behavior of the DW-algorithm up to stage $m-1$, the probability that $H_{i+\ell} \geq H_{i+\ell-1} vk^3$ is bounded above by the probability that a Brownian motion started at $H_{i+\ell -1}$ hits $H_{i+\ell -1}v k^3$ before $0$, and is therefore $\leq v^{-1} k^{-3}$. This is the same as the quantity $k^{-3} v^{-1}$ in \eqref{rd9.10b}. One then completes the proof for (R7) as in the case of (R4).

   Finally, we turn to (R3). Looking back to the four types of episodes of order $k$, we see that during such an episode of type 1, 3 or 4, $X$ moves between levels that are at least $2^{-k}$ units apart, and the probability that a Brownian motion would do this during less than $v^{-2} k^{-7} 2^{-2k}$ units of time is $\leq e^{-cv^2k^7}$. Therefore, we handle these three types of episodes in the same way as we did (R2), replacing the definition of $\ti G_{m,k}$ by ``Stage $m$ produces an episode of order $k$ of type 1, 3 or 4 that fails (R3)."
	
	In order to handle episodes of type 2, redefine  $\ti G_{m,k}$ to be the event ``Stage $m$ produces an episode of order $k$ of type 2 that fails (R3)." If $H_{m-1} - H_{m-2} < 2^{1-k} - v^{-3/4} k^{-3} 2^{-k}$, then $X$ moves between two levels that are $v^{-3/4} k^{-3} 2^{-k}$ units apart, and the probability that a Brownian motion would do this during less than $v^{-2} k^{-7} 2^{-2k}$ units of time is $\leq e^{-cv^{1/2} k}$, which replaces the factor $e^{-cv^2 k^7}$in the discussion of types 1, 3 and 4. We now consider the probability that $2^{1-k} - v^{-3/4} k^{-3} 2^{-k} \leq H_{m-1} - H_{m-2}$. In this case, $2^{1-k} - v^{-3/4} k^{-3} 2^{-k} \leq H_{m-1} < 2^{1-k}$. We replace the event $\{H_{m-1} \leq H_m < H_{m-1} (1- \frac{1}{v k^3})^{-1}\}$ in \eqref{rd9.10a} with $\{2^{1-k} - v^{-3/4} k^{-3} 2^{-k} \leq H_{m-1} < 2^{1-k}\}$, which replaces the factor $k^{-3}v^{-1}$ in \eqref{rd9.10b} with $k^{-3} v^{-3/4}$. We then obtain \eqref{e9.14} with $v^{-1/2}$ replaced by $v^{-1/4}$.
	
	This completes the proof of Proposition \ref{rdprop38}.
\end{proof}

\end{section} 
\eject

\begin{section}{Lower bound for the Brownian sheet: the one-point estimate}\label{sec10}

   In this section, we will establish the lower bound needed to implement the second-moment argument: this will be done in Proposition \ref{rdprop41}, which is the main result of this section.
	
   Fix $r \in [2,3]^2$. For $u \geq -2,$ define
$$
   Z_1^r(u) = W(r_1+u, r_2) - W(r), \qquad Z_2^r(u) = W(r_1, r_2 + u) - W(r).
$$
\index{$Z_1^r$} \index{$Z_2^r$}Fix $x_0 > 0$ and $(s_1, s_2) \in [-2, + \infty[^2$. Set 
\begin{equation}\label{rd9.11a}
   X^r(s_1, s_2) = Z_1^r(s_1) + Z_2^r(s_2).
\end{equation}
\index{$X^r$}This is an additive Brownian motion which is not standard, since $(Z_i^r(u),\ u \in [-2,0])$, $i=1,2$, are correlated. We are going to construct a standard additive Brownian motion
$$
   \tilde X^r(s_1, s_2) = \tilde Z^r_1(s_1) + \tilde Z^r_2(s_2), \qquad (s_1, s_2) \in [-2, + \infty[^2,
$$ 
\index{$\tilde X^r$}which is close to $X^r$ with high probability. In particular, by Proposition \ref{rdprop37}, the DW-algorithms for $\tilde X^r$ and $X^r$ started at $(0,0)$ with value $x_0$ will typically behave similarly.

   Since $r$ is fixed, we omit the superscript $r$, and we shall use a superscript with a different meaning in our construction. Let $W^\prime$ be a Brownian sheet that is independent of $W$. We recall \cite{walsh} that $W^\prime$ can be defined from its associated white noise $\dot W^\prime.$

   Set $W^1 = W$, $X^1 = X$, $Z^1_i = Z_i$, $i = 1,2,$ and run Stage 1 of the DW-algorithm started at $(0,0)$ with value $x_0$ for $X^1.$ This produces in particular the two random variables $U_1$ and $U^\prime_1$ with $U_1 < 0 < U^\prime_1$. 

  Define a new Brownian sheet $W^2$ by letting its associated white noise be
$$
  \dot W^2 = \left\{\begin{array}{ll}
               \dot W^\prime &\mbox{ on } [r_1+ U_1, r_1] \times [0, r_2],\\
               \dot W^1 &\mbox{ elsewhere}.
             \end{array}\right.
$$
We define two Brownian motions $Z^2_1$ and $Z_2^2$ by
$$
Z_1^2(u) = Z^1_1(u), \qquad Z_2^2(u) = W^2(r_1, r_2+u) - W^2(r), \qquad u \geq -2.
$$
Note that $Z_2^2(u) = Z^1_2(u)$ if $u \geq 0$, and that $(Z_1^2(u),\ u \geq U_1)$ and $(Z_2^2(u), u \geq -2)$ are independent. We define an ABM $X^2$ by setting
$$
   X^2(s_1, s_2) = Z_1^2(s_1) + Z_2^2(s_2), \qquad (s_1, s_2) \in [-2, + \infty[^2.
$$
We note that Stage 1 of the DW-algorithm started at $(0,0)$ with value $x_0$ is the same for $X^2$ and $X^1$ (see also Figure \ref{rdfig1}).
\begin{figure}
\begin{center}
\begin{picture}(200,200)
\put(20,20){\line(0,1){190}}
\put(20,20){\line(1,0){190}}

\put(200,200){\circle*{3}}

\put(20,200){\line(1,0){180}}  \put(6,198){$r_2$}
\put(20,170){\line(1,0){150}}  \put(-20,168){$r_2+V_1$}
\put(20,140){\line(1,0){120}}  \put(-20,138){$r_2+V_2$}
\put(20,110){\line(1,0){90}}   \put(-20,108){$r_2+V_3$}

\put(200,20){\line(0,1){180}}  \put(200,8){$r_1$}
\put(170,20){\line(0,1){180}}  \put(160,8){$r_1+U_1$}
\put(140,20){\line(0,1){150}}  \put(120,8){$r_1+U_2$}
\put(110,20){\line(0,1){120}}  \put(80,8){$r_1+U_3$}

\multiput(170,20)(0,6){26}{\line(1,1){30}} \put(170,176){\line(1,1){24}} \put(170,182){\line(1,1){18}} \put(170,188){\line(1,1){12}}
\put(176,20){\line(1,1){24}} \put(182,20){\line(1,1){18}} \put(188,20){\line(1,1){12}} 
 
\multiput(20,200)(6,0){21}{\line(1,-1){30}} \put(146,200){\line(1,-1){24}} \put(152,200){\line(1,-1){18}} \put(158,200){\line(1,-1){12}} 
\put(20,194){\line(1,-1){24}} \put(20,188){\line(1,-1){18}} \put(20,182){\line(1,-1){12}} 

\multiput(140,20)(0,6){25}{\line(1,0){30}}

\multiput(20,140)(6,0){20}{\line(0,1){30}}

\multiput(110,20)(0,6){16}{\line(1,1){30}} \put(110,116){\line(1,1){24}} \put(110,122){\line(1,1){18}} \put(110,128){\line(1,1){12}}
\put(116,20){\line(1,1){24}} \put(122,20){\line(1,1){18}} \put(128,20){\line(1,1){12}}

\multiput(20,140)(6,0){11}{\line(1,-1){30}} \put(86,140){\line(1,-1){24}} \put(92,140){\line(1,-1){18}} \put(98,140){\line(1,-1){12}} 
\put(20,134){\line(1,-1){24}} \put(20,128){\line(1,-1){18}} \put(20,122){\line(1,-1){12}} 
\end{picture}
\end{center}

\caption{Constructing a standard ABM from a Brownian sheet. \label{rdfig1}} 
\end{figure}
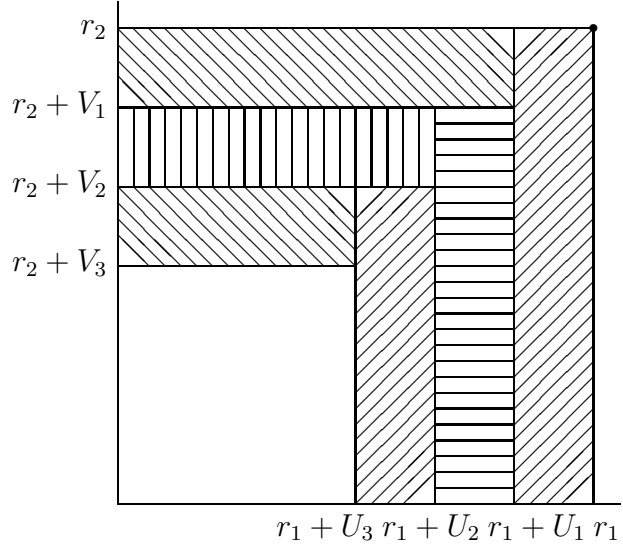

   We now run Stage 2 of the DW-algorithm for $X^2$. This produces in particular the two random variables $V_1 < 0 < V^\prime_1.$ We define a new Brownian sheet $W^3$ by letting its associated white noise be
$$
   \dot W^3 = \left\{\begin{array}{ll}
                      \dot W^\prime &\mbox{ on  } [0,r_1+U_1] \times [r_2+V_1, r_2], \\
                      \dot W^2 &\mbox{ elsewhere}.
                 \end{array}\right.
$$
We define two Brownian motions $Z_1^3$ and $Z_2^3$ by
$$
        Z_1^3(u) = 
        \left\{ \begin{array}{ll}
                   Z_1^2(u)& \mbox{ if  } u \geq U_1,\\
                                    \\
                   Z_1^2(U_1) + W^3 (r_1+u, r_2) - W^3(r_1+U_1,r_2) & \mbox{ if } -2 \leq u < U_1,
                \end{array}\right.
$$
and $Z_2^3(u) = Z_2^2(u)$, for $u \geq -2$.
We note that $(Z_1^3(u),\ u \geq -2)$ and $(Z_2^3(u),\ u \geq V_1)$ are independent. We define an ABM $X^3$ by setting
$$
   X^3(s_1, s_2) = Z_1^3(s_1) + Z^3_2(s_2), \qquad (s_1,s_2) \in [-2, + \infty[^2,
$$
and we note that Stages 1 and 2 for the DW-algorithm started at $(0,0)$ with value $x_0$ are the same for $X^2$ and $X^3$.

   We now proceed by induction, assuming that for some $n > 1,$ we have continued with the previous construction up to stage $2n-2$. We have constructed a Brownian sheet $W^{2n-1}$ and an ABM $X^{2n-1}(s_1, s_2) = Z^{2n-1}_1(s_1) + Z^{2n-1}_2(s_2)$ such that $(Z_1^{2n-1}(u),\ u \geq -2)$ and $(Z^{2n-1}_2(u),\ u \geq V_{n-1})$ are independent.

  We then run Stage $2n-1$ of the DW-algorithm for $X^{2n-1}$. This produces in particular the two random variables $U_n$ and $U^\prime_n$ such that $U_n < U_{n-1} < 0 < U^\prime_{n-1} < U^\prime_n$. We define a new Brownian sheet $W^{2n}$ by letting its associated white noise be
$$
   \dot W^{2n} = 
      \left\{ \begin{array}{ll}
                   \dot W^\prime &\mbox{ on } [r_1+U_n, r_1+U_{n-1}] \times [0, r_2+V_{n-1}],\\
                   \dot W^{2n-1} &\mbox{ elsewhere}.
              \end{array}\right.
$$
We define two Brownian motions $Z^{2n}_1$ and $Z_2^{2n}$ by $Z^{2n}_1(u) = Z_1^{2n-1}(u),\ u \geq -2$, 
$$
   Z_2^{2n}(u) = 
       \left\{ \begin{array}{ll}
                Z^{2n-1}_2(u) & \mbox{ if } u \geq V_{n-1},\\
                         \\
                  Z_2^{2n-1}(V_{n+1})+ W^{2n}(r_1, r_1 + u) - W^{2n}(r_1,r_2 + V_{n+1})& \mbox{ if } -2 \leq u < V_{n-1}.
                 \end{array}\right.
$$
Note that $(Z_1^{2n}(u),\ u \geq U_{n-1})$ and $(Z_2^{2n}(u),\ u \geq -2)$ are independent. We define an ABM $X^{2n}$ by setting
$$
   X^{2n}(s_1, s_2) = Z_1^{2n}(s_1) + Z_2^{2n}(s_2), \qquad (s_1, s_2) \in [-2, + \infty[^2.
$$
We note that Stages 1 to $2n-1$ of the DW-algorithm started at $(0,0)$ with value $x_0$ are the same for $X^{2n}$ and $X^{2n-1}.$

   We now run Stage $2n$ of the DW--algorithm for $X^{2n}$. This produces in particular the two random variables $V_n$ and $V^\prime_n$ such that $V_n < V_{n-1} < 0 < V^\prime_{n-1} < V^\prime_n.$ We define a new Brownian sheet $W^{2n+1}$ by letting its associated white noise be
$$
   \dot W^{2n+1} = \left\{\begin{array}{ll}
     \dot W^\prime &\mbox{ on } [0, r_1+U_n] \times [r_2+V_n, r_2+V_{n-1}],\\
     \dot W^{2n} &\mbox{ elsewhere}.
     \end{array}\right.
$$
We define two Brownian motions $Z_1^{2n+1}$ and $Z^{2n+1}_2$ by
$$
   Z_1^{2n+1}(u) = \left\{ \begin{array}{ll}
              Z_1^{2n}(u) & \mbox{ if } u \geq U_n,\\
                      \\
              Z_1^{2n}(U_n) + W^{2n+1}(r_1+u,r_2)-W^{2n+1}(r_1+U_n, r_2) & \mbox{ if } -2 \leq u < U_n,
          \end{array}\right.
$$
and $Z_2^{2n+1}(u) = Z_2^{2n}(u)$, $u \geq -2.$ Note that $(Z_1^{2n+1}(u),\ u \geq -2)$ and $(Z_2^{2n+1}(u),\ u \geq V_n)$ are independent. We define an ABM $X^{2n+1}$ by setting
$$
   X^{2n+1}(s_1, s_2) = Z_1^{2n+1}(s_1) + Z_2^{2n+1}(s_2), \qquad (s_1, s_2) \in [-2, + \infty[^2.
$$
This completes the inductive construction, which continues until the DW-algorithm terminates at Stage $2N$, say. At this point, we set
$$
   \tilde Z_1^r(u) = \left\{\begin{array}{ll}
               Z_1^{2N}(u) & \mbox{ if } u \geq U_N,\\
              \\
          Z_1^{2N}(U_N) + W^\prime(r_1+u, r_2) - W^\prime(r_1+U_N, r_2) & \mbox{ if } -2 \leq u \leq U_N,
              \end{array}\right.
$$
$$
   \tilde Z_2^r(u) = \left\{\begin{array}{ll}
              Z^{2N}_2(u) & \mbox{ if } u \geq V_N,\\
                            \\
           Z_2^{2N}(V_N) + W^{\prime \prime}(r_1, r_2+u) - W^{\prime \prime}(r_1, r_2 + V_N) & \mbox{ if } -2  \leq u \leq V_n,
\end{array}\right.
$$
where $W^{\prime \prime}$ is a Brownian sheet independent of all previously considered processes. We define an ABM $\tilde X^r$ by
\begin{equation}\label{08_08_sabm}
   \tilde X^r(s_1, s_2) = \tilde Z^r_1(s_1) + \tilde Z_2^r(s_2), \qquad (s_1, s_2) \in [-2, + \infty[^2.
\end{equation}
This ABM is standard (up to a deterministic time change) and Stages 1 to $2N$ of the DW-algorithms for $X^{2N}$ and $\tilde X^r$ are the same.

\begin{lemma} Let $W$ be a Brownian sheet. Fix $t_2 > 0.$ For $ 0 \leq v \leq t_2,$ let
$$
\F_v = \sigma(W(1, t_2-u) - W(1, t_2), \ 0 \leq u \leq v) \vee \G,
$$
where $\G$ is a $\sigma$-field that is independent of $W$. Let $V$ be a stopping time relative to $(\F_v)$ and let $0 \leq S_1 < S_2 \leq 1$ be random times that are $\cH \vee \F_V$-measurable, where $\cH$ is independent of $\sigma (\dot W\vert_{[0,1]\times [t_2 - V, t_2]})$.
Then for $x > 0$, $a \in [0, 1]$ and $v_0 > 0$,
\begin{eqnarray}
 &&  P\{\sup_{S_1 \leq u \leq S_2 \atop{0 \leq v \leq V}} \vert W([u, S_2] \times [t_2-v, t_2]) \vert \geq x \sqrt{v_0 a} + a \sup_{0 \leq v \leq V} \vert W(1, t_2-v)- W(1, t_2)\vert, \nonumber\\
 &&{\hskip 2in} S_2-S_1 \leq a,\ V \leq v_0 \mid \F_V \}
 \label{rd10.1}
\end{eqnarray}
is bounded above by $8 e^{-x^2/2}$.
\label{rdlem39}
\end{lemma}

\begin{figure}
\begin{center}
\begin{picture}(200,200)
\put(20,20){\line(0,1){190}}
\put(20,20){\line(1,0){190}}

\put(5,170){$t_2$}

\put(-15,150){$t_2 - v$}
\put(19,150){\line(1,0){2}}

\put(-15,130){$t_2 - V$}
\put(20,130){\line(1,0){150}}

\put(20,170){\line(1,0){150}}
\put(170,20){\line(0,1){150}}

\put(168,9){$1$}

\put(117,7){$S_2$}
\put(120,19){\line(0,1){2}}

\put(100,9){$u$}
\put(100,19){\line(0,1){2}}

\put(77,7){$S_1$}
\put(80,19){\line(0,1){2}}

\put(100,150){\line(1,0){20}}
\put(100,150){\line(0,1){20}}
\put(120,150){\line(0,1){20}}

\end{picture}
\end{center}
\caption{Illustration of Lemma \ref{rdlem39}. \label{figrdlem39}} 
\end{figure}
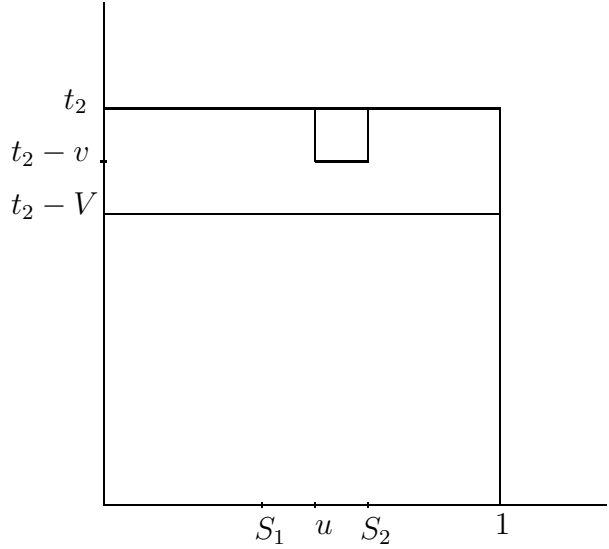

\begin{remark} As becomes clear from Figure \ref{figrdlem39}, this lemma applies with the roles of the coordinates exchanged.  Also it can be applied to the case where the first coordinate of $W$ is fixed at a value other than $1$, after trivial scaling. 
\label{remrdlem39} 
\end{remark}

\begin{proof} Observe by computing covariances that the process $Y(s_1, s_2) = W(s_1,s_2)-s_1 W(1, s_2)$ is independent of $\F_{t_2},$ and that
\begin{eqnarray*}
W([u, S_2] \times [t_2-v, t_2]) &=& Y (u, t_2-v) - Y(u,t_2) - Y (S_2, t_2-v) + Y(S_2, t_2)\\
&&\qquad\qquad - \, (S_2-u)(W(1, t_2-v)-W(1,t_2)).
\end{eqnarray*}
Therefore, the conditional probability in (\ref{rd10.1}) is no greater than
$$
  P\{ \sup_{0 \leq u \leq a \atop{0 \leq v \leq v_0}} \vert Y([u, a] \times [t_2-v, t_2]) \vert \geq x \sqrt{v_0a} \},
$$
so the lemma follows from \cite[Lemma 1.2]{OP}. 
\end{proof}
\vskip 12pt

\noindent{\em Some random partitions}
\vskip 12pt

The DW-algorithm for the standard ABM $\tilde X^r$ defined in \eqref{08_08_sabm} constructs a random partition
\begin{equation}\label{rd10.2}
  [\tilde U_N, \tilde U_{N-1}], \ldots, [\tilde U_1, \tilde U_0],\quad [\tilde U_0^\prime, \tilde U_1^\prime], \ldots, [\tilde U_{N-1}, \tilde U_N^\prime]
\end{equation}
of $[\tilde U_N, \tilde U^\prime_N]$. By using episodes, we have seen that we obtain a refinement of this partition.

   Enumerate the endpoints of the horizontal episodes for $\tilde X^r$ in $[-2, 0]$ in decreasing order: $\cdots < \tilde U^{(k)} < \tilde U^{(k-1)} < \cdots < \tilde U^{(1)} < \tilde U^{(0)} = 0,$ and the endpoints of the horizontal episodes for $\tilde X^r$ in $[0, + \infty[$ in increasing order:
$$
0 = \tilde U^{\prime(0)} < \tilde U^{\prime(1)} < \cdots < \tilde U^{\prime(k-1)} < \tilde U^{\prime(k)} < \cdots,
$$
so that every horizontal episode for $\tilde X^r$ is of the form $[\tilde U^{(k)}, \tilde U^{(k-1)}]$ or $[\tilde U^{\prime(k-1)}, \tilde U^{\prime(k)}]$. These episodes form a partition of $[\tilde U_N, \tilde U^\prime_N]$ that is a refinement of the partition in (\ref{rd10.2}).

   Similarly, we write the vertical episodes $[\tilde V^{(k)}, \tilde V^{(k-1)}]$ and $[\tilde V^{\prime(k-1)}, \tilde V^{\prime(k)}]$, with $\tilde V^{(k)} < \tilde V^{(k-1)} < 0 < \tilde V^{\prime(k-1)} < \tilde V^{\prime(k)}$. These intervals form a partition of $[\tilde V_N, \tilde V^\prime_N]$ that is a refinement of
$$
[\tilde V_N, \tilde V_{N-1}], \ldots, [\tilde V_1, \tilde V_0],\quad [\tilde V_0^\prime, \tilde V^\prime_1], \ldots, [\tilde V^\prime_{N-1}, \tilde V^\prime_N].
$$
The order of an episode $[\tilde U^{(k)}, \tilde U^{(k-1)}]$ (resp. $[\tilde U^{\prime(k-1)}, \tilde U^{\prime(k)}])$ is denoted $p(k)$\index{$p(k)$} (resp. $p^\prime(k)$\index{$p^\prime(k)$}), and that of $[\tilde V^{(k)}, \tilde V^{(k-1)}]$ (resp.~$[\tilde V^{\prime(k-1)}, \tilde V^{\prime(k)}]$) is denoted $q(k)$\index{$q(k)$} (resp.~$q^\prime(k)$\index{$q^\prime(k)$}). Note that the episode $[\tilde U^{(k)}, \tilde U^{(k-1)}]$ was produced when $\tilde X^r$ was no greater than $2^{-q(k)}$, and $k \mapsto p(k)$ and $k \mapsto q(k)$ are decreasing functions of $k$.

   The set of all episodes, vertical and horizontal, can be ordered according to when they occur during the DW-algorithm. We write $I \prec J$ to say that episode $I$ either occurs in an earlier stage than episode $J$, or occurs during the same stage but before episode $J$.
\vskip 12pt

\noindent{\em Increments of $W$ and $W^\prime$ over products of episodes}
\vskip 12pt

   Define
$$
   W_r(u_1,u_2) = W(r_1+u_1,r_2+u_2), \qquad W^\prime_r(u_1,u_2) = W^\prime(r_1+u_1,r_2+u_2),
$$
and let
$$
   E^1_{k, \ell} = \sup_{\tilde U^{\prime(k-1)} < u < \tilde U^{\prime(k)} \atop{\tilde V^{\prime(\ell-1)} < v < \tilde V^{\prime(\ell)}}} \vert W_r([\tilde U^{\prime(k-1)}, u] \times [\tilde V^{\prime(\ell-1)}, v]) \vert \vee \vert W^\prime_r([\tilde U^{\prime(k-1)}, u] \times [\tilde V^{\prime(\ell-1)}, v])\vert,
$$
$$
  E^2_{k,\ell} = \sup_{\tilde U^{(k)} < u < \tilde U^{(k-1)}\atop{\tilde V^{\prime(\ell-1)} < v < \tilde V^{\prime(\ell)}}} \vert W_r([u,\tilde U^{(k-1)}] \times [\tilde V^{\prime(\ell-1)}, v]) \vert \vee \vert W^\prime_r([u, \tilde U^{(k-1)}] \times [\tilde V^{\prime(\ell-1)}, v]) \vert .
$$
The random variables $E^3_{k, \ell}$ and $E^4_{k, \ell}$ are defined similarly, relative to quadrants 3 and 4, respectively.

   Let
$$
a^1_{k \ell} = 2^{-\frac{3}{4}(p^\prime(k)+q^\prime(\ell))}, \ a^2_{k, \ell} = 2^{- \frac{3}{4}(p(k)+q^\prime(\ell))}, \ a^3_{k,\ell} = 2^{- \frac{3}{4}(p(k)+q(\ell))}, \ a^4_{k \ell} = 2^{-\frac{3}{4}(p^\prime(k) + q(\ell))},
$$
\begin{eqnarray} \label{rd10.2b}
K^1 &=& \inf \{k : \exists \ell \ \mbox{ with } \ q^\prime(\ell) \geq p^\prime(k) \ \mbox{ and } \ E^1_{k, \ell} \geq a^1_{k, \ell},\\ \nonumber
&&\qquad\qquad \mbox{ or } \ \exists \ell \ \mbox{ with } q(\ell) \geq p^\prime(k) \ \mbox{ and } \ E^4_{k, \ell} \geq a^4_{k, \ell} \},\\ \nonumber
\\ \nonumber
K^2 &=& \inf\{ \ell: \exists k \ \mbox{ with } p^\prime(k) \geq q^\prime(\ell) \ \mbox{ and } E^1_{k, \ell} \geq a^1_{k, \ell},\\ \nonumber
&&\qquad\qquad \mbox{ or } \ \exists k \ \mbox{ with } \ p(k) \geq q^\prime(\ell) \ \mbox{ and } E^2_{k, \ell} \geq a^2_{l, \ell}\},\\ \nonumber
\\ \nonumber
K^3 &=& \inf \{k: \exists \ell \ \mbox{ with } \ q^\prime(\ell) \geq p(k) \ \mbox{ and } E^2_{k, \ell} \geq a^2_{k,\ell},\\ \nonumber
&&\qquad\qquad \mbox{ or } \ \exists \ell \ \mbox{ with } \  q(\ell) \geq p(k) \ \mbox{ and } \ E^3_{k, \ell} \geq a^3_{k, \ell} \},\\ \nonumber
\\ \nonumber
K^4 &=& \inf \{ \ell: \exists k \ \mbox{ with } \ p(k) \geq q(\ell) \ \mbox{ and } \ E^3_{k, \ell} \geq a^3_{k, \ell},\\ \nonumber
&&\qquad\qquad \mbox{ or } \ \exists k \ \mbox{ with } \ p^\prime(k) \geq q(\ell) \mbox{ and } \ E^4_{k, \ell} \geq a^4_{k, \ell} \}.
\end{eqnarray}
In particular, $K^1$ is the smallest $k$ such that there is a vertical episode $I^{(\ell)}$ that occurs earlier in the DW-algorithm such that $W_r$ or $W^\prime_r$ has an unusually large rectangular increment in $[\tilde U'^{(k-1)},\tilde U'^{(k)}] \times I^{(\ell)}$.

   We denote
\begin{equation}\label{10.6rd}
   \cO(\tilde X^r, 2^{-n}, m,v)
\end{equation}
the event ``the DW-algorithm for $\tilde X^r$ started at level $2^{-n}$ is $v$-robust above order $m$." 
  
\begin{lemma} Set $\tilde Q = p^\prime(K^1) \vee q^\prime(K^2) \vee p(K^3) \vee q(K^4)$.\index{$\tilde Q$} For all $v \geq 1$, for all $m \geq 1$, there exists $c(v,m)$ such that for all $n \geq m$,
$$
  P_{2^{-n}} (\{\tilde \ultau^{2^{-m}}  < \infty,\ \tilde Q \geq m\} \cap \cO(\tilde X^r, 2^{-n}, m,v)) \leq c(v,m) 2^{(m-n) \lambda_1}
$$
and $\lim_{m \to \infty} c(v,m)= 0.$
\label{rdlem40}
\end{lemma}
   
  \begin{proof} We first seek to show that 
\begin{equation}\label{rd10.2a}
   P_{2^{-n}}(\{\tilde \ultau^{2^{-m}} < \infty,\ \tilde Q=m\} \cap \cO(\tilde X^r, 2^{-n}, m,v)) \leq C 2^{(m-n)\lambda_1} \sqrt{v}\, m \, \exp(-2^{4m/5}).
\end{equation}
When $\tilde Q = m$, one of the four quantities in the formula for $\tilde Q$ is equal to $m$ and the others are no greater than $m$. If, for instance, $p(K^3) = m$, then there is $\ell$ such that either $q^\prime(\ell) \geq m$ and $E^2_{k, \ell} \geq a^2_{k, \ell}$, or $q(\ell) \geq m$ and $E^3_{k, \ell} \geq a^3_{k, \ell}$, where $k = K^3$. We can therefore decompose the event $\{\tilde \ultau^{2^{-m}} < \infty,\ \tilde Q=m\} \cap \cO(\tilde X^r, 2^{-n}, m,v)$ into the union of eight events and separately bound the probability of each. We treat explicitly only one of these events, namely
$$
  F = \bigcup_{\ell:\, q(\ell) \geq m} F_\ell,
$$
where
\begin{align*}
   F_\ell &= \{\tilde \ultau^{2^{-m}} < \infty,\ p(K^3) = m,\  p^\prime(K^1) \vee q^\prime (K^2) \vee q(K^4) \leq m,\ E^3_{K^3, \ell} \geq a^3_{K^3, \ell}\} \\
	&\qquad \cap \cO(\tilde X^r, 2^{-n}, m,v).
\end{align*}
On this event, the DW-algorithm for $\tilde X^r$ is in particular $v$-robust above order $m$. On $F_\ell$, there are two further cases to consider: either $[\tilde U^{(k)}, \tilde U^{(k-1)}] \prec [\tilde V^{(\ell)}, \tilde V^{(\ell-1)}]$ (only possible if $q(\ell) = m$, too) or the opposite relation holds. We only consider the latter case, so we bound
$$
  P_{2^{-n}}\left(\bigcup_{\ell:\, q(\ell) \geq m} (F_\ell \cap \{[\tilde V^{(\ell)}, \tilde V^{(\ell-1)}] \prec [\tilde U^{(k)}, \tilde U^{(k-1)}]\}) \right).
$$

  Let 
$$
  E^{3, \ast}_{k, \ell} = \sup_{{\tilde U^{(k)}< u < \tilde U^{(k-1)}}\atop{\tilde V^{(\ell)} < v < \tilde V^{(\ell-1)}}} \vert W_r([u, \tilde U^{(k-1)}] \times [v, \tilde V^{(\ell -1)}])\vert,
$$
and let $E^{3, \ast \ast}_{k, \ell}$ be defined in the same way but with $W$ replaced by $W^\prime$. Then 
$$
  \{E^3_{k, \ell} \geq a^3_{k, \ell}\} = \{E^{3,\ast}_{k, \ell} \geq a^3_{k, \ell}\} \cup \{E^{3, \ast \ast}_{k, \ell} \geq a^3_{k, \ell}\}.
$$
Therefore, we can further split $F_\ell$ into the union of two events $F^\ast_\ell$ and $F^{\ast \ast}_\ell$, where $E^3_{k \ell}$ is replaced by $E^{3\ast}_{k, \ell}$ and $E^{3 \ast \ast}_{k, \ell}$, respectively.
  
  Let $H$ be the event ``the DW-algorithm for $\tilde X^r$ started at $2^{-n}$ reaches level $2^{-m}$ and is $v$-robust above order $m$", and let $G^\ast_\ell$ be the intersection of $H$ and the event ``$p(K^3) = m,$ there exists $\ell$ such that $q(\ell) \geq m$ and $[\tilde V^{(\ell)}, \tilde V^{(\ell-1)}] \prec [\tilde U^{(k)}, \tilde U^{(k-1)}],$ and $E^{3, \ast}_{k, \ell} \geq a^3_{k, \ell}$'' (where $k=K^3$). Let $G^{\ast \ast}_\ell$ be the event defined in the same way as $G^\ast_\ell$ but with $E^{3, \ast}_{k, \ell}$ replaced by $E^{3, \ast \ast}_{k, \ell}$. Then $F^\ast_\ell \subset G^\ast_\ell$ and $F^{\ast \ast}_\ell \subset G^{\ast \ast}_\ell.$
  
  Using property (R2) of $v$-robustness, we observe that $G^{\ast \ast}_\ell$ is contained in the event
\begin{eqnarray*}
 && H \cap \{E^{3, \ast \ast}_{k, \ell} \geq a^3_{k, \ell},\ \sup_{\tilde U^{(k)} < u < \tilde U^{(k-1)}} \vert \tilde X^r(u, r_2)-\tilde X^r(\tilde U^{(k-1)}, r_2)\vert \leq 2^{1-m},\\
 &&\qquad\qquad 0 \leq \tilde U^{(k-1)}-\tilde U^{(k)} \leq v m 2^{-2m},\ 0 \leq \tilde V^{(\ell-1)} - \tilde V^{(\ell)} \leq v q(\ell)2^{-q(\ell)}\},
\end{eqnarray*}
and that this is contained in
\begin{eqnarray}\nonumber
  &&H \cap \{E^{3, \ast \ast}_{k, \ell} \geq a^3_{k, \ell} - v q(\ell) 2^{-2 q(\ell)}2^{1-m}\\
  &&\qquad\qquad\qquad\qquad\qquad +\, v q(\ell) 2^{-2q(\ell)} \sup_{\tilde U^{(k)}< u < \tilde U^{(k-1)}} \vert \tilde X^r(u, r_2) - \tilde X^r(\tilde U^{k-1}, r_2) \vert, \nonumber\\
 &&\qquad\qquad 0 \leq \tilde U^{(k-1)} - \tilde U^{(k)} \leq v m 2^{-2m},\ 0 \leq \tilde V^{(\ell-1)}-\tilde V^{(\ell)} \leq v q(\ell) 2^{-2q(\ell)}\}.\qquad
\label{rd10.3}
\end{eqnarray}
Let $\G^{\ast \ast}$ be the $\sigma$-field generated by increments of $\tilde X^r$ in the DW-algorithm for $\tilde X^r$ up to episode $k$ not included, and let $\F^{\ast \ast}_k$ be the $\sigma$-field generated by increments of $\tilde X^r$ over $[\tilde U^{(k)}, \tilde U^{(k-1)}] \times \{r_2\}$. We note that $H \in \G^{\ast \ast} \vee \F_k^{\ast \ast}$ and therefore we can use Lemma \ref{rdlem39} to bound the probability of the event in (\ref{rd10.3}) (see also Remark \ref{remrdlem39}).

 Indeed, refering to the notations of Lemma \ref{rdlem39} (with the roles of the axes exchanged), $\F^{\ast \ast}_k$, $\G^{\ast \ast}$, $\tilde U^{(k-1)} - \tilde U^{(k)}$, $\tilde V^{\ell-1)} - \tilde V^{(\ell)}$, $v q(\ell) 2^{-2 q(\ell)}$, $v m 2^{-2m}$ play respectively the role of $\F_v$, $\G$, $V$, $S_2-S_1$, $a$ and $v_0$.
  
  Choose $x$ so that
$$
  x \sqrt{v q(\ell) 2^{-2q(\ell)} v m 2^{-2m}} = a^3_{k,\ell}-v q(\ell) 2^{-2q(\ell)} 2^{1-m},
$$
that is
$$
  x = \frac{1}{v \sqrt{m q(\ell)}} \left(2^{(m+q(\ell))/4} - v q(\ell)2^{1-q(\ell)}\right) \geq c_0 2^{(m+q(\ell))/5}.
$$
Then Lemma \ref{rdlem39} and Theorem \ref{prop1} imply that
$$
  P_{2^{-n}} (G^{\ast \ast}_\ell) \leq C E_{2^{-n}}( 1_H \exp(-2^{2(m+q(\ell))/5})).
$$

   Turning to $G^\ast_\ell$, observe that $G^\ast_\ell$ is contained in the event 
\begin{eqnarray*}
 && H \cap \{E^{3, \ast}_{k, \ell} \geq a^3_{k, \ell},\ \sup_{\tilde V^{(\ell)} < \tilde v < \tilde V^{(\ell-1)}} \vert \tilde X^r(r_1, \tilde v) - \tilde X^r(r_1, \tilde V^{(\ell-1)})\vert \leq 2^{1-q(\ell)}, \\
 &&\qquad\qquad\qquad 0 \leq \tilde U^{(k-1)}-\tilde U^{(k)} \leq v m 2^{-2m},\ 0 \leq \tilde V^{(\ell-1)}-\tilde V^{(\ell)} \leq v q(\ell) 2^{-2q(\ell)}\},
\end{eqnarray*}
and that this is contained in
\begin{eqnarray}\nonumber
 && H \cap \{E^{3, \ast}_{k, \ell} \geq a^3_{k, \ell} - v m 2^{-2m} 2^{1-q(\ell)}\\
 &&\qquad\qquad\qquad\qquad  +\, v m 2^{-2m} \sup_{\tilde V^{(\ell)} < \tilde v < \tilde V^{(\ell-1)}} \vert \tilde X^r(r_1, \tilde v) - \tilde X^r(r_1, \tilde V^{(\ell-1)})\vert, \nonumber\\
 &&\qquad\qquad 0 \leq \tilde U^{(k-1)}-\tilde U^{(k)} \leq v m 2^{-2m},\ 0 \leq \tilde V^{(\ell-1)}- \tilde V^{(\ell)} \leq v q(\ell) 2^{-2 q(\ell)}\}.\qquad
\label{rd10.4}
\end{eqnarray}
Let $\G^\ast$ be (somewhat informally) the $\sigma$-field generated by increments of $\tilde X^r$ used by the DW-algorithm up to episode $k$ included, {\em except} for those of $\tilde X^r$ over $\{r_1\} \times [\tilde V^{(\ell)}, \tilde V^{(\ell-1)}],$ and let $\F^\ast_\ell$ be the $\sigma$-field generated by increments of $\tilde X^r$ over this segment. We note that $H \in \G^\ast \vee \F^\ast_\ell$ and therefore we can use Lemma \ref{rdlem39} to bound the probability of the event in (\ref{rd10.4}). Referring again to the notations of Lemma \ref{rdlem39}, $\F^\ast_\ell$, $\G^\ast$, $\tilde V^{\ell-1)}-\tilde V^{(\ell)}$, $\tilde U^{(k-1)}-\tilde U^{(k)}$, $v m 2^{-2m}$, $v q(\ell) 2^{-2q(\ell)}$ play respectively the role of $\F_V$, $\G$, $V$, $S_2-S_1$, $a$ and $v_0$.
  
  Choose $x$ so that
$$
  x \sqrt{v m 2^{-2m}\, v q(\ell) 2^{-2q(\ell)}} = a^3_{k, \ell} - vm 2^{-2m} 2^{1-q(\ell)},
$$
that is
$$
  x = \frac{1}{v \sqrt{m q(\ell)}} \left(2^{(m+q(\ell))/4} - v m 2^{1-m}\right) \geq c \ 2^{(m+q(\ell))/5}.
$$
Then Lemma \ref{rdlem39} implies that
\begin{equation}\label{rd10.4cp}
  P_{2^{-n}}(G^\ast_\ell) \leq C E_{2^{-n}}(1_H \exp(-2^{2(m+q(\ell))/5})).
\end{equation}
Taking into account property (R1) of $v$-robustness, we conclude from Theorem \ref{prop1} that
\begin{equation}\label{rd10.4bp}
  P_{2^{-n}}(F) \leq C \ 2^{(m-n)\lambda_1} \sum^\infty_{q=m} q \sqrt{v} \exp(-2^{2(m+q)/5}).
\end{equation}
The series is bounded by a constant times $\sqrt{v} \ m \ \exp(-2^{4m/5})$, establishing \eqref{rd10.2a}, and therefore
\begin{align}\nonumber
  &P_{2^{-n}} (\{\tilde\ultau^{2^{-m}}< \infty,\ \tilde Q \geq m\} \cap \cO(\tilde X^r, 2^{-n}, m,v)) \\
	&\qquad\leq \sum^n_{k=m} P_{2^{-n}} (\{\tilde\ultau^{2^{-m}}< \infty,\ \tilde Q = k\} \cap \cO(\tilde X^r, 2^{-n}, k,v)) \\
	&\qquad\leq 2^{(m-n)\lambda_1} \sum^\infty_{k=m} C \sqrt{v} \ k \exp(-2^{4k/5}),
\label{rd10.4ap}
\end{align}
and this series is the desired constant $c(v,m).$
\end{proof}

    We let $\cC_r$ denote the connected component of $\{s \in \rtoo: W(s) >0 \}$ that contains $r$, and we let $\partial \cC_r^\alpha$ denote the subset of points in $\partial \cC_r$ that can be reached from $r$ by a curve contained in $\cC_r \cap ([r_1,r_1+\alpha[\,\times [r_2,r_2 + \alpha[)$.
    
    Fix $r \in [2,3]^2,$  consider the additive Brownian motion $X^r$ defined in (\ref{rd9.11a}), and let $x_0 = 2^{-n}$. Define
$$
  A_1(r,n) = \{W(r) \in [-2^{1-n}, -2^{-n}[\}.
$$
\index{$A_1(r,n)$}For $k_0 \in \IN$ and $n \geq k_0$, let $A_2(r,n,k_0,v)$\index{$A_2(t,n,k_0,v)$} be the intersection of the event ``the DW-algorithm applied to $X^r$ started at level $2^{-n}$, reaches level $2^{-k_0}$ before escaping $\cR(2^{-2k_0})$ or escapes $\cR(2^{-2k_0})$, and is $v$-robust above order $k_0$'' and $\{Q \leq k_0\}$, where $Q$ is defined in the same way as $\ti Q$ in Lemma \ref{rdlem40}, but after removing the terms involving $W'$ in the definition of the $E^j_{k,\ell}\,$, using episodes for $X^r$ instead of episodes for $\ti X^r$, and only episodes ``seen" by the DW-algorithm up to $\ultau^{2^{-n},r,2^{-k_0 }}$.

  As in Lemma \ref{lem6.20}, let $\gamma^\ast$ be the one-to-one parameterization by arc-length, such that $\gamma^\ast(0) = r$, of the path $\Gamma^\ast$ constructed by the DW-algorithm applied to $X^r$. Let
$$
   \rho = 3
$$
(this is the exponent of $k$ that appears in property (R7) of Section \ref{sec9}), let $d(u)$ be the distance (in $\ell^1$-norm) between $\gamma^\ast(u)$ and $r$, and set
\begin{align*}
 & A_3(r,n,k_0,v) \\
 &\qquad = \bigcap_{k=k_0}^n \Big\{W(\gamma^\ast(u))-W(r) \geq \frac{\tilde c\, \sqrt{d(u)}}{v^{7/4} \log_2^9(\frac{1}{d(u)})},  
  \quad\frac{\tilde c\, 2^{-2k}}{v^{10}\, k^{14+8\rho}} \leq d(u) \leq 48 v^{3/2} k^2 2^{-2k},\\
  &\qquad\qquad\qquad\qquad\qquad \mbox{ for all } u \in [\alpha_n^{n-k,r}, \alpha_n^{n-k+1,r}]\Big\},
\end{align*}
where $\alpha_n^{n-k,r}$ is defined in Lemma \ref{lem6.20} and $\tilde c$ is a constant specified in the proof of Lemma \ref{rdlem42} below.

   Finally, let $A_4(r,n,k_0)$ be the event ``there is a path with extremities (1,1) and $\alpha_n^{n-k_0+1,r}$ contained in $[1,4]^2$ along which $W(\cdot) - W(r) \in [\frac{2^{-k_0-1}}{2}, 11]$" $(\alpha^{j,r}_n$ is defined as in Lemma \ref{lem6.20}). Finally, set
\begin{equation}\label{rd10.4dp}
  A(r,n,k_0,v) = A_1(r,n) \cap A_2(r,n,k_0,v) \cap A_3(r,n,k_0,v) \cap A_4(r,n,k_0).
\end{equation}

   Notice that if $A(r,n,k_0,v)$ occurs, then $r$ is no more that $\tilde{c}^{-1}\, 2^{18}\, v^{7/2}\, n^{18}\, 2^{-2n}$ units away from a point in $\partial \cC_{(1,1)}^4$.
   
   The following proposition is the main result of this section.

\begin{prop} Fix a compact subset $I$ of $]0,+\infty[$. There are $v \geq 1$, $k_0 \in \IN \setminus \{0\}$ and $c = c_{I,v,k_0}>0$ such that for all $r \in [2,3]^2$, all sufficiently large $n$, and all $y \in I$,
$$
  P(A(r,n,k_0,v) \mid W(1,1) = y) \geq c \ 2^{-n(1+\lambda_1)}.
$$
\label{rdprop41}
\end{prop}

     A key ingredient in the proof of this proposition is the following lemma.
     
\begin{lemma} There exists $v \geq 1$, $k_0 \in \IN \setminus \{0\}$ and $c > 0$ such that for all $r \in [2,3]^2$ and all sufficiently large $n$,
\begin{equation}\label{rd10.5}
  P(A_2(r,n,k_0,v) \cap A_3(r,n,k_0,v)) \geq c \ 2^{(k_0-n)\lambda_1}.
\end{equation}
\label{rdlem42}
\end{lemma}


\noindent{\em Proof of Lemma \ref{rdlem42}}. Let $\tilde Q$ be the random variable defined in Lemma \ref{rdlem40}. Let $\ti X^r$ be the ABM defined in \eqref{08_08_sabm} (with $x_0 = 2^{-n}$). Given $k_0 \geq 1$, let $\tilde F_1 := \cO(\tilde X^r,2^{-n},k_0,\frac{v}{9})$ 
and set
\begin{equation}\label{08_08_tiF}
  \tilde F = \{\underline{\tilde \tau}^{2^{-k_0}} < \underline{\tilde {\sigma}}^{\cR(2^{-2k_0})}\}\cap \tilde F_1 \cap \{\tilde Q \leq k_0\}.
\end{equation}
We note that
\begin{eqnarray*}
  P_{2^{-n}}(\tilde F) &\geq& P_{2^{-n}}\{\underline{\tilde \tau}^{2^{-k_0}} < \underline{\tilde \sigma}^{\cR(2^{-2k_0})}\} - P_{2^{-n}}(\{\underline{\tilde \tau}^{2^{-k_0}} < \infty \} \cap \tilde F_1^c)\\
  &&{\hskip 0.5in} - P_{2^{-n}}(\{\underline{\tilde \tau}^{2^{-k_0}} < \infty,\ \tilde Q > k_0\} \cap \tilde F_1 ).
\end{eqnarray*}
Apply Lemma \ref{lem6.18}, Proposition \ref{rdprop38} and Lemma \ref{rdlem40}, to see that for all large $n$, this is bounded below by
$$
  c_0 2^{(k_0-n)\lambda_1} - c\left(\frac{v}{9}\right) 2^{(k_0-n)\lambda_1} - c\left(\frac{v}{9}, k_0\right) 2^{(k_0-n)\lambda_1}.
$$
Fix $v$ large enough so that $c(\frac{v}{9}) < \frac{c_0}{4}$, then $k_0$ large enough,  so that $c(\frac{v}{9},k_0) < \frac{c_0}{4}$, to conclude that for all large $n$,
\begin{equation}\label{08_08PtiF}
  P_{2^{-n}}(\tilde F) \geq \frac{c_0}{2} \ 2^{(k_0-n) \lambda_1}.
\end{equation}

   With $v$ and $k_0$ fixed as above, 
we are going to show that
\begin{equation}\label{10.16rd}
  \tilde F \subset A_2(r,n,k_0 +1,v) \cap A_3(r, n, k_0 +1,v).
\end{equation}
This will establish (\ref{rd10.5}), with $c = c_0/2.$ 

   Let $[\tilde U^{(k)}, \tilde U^{(k-1)}]$ be a horizontal episode with $p(k) \geq k_0$. For $u\in [\tilde U^{(k)}, \tilde U^{(k-1)}]$, the construction of $\tilde X^r$ implies that on $\tilde F_1 \cap \{\tilde Q \leq k_0\}$,
$$
  \vert X^r(u,0) - \tilde X^r(u, 0)\vert \leq \sum E^3_{m, q},
$$
where the sum is over all horizontal episodes with $m \geq k$ and all vertical episodes $[\tilde V^{(q)}, \tilde V^{(q-1)}]$ with $[\tilde V^{(q)}, \tilde V^{(q-1)}] \prec [\tilde U^{(k)}, \tilde U^{(k-1)}]$. Since $\tilde Q \leq k_0$ on $\tilde F$, we deduce from properties (R1) and (R2) of $\frac{v}{9}$-robustness that
\begin{eqnarray}
  \sum E^3_{m, q} &\leq& \sum^\infty_{p=p(k)} \sqrt{\frac{v}{9}} \ p \sum^\infty_{q=p(k)} \sqrt{\frac{v}{9}} \ q\, 2^{-\frac{3}{4}(p+q)} \nonumber \\
  &\leq& C \ \frac{v}{9}\, p(k)^2\, 2^{-\frac{3}{2} p(k)},
  \label{rd10.8}
\end{eqnarray}
where $C$ is a universal constant. Therefore, for $u\in [\tilde U^{(k)}, \tilde U^{(k-1)}]$ with $p(k) \geq k_0$,
\begin{equation}\label{10_10.19}
   \vert X^r(u,0) - \tilde X^r(u, 0)\vert \leq  C \ \frac{v}{9}\, p(k)^2\, 2^{-\frac{3}{2} p(k)}.
\end{equation}

   We are now going to check the assumptions of Proposition \ref{rdprop37}, with $\tilde X^r$ (resp. $X^r$) playing the role of $X$ (resp. $\tilde X$) there, and $\ell_0$ there replaced by $k_0$.
   
   Let $\ultau$ be defined as in Proposition \ref{rdprop37} (relative to $\tilde X^r$). Since, for $u\geq 0$, $X^r(u,0) = \tilde X^r(u,0)$ and $X^r(0,u) = \tilde X^r(0,u)$, \eqref{rd9.1} will be checked if we establish that for $k_0 \leq \ell \leq n$,
\begin{align}\nonumber
  & \max\Big[\sup_{-(\ultau_3 \wedge 2^{-2\ell}) \leq u \leq 0} \vert X^r(u,0) - \tilde X^r(u, 0)\vert,\\
   &\qquad\qquad \sup_{-(\ultau_4 \wedge 2^{-2\ell}) \leq u \leq 0} \vert X^r(0,u) - \tilde X^r(0,u)\vert\Big] \leq f_{15}(\ell).
\label{10_10.20}	
\end{align}

   Let $u \in [-\ultau_3,0]$ and let $[\tilde U^{(k)}, \tilde U^{(k-1)}]$ be the horizontal episode containing $u$, which must be such that $p(k) \geq k_0$. Suppose that $u \in [-2^{-2 \ell}, -2^{-2(\ell+1)}]$. We are going to show that because of $\frac{v}{9}$-robustness, $\ell$ is related to $p(k)$ by the following inequalities:
\begin{equation}\label{10_10.21}
   \ell - (7+4\rho) \log_2 \ell - 5 \log_2 \frac{v}{9} - c \leq p(k) \leq \ell + \log_2 \ell + \frac{3}{4} \log_2 \frac{v}{9} +c,
\end{equation}
where $c$ is a universal constant.

   Indeed, as in \eqref{sum_ep}, the sum of all lengths of episodes of order greater than $p(k)$ is at most $12 (\frac{v}{9})^{3/2} p(k)^2\, 2^{-2p(k)}$, so 
\begin{equation}\label{10_10.22}
   2^{-2(\ell+1)} \leq 12 \left(\frac{v}{9}\right)^{3/2} p(k)^2\, 2^{-2p(k)},
\end{equation}
implying
\begin{equation}\label{10_10.23}
   p(k) \leq \ell + \log_2 p(k) + \half \log_2\left(12 \left(\frac{v}{9}\right)^{3/2}\right) +1
\end{equation}

   It is slightly more subtle to establish a lower bound on $p(k)$, since it could be possible that $[\tilde U^{(k)}, \tilde U^{(k-1)}]$ is the first episode in a stage and that previous horizontal episodes were of substantially higher order. However, this order is controlled by property (R7). Indeed, suppose that $[\tilde U^{(k)}, \tilde U^{(k-1)}]$ occurs in stage $m$, and $H_{m-1} \in [2^{-k_1}, 2^{1-k_1}]$, $H_{m-2} \in [2^{-k_2}, 2^{1-k_2}]$, with $p(k) \leq k_1 \leq k_2$. By (R7),
\begin{equation}\label{10_10.24}
   2^{-p(k)} \leq 2^{1-k_1} \frac{v}{9}\, k_1^\rho \leq 2^{2-k_2} \left(\frac{v}{9}\right)^2 (k_1k_2)^\rho \leq 2^{2-k_2} \left(\frac{v}{9}\right)^2 k_2^{2\rho}.
\end{equation}
In addition, the horizontal episode preceding $[\tilde U^{(k)}, \tilde U^{(k-1)}]$ has order at most $k_2$ and did not reach $- 2^{-2\ell}$, so by (R3),
\begin{align*}
   2^{-2\ell} \geq \frac{2^{-2k_2}}{(\frac{v}{9})^2\, k_2^7} \geq \frac{2^{-4-2p(k)}}{(\frac{v}{9})^2\, k_2^7\, (\frac{v}{9})^4\, k_2^{4\rho}} = \frac{2^{-4}}{(\frac{v}{9})^6\, k_2^{7 + 4\rho}}\, 2^{-2p(k)}.
\end{align*}
Provided $k_0$ is large enough, inequality \eqref{10_10.24} implies that $p(k) \geq k_2 / 2$, so we deduce that
$$
   2^{-2\ell} \geq  \frac{2^{-4}}{(\frac{v}{9})^6\, (2p(k))^{7 + 4\rho}}\, 2^{-2p(k)},
$$
which is equivalent to
\begin{equation}\label{10a_10.23}
   p(k) \geq \ell + \frac{1}{2}\, \log_2\left(2^{-11-4\rho} \left(\frac{9}{v}\right)^6\right) - \frac{1}{2}(7+4\rho) \log_2 p(k).
\end{equation}
Since \eqref{10_10.23} implies that $p(k) \leq 2\ell$ and \eqref{10a_10.23} implies $p(k) \geq \ell/2$, we get \eqref{10_10.21} from \eqref{10a_10.23} and \eqref{10_10.23}, for some universal constant $c$.

   We now combine \eqref{10_10.19} and \eqref{10_10.21} to see that for $u \in [-(2^{-2\ell} \wedge \ultau_3),-(2^{-2(\ell+1)} \wedge \ultau_3)[$ with $\ell \geq k_0$,
\begin{align*}
   \vert X^r(u,0) - \tilde X^r(u, 0)\vert &\leq C\, \frac{v}{9}\, (2\ell)^2\, 2^{-\frac{3}{2}(\ell - (7+4\rho)\log_2 \ell - 5 \log_2 \frac{v}{9})} = \tilde C\, \left(\frac{v}{9}\right)^{17/2}\, \ell^{13 + 6\rho}\, 2^{-\frac{3}{2}\ell}\\
   &\leq f_{15}(\ell)
\end{align*}
provided $\ell \geq k_0$ and $k_0$ is large enough.

   Proceeding similarly for vertical episodes, this verifies the assumption of Proposition \ref{rdprop37}, and we conclude that on $\tilde F$, the DW-algorithm applied to $X^r$ reaches level $2^{-k_0-1}$ before exiting $\cR(2^{-2k_0-1})$ or escapes $\cR(2^{-2k_0-1})$, and behaves $v$-robustly above order $k_0+1$. Since $\{\tilde Q \leq k_0\} \subset \{ Q \leq k_0\}$, it follows that $\tilde F \subset A_2(r,n,k_0+1,v)$.
   
   We now show that $\tilde F \subset A_3(r, n, k_0+1,v)$. Suppose that the DW-algorithm for $X^r$ reaches level $2^{-k}$ ($k \geq k_0+1)$ during an odd stage $2m-1,$ at point $\gamma^{\ast} (\alpha_n^{n-k,r})$. Suppose without loss of generality that $T^m_1 \in [U^\prime_{m-1}, U^\prime_m].$ By property (R5)(a) of $v$-robustness, $X^r \geq \frac{1}{v} f_8(k)$ on the segment from $\gamma^{\ast}(\alpha_n^{n-k,r})$ to $(T^m_1, T^{m-1}_2).$ Suppose without loss of generality that $T^m_2 \in [V^\prime_{m-1}, V^\prime_m].$ On the segment $\{T^m_1\} \times [T_2^{m-1}, V^\prime_{m-1}],$ we have (see (\ref{X_0})):
$$
   X^r \geq H_{2m-1} - H_{2m-2} \geq \frac{1}{v} f_3(k)
$$ 
by property (R4) of $v$-robustness. Then on $\{T_1^m\} \times [V^\prime_{m-1}, T^m_2]$, by property (R6) of $v$-robustness, $X^r \geq \frac{1}{v} f_8(k).$ In particular, along the path $\Gamma^\ast$, after reaching level $2^{-k}$ and until reaching level $2^{1-k},$ the inequality $X^r \geq \frac{1}{v} f_8(k)$ holds. At those positions along the path $\Gamma^\ast$,
$$
   W(\cdot) - W(r) \geq X^r - \sum^4_{i=1} \sum E^i_{m,\ell},
$$
where the inner sum is over all horizontal and vertical episodes of order at least $k$. As in (\ref{rd10.8}), this sum is bounded by $C v k^2 2^{-\frac{3}{2} k},$ so for $u \in [\alpha_n^{n-k,r}, \alpha_n^{n-k+1,r}]$ and $k$ large enough,
\begin{equation}\label{10_10.25}
   W(u) - W(r) \geq \frac{1}{9v} f_8 (k) - C  v k^2 2^{-3k/2 } \geq \frac{2^{-k}}{10 v k^8}.
\end{equation}

   Suppose that $d(u) \in [2^{-2(\ell+1)},2^{-2\ell}]$. The arguments that led to \eqref{10_10.21} show that
\begin{equation}\label{10_10.26}
   \ell - (7+4\rho)\log_2 \ell - 5 \log_2 v - c \leq k \leq \ell + \log_2 \ell + \frac{3}{4} \log_2 v +c,
\end{equation}
so 
$$
   2^{-k} \geq \frac{2^{-\ell}}{\ell\, v^{3/4}\, 2^c}\qquad
   \mbox{and}\qquad k \leq 2 \ell.
$$
From \eqref{10_10.25}, we see that for $u \in [\alpha_n^{n-k,r}, \alpha_n^{n-k+1,r}]$, 
\begin{equation}\label{10_10.27}
   W(u) - W(r) \geq \frac{2^{-\ell}}{\ell\, v^{3/4}\, 2^c\, 10\, v\, (2\ell)^8} = \frac{2^{-\ell}}{2^{8+c}\, 10\,v^{7/4}\,\ell^9}
   \geq \frac{\tilde c}{v^{7/4}}\, \frac{\sqrt{d(u)}}{\log_2^9\left(\frac{1}{d(u)}\right)}\, ,
\end{equation}
for $\tilde c$ small enough. We note that 
$$
   d(u) \geq 2^{-2(\ell + 1)},
$$
and by \eqref{10_10.26}, 
$$
  2^{-2(\ell + 1)} \geq 2^{-2k}\, 2^{-2}\, 2^{-2(7+4\rho)\log_2 \ell - 10 \log_2 v - 2c} = \frac{2^{-2k}\, 2^{-2-2c}}{v^{10}\, \ell^{14+8\rho}}\, .
$$
Since \eqref{10_10.26} also implies that $\ell \leq 2k$, we conclude that
\begin{equation}\label{10_10.28}
   d(u) \geq \frac{\tilde c}{v^{10}}\, \frac{2^{-2k}}{k^{14+8\rho}}\, ,
\end{equation}
for $\tilde c$ small enough. Finally, as in \eqref{10_10.22}, we have
\begin{equation}\label{10_10.29}
d(u) \leq 2^2\, 2^{-2(\ell+1)} \leq 2^2\, 12 \, v^{3/2}\, k^2\, 2^{-2k},
\end{equation}
and we conclude from \eqref{10_10.27}, \eqref{10_10.28} and \eqref{10_10.29} that $\tilde F \subset A_3(r,n,k_0+1,v)$. This completes the proof.
\hfill $\Box$
\vskip 16pt

\noindent{\em Proof of Proposition \ref{rdprop41}}. Let $v$ and $k_0$ be such that Lemma \ref{rdlem42} holds. Let $\G = \sigma(W(t) - W(r),\ t \in [\frac{3}{2},4] \times [1,4])$. Clearly, $A_i(r,n,k_0,v) \in \G$, $i = 2,3$.


   Set $\ultau = \ultau^{2^{-n},r,2^{-k_0-1}}$ (relative to the DW-algorithm for $X^r$), and let $H_1,\dots,H_4$ be defined as following \eqref{eq6.12}, but with $X^t$ replaced there by $X^r$ and level $1$ replaced by $2^{-k_0-1}$.
   
   Set $\cH = \sigma(W(t)-W(r),\ t \in \cR_r(\ultau))$. Let $\hat A = A_3(r,n,k_0,v) \cap A_2(r,n,k_0,v)$. Then $H_i \cap \hat A \in \cH$. We claim that there is $c >0$ such that for all large $n$, $r \in [2,3]^2$ and $y \in I$,
\begin{equation}\label{08_08_20_1}
   P(A_1(r,n) \cap A_4(r,n,k_0) \mid \cH, W(1,1) = y) \geq c\, 2^{-n} \qquad\mbox{on } H_i \cap \hat A.
\end{equation}
Using Lemma \ref{rdlem42}, this will complete the proof of Proposition \ref{rdprop41}, just as \eqref{eq6.14} quickly led to the proof of the lower bound in Proposition \ref{lowerlem1}(a).

   We now prove \eqref{08_08_20_1}. We only consider the case where $i=1$, $\ultau$ occurs on a horizontal stage and $2^{-n} + X^r(r_1 + \ultau_1,T_2^n) = 2^{-k_0 - 1}$, since the other cases are similar. Consider the following points (see Figure \ref{08_08_figJ}):
$$
  \rho_1 = (r_1 + \ultau_1,T_2^n),\quad \rho_2 = (4,T_2^n), \quad \rho_3 = (4, r_2 - \ultau_4), \quad \rho_4 = (4, 1),
$$
$$
    \quad \rho_5 = (r_1 + \ultau_1, 1),\quad \rho_6 = (r_1 - \ultau_3,  1),\quad \rho_7 = (3/2,1),\quad \rho_8 = (1, 1),
$$
and the following straight line segments (which are also shown in Figure \ref{08_08_figJ}):
$$
   \Gamma_i = \langle \rho_i,\rho_{i+1} \rangle,\qquad i=1,\dots,7.
$$
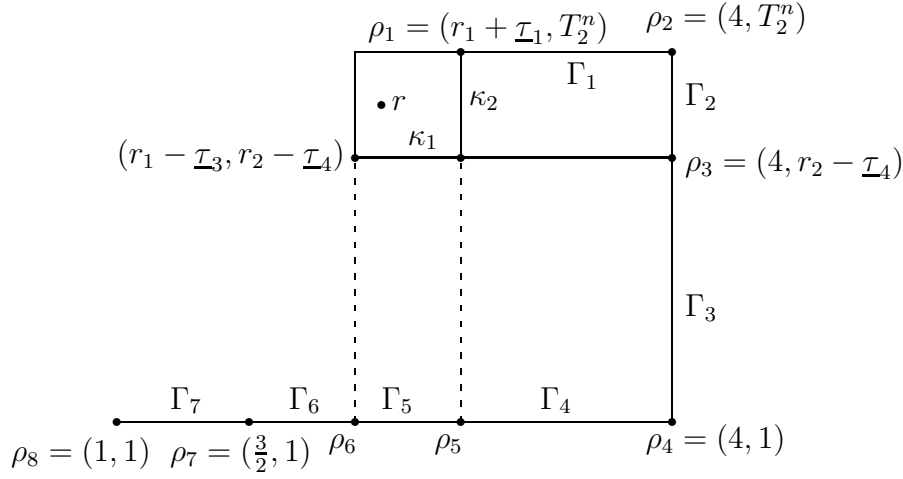
\begin{figure}
\begin{center}
\begin{picture}(320,200)

\put(20,20){\line(1,0){210}}

\put(40,25){$\Gamma_7$}
\put(20,20){\circle*{3}}
\put(-20,5){$\rho_8 = (1,1)$}

\put(85,25){$\Gamma_6$}
\put(70,20){\circle*{3}}
\put(40,5){$\rho_7=(\frac{3}{2},1)$}

\put(120,25){$\Gamma_5$}
\put(110,20){\circle*{3}}
\put(100,10){$\rho_6$}

\put(180,25){$\Gamma_4$}
\put(150,20){\circle*{3}}
\put(140,10){$\rho_5$}

\put(230,20){\circle*{3}}
\put(220,10){$\rho_4=(4,1)$}

\put(230,20){\line(0,1){140}}

\put(235,60){$\Gamma_3$}

\put(230,120){\circle*{3}} \put(235,115){$\rho_3 = (4,r_2 - \ultau_4)$}

\put(235,140){$\Gamma_2$}

\put(230,160){\circle*{3}} \put(220,170){$\rho_2 = (4,T_2^n)$}

\put(110,120){\line(1,0){120}}
\put(110,160){\line(1,0){120}}
\put(110,120){\line(0,1){40}}
\put(150,120){\line(0,1){40}}

\put(150,120){\circle*{3}}

\put(110,120){\circle*{3}}
\put(20,118){$(r_1 - \ultau_3,r_2 - \ultau_4)$}

\put(115,166){$\rho_1 = (r_1 + \ultau_1,T_2^n)$}    \put(190,148){$\Gamma_1$}
\put(150,160){\circle*{3}}

\multiput(110,20)(0,6){17}{\line(0,1){2}}
\multiput(150,20)(0,6){17}{\line(0,1){2}}

\put(120,140){\circle*{3}}
\put(124,138){$r$}

\put(130,125){$\kappa_1$}   \put(153,140){$\kappa_2$}

\end{picture}
\end{center}
\caption{The points $\rho_1,\dots,\rho_8$. \label{08_08_figJ}} 
\end{figure}

   Define 
\begin{eqnarray*}
   G_1 &=& \left\{ W(\cdot) - W(\rho_1) \in \left[-\frac{2^{-k_0-1}}{2},5 \right]\quad\mbox{on } \Gamma_1\right\} \\
   &&\qquad \cap \left\{W(\rho_2) - W(\rho_1) \geq 4 \right\},
\end{eqnarray*}
and for $i = 2,\dots,6$, define the following events:
$$
   G_i = \left\{ W(\cdot) - W(\rho_i) \in \left[-\frac{1}{2},1 \right] \quad\mbox{on } \Gamma_i\right\},
$$
and, finally, set $G_7 = \{W(\cdot) > 0$ on $\Gamma_7\}$.
We note that 
$$
   \cap_{i=1}^7 G_i \cap H_1 \cap \hat A \subset A_4(r,n,k_0).
$$

   The event $\cap_{i=1}^6 G_i \cap H_1 \cap \hat A$ is $\G$-measurable, and on this event, because $W(\rho_2) \geq 4$, the definitions of the $G_i$ imply that $W(\frac{3}{2},1) \geq \frac{3}{2}$.
	
	 Define $Z = W(\frac{3}{2},1) - W(1,1)$, so that $W(r) = W(1,1) + Z + Y$, where $Y$ is $\G$-measurable and $(W(1,1),Z)$ is independent of $\G$. Let $\cH_8 = \G \vee \sigma(W(1,1),Z)$. Then
$$
	 \cap_{i=1}^6 G_i \cap H_1 \cap \hat A \cap A_1(r,n) \in \cH_8,
$$
and on this event, $P(G_7 \mid \cH_8,\, W(1,1) = y)$ is bounded below by the probability that a Brownian bridge from $y$ to $\frac{3}{2}$ on the time interval $[0,\frac{1}{2}]$ stays positive. Therefore, there is $c_0 >0$ such that on this event, for all $y \in I$, 
$$
   P(G_7 \mid \cH_8,\, W(1,1) = y) \geq c_0.
$$
It follows that 
\begin{align*}
   & P(A_1(r,n \cap A_4(r,n,k_0) \mid \cH,\, W(1,1) = y) \\
	 & \qquad \geq c_0 P(\cap_{i=1}^6 G_i \cap H_1 \cap \hat A \cap A_1(r,n) \mid \cH,\, W(1,1) = y).
\end{align*}
Now let $\cH_7 = \G \vee \sigma(W(1,1))$. Then $\cap_{i=1}^6 G_i \cap H_1 \cap \hat A \in \cH_7$, $Z$ is independent of $\cH_7$ and
\begin{align}\label{e10.33}
    & P(A_1(r,n) \mid \cH_7,\, W(1,1) = y) \\ \nonumber
		&\qquad = P\{Z \in [-2^{1-n} - W(1,1) - Y,\, -2^{-n} - W(1,1) - Y[\, \mid  \cH_7,\, W(1,1) = y\}.
\end{align}
On $A_1(r,n) \cap \cap_{i=1}^6 G_i \cap H_1 \cap \hat A$, by definition of the $G_i$, $\vert Y \vert \leq 10$ on this event. Therefore, the conditional probability in \eqref{e10.33} is bounded below by $c_1 2^{-n}$. Therefore,
\begin{align*}
   & P(A_1(r,n \cap A_4(r,n,k_0) \mid \cH_7,\, W(1,1) = y) \\
	 &\qquad \geq c_0 c_1 2^{-n} P(\cap_{i=1}^6 G_i \cap H_1 \cap \hat A \mid \cH_7,\, W(1,1) = y) \\
	  &\qquad = c_0 c_1 2^{-n} P(\cap_{i=1}^6 G_i \cap H_1 \cap \hat A\mid \cH_7),
\end{align*}
since the event on the right-hand side is independent of $W(1,1)$.

Set $\cH_6 = \sigma(W(t) - W(r),\ t \geq (r_1 - \ultau_3, 1))$. We note that $G_1 \cap \cdots \cap G_5 \in \cH_6$. By a (strong) Markov property of $W$, given $\cH_6$, $W\vert_{\Gamma_6} - W(\rho_6)$ (moving towards $(1,1)$) is a Brownian motion, and so there is $c_8 >0$ (not depending on $r \in [2,3]^2$ or $n$) such that
$$
   P(G_6 \mid \cH_6) \geq c_8\qquad \mbox{on } \cap_{i=1}^5 G_i \cap H_1 \cap \hat A.
$$

   Consider the two segments (shown in Figure \ref{08_08_figJ})
\begin{eqnarray*}
   \kappa_1 &=& \langle (r_1 - \ultau_3, r_2 - \ultau_4), (r_1 + \ultau_1, r_2 - \ultau_4)\rangle,\\
    \kappa_2 &=& \langle (r_1 + \ultau_1, r_2 - \ultau_4), \rho_1\rangle.
\end{eqnarray*}
For $j=2,\dots,5$, set
$$
   \cH_j = \cH \vee \sigma(W\vert_{\Gamma_i} - W(\rho_i),\ i = 1,\dots,j-1).
$$
Then $W\vert_{\Gamma_5} - W(\rho_5)$ is conditionally independent of $\cH_5$ given $W\vert_{\kappa_1} - W(r_1 - \ultau_3,r_2 - \ultau_4)$, and using in particular Lemma \ref{lem5.3prime} and the fact that $Q \leq k_0+1$ on $A_2(r,n,k_0,v)$, we have $2^{1-k_0} \geq W\vert_{\kappa_1} - W(r_1 - \ultau_3,r_2 - \ultau_4) \geq - 2^{-k_0}$ on $H_1 \cap \hat A$. The law of $W\vert_{\Gamma_5} - W(\rho_5)$ given $W\vert_{\kappa_1} - W(r_1 - \ultau_3,r_2 - \ultau_4)$ is that of a time-reversed Brownian sheet (see\cite[Section6]{DW2002}, in particular, Theorems 6.1 to 6.7 in this reference), and because the distance between $\Gamma_5$ and $\kappa_1$ is $\geq 1$, there is $c_7>0$ (not depending on $r$ or $n$) such that
$$
   P(G_5 \mid \cH_5) \geq c_7 \qquad\mbox{on } G_1 \cap \cdots\cap G_4 \cap H_1 \cap \hat A.
$$

   Notice now that $W\vert_{\Gamma_4} - W(\rho_4)$ is conditionally independent of $\cH_4$ given $W\vert_{\Gamma_1} - W(\rho_1)$, and $5 \geq W\vert_{\Gamma_1} - W(\rho_1) \geq -2^{-k_0-2}$ on $G_1 \cap H_1 \cap \hat A$. Using again properties of a time-reversed Brownian sheet, we see that there is $c_6 >0$ (not depending on $r$ or $n$) such that
$$
   P(G_4 \mid \cH_4) \geq c_6 \qquad\mbox{on } G_1 \cap G_2 \cap G_3 \cap H_1 \cap \hat A.
$$

   Similarly, $W\vert_{\Gamma_3} - W(\rho_3)$ is conditionally independent of $\cH_3$ given $W\vert_{\Gamma_1} - W(\rho_1)$ and $W\vert_{\kappa_1} - W(r_1 - \ultau_3, r_2 - \ultau_4)$. Given these two processes, $W\vert_{\Gamma_3} - W(\rho_3)$ is the sum of a Brownian bridge and an independent Brownian motion, so there is $c_5 >0$ (not depending on $r$ or $n$) such that
$$
   P(G_3 \mid \cH_3) \geq c_5 \qquad\mbox{on } G_1 \cap G_2 \cap H_1 \cap \hat A.
$$
In the same way, $W\vert_{\Gamma_2} - W(\rho_2)$ is conditionally independent of $\cH_2$ given $\sigma(W\vert_{\kappa_2} - W(\rho_1)) \vee \sigma(W\vert_{\Gamma_1} - W(\rho_1))$. Note that on $H_1 \cap \hat A$, $2^{1-k_0} \geq W\vert_{\kappa_2} - W(\rho_1) \geq - 2^{-k_0}$, and for $t=(t_1,t_2) \in \Gamma_2$,
$$
   W(t) - W(\rho_2) = Y(t) + W(r_1 + \ultau_1, t_2) - W(\rho_1),
$$
where
$$
   Y(t) = \Delta_{[r_1 + \ultau_1, 4]\times [t_2,T_2^n]} W.
$$  
We note that the conditional law given $\cH_2$ of $Y(\cdot)$ is that of a Brownian bridge with speed in $[\half,2]$. Therefore, there is $c_4 >0$ (not depending on $r$ or $n$) such that
$$
   P(G_2 \mid \cH_2) \geq c_4 \qquad\mbox{on } G_1 \cap H_1 \cap \hat A.
$$

   We note that $G_1$ is independent of $\cH$, and $W\vert_{\Gamma_1}$ is a Brownian motion with speed in $[2,4]$, so there is $c_3 >0$ (not depending on $r$ or $n$) such that
$$
   P(G_1 \mid \cH) = P(G_1) \geq c_3 \qquad\mbox{on } H_1 \cap \hat A.
$$ 
We also note that all the constants $c_3,\dots,c_8$ depend on $k_0$, but $k_0$ is fixed.  
   
By iteration of conditional probabilities, the above considerations establish \eqref{08_08_20_1}. Proposition \ref{rdprop41} is proved.
\hfill $\Box$
\vskip 16pt

\end{section}
\eject

\begin{section}{Lower bound for the Brownian sheet: the two-point estimate}\label{sec11}

   In this section, we establish the second key ingredient needed to implement the second-moment argument, which is the upper bound in Proposition \ref{rdprop42} below.


  Let $A_2(t, n, k_0,v)$\index{$A_2(t,n,k_0,v)$} be the event ``the DW-algorithm for $X^t$ started at level $2^{-n}$, reaches level $2^{-k_0}$ before escaping the square with sides of length $2^{-2k_0}$ or escapes this square, and is $v$-robust above order $k_0$" (this is not quite the same definition as in Lemma \ref{rdlem42}, because there is no condition on the variable $Q$, but defines a larger event. Since we are seeking an upper bound, this will be sufficient).

   The following proposition is the principal objective of this section.
\vskip 12pt
 
\begin{prop} Let $A(t,n,k_0,v)$ be defined as in \eqref{rd10.4dp} (but using $A_2(t, n, k_0,v)$ as above). For all $v \geq 1,$ there is $k_0 \in \IN$ and $C=C_{k_0,v} > 0$ such that for all large $n \in \IN$, $1 \leq k \leq \ell \leq n-k_0$ and $(r,t) \in \ID_n(k,\ell),$
\begin{equation}\label{eq11.1}
 P(A(t, n,k_0, v) \cap A(r, n,k_0, v)) \leq C\ 2^{-n-\ell-2 k\lambda_1-(n-\ell)\lambda_1}.
\end{equation}
\label{rdprop42}
\end{prop}

\noindent{\em The two-point DW-algorithm}
\vskip 12pt

   We now work towards the proof of Proposition \ref{rdprop42}. Since the events on the left-hand side of \eqref{eq11.1} are statements about the values of $W(t)$, $W(r)$ and of $DW$-algorithms applied to $X^t$ and $X^r$ (with $x_0 = 2^{-n}$), (\ref{eq11.1}) can be proved without requiring the actual growth of the Brownian sheet everywhere along this path, which is a substantial simplification. To be precise, we note that in view of the definition of $A(r,n,k_0,v),$ it suffices to prove that
\begin{equation}\label{rd10.9}
   P(A_1(t,n) \cap A_2(t,n,k_0,v) \cap A_1(r,n) \cap A_2(r,n,k_0,v)) \leq C 2^{-n-\ell-2k\lambda_1-(n-\ell)\lambda_1}.
\end{equation}

   Fix $ \ell \geq k$ and $(r,t) \in \ID_n(k,\ell)$. If $r \leq t$ or $t \leq r,$ then $X^t$ and $X^r$ restricted to $[-2^{2(k-n-2)}, 2^{2(k-n-2)}]$ are independent. However, if neither $r \leq t$ nor $t \leq r$, then this independence property no longer holds, and we will construct independent standard ABM's $\hat X^t$ and $\hat X^r$ that are close to $X^t$ and $X^r$, respectively (when $r \leq t$ or $t \leq r$, the $\hat X$'s can simply be taken equal to the $\tilde X$'s defined in Section \ref{sec10}). Without loss of generality, we shall only discuss the case where
\begin{equation}\label{rd10.10}
   t_2-r_2 \in\, ]2^{2(\ell-n-1)}, 2^{2(\ell-n)}] 
   \qquad \mbox{and} \qquad r_1-t_1 \in\, ]2^{2(k-n-1)}, 2^{2(k-n)}].
\end{equation}


   The construction of the $\hat X$'s uses the ideas developped in the construction of $\tilde X^r$ at the beginning of Section \ref{sec10}, but in addition, accounts for the dependence between $X^r$ and $X^t$.

   Construct $\hat X^r$ in the same way as $\tilde X^r$, using an independent Brownian sheet $W^\prime$ to make $\hat X^r$ standard. Do this until the DW-algorithm for $\hat X^r$ terminates or achieves level $2^{1-n},$ having explored the rectangle
$$
\R_r(\underline{ \hat \tau}^{2^{-n}, r, 2^{1-n}}) := \ I^{r,1} \times J^{r,1}. 
$$
Assume
$$
I^{r,1} = [u^{r,1}_1, v_1^{r,1}],\qquad J^{r,1} = [u^{r,1}_2, v_2^{r,1}]
$$
(see Figure \ref{fig10.1}).
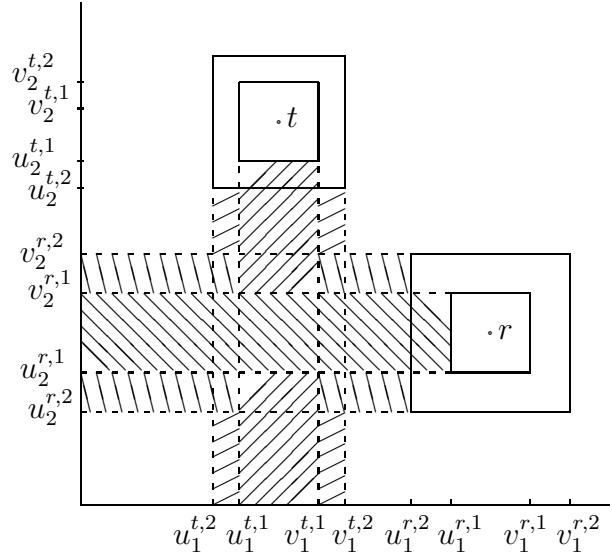
\begin{figure}[h]
\begin{center}
\begin{picture}(200,200)
\put(20,20){\line(0,1){190}}
\put(20,20){\line(1,0){200}}

\put(95,165){\circle{1}}
\put(98,163){$t$}

\put(80,150){\line(1,0){30}}
\put(80,150){\line(0,1){30}}
\put(80,180){\line(1,0){30}}
\put(110,150){\line(0,1){30}}

\put(70,140){\line(1,0){50}}
\put(70,140){\line(0,1){50}}
\put(70,190){\line(1,0){50}}
\put(120,140){\line(0,1){50}}

\put(175,85){\circle{1}}
\put(178,83){$r$}

\put(160,70){\line(1,0){30}}
\put(160,70){\line(0,1){30}}
\put(160,100){\line(1,0){30}}
\put(190,70){\line(0,1){30}}

\put(145,55){\line(1,0){60}}
\put(145,55){\line(0,1){60}}
\put(145,115){\line(1,0){60}}
\put(205,55){\line(0,1){60}}


\multiput(20,115)(6,0){21}{\line(1,0){2}}
\multiput(20,100)(6,0){24}{\line(1,0){2}}
\multiput(20,70)(6,0){24}{\line(1,0){2}}
\multiput(20,55)(6,0){21}{\line(1,0){2}}

\multiput(70,20)(0,6){21}{\line(0,1){2}}
\multiput(80,20)(0,6){23}{\line(0,1){2}}
\multiput(110,20)(0,6){23}{\line(0,1){2}}
\multiput(120,20)(0,6){21}{\line(0,1){2}}

\put(19,180){\line(1,0){2}}
\put(-6,180){$v_2^{t,2}$}
\put(19,170){\line(1,0){2}}
\put(0,168){$v_2^{t,1}$}
\put(19,150){\line(1,0){2}}
\put(-6,148){$u_2^{t,1}$}
\put(19,140){\line(1,0){2}}
\put(0,136){$u_2^{t,2}$}

\put(-3,112){$v_2^{r,2}$}
\put(0,98){$v_2^{r,1}$}
\put(-3,68){$u_2^{r,1}$}
\put(0,53){$u_2^{r,2}$}

\put(55,5){$u_1^{t,2}$}
\put(75,5){$u_1^{t,1}$}
\put(97,5){$v_1^{t,1}$}
\put(115,5){$v_1^{t,2}$}

\put(145,20){\line(0,1){2}}
\put(135,5){$u_1^{r,2}$}
\put(160,20){\line(0,1){2}}
\put(155,5){$u_1^{r,1}$}
\put(190,20){\line(0,1){2}}
\put(180,5){$v_1^{r,1}$}
\put(205,20){\line(0,1){2}}
\put(200,5){$v_1^{r,2}$}

\multiput(20,100)(6,0){19}{\line(1,-1){29}} \put(134,100){\line(1,-1){26}} \put(140,100){\line(1,-1){20}} \put(146,100){\line(1,-1){14}}
\put(20,94){\line(1,-1){23}} \put(20,88){\line(1,-1){17}} \put(20,82){\line(1,-1){11}} 

\put(80,136){\line(1,1){14}}
\put(80,130){\line(1,1){20}}
\put(80,124){\line(1,1){26}}
\multiput(80,100)(0,6){4}{\line(1,1){30}}
\put(86,100){\line(1,1){24}} \put(92,100){\line(1,1){18}} \put(98,100){\line(1,1){12}}

\put(80,56){\line(1,1){14}}
\put(80,50){\line(1,1){20}}
\put(80,44){\line(1,1){26}}
\multiput(80,20)(0,6){4}{\line(1,1){30}}  \put(86,20){\line(1,1){24}} \put(92,20){\line(1,1){18}} \put(98,20){\line(1,1){12}}

\multiput(20,115)(6,0){10}{\line(1,-4){3.7}} 
\multiput(110,115)(6,0){6}{\line(1,-4){3.7}} 

\multiput(20,70)(6,0){10}{\line(1,-4){3.7}} 
\multiput(110,70)(6,0){6}{\line(1,-4){3.7}} 

\multiput(70,115)(0,6){4}{\line(2,1){10}} 
\multiput(70,20)(0,6){6}{\line(2,1){10}}

\multiput(110,115)(0,6){4}{\line(2,1){10}}
\multiput(110,20)(0,6){6}{\line(2,1){10}}

\end{picture}
\end{center}
\caption{Construction of the $\hat X$'s. \label{fig10.1}} 
\end{figure}

   Then, on the rectangle $[0, u^{r,1}_1] \times J^{r,1}$, replace all remaining white noise $\dot{W}$ by $\dot W^\prime$. Using this new white noise, $X^t$ is replaced by $X^{t,1}$, an ABM that is initially independent of $X^r$. With this ABM, we begin to construct the standard ABM $\hat X^t$, proceeding as described in Section \ref{sec10}, until either the DW-algorithm for $\hat X^t$ terminates, or it achieves level $2^{1-n}$, having explored the rectangle
$$
   I^{t,1} \times J^{t,1}  := \R_t (\underline{\hat \tau}^{2^{-n}, t, 2^{1-n}}).
$$
Assume that
$$
I^{t,1}= [u^{t,1}_1, v^{t,1}_1], \qquad J^{t,1}= [u^{t,1}_2, v_2^{t,1}].
$$
We then replace the white noise $\dot W$ by $\dot W^\prime$ on the set
$$
I^{t,1} \times ([0, u^{t,1}_2] \setminus J^{r,1}).
$$
Using this new white noise, we return to the construction of $\hat X^r$ from where we left off; we use the method of Section \ref{sec10} until the DW-algorithm for this new $\hat X^r$ terminates or achieves level $2^{2-n}$, having explored the rectangle
$$
   I^{r,2} \times J^{r,2}  := \R_r(\underline{\hat \tau}^{2^{-n}, r, 2^{2-n}}).
$$
Assume
$$
I^{r,2} = [u^{r,2}_1, v^{r,2}_1],\qquad J^{r,2} = [u^{r,2}_2, v^{r,2}_2].
$$
We then replace $\dot W$ by $\dot W^\prime$ on the set
$$
([0, u^{r,2}_1] \setminus I^{t,1}) \times (J^{r,2} \setminus J^{r,1}).
$$
Using this new white noise, we return to the construction of $\hat X^t$ from where we left off: we use the method of Section \ref{sec10} until the DW-algorithm for this new $\hat X^t$ terminates or achieves level $2^{2-n}$, having explored the rectangle
$$
  I^{t,2} \times J^{t,2}  := \R_t(\hat \ultau^{2^{-n}, t, 2^{2-n}}).
$$
Assume
$$
I^{t,2} = [u^{t,2}_1, v^{t,2}_1], \qquad J^{t,2} = [u^{t,2}_2, v^{t,2}_2].
$$
We then replace the white noise $\dot W$ by $\dot W^\prime$ on the set 
$$
(I^{t,2} \setminus I^{t,1}) \times ([0, u_2^{t,2}] \setminus J^{r,2}),
$$
and return to the construction of $\hat X^r$ where we left off.

   This construction continues until either of the two DW-algorithms terminates or until one of paths escapes the square with sides of length $2^{2(k-n)-2}$ centered around its starting point. If the latter occurs, we say that this DW-algorithm has escaped this square, and we no longer continue with this algorithm. However, we continue with the other DW-algorithm, either until it terminates, or until it escapes the square with sides of length $2^{2(k-n)-2}$ centered around its own starting point. If this last occurs, we say that both DW-algorithms have escaped to a distance of $2^{2(k-n)-2}$ (see Figure \ref{fig10.2}). 
\begin{figure}[h]
\begin{center}
\begin{picture}(200,200)
\put(0,20){\line(0,1){190}}
\put(0,20){\line(1,0){190}}

\put(30,180){\vector(1,0){20}}
\put(30,180){\vector(-1,0){0}}

\put(30,150){\line(1,0){20}}
\put(30,170){\line(1,0){20}}
\put(30,150){\line(0,1){20}}
\put(50,150){\line(0,1){20}}

\put(40,160){\circle{.6}}
\put(42,156){$t$}

\put(75,85){$2^{2(k-n)-1}$}
\put(70,80){\vector(1,0){20}}
\put(70,80){\vector(-1,0){0}}

\put(70,50){\line(1,0){20}}
\put(70,70){\line(1,0){20}}
\put(70,50){\line(0,1){20}}
\put(90,50){\line(0,1){20}}

\put(80,60){\circle{.6}}
\put(82,56){$r$}

\put(40,13){\vector(1,0){40}}
\put(40,13){\vector(-1,0){0}}

\put(40,0){$\sim 2^{2(k-n)}$}

\put(5,60){\vector(0,1){100}}
\put(5,60){\vector(0,-1){0}}

\put(8,110){$\sim 2^{2(\ell-n)}$}

\end{picture}
\end{center}
\caption{The two DW-algorithms. \label{fig10.2}} 
\end{figure}
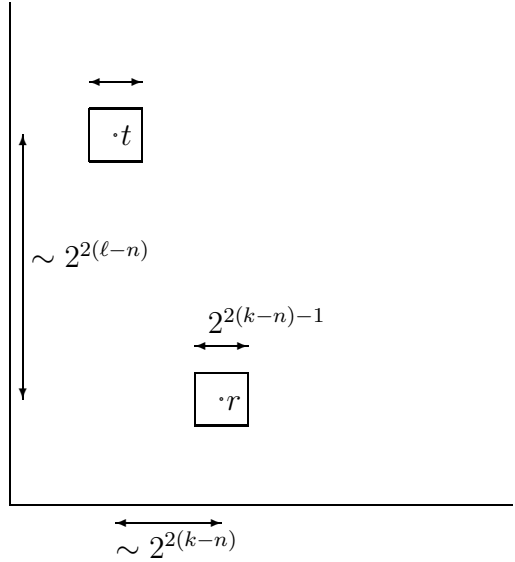
   
   Observe that any piece of white noise is used no more than once, which guarantees independence of $\hat X^r$ and $\hat X^t$, and that these two ABM's are standard (up to a deterministic rescaling of time).
\vskip 12pt

\noindent {\em The probability that the DW-algorithm for $\hat X^t$ and $\hat X^r$ both escape to $2^{2(k-n)-2}$}
\vskip 12pt

We will see that on $A_2(t,n,k_0,v) \cap A_2(r,n,k_0,v)$, unless certain white noise increments of $\dot W$ or $\dot W'$ are unusually large, then both DW-algorithms for $\hat X^t$ and $\hat X^r$ escape to $2^{2(k-n)-3}$. This will lead to an upper bound on the probability of $A_2(t,n,k_0,v) \cap A_2(r,n,k_0,v)$. We use ideas similar to those used for the proof of Lemma \ref{rdlem40}.

   The DW-algorithm's for $\hat X^t$ and $\hat X^r$ construct respectively partitions into horizontal and vertical episodes
$$
\ldots, [\hat U^{(\ti k),t}, \hat U^{(\ti k-1),t}], \ldots, [\hat U^{(1),t}, \hat U^{(0),t}], [\hat U^{\prime(0),t}, \hat U^{\prime(1),t}], \ldots, [\hat U^{\prime(\ti k-1),t}, \hat U^{\prime(\ti k),t}] \ldots,
$$
$$
\ldots, [\hat V^{(\ti k),t}, \hat V^{(\ti k-1),t}], \ldots, [\hat V^{(1), t}, \hat V^{(0),t}], [\hat V^{\prime(0),t}, \hat V^{\prime(1),t}], \ldots, [\hat V^{\prime(\ti k-1),t}, \hat V^{\prime(\ti k),t}], \ldots,
$$\index{$\hat U^{(\ti k),t}$}\index{$\hat U^{\prime(\ti k),t}$}\index{$\hat V^{(\ti k),t}$}\index{$\hat V^{\prime(\ti k),t}$}
and
$$
\ldots, [\hat U^{(\ti k),r}, \hat U^{(\ti k-1),r}], \ldots, [\hat U^{(1),r}, \hat U^{(0),r}], [\hat U^{\prime(0),r}, \hat U^{\prime(1),r}], \ldots, [\hat U^{\prime(\ti k-1),r}, \hat U^{\prime(\ti k),r}], \ldots,
$$
$$
\ldots, [\hat V^{(\ti k),r}, \hat V^{(\ti k-1),r}], \ldots, [\hat V^{(1),r}, \hat V^{(0),r}], [\hat V^{\prime(0),r}, \hat V^{\prime(1),r}], \ldots, [\hat V^{\prime(\ti k-1),r}, \hat V^{\prime(\ti k),r}], \ldots \,.
$$
which we shift to Brownian sheet coordinates by adding $t_1$, $t_2$, $r_1$, or $r_2$ as appropriate, so that $\hat U^{(0),t} = t_1 = \hat U^{\prime(0),t}$, $\hat V^{(0),t} = t_2 = \hat V^{\prime(0),t}$, $\hat U^{(0),r} = r_1 = \hat U^{\prime(0),r}$, $\hat V^{(0),r} = r_2 = \hat V^{\prime(0),r}$. The order of an episode is denoted using the $p(\ti k)$, $q(\ti k)$, $p^\prime(\ti k)$ and $q^\prime(\ti k)$ introduced below (\ref{rd10.2}), except that we write $p^t(\ti k)$\index{$p^t(\ti k)$}, $q^t(\ti k)$\index{$q^t(\ti k)$}, $p^{\prime,t}(\ti k)$\index{$p^{\prime,t}(\ti k)$} and $q^{\prime,t}(\ti k)$\index{$q^{\prime,t}(\ti k)$} (resp. $p^r(\ti k)$, $q^r(\ti k)$, $p^{\prime,r}(\ti k)$ and $q^{\prime,r}(\ti k)$) for episodes relative to $\hat X^t$ (resp.~$\hat X^r$). For $j=1, \ldots, 4,$ we define $E^{j,t}_{\ti k, \ti \ell}$ and $E^{j,r}_{\ti k, \ti \ell}$ as we did for $E^j_{k, \ell}$ below (\ref{rd10.2}). However, since the two DW-algorithms interact, we also have to consider cartesian products of episodes from the DW-algorithm for $\hat X^t$ with episodes from the DW-algorithm for $\hat X^r$. These can be of four kinds: let
$$
   E^{1,1,t,r}_{\ti k, \ti \ell} = \sup_{\underset{\hat V^{\prime(\ti \ell-1),r}< v < \hat V^{\prime(\ti \ell),r}}{\hat U^{\prime(\ti k-1),t} < u < \hat U^{\prime(\ti k),t}}} \vert W([\hat U^{\prime(\ti k-1),t}, u] \times [\hat V^{\prime(\ti \ell-1),r}, v]) \vert \vee \vert W^\prime([\hat U^{\prime(\ti k-1),t},u] \times [\hat V^{\prime(\ti \ell-1),r},v]\vert.
$$
\index{$E^{1,1,t,r}_{\ti k, \ti \ell}$}We define $E_{\ti k, \ti \ell}^{1, 2, t, r}$,\index{$E_{\ti k, \ti \ell}^{1, 2, t, r}$} $E_{\ti k, \ti \ell}^{2,1,t,r}$,\index{$E_{\ti k, \ti \ell}^{2,1,t,r}$} $E_{\ti k, \ti \ell}^{2,2,t,r}$\index{$E_{\ti k, \ti \ell}^{2,2,t,r}$} analogously (see Figure \ref{fig10.3}).
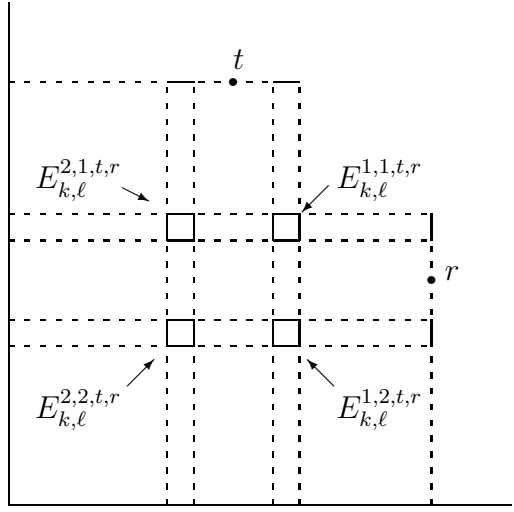
\begin{figure}[h]
\begin{center}
\begin{picture}(200,200)
\put(20,20){\line(0,1){190}}
\put(20,20){\line(1,0){190}}

\put(105,185){$t$}
\put(80,180){\line(1,0){10}}
\put(105,180){\circle*{3}}
\put(120,180){\line(1,0){10}}

\put(185,105){$r$}
\put(180,120){\line(0,1){10}}
\put(180,105){\circle*{3}}
\put(180,80){\line(0,1){10}}

\put(30,140){$E_{k,\ell}^{2,1,t,r}$}
\put(63,140){\vector(2,-1){10}}

\put(30,53){$E_{k,\ell}^{2,2,t,r}$}
\put(65,65){\vector(1,1){10}}

\put(144,53){$E_{k,\ell}^{1,2,t,r}$}
\put(143,65){\vector(-1,1){10}}

\put(144,140){$E_{k,\ell}^{1,1,t,r}$}
\put(141,141){\vector(-1,-1){10}}

\put(80,80){\line(1,0){10}}
\put(80,80){\line(0,1){10}}
\put(80,90){\line(1,0){10}}
\put(90,80){\line(0,1){10}}

\put(80,120){\line(1,0){10}}
\put(80,120){\line(0,1){10}}
\put(80,130){\line(1,0){10}}
\put(90,120){\line(0,1){10}}

\put(120,120){\line(1,0){10}}
\put(120,120){\line(0,1){10}}
\put(120,130){\line(1,0){10}}
\put(130,120){\line(0,1){10}}

\put(120,80){\line(1,0){10}}
\put(120,80){\line(0,1){10}}
\put(120,90){\line(1,0){10}}
\put(130,80){\line(0,1){10}}

\multiput(20,80)(6,0){27}{\line(1,0){2}}
\multiput(20,90)(6,0){27}{\line(1,0){2}}
\multiput(20,120)(6,0){27}{\line(1,0){2}}
\multiput(20,130)(6,0){27}{\line(1,0){2}}
\multiput(20,180)(6,0){18}{\line(1,0){2}}

\multiput(80,20)(0,6){27}{\line(0,1){2}}
\multiput(90,20)(0,6){27}{\line(0,1){2}}
\multiput(120,20)(0,6){27}{\line(0,1){2}}
\multiput(130,20)(0,6){27}{\line(0,1){2}}
\multiput(180,20)(0,6){18}{\line(0,1){2}}

\end{picture}
\end{center}
\caption{Increments over products of episodes. \label{fig10.3}} 
\end{figure}

   We define $a^{j,t}_{\ti k,\ti\ell}$ (resp.~$a^{j,r}_{\ti k,\ti\ell}$) as for the $A^j_{k,\ell}$ above \eqref{rd10.2b}, but with $p'(k)$, etc., replaced by $p^{\prime,t}(\ti k)$ (resp.~$p^{\prime,r}(\ti k)$), etc. We also define
$$
   a^{1,t,r}_{\ti k,\ti\ell} = 2^{-\frac34 (p^{\prime,t}(\ti k) + q^{\prime,r}(\ti \ell))}
$$
and similarly, $a^{2,t,r}_{\ti k,\ti\ell}$, $a^{3,t,r}_{\ti k,\ti\ell}$, $a^{4,t,r}_{\ti k,\ti\ell}$.
	
   As in (\ref{rd10.2b}), we define
\begin{eqnarray*}
K^{1,t} &=& \inf \{\ti k :  \exists \ti \ell \ \mbox{with} \ q^{\prime,t}(\ti \ell) \geq p^{\prime,t}(\ti k) \ \mbox{and} \ E^{1,t}_{\ti k, \ti \ell} \geq a_{\ti k, \ti \ell}^{1,t},\\
&&\qquad\qquad \mbox{or} \ \exists \ti \ell \ \mbox{with} \ q^t(\ti \ell) \geq p^{\prime,t}(\ti k) \ \mbox{and} \ E^{4,t}_{\ti k, \ti \ell} \geq a^{4,t}_{\ti k, \ti \ell},\\
&&\qquad\qquad  \mbox{or} \ \exists \ \ti \ell \ \mbox{with} \ q^{\prime,r}(\ti \ell) \geq p^{\prime,t}(\ti k) \ \mbox{and} \ E_{\ti k, \ti \ell}^{1,1,t,r} \geq a^{1,t,r}_{\ti k, \ti \ell},\\
&&\qquad\qquad \mbox{or} \ \exists \ti \ell \ \mbox{with} \ q^r(\ti \ell) \geq p^{\prime,t}(\ti k) \ \mbox{and} \ E_{\ti k, \ti \ell}^{1,2,t,r} \geq a^{4,t,r}_{\ti k, \ti \ell}\}.
\end{eqnarray*}
Similarly, we define $K^{3,t}$, while $K^{2,t}$ and $K^{4,t}$ only involve episodes for $\hat X^t$ as in (\ref{rd10.2b}). Relative to $r$, we define the analogous random variables $K^{j,r}$, $j = 1, \ldots, 4.$

\begin{prop} (a) Let $\cO(\hat X^t,2^{-n},m,v)$ be defined as in \eqref{10.6rd} (with $\ti X^r$ there replace by $\hat X^t$), and let
\begin{align*}
   \hat Q &= \max(p^{\prime,t}(K^{1,t}), q^{\prime,t}(K^{2,t}), p^t(K^{3,t}), q^t(K^{4,t}),\\
	   &\qquad\qquad p^{\prime,r}(K^{1,r}), q^{\prime,r}(K^{2,r}), p^r(K^{3,r}), q^r(K^{4,r}) ).
\end{align*}
Then for all $v > 1$ and for all large $ m \geq 1$, there exists $C< \infty$ such that for all $n \geq m$ with $m \geq n-k,$ 
\begin{align}\nonumber
   &P_{2^{-n}}\big(\{ \mbox{both the DW-algorithms for } \hat X^t \mbox{ and } \hat X^r \mbox{ escape the square with sides} \\ \nonumber
	 & \qquad\qquad \mbox { of length }  2^{-2m} \mbox { or reach level } 2^{-m} \mbox{ before escaping this square } \}   \\  \nonumber
	 & \qquad\qquad  \cap \{ \hat Q = m\} \cap \cO(\hat X^t,2^{-n},m,v)) \cap \cO(\hat X^r,2^{-n},m,v)\big) \\
   &\qquad \leq C 2^{2(m-n)\lambda_1} \exp (-2^{4m/5}).
\label{rd10.11}
\end{align}

%

   (b) 
	Letting $\hat Q$ be as in (a), for $k_0$ large and $n-k \geq k_0$,
\begin{align}\nonumber
  &P_{2^{-n}}(\{\hat Q \geq n-k\} \cap A_1(t,n) \cap A_1(r,n) \cap A_2(t,n,k_0,v) \cap A_2(r,n,k_0,v)) \\
	&\qquad \leq C 2^{-n-\ell-2k\lambda_1} \exp (-2^{4(n-k)/5}).
	\label{rd10.12a}
\end{align}
\label{08.prop11.3}
\end{prop}


\begin{proof} (a) The proof of (\ref{rd10.11}) is similar to that of (\ref{rd10.2a}), so we only explain the main differences with that proof. 

   The event $\{\hat Q = m\}$ occurs either because 
$$
   \max(p^{\prime,t}(K^{1,t}), q^{\prime,t}(K^{2,t}), p^t(K^{3,t}), q^t(K^{4,t})) = m
$$ 
or because $\max(p^{\prime,r}(K^{1,r}), q^{\prime,r}(K^{2,r}), p^r(K^{3,r}), q^r(K^{4,r})) = m$. Therefore, instead of the eight events that decomposed $\{\hat Q = m\}$ in the proof of (\ref{rd10.2a}), there are now twice as many: eight for each of the DW-algorithms. However, there are additional events to put in the decomposition of $\{\hat Q = m\}$, corresponding to the cases where it is one of the four increments $E_{\ti k,\ti \ell}^{i,j,t,r}$ which is excessively large. So in fact, $\{\hat Q = m\}$ is decomposed into 32 events which correspond to the $F_\ell$ in the proof of (\ref{rd10.2a}), and each of these is decomposed into the union of two events which correspond to the $F^\ast_\ell$ and $F^{\ast \ast}_\ell$ of that proof. The event $H$ of that proof becomes $\hat H^t \cap \hat H^r$, where $\hat H^t$ is the event ``the DW-algorithm for $\hat X^t$ escapes the square with sides of length $2^{-2m}$ and is $9v$-robust above order $m$."
The events $G^\ast_\ell$ and $G^{\ast \ast}_\ell$ are similarly replaced by $G^{\ast, t}_{\ti \ell}$ and $G^{\ast \ast, t}_{\ti \ell}$, or $G^{\ast, r}_{\ti \ell}$ and $G^{\ast \ast,r}_{\ti \ell}$, according to which DW-algorithm first ``sees" the excessively large white noise increment. Lemma \ref{rdlem39} is used as before to get the following analogue of (\ref{rd10.4cp}): 
$$
   P_{2^{-n}}(G^{\ast, t}_{\ti \ell}) \leq C \ E_{2^{-n}}(1_{\hat H^t \cap \hat H^r} \exp(-2^{2(m+q^t(\ti \ell))/5}).
$$
Using the bounds on the number of episodes of each order under $9v$-robustness, the independence of $\hat H^t$ and $\hat H^r$, and Theorem \ref{thm3}, we obtain (\ref{rd10.11}) as the analogue of (\ref{rd10.4bp}).

   We now turn to the proof of (\ref{rd10.12a}). We observe that for $n \in \{m, \ldots, m+k-1\}$,
\begin{eqnarray*}
  && \{\hat Q = m\} \cap A_1(t,n) \cap A_1(r,n) \cap A_2(t,n,k_0,v) \cap A_2(r,n,k_0,v)\\
  && \qquad =  F_1 \cap \{W(r) \in [2^{-n}, 2^{1-n}]\} \cap \{X + Y \in [2^{-n}, 2^{1-n}]\},
\end{eqnarray*}
where
$$
   F_1 = \{\hat Q = m\} \cap A_2(t,n,k_0,v) \cap A_2(r,n,k_0,v)
$$
and
\begin{eqnarray*}
X &=& W([0,1] \times [r_2 + 2^{2(k-n)-4}, t_2 - 2^{2(k-n)-4}]),\\
Y &=& W(t)-X.
\end{eqnarray*}
By \eqref{rd10.10}, $t_2 - r_2 - 2^{2(k-n)-3} \geq 2^{2(\ell-n)-3}$, and $X$ is independent of $\sigma(Y)$ and $F_1 \cap \{W(r) \in [2^{-n}, 2^{1-n}]\}.$ Therefore, since $\{ X+Y \in [2^{-n},2^{1-n}] \} = \{Y - 2^{-n} \leq X \leq Y - 2^{1-n}\}$ and Var~$X \geq 2^{2(\ell - n)-3}$,
\begin{eqnarray*}
   &&P(F_1 \cap \{W(r) \cap [2^{-n},2^{1-n}]\} \cap \{X + Y \in [2^{-n}, 2^{1-n}]\})\\
   &&\qquad \leq c_0 2^{-\ell} P(F_1 \cap \{W(r) \in [2^{-n}, 2^{1-n}]\}).
\end{eqnarray*}

   Let $Z = W(r) - W(1,1)$. Since $W(1,1)$ is independent of $Z$ and $F_1$, the right-hand side is bounded above by
\begin{equation}\label{11.6rd}
c_1 2^{-\ell} 2^{-n} P_{2^{-n}}(F_1).
\end{equation}
	
   We now examine $P(F_1)$. Observe first that for $m > n-k\geq k_0$, $F_1 \subset F_2$, where 
\begin{align*}
   F_2 = \{\hat Q = m\} \cap A_2(t,n,m,v)  \cap A_2(r,n,m,v) .
\end{align*}
Indeed, on $A_2(t,n,k_0,v)$, either the DW-algorithm for $X^t$ escapes to $2^{-2m}$, or it does not escape to $2^{-2m}$ but then level $2^{-m} \leq 2^{-k_0}$ is reached within this rectangle.

   On $A_2(t,n,m,v)  \cap A_2(r,n,m,v)$, $\cO(X^t,2^{-n},m,v)$ occurs. Therefore, 
using the same argument as used to prove \eqref{10_10.20},
we see via Proposition \ref{rdprop37} that $\hat X^t$ and $\hat X^r$ are $9v$-robust above order $m$.
Therefore,
\begin{align*}
   F_1 \subset \{\hat Q = m\} 
			\cap \hat A_2(t,n,m,9v) \cap \hat A_2(r,n,m,9v),
\end{align*}
where $\hat A_2(t,n,m,9v)$ is defined in the same way as $A_2(t,n,m,9v)$, but relative to $\hat X^t$ instead of $X^t$.
Therefore, using \eqref{rd10.11},
we see that 
\begin{align*}
   P_{2^{-n}}(F_1) &\leq P_{2^{-n}}(\{\hat Q = m \} \cap \hat A_2(t,n,m,9v) \cap \hat A_2(r,n,m,9v)) \\
	 &\leq C  2^{2(m-n) \lambda_1} \exp(-2^{4m/5}).
\end{align*}
Combining this with \eqref{11.6rd}, we see that for $m \geq n-k \geq k_0$,
\begin{eqnarray*}
 && P_{2^{-n}}(\{\hat Q = m\} \cap A_1(t,n) \cap A_1(r,n) \cap A_2(t,n,k_0,v) \cap A_2(r,n,k_0,v))\\
 &&\qquad\qquad \leq C  2^{-\ell-n} 2^{2(m-n) \lambda_1} \exp(-2^{4m/5}).
\end{eqnarray*}
Sum this over $m \in \{n-k, \ldots, n\}$ to obtain (\ref{rd10.12a}).
\end{proof}

\vskip 12pt

\noindent{\em The maximum of $X^t$}
\vskip 12pt

   Set $R_k = [-2^{2(k-n)-4}, 2^{2(k-n)-4}]^2,$ and let
\begin{eqnarray*}
\R_t  (\underline{\hat \tau}^{2^{-n}, t, 2^{k-n-2}} \wedge \underline {\hat \sigma}^{2^{-n}, t, R_k}) &=& [x_1^t, y^t_1] \times [x_2^t, y_2^t],\\
\R_r  (\underline{\hat \tau}^{2^{-n},r,2^{k-n-2}} \wedge \underline {\hat \sigma}^{2^{-n}, r, R_k}) &=& [x^r_1, y^r_1] \times [x^r_2, y^r_2]
\end{eqnarray*}
be the rectangles (in coordinates for $W$) explored by the two DW-algorithms for $\hat X^t$ and $\hat X^r$ until either reaching level $2^{k-n}$ or escaping $R_k$. 
We are interested in the random variable\index{$\bar X^{r,t}$}
$$
   \bar X^{r,t} = \sup_{\underset{x_2^r-t_2 \leq u_2 \leq y_2^t-t_2}{x_1^t-t_1 \leq u_1 \leq y_1^r-t_1}}(2^{-n}+ X^t(u_1,u_2))
$$
(see Figure \ref{fig10.4}). Indeed, the rectangle $[x_1^t, y_1^r] \times [x_2^r, y_2^t]$ corresponds to a region in which the two DW-algorithms for $X^t$ and $X^r$ (no ``hats") overlap substantially, and we will not say anything about their behaviors there, except to bound the maximum height achieved by $X^t$ there. This idea has already been used in the proof of Proposition \ref{lembivariate}.
\begin{figure}
\begin{center}
\begin{picture}(200,200)
\put(20,20){\line(0,1){220}}
\put(20,20){\line(1,0){220}}

\put(90,210){$\sim 2^{2(k-n)}$}  \put(175,130){$\sim 2^{2(\ell-n)}$}
\put(70,200){\vector(1,0){80}}  \put(170,87){\vector(0,1){88}}
\put(70,200){\vector(-1,0){0}}  \put(170,87){\vector(0,-1){0}}

\put(70,175){\circle*{3}}
\put(70,178){$t$}

\put(60,160){\line(1,0){30}}
\put(60,160){\line(0,1){30}}
\put(60,190){\line(1,0){30}}
\put(90,160){\line(0,1){30}}

\put(20,190){\line(1,0){40}}
\put(5,188){$y_2^t$}

\put(20,175){\line(1,0){50}}
\put(5,173){$t_2$}

\put(20,160){\line(1,0){40}}
\put(5,158){$x_2^t$}

\put(60,19){\line(0,1){2}}
\put(50,5){$x_1^t$}

\put(70,20){\line(0,1){155}}
\put(68,5){$t_1$}

\put(90,20){\line(0,1){140}}
\put(85,5){$y_1^t$}

\put(150,87){\circle*{3}}
\put(150,90){$r$}

\put(130,70){\line(1,0){30}}
\put(130,70){\line(0,1){30}}
\put(130,100){\line(1,0){30}}
\put(160,70){\line(0,1){30}}

\put(20,100){\line(1,0){110}}
\put(5,98){$y_2^r$}

\put(19,87){\line(1,0){2}}
\put(5,85){$r_2$}

\put(20,70){\line(1,0){110}}
\put(5,68){$x_2^r$}

\put(130,20){\line(0,1){70}}
\put(125,5){$x_1^r$}

\put(150,19){\line(0,1){2}}
\put(146,5){$r_1$}

\put(160,19){\line(0,1){2}}
\put(158,5){$y_1^r$}

\multiput(60,70)(0,6){15}{\line(0,1){2}}
\multiput(160,100)(0,6){13}{\line(0,1){2}}
\multiput(70,175)(6,0){15}{\line(1,0){2}} \put(159,175){\line(1,0){1}}

\put(35,166){$_{(1')}$}
\put(35,130){$_{(2')}$} \put(60,130){$_{(3')}$} \put(72,110){$_{(1)}$} 
\put(35,80){$_{(4')}$}

\put(120,140){$_{(5)}$}
\put(100,80){$_{(4)}$}
\put(100,50){$_{(3)}$}

\put(132,80){$_{(2)}$}

\end{picture}
\end{center}
\caption{The various regions. \label{fig10.4}} 
\end{figure}
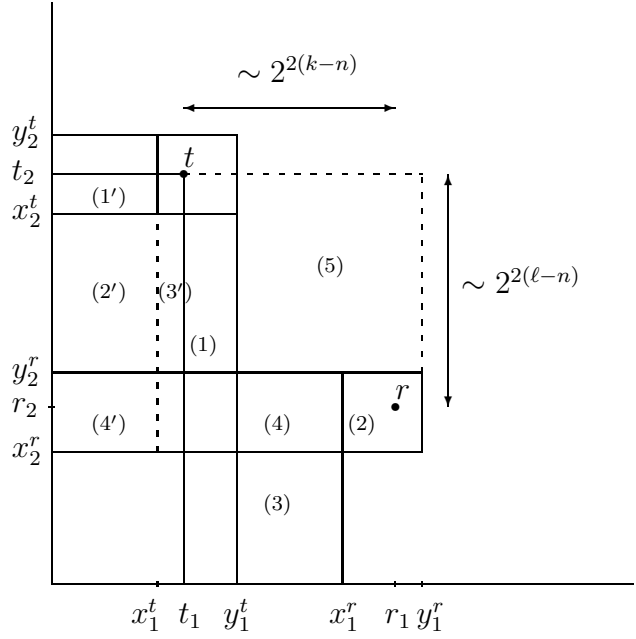

   For $v$, $t$, $r$ fixed, set
$$
   \hat \cO_n(m) = \cO(\hat X^t,2^{-n},m,9v) \cap \cO(\hat X^r,2^{-n},m,9v).
$$

\begin{lemma} Let $\cH$ be the $\sigma$-field generated by the white noise increments (of $\dot W$ and $\dot W'$) used by the two DW-algorithms for $\hat X^t$ and $\hat X^r$ up to escaping the rectangles $[x^t_1 - t_1,y^t_1 - t_1] \times [x^t_2 - t_2, y^t_2 - t_2]$ and $[x^r_1 - r_1, y^r_1 - r_1] \times [x^r_2 - r_2, y^r_2 - r_2]$, respectively. For $k_0$ large and $n - k_0 \geq \ell \geq k \geq 1$, on $\{\hat Q < n-k \} \cap \hat \cO_n(n - k)$,
$$
   P\{\bar X^{r,t} \geq x 2^{\ell-n} \mid \cH \} \leq c \, e^{-x^2/7}.
$$
\label{lem11.4}
\end{lemma}

\proof Clearly, the maximum of $X^t$ over $[x_1^t - t_1, y_1^r - t_1]\times [x^r_2 - t_2, y_2^t - t_2]$ is bounded by
\begin{equation}\label{rd10.13}
\sup_{0 \leq u_1 \leq y_1^r-t_1} X^t(u_1,0) + \sup_{0 \leq u_2 \leq y_2^t-t_2} X^t(0, u_2) + \sup_{0 \leq u_1 \leq t_1-x^t_1} X^t(-u_1,0) + \sup_{0 \leq u_2 \leq t_2-x^r_2} X^t(0,-u_2).
\end{equation}
Further, looking at the third term, we see that
$$
\sup_{0 \leq u_1 \leq t_1-x_1^t} X^t(-u_1, 0) \leq 2^{k-n} + \sup_{0 \leq u_1 \leq t_1-x^t_1} \vert X^t(-u_1,0) - \hat X^t(-u_1,0)\vert,
$$
and as in (\ref{10_10.20}) (with $\ell$ there replaced by $k$), on $\{\hat Q < n-k\} \cap \hat \cO_n(n - k)$, the second term on the right-hand side is bounded by $2^{k-n}.$ Therefore, on $\{\hat Q < n-k\}  \cap \hat \cO_n(n - k)$,
$$
\sup_{0 \leq u_1 \leq t_1-x^t_1} X^t (-u_1, 0) \leq 2^{k-n+1},
$$
and similarly,
$$
\sup_{0 \leq u_2 \leq y^t_2 - t_2} X^t(0,u_2) \leq 2^{k-n+1}.
$$
Let\index{$Y_i$} 
$$
\begin{array}{ll}\displaystyle
   Y_1 = \max_{0 \leq u_1 \leq y_1^t-t_1}  X^t(u_1,0), & \displaystyle Y_2 = \max_{x_1^r \leq u_1 \leq y^r_1}\vert W(u_1, y_2^r) - W(x_1^r, y_2^r)\vert,\\
   \\
  \displaystyle Y_3 = \max_{y_1^t \leq u_1 \leq x^r_1} \vert W([y_1^t, u_1] \times [0, x^r_2])\vert,& \displaystyle Y_4 = \max_{y_1^t \leq u_1 \leq x_1^r} \vert W([y_1^t,u_1] \times [x^r_2, y_2^r])\vert, \\ 
   \\
  \displaystyle Y_5 = \max_{y^t_1 \leq u_1 \leq y_1^r} \vert W([y_1^t, u_1] \times [y_2^r, t_2])\vert. & 
\end{array}
$$
These random variables correspond to certain maximal increments of the sheet over regions labelled $(1), \ldots, (5)$, respectively, in Figure \ref{fig10.4}. Clearly,
\begin{equation}\label{rd10.14}
   \sup_{0 \leq u_1 \leq y_1^r -t_1} X^t(u_1, 0) \leq Y_1 + \cdots + Y_5.
\end{equation}
In the same way as for the second and third terms in (\ref{rd10.13}), on $\{\hat Q < n-k\} \cap \hat \cO_n(n - k)$, $Y_1 \leq 2^{k-n+1}$ and $Y_2 \leq 2^{k-n+1}$.

   The random variable $Y_3$ is conditionally independent of $\cH$ given $y_1^t$, $x^r_1$ and $x_2^r$. 
The conditional distribution of $Y_3$ is that of the maximum of a Brownian motion with speed $x^r_2 \leq 3,$ over a time-interval of length at most $2^{2(k-n)}$. Therefore, on $\{\hat Q < n-k\} \cap \hat \cO_n(n - k)$,
\begin{equation}\label{08rd10.15}
   P(Y_3 \geq x 2^{k-n} \vert \cH) \leq C e^{-x^2/6}.
\end{equation}
Observe that
\begin{align}\nonumber
   Y_4 &\leq \sum \sup_{y_1^t \leq u_1 \leq x^r_1} \vert W([y_1^t, u_1] \times [\hat V^{(n),r}, \hat V^{(n-1),r}])\vert\\
 &\qquad\qquad\qquad + \sum \sup_{y_1^t \leq u_1 \leq x^r_1} \vert W([y_1^t, u_1] \times [\hat V^{\prime(m-1),r}, \hat V^{\prime(m),r}])\vert,
 \label{11.6a}
\end{align}
where the sums are over all vertical episodes for $\hat X^r$ of order above $n-k$.

   If each term in these sums that comes from an order $m$ episode is no greater than $m^{-3} v^{-1/2} 2^{k-n}$, then since on $\{\hat Q < n-k\} \cap \hat \cO_n(n - k)$, there are no more than $m \sqrt{v}$ episodes of order $m$, we would have
$$
   Y_4 \leq 2 \sum^n_{m=n-k} m \sqrt{v}\, m^{-3} v^{-1/2} 2^{k-n} < 2^{k-n}.
$$
The area of a rectangle that appears in these sums and that comes from an order $m$ episode is no greater than $2^{2(k-n)} v\, m \ 2^{-2m},$ so by Lemma \ref{rdlem39},
\begin{equation}\label{11.6aa}
   P(Y_4 \geq 2^{k-n} \vert \cH) \leq \sum^n_{m=n-k} m \sqrt{v} \exp(-\frac{2^{2m}}{m^8v^3}) \leq C \exp(-2^{n-k})
\end{equation}
for $n-k$ sufficiently large. Finally, given $y_1^t, y_1^r$ and $y_2^r$, the conditional distribution of $Y_5$ is that of the maximum of a Brownian motion with speed $t_2-y_2^r \simeq 2^{2(\ell-n)}$ over a time interval of length $y^r_1-y_1^t \leq 2^{2(k-n+2)}$. The maximal variance is $2^{2(\ell-n)} \cdot 2^{2(k-n+2)}$, so 
$$
   P(Y_5 \geq 2^{\ell-n} \vert \cH) \leq C \exp(-2^{n-k}).
$$
We now have bounds on conditional tail probabilities for all of the terms in (\ref{rd10.14}). 

   It remains to do something similar for the fourth term in (\ref{rd10.13}). Let\index{$Y^\prime_i$}
$$
\begin{array}{ll}\displaystyle
   Y^\prime_1 = \sup_{0 \leq u_2 \leq t_2-x_2^t} X^t(0,-u_2), &\displaystyle Y^\prime_2 = \sup_{y^r_2 \leq u_2 \leq x_2^t} (W(x^t_1, u_2) - W(x_1^t, x_2^t)),\\
\\
\displaystyle Y^\prime_3 = \sup_{y_2^r \leq u_2 \leq x^t_2} \vert W([x^t_1, t_1] \times [u_2, x_2^t])\vert,&\displaystyle Y_4^\prime = \sup_{t_2-y_2^r \leq u_2 \leq t_2 - x^r_2} (X^t(0,-u_2)-X^t(0,y_2^r - t_2)).
\end{array}
$$
These random variables correspond to maximal increments of the sheet over regions labelled $(1^\prime), \ldots, (4^\prime)$ respectively, in Figure \ref{fig10.4}. Note the absence of absolute values in the definition of $Y^\prime_2$ and $Y^\prime_4$. Clearly, 
$$
\sup_{0 \leq u_2 \leq t_2-x_2^r} X^t(0, -u_2) \leq Y^\prime_1 + \cdots + Y^\prime_4.
$$
On $\{\hat Q < n-k\} \cap \hat \cO_n(n - k)$, as for $Y_1$, we have $Y^\prime_1 \leq 2^{k-n+1}.$ The random variable $Y^\prime_2$ is conditionally independent of $\cH$ and $Y^\prime_3$ given $x_1^t, x_2^t, x_2^r$ and $y_2^r$, with conditional distribution equal to that of the maximum of a Brownian motion with speed $x_1^t \leq 3,$ over a time interval of length $x_2^t-y_2^r \leq 2^{2(\ell-n)}$. Therefore,
\begin{equation}\label{08rd11.8}
   P(Y^\prime_2 \geq x \, 2^{\ell-n} \vert \cH) \leq C \, e^{-x^2/6}.
\end{equation}
Arguing as for $Y_4$, we find that on $\{\hat Q < n-k\} \cap \hat \cO_n(n - k)$,
\begin{equation}\label{08rd11.9}
P(Y^\prime_3 \geq 2^{\ell-n} \vert \cH) \leq C \, \exp(-2^{n-k}).
\end{equation}

   We now turn to the term $Y^\prime_4$. Each increment appearing in the definition of $Y^\prime_4$ is equal to an increment $A_1$ of $X^r(0,\cdot)$ minus an increment $A_2$ that contributes to the term $Y_2$ in \eqref{rd10.14}, minus an increment $A_4$ that contributes to the term $Y_4$ in \eqref{rd10.14}, minus an increment $A_3 = W([t_1,y_1^t] \times [u_2,y_2^r])$.
	
	 On the event $\{\hat Q < n-k\} \cap \hat \cO_n(n - k)$, using the bounds on the variables $E_{\ti k,\ti \ell}^{i,j,t,r}$, we see that $A_3$ is bounded by $(n-k)^2\, 2^{\frac{3}{2} (k-n)} \ll 2^{k-n}$ for $n-k \geq k_0$. Further, $\vert A_1 \vert \leq 2^{k-n}$ and we have seen that $\vert A_2 \vert \leq 2^{k-n+1}$. Therefore, 
$$
   P(Y^\prime_4 \geq 2^{2+k-n} \mid \cH) \leq P\{Y_4 \geq 2^{k-n} \mid \cH \} \leq C \, \exp(-2^{n-k})
$$
by \eqref{11.6aa}.

  
  To summarize, on $\{\hat Q < n-k\} \cap \hat \cO_n(n - k)$,
$$
  \Xbar^{r,t} \leq c \ 2^{k-n} + Y_3 + Y^\prime_2 + Z,
$$
where $P\{Z \geq 2^{2+\ell-n} \vert {\cal{H}}\} \leq \exp(-2^{n-k}) \leq \exp(-2^{n-\ell})$, and $Y^\prime_2$ (resp. $Y_3$) satisfies \eqref{08rd11.8} (resp.~\eqref{08rd10.15}). This establishes Lemma \ref{lem11.4}. 
\hfill $\Box$
\vskip 16pt
  
  Using the notation from the proof of Lemma \ref{lem11.4}, let 
$$
    \Ybar^{r,t} = \max (Y_1, Y_2, Y_4, Y_5, Y^\prime_1, Y^\prime_3, Y^\prime_4),
$$ 
so that
\begin{equation}\label{11.14rd}
  \Xbar^{r,t} \leq 2^{-n} + 7\, \Ybar^{r,t} + Y_3 + Y^\prime_2,
\end{equation}
and 
\begin{equation}\label{08.0.1}
  P\{\Ybar^{r,t} \geq 2^{2+\ell-n} \mid {\cal{H}}\} \leq \exp(-2^{n-k}) \qquad \mbox{ on } \{\hat Q < n-k\} \cap \hat \cO_n(n - k).
\end{equation}
Notice that $Y_3$ and $Y^\prime_2$ are conditionally independent and conditionally independent of ${\cal{H}} \vee \sigma(\Ybar^{r,t})$ given $x^t_1$, $y^t_1$, $x^r_1$, $x^r_2$, $y^r_2$, $x^t_2$.

\begin{lemma} The following inequality holds:
  \begin{eqnarray*}
  && P(\{\Ybar^{r,t} \geq 2^{2+\ell-n}\} \cap A_1(t,n) \cap A_1(r,n)\\
  &&\qquad\qquad\qquad \cap\, A_2(t,n,n-k,v) \cap A_2(r,n,n-k,v) \cap\{\hat Q < n-k\} )\\
  &&\qquad\qquad \leq C \ 2^{-n-\ell-2k\lambda_1} \exp(-2^{n-k}).
  \end{eqnarray*}
\label{08.lem1}
\end{lemma}
  
\proof We write
$$
  W(t) = A^t_0 + A_1^t, 
$$
where $A_0^t = W([0, x^t_1] \times [y^r_2,x^t_2])$, so $A_0^t$ is conditionally independent of $\IH_1 := \IH \vee \sigma(A_1^t, W(r), \Ybar^{r,t})$ given $(x^t_1, y^r_2, x^t_2)$. Let
$$
   F_1 = \{\Ybar^{r,t} \geq 2^{2+\ell-n}\} \cap A_1(r,n) \cap A_2(t,n, n-k,v) \cap A_2(r,n, n-k,v) \cap \{\hat Q < n-k\}.
$$
Then $F_1$ is $\IH_1$-measurable, so the probability in the statement of the lemma can be written
$$
    E(1_{F_1} P\{A_0^t \in [2^{-n} - A_1^t,\ 2^{1-n} - A_1^t [\, \mid x^t_1, y^r_2, x^t_2\}).
$$
Further, the conditional law of $A^t_0$ given $x^t_1, y^r_2, x^t_2$ is Normal with mean $0$ and variance 
\begin{equation}\label{11.15a}
   x^t_1 (x^t_2 - y^r_2) \geq t_2 - 2^{2(k-n)-4} - (r_2 + 2^{2(k-n)-4}) \geq \frac{1}{4} 2^{2(\ell-n)}.
\end{equation}
Therefore, letting $Z$ be a standard Normal random variable, 
\begin{align*}
   & P\{A_0^t \in [2^{-n} - A_1^t,\ 2^{1-n} - A_1^t [\, \mid x^t_1, y^r_2, x^t_2\} \\
	  &\qquad \leq \sup_{x \in \IR} P\{ 2^{\ell - n - 1} Z \in  [-2^{1-n} + x, -2^{-n} + x [\} 
		\leq c 2^{-\ell}.
\end{align*}

%

   We now write $W(r) = W(1,1) + A^r$, where $A^r$ is a sum of Brownian sheet increments that are independent of $W(1,1)$ but not necessarily of $\IH$. Then
$$
   P(F_1) = P(W(1,1) \in [-2^{1-n}-A^r,\, -2^{-n} -A^r[\, \mid F_2) P(F_2),
$$
where $F_2 = \{\Ybar^{r,t} \geq 2^{2+\ell-n}\} \cap A_2(t,n,n-k,v) \cap A_2(r,n,n-k,v) \cap \{\hat Q < n-k\}$, and the conditional probability is no greater than
$$
   \sup_{x \in \IR} P\{W(1,1) \in [-2^{1-n} + x, -2^{-n} + x[\} \leq 2^{-n} .
$$
   
   It remains only to show that $P(F_2) \leq c \ 2^{-2k \lambda_1} \exp(-2^{n-k})$. 
The key point is now that on $\{\hat Q < n-k\} \cap A_2(t,n,n-k,v) \cap A_2(r,n,n-k,v)$, the DW-algorithms for $\hat X^t$ and $\hat X^r$are $9v$-robust above order $n-k$, so as in the proof of \eqref{10.16rd}, $\hat \cO_n(n - k) \cap \hat A_2(r,n,n-k,9v) \cap \hat A_2(t,n,n-k,9v)$ occurs.
Since $\hat X^t$ and $\hat X^r$ are independent, we deduce that
\begin{align*}
   P(F_2) &\leq P(\{\Ybar^{r,t} \geq 2^{2+\ell-n}\} \cap \{\hat Q < n-k\}\\
	 & \qquad \qquad \cap \hat \cO_n(n - k) \cap \hat A_2(t,n,n-k,9v) \cap \hat A_2(r,n,n-k,9v)).
\end{align*}
By \eqref{08.0.1},
\begin{align*}
   P(F_2) &\leq \exp(-2^{n-k}) P(\hat A_2(r,n,n-k,9v) \cap \hat A_2(t,n,n-k,9v)) \\
	  & \leq \exp(-2^{n-k}) c (2^{-k\lambda_1})^2,
\end{align*}
by independence of $\hat X^t$ and $\hat X^r$ and  Theorem \ref{thm3}. This completes the proof of Lemma \ref{08.lem1}. 
\hfill $\Box$
\vskip 16pt
   
   Set
$$
   F^* = A_1(t,n) \cap A_2(t,n,n-k,v) \cap A_1(r,n) \cap A_2(r,n,n-k,v).
$$
   
\begin{lemma} (a) For some universal constants $c$ and $c'$, for all $M_1 \geq 0$ and $M_2 \geq 0$,
\begin{eqnarray*}
 && P(F^* \cap \{\hat Q < n-k\} \cap \{\Ybar^{r,t} < 2^{2+\ell-n}\} \cap \{Y_3 \geq M_1 2^{\ell-n}\} 
\cap\, \{Y^\prime_2 \geq M_2 2^{\ell-n}\} )\\
 && \qquad
   \leq c' \ 2^{-n-\ell-2k \lambda_1} \,  e^{-c(M^2_1+M_2^2)}.
\end{eqnarray*}

(b) For all $M \geq 0$,
\begin{align}\nonumber
 & P(F^* \cap \{\hat Q \leq n-k\} \cap \{\Ybar^{r,t} < 2^{2+\ell-n}\} \cap \{\bar X^{r,t} \geq M 2^{\ell-n}\} )\\
   & \qquad
   \leq c' \ 2^{-n-\ell-2k \lambda_1}   e^{-cM^2}.
\label{11.15rd}
\end{align}

\label{08.lem2}
\end{lemma}

   
\proof (a) Recall that
$$
   W(t) = A^t_0 + A^t_1, 
$$
where $A_0^t$ and $A_1^t$ are defined in the proof of Lemma \ref{08.lem1}. Set
\begin{eqnarray*}
 \tilde F_1 &=& \{\Ybar^{r,t} < 2^{2+\ell-n}\} \cap A_1(r,n) \cap A_2(r,n,n-k,v) \cap A_2(t,n,n-k,v) \\
 &&\qquad \cap \ \{\hat Q < n-k\} \cap \{Y_3 \geq M_1 2^{\ell-n}\}. 
\end{eqnarray*} 
Observe that $(A^t_0, Y'_2)$ is conditionally independent of $\IH_1 := \IH \vee \sigma(Y_3, \Ybar^{r,t}, A^t_1)$ given $x^t_1, x^t_2, y^r_2$, so that the probability in the statement of the lemma is equal to
\begin{equation}\label{11.18a}
   E(1_{ \tilde F_1} P\{ Y'_2 \geq M_2 2^{\ell - n}, A^t_0 \in [-2^{1-n} - A^t_1, -2^{-n} - A^t_1] \mid x^t_1, x^t_2, y^r_2 \}),
\end{equation}
and the conditional probability above is bounded above by
\begin{align}\nonumber
   &\sup _{x \in \IR} P\{Y'_2 \geq M_2 2^{\ell - n}, A^t_0 \in [-2^{1-n} + x, -2^{-n} + x] \mid x^t_1, x^t_2, y^r_2 \} \\ \label{11.19a}
	 &\qquad = \sup _{x \in \IR} \sum_{m=0}^\infty P\{Y'_2 \in [ (M_2+m) 2^{\ell - n},(M_2+m+1) 2^{\ell - n} [, \\
	 &\qquad\qquad\qquad\qquad A^t_0 \in [-2^{1-n} + x, -2^{-n} + x] \mid x^t_1, x^t_2, y^r_2 \}.
	\nonumber
\end{align}
The conditional law of $u_2 \mapsto W(x^t_1, u_2) - W(x^t_1,x^t_2)$, as $u_2$ decreases from $x^t_2$ to $y^r_2$, given $x^t_1, x^t_2, y^r_2$, is that of a Brownian motion with speed $x^t_1$, so we can apply Lemma \ref{supBM} to obtain
\begin{align*}
   P\{A^t_0 \in [-2^{1-n} + x, -2^{-n} + x] \mid Y'_2, x^t_1, x^t_2, y^r_2 \} \leq \frac{2^{-n}}{\sqrt{x^t_1}} \frac{2 Y'_2/\sqrt{x^t_1}}{x^t_2 - y^r_2} \leq \frac{8\, 2^{-n} Y'_2}{2^{2(\ell - n)}},
\end{align*}
where we have used \eqref{11.15a}. Notice that this bound no longer depends on $x$. Therefore, \eqref{11.19a} is bounded above by
\begin{align*}
   &\sum_{m=0}^\infty  \frac{8\, 2^{-n}}{2^{2(\ell - n)}} (M_2 + m + 1) 2^{\ell - n} P\{Y'_2 \geq (M_2+m) 2^{\ell - n} \mid x^t_1, x^t_2, y^r_2 \} \\
	&\qquad 8\, 2^{-\ell} \sum_{m=0}^\infty (M_2 + m + 1) \exp(-c (M_2 + m)^2),
\end{align*}
where we have used \eqref{08rd11.8}. Clearly, for some universal constant $c'>0$, and slightly smaller $c>0$, this is $\leq c' 2^{-\ell } e^{-cM_2^2}$. Therefore, \eqref{11.18a} is bounded above by $c' 2^{-\ell } e^{-cM_2^2} P(\tilde F_1)$.



   We now write
\begin{equation}\label{08.0.3}
 P(\tilde F_1) \leq \sup_{x \in \IR} P(W(1,1) \in [-2^{1-n} + x, -2^{-n} + x[ \ \vert \ \tilde F_2) P(\tilde F_2),
\end{equation}
where
\begin{eqnarray*}
 \tilde F_2 &=& \{\Ybar^{r,t} < 2^{2+\ell-n}\} \cap A_2(r,n,n-k,v) \cap A_2(t,n,n-k,v)) \cap \{\hat Q \leq n-k\}\\
 && \qquad\cap\, \{Y_3 \geq M_1 2^{\ell-n}\}.
\end{eqnarray*}
 Since $\tilde F_2$ is independent of $W(1,1)$, the conditional probability in (\ref{08.0.3}) is bounded by $2^{-n}$. The 
event $\{\Ybar^{r,t} < 2^{2+\ell-n}\}$ can be omitted, and we obtain using (\ref{08rd10.15}) and the $9v$-robustness property that 
$$
 P(\tilde F_2) \leq c' \ \exp(-c M_1^2) P(\{\hat Q \leq n-k\} \cap A_2(r,n,n-k,v) \cap A_2(t,n,n-k,v)).
$$
With the argument used in the end of the proof of Lemma \ref{08.lem1}, we conclude that
$$
 P(\tilde F_2) \leq C \exp(- c M_1^2) \, 2^{-2k\lambda_1}.
$$
 This proves (a).

   (b) Observe from \eqref{11.14rd} that the event on the left-hand side of \eqref{11.15rd} is contained in
$$
   P(F \cap \{\hat Q \leq n-k \} \cap \{\Ybar^{r,t} < 2^{2+\ell -n} \}\cap \{Y_3 + Y^\prime_2 \geq M 2^{\ell -n} - 2^{-n} - 7\cdot 2^{2 + \ell -n} \}
$$
and the last event is included in 
$$
   \{Y_3 + Y^\prime_2 \geq (M - 30) 2^{\ell -n} \vee 0 \} \subset \cup_{i=0}^{M - 30} \{Y_3 \geq i 2^{\ell - n},\, Y^\prime_2 \geq (M - 30-i)  2^{\ell - n} \}.
$$	
By (a), the left-hand side of \eqref{11.15rd} is bounded above by
$$
	\sum_{i=0}^{M - 30} c' 2^{-n-\ell -2k \lambda_1} \exp(-c(i^2 + (M - 30 - i)^2)).
$$
Use the inequality $\frac12 (a+b)^2 \leq a^2 + b^2$ to see that this is bounded above by
\begin{align*}
   \sum_{i=0}^{M-30} c' 2^{-n-\ell -2k \lambda_1} \exp(-\frac{c}{2} (M-30)^2)) 
	   &\leq \tilde c' 2^{-n-\ell -2k \lambda_1} M \exp(-\hat c M^2)\\
		 &\leq \tilde C'' 2^{-n-\ell -2k \lambda_1} \exp(-\hat c' M^2).
\end{align*}
This proves Lemma \ref{08.lem2}. 
\hfill $\Box$
\vskip 16pt

   Set
\begin{eqnarray}\label{defG_M}
 G_M &=& A_2(t,n,k_0,v) \cap A_2(r,n,k_0,v)\cap \{\hat Q < n-k\} \\
    &&\qquad  \cap\, \{\Ybar^{r,t} < 2^{\ell-n}\} \cap \{\Xbar^{r,t} \in [ (M-1)2^{\ell-n}, M 2^{\ell-n}[\}.
\nonumber
\end{eqnarray}
On $G_M$,\index{$G_M$}  within $[x^t_1 - t_1, y^r_1 - t_1] \times [x^r_2 - t_2, y^t_2 - t_2 ]$, the DW-algorithm for $X^t$ goes no higher than the level $M2^{\ell - n}$. In coordinates for $X^t$, we denote the ``information rectangle" explored by $X^t$ up to escaping this rectangle by
$$
  [U_m, U^\prime_m] \times [V_m, V^\prime_m], \qquad \mbox{for some } m \in \IN.
$$
Since ``information rectangles" increase, the behavior of the DW-algorithm for $X^t$ {\em after} escaping $[x^t_1 - t_1, y^t_1 - t_1] \times [x^r_2 - t_2, y^t_2 - t_2 ]$ is mainly determined by the increments\index{$I_1(u_1)$}\index{$I_2(u_2)$}
\begin{eqnarray*}
  I_1(u_1) &=& X^t(\tilde U_m' + u_1, 0) - X^t(\tilde U^\prime_m, 0), \qquad u_1 \geq 0,\\
  I_1(u_1) &=& X^t (\tilde U_m + u_1, 0) - X^t(\tilde U_m, 0), \qquad u_1 \leq 0,\\
  I_2(u_2) &=& X^t(0, \tilde V^\prime_m +  u_2) - X^t(0, \tilde V^\prime_m), \qquad u_2 \geq 0,\\
  I_2(u_2) &=& X^t(0, \tilde V_m+u_2) - X^t(0, \tilde V_m), \qquad u_2 \leq 0,
\end{eqnarray*}
where $\tilde U^\prime_m = U^\prime_m \vee (y^r_1-t_1)$, $\tilde U_m = U_m \wedge (x^t_1-t_1)$, $\tilde V^\prime _m = V^\prime _m \vee (y^t_2-t_2)$, $\tilde V_m = V_m \wedge (x^r_2 - t_2)$, and certain other increments (for instance, if $U^\prime_m < \tilde U^\prime_m$, then increments from $U^\prime_m$ to $y^r_1 - t_1$ play a role, but these are bounded by $M 2^{\ell - n}$).

  Further, for $(v_1, v_2) \not\in ([\tilde U_m, \tilde U^\prime_m] \times \IR_+) \cup (\IR_+ \times [\tilde V_m, \tilde V^\prime_m])$, say for instance that $v_1 > \tilde U^\prime_m$ and $v_2 > \tilde V^\prime_m$, then
$$
  X^t(v_1, v_2) = X^t(\tilde U^\prime_m, \tilde V^\prime_m) + I_1(v_1-\tilde U^\prime_m) + I_2(v_2-\tilde V^\prime_m),
$$
so on $G_M$, $X^t(v_1, v_2) > 0$ implies that
\begin{equation}\label{08.0.6}
  3 M 2^{\ell - n} + I_1(v_1- \ti U^\prime_m) + I_2(v_2- \ti V^\prime_m) > 0.
\end{equation}
This suggests to consider a new additive process\index{$I(u_1, u_2)$}
$$
  I(u_1, u_2) = I_1(u_1) + I_2(u_2)
$$
and to start it at value $2^{-k_{M,n}}$, where $k_{M,n}$ is chosen so that
$$
   2^{-k_{M,n}} = 3 M 2^{\ell - n}.
$$
Indeed, on the event $G_M$, the DW-algorithm applied to $I(\cdot, \cdot)$ started at value $3 M 2^{\ell - n}$ will escape the square with side of length $2^{-2k_0-1} - 2^{2(\ell - n)}$.
  
  Given $\IH$ and $\Xbar^{r,t}$, $I_1$ and $I_2$ are not independent and are not Brownian motions. We observe how these Brownian sheet increments interact with previously used increments by examining Figure \ref{figureF}. In particular, the white noise increments that will be used by $I_1(\cdot)$ up to escaping to $2^{-2 k_0 -1}$ do not involve subsets of the vertical strip $[0, \frac{3}{2}] \times \IR_+$.
\begin{figure}
\begin{center}
\begin{picture}(400,420)
\put(20,20){\line(0,1){5}} \multiput(20,25)(0,6){2}{\line(0,1){2}} \put(20,38){\line(0,1){25}} \multiput(20,63)(0,6){2}{\line(0,1){2}} \put(20,75){\line(0,1){65}} 
\put(20,140){\line(0,1){240}}

\put(20,20){\line(1,0){5}} \multiput(25,20)(6,0){2}{\line(1,0){2}} \put(38,20){\line(1,0){25}} \multiput(63,20)(6,0){2}{\line(1,0){2}} \put(75,20){\line(1,0){65}} 
\put(140,20){\line(1,0){240}}

\put(215,340){\vector(0,1){35}}  \put(200,370){$I_2$}
\put(215,180){\vector(0,-1){35}} \put(200,145){$I_2$}
\put(300,315){\vector(1,0){35}}  \put(330,300){$I_1$}
\put(170,315){\vector(-1,0){35}} \put(135,300){$I_1$}

\put(170,180){\line(1,0){130}}
\put(170,180){\line(0,1){160}}
\put(170,340){\line(1,0){130}}
\put(300,180){\line(0,1){160}}

\put(20,330){\line(1,0){180}}

\put(215,315){\circle*{3}}
\put(219,312){$t$}

\put(20,240){\line(1,0){180}}

\put(200,300){\line(1,0){30}}
\put(200,300){\line(0,1){30}}
\put(200,330){\line(1,0){30}}
\put(230,300){\line(0,1){30}}

\put(265,225){\circle*{3}}
\put(269,222){$r$}

\put(250,210){\line(1,0){30}}
\put(250,210){\line(0,1){30}}
\put(250,240){\line(1,0){30}}
\put(280,210){\line(0,1){30}}

\multiput(200,210)(6,0){9}{\line(1,0){2}}
\multiput(200,210)(0,6){16}{\line(0,1){2}}
\multiput(230,300)(6,0){9}{\line(1,0){2}}
\multiput(280,240)(0,6){10}{\line(0,1){2}}

\put(50,19){\line(0,1){2}}
\put(48,7){$1$}

\put(100,19){\line(0,1){2}}
\put(98,6){$\frac{3}{2}$}

\put(160,19){\line(0,1){2}}
\put(157,7){$2$}

\put(170,19){\line(0,1){2}}
\put(160,-5){$t_1+ \ti U_m$} \put(174,5){\vector(-1,3){4}}

\put(200,19){\line(0,1){2}}
\put(197,7){$x^t_1$}

\put(280,19){\line(0,1){2}}
\put(277,7){$y^r_1$}

\put(300,19){\line(0,1){2}}
\put(290,-5){$t_1+ \ti U_m'$} \put(304,5){\vector(-1,3){4}}

\put(350,19){\line(0,1){2}}
\put(347,7){$3$}

\multiput(160,160)(6,0){32}{\line(1,0){2}}
\multiput(160,160)(0,6){32}{\line(0,1){2}}
\multiput(160,350)(6,0){32}{\line(1,0){2}}
\multiput(350,160)(0,6){32}{\line(0,1){2}}

\put(19,350){\line(1,0){2}}
\put(10,347){$3$}

\put(19,340){\line(1,0){2}}
\put(-30,337){$t_2 + \ti V_m'$}

\put(19,330){\line(1,0){2}}
\put(7,327){$y_2^t$}

\put(19,300){\line(1,0){2}}
\put(7,297){$x_2^t$} 

\put(19,240){\line(1,0){2}}
\put(7,237){$y_2^r$}

\put(19,210){\line(1,0){2}}
\put(7,207){$x_2^r$}

\put(19,180){\line(1,0){2}}
\put(-20,177){$t_2 + \ti V_m$}

\put(19,160){\line(1,0){2}}
\put(10,157){$2$}

\put(19,100){\line(1,0){2}}
\put(10,97){$\frac{3}{2}$}

\put(19,50){\line(1,0){2}}
\put(10,47){$1$}

\end{picture}
\end{center}
\caption{Relationship between the additive process $I(\cdot, \cdot)$ and previously used increments.\label{figureF}} 
\end{figure}
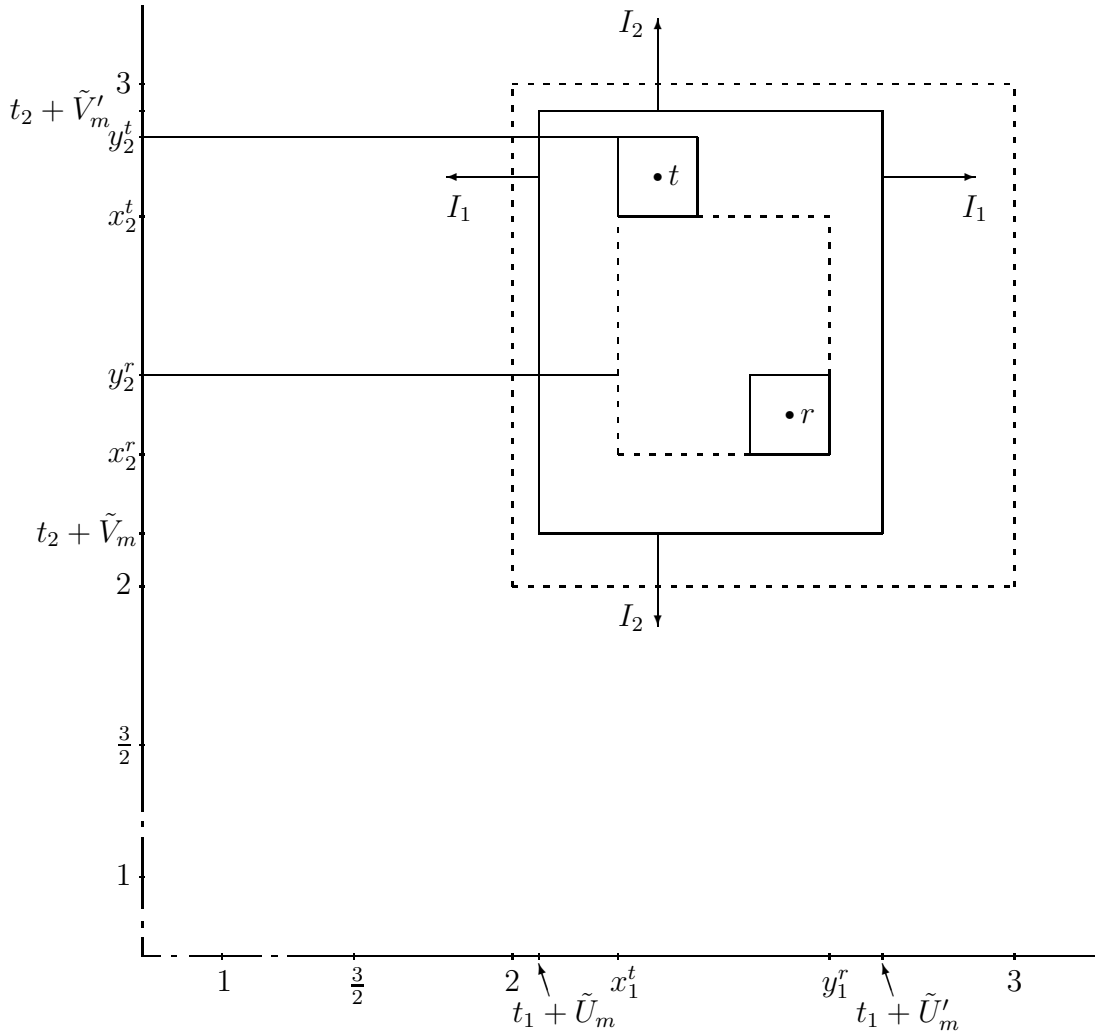
  
\begin{remark} On $G_M$,
$$
\max(-\tilde U_m, \tilde U^\prime_m, - \tilde V_m, \tilde V^\prime_m) \leq 2^{8} v^{3/2} M^2\, (n-\ell)^2\, 2^{2(\ell-n)}.
$$
Indeed, due to $v$-robustness, the length of an order $j$ episode is no more than $ v j 2^{-2j}$ and there are no more than $j \sqrt{v}$ episodes of order $j$, so the maximal distance from $t$ before reaching level $3 M 2^{\ell - n}$ is bounded (using Lemma \ref{rdlem35}(c)) by 
\begin{eqnarray*}
\sum^n_{j=k_{M,n}} v \ j \ 2^{-2j} j\sqrt{v} &\leq& v^{3/2} \sum^\infty_{j=k_{M,n}} j^2 2^{-2j} \leq 2 \ v^{3/2} \ k_{M,n}^2 \ 2^{-2k_{M,n}}\\
&\leq& 2^{5} v^{3/2} M^2 (n-\ell)^2 2^{2(\ell-n)}.
\end{eqnarray*}
\label{rem11.8}
\end{remark}
\vskip 12pt

\noindent{\em The {\em boosted} DW-algorithm}
\vskip 12pt

   We are going to use the additive process $(I(u_1, u_2))$ to construct a standard ABM $(\tilde X(u_1, u_2))$, and an associated path $\tilde \Gamma,$ such that on $G_M$, $2^{-k_{M,n}} + \tilde X(\cdot, \cdot)$ essentially escapes to $2^{-2 k_0}$. This ABM $\tilde X(\cdot, \cdot)$ will be independent of previously used white noise increments. However, since the DW-algorithm for $\tilde X$ started at level $2^{-k_{M,n}}$ is not {\em a priori} $v$-robust, we will guarantee that on $A_2(t,n,k_0,v)$, $2^{-k_{M,n}} + \tilde X(\cdot, \cdot)$ escapes to $2^{-2 k_0}$ (or reaches level $2^{- k_0}$) with high probability, by progressively increasing (``boosting") the value $2^{-k_{M,n}}$ during the construction of $\tilde \Gamma$, yet without significantly changing the gambler's ruin or escape probabilities. This boosting compensates for differences coming from the increments of the different white noise that will be used by $\ti X(\cdot, \cdot)$ but not by $I(\cdot, \cdot)$. The precise construction is as follows.

   Let $\tilde W$ be a Brownian sheet that is independent of $W$ and $W^\prime$, and let $\dot{\tilde W}$ be its associated white noise.

\begin{figure}
\begin{center}
\begin{picture}(200,200)
\put(30,20){\line(0,1){190}}
\put(30,20){\line(1,0){190}}

\put(90,140){\line(1,0){30}}
\put(90,170){\line(1,0){30}}
\put(90,140){\line(0,1){30}}
\put(120,140){\line(0,1){30}}

\put(105,155){\circle*{3}}
\put(108,152){$t$}

\put(-10,178){$t_2 + \ti V_m'$}
\put(29,180){\line(1,0){2}}

\put(18,152){$t_2$}
\put(29,155){\line(1,0){2}}

\put(-10,58){$t_2 + \ti V_m$}
\put(29,60){\line(1,0){2}}

\put(105,19){\line(0,1){2}}
\put(104,5){$t_1$}

\put(150,70){\line(1,0){30}}
\put(150,100){\line(1,0){30}}
\put(150,70){\line(0,1){30}}
\put(180,70){\line(0,1){30}}

\put(165,85){\circle*{3}}
\put(168,82){$r$}

\put(80,60){\line(1,0){110}}
\put(80,180){\line(1,0){110}}
\put(80,60){\line(0,1){120}}
\put(190,60){\line(0,1){120}}

\put(80,19){\line(0,1){2}}
\put(55,5){$t_1 + \ti U_m$}

\put(190,19){\line(0,1){2}}
\put(174,5){$t_1 + \ti U_m'$}

\end{picture}
\end{center}
\caption{Situation for the boosted DW-algorithm.\label{figureG}} 
\end{figure}
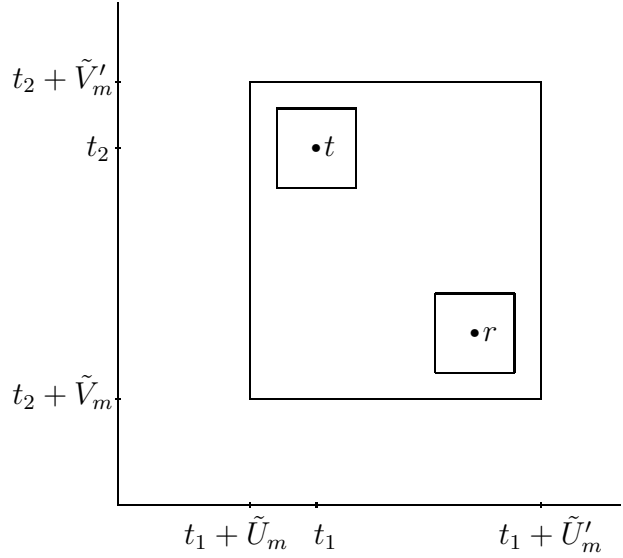
   In the rectangle $[0, t_1 + \tilde U_m] \times [t_2 + \tilde V_m, t_2]$ (see Figure \ref{figureG}), we immediately replace the white noises $\dot W$ or $\dot W^\prime$ by $\dot{\tilde W}$ and use this modified white noise to construct a Brownian sheet $W^1$. We set
$$
Z^1_1(u_1) = \left\lbrace \begin{array}{lll} W^1(t_1 + \tilde U_m + u_1,t_2) - W^1(t_1 + \tilde U_m, t_2) &\mbox{ if }& u_1 \leq 0,\\
I_1(u_1) &\mbox{ if }& u_1 \geq 0,
\end{array}
\right.
$$
and $Z^1_2(u_2) = I_2(u_2)$, $u_2 \geq-2.$ We run Stage 1 of the DW-algorithm started at $(0,0)$ with value $2^{-k_{M,n}}$ for $X^1(u_1, u_2) = 2^{-k_{M,n}} + Z^1_1(u_1) + Z^1_2(u_2).$ This produces two random variables $U_1$ and $U^\prime_1$ with $U_1 < 0 < U^\prime_1$.\\
We now define a new Brownian sheet $W^2$ by letting its associated white noise be
$$
\dot W^2 = \left\{\begin{array}{ll}
\dot{\tilde W} & \mbox{ on } [t_1 + \tilde U_m + U_1, t_1 + \tilde U_m] \times [0, t_2 + \tilde V_m],\\
\dot W^1 & \mbox{elsewhere}
\end{array}
\right.
$$
(recall that the white noise in $[t_1 + \tilde U_m, t_1] \times [0, t_2 + \tilde V_m]$ which was used during the construction of $\tilde U_m$, $\tilde V_m$ has already been replaced by unused independent white noise). We define two Brownian motions $Z^2_1$ and $Z^2_2$ by
$$
  Z^2_1 (u) = Z^1_1(u),
  \qquad Z^2_2(u) = \left\{ \begin{array}{lll}
W^2(t_1, t_2 + \tilde V_m + u) - W^2(t_1, t_2 +\tilde V_m) & \mbox{ if } & u \leq 0,\\
I_2(u) &\mbox{ if }  &u \geq 0.
\end{array}
\right.
$$
We define an ABM $X^2$ by setting
$$
X^2(s_1, s_2) = 2^{-k_{M,n}} + Z^2_1(s_1) + Z^2_2(s_2), \qquad (s_1, s_2) \in [-2, + \infty[^2.
$$
We note that Stage 1 of the DW-algorithm is the same for $X^2$ and $X^1$.

   We now run Stage 2 of the DW-algorithm for $X^2$. This produces in particular two random variables $V_1 < 0 < V_1'$. We define a new Brownian sheet $W^3$ by letting its associated white noise be
$$
\dot W^3 = \left\{\begin{array}{ll}
\dot{\tilde W} & \mbox{ on } [0, t_1 + \tilde U_m + U_1] \times [t_2 + \tilde V_m +V_1, t_2 + \tilde V_m],\\
\dot W^2 & \mbox{elsewhere.}
\end{array}
\right.
$$
We define two Brownian motions $Z^3_1$ and $Z^3_2$ by
$$
 Z^3_1(u) = \left\{ \begin{array}{lll}
              Z^2_1(u) & \mbox{ if } & u \geq U_1,\\
    Z^2_1(U_1) + W^3(t_1 +\tilde U_m +u, t_2) - W^3(t_1 +\tilde U_m +U_1, t_2)
           &\mbox{ if }  &u \leq U_1,
\end{array}
\right.
$$
and $Z^3_2(u) = Z^2_2(u)$ for $u \geq -2$. We note that $(Z^3_1(u),\ u \geq -2)$ and $(Z^3_2(u),\ u \geq V_1)$ are independent. We define an ABM $X^3$ by setting
$$
   X^3(s_1,s_2) = 2^{-k_{M,n}} + Z^3_1(s_1) + Z^3_2(s_2), \qquad (s_1,s_2) \in [-2,+\infty[^2,
$$
and we note that Stages 1 and 2 for the DW-algorithm are the same for $X^2$ and $X^3$.

   The construction now proceeds by induction, similar to the construction in Section \ref{sec10}, with the following significant change. The first time that the ABM that we are currently using reaches level $2^{1-k_{M,n}}$, we continue the DW-algorithm and the construction with the starting value $x_0 = 2^{-k_{M,n}}$ replaced by $x_1 = 2^{-k_{M,n}} + (k_{M,n}-1)^{-2} \ 2^{1-k_{M,n}}$ until we reach level $2^{2-k_{M,n}}$ (with the new starting value), at which time we replace the starting value $x_1$ by $x_2 = x_1 + (k_{M,n}-2)^{-2} \ 2^{2-k_{M,n}}$ and so forth : the first time we reach level $2^{-j} (j < k_{M,n})$, we replace the starting value $x_{k_{M,n}-1-j}$ by $x_{k_{M,n}-j} = x_{k_{M,n}-1-j} + j^{-2} \ 2^{-j}$ and we continue the algorithm with this starting value until we reach level $2^{1-j}$, etc. By the time we reach level $2^{-j}$, we are using the starting value
\begin{equation}\label{08.0.11.17}
   x_{k_{M,n}-j} = 2^{-k_{M,n}} + \sum^{k_{M,n}-1}_{h=j} h^{-2} \, 2^{-h}
\end{equation}
(equivalently, $x_{j} = 2^{-k_{M,n}} + \sum_{i=1}^{j} (k_{M,n}- i)^{-2}\, 2^{i - k_{M,n}}$). We refer to this modified DW-algorithm as the {\em boosted DW-algorithm.} 

   It will be important to be slightly more precise about exactly what is done when each ``boost" occurs. If we are using the ABM $X^i$ at the first time that the algorithm reaches level $2^{1-k_{M,n}}$, and if this level is reached during an odd stage $2\ti n+1$, then there are three cases:
	
	{\em Case 1.} If $x_0 + X^i(\ultau_1^{2^{1-k_{M,n}}}, T^{\ti n}_2) = 2^{1-k_{M,n}}$ and $x_0 + X^i(\ultau_3^{2^{1-k_{M,n}}}, T^{\ti n}_2) = 0$, so that $\ultau_3^{2^{1-k_{M,n}}} = U_{\ti n + 1}$, then we keep this value of $U_{\ti n + 1}$, and we replace $x_0$ by $x_1$ before constructing $U^\prime_{\ti n+1}$.
	
	{\em Case 2.} If $x_0 + X^i(\ultau_1^{2^{1-k_{M,n}}}, T^{\ti n}_2) = 0$ and $x_0 + X^i(\ultau_3^{2^{1-k_{M,n}}}, T^{\ti n}_2) = 2^{1-k_{M,n}}$, so that $\ultau_1^{2^{1-k_{M,n}}} = U^\prime_{\ti n + 1}$, then we keep this value of $U^\prime_{\ti n + 1}$, and we replace $x_0$ by $x_1$ before constructing $U_{\ti n+1}$.
	
	{\em Case 3.} If $x_0 + X^i(\ultau_1^{2^{1-k_{M,n}}}, T^{\ti n}_2) = 2^{1-k_{M,n}}$ and $x_0 + X^i(\ultau_3^{2^{1-k_{M,n}}}, T^{\ti n}_2) = 2^{1-k_{M,n}}$, then we replace $x_0$ by $x_1$ before constructing $U^\prime_{\ti n+1}$ and $U_{\ti n+1}$.
\vskip 12pt

\noindent If the level $2^{1-k_{M,n}}$ is reached during an even stage, then we proceed analogously, and, similarly, at each other level where ``boosting" occurs.

  When this boosted DW-algorithm STOPS, we will have constructed a standard ABM 
$$
   \tilde X(s_1, s_2) = \tilde Z_1(s_1) + \tilde Z_2(s_2)
$$ 
and a path $\tilde \Gamma$ along which this ABM, to which we add the appropriate ``boosted" starting value as we move along the path, remains positive. We have also constructed the variables $U_{\ti n}$, $U^\prime_{\ti n}$, $V_{\ti n}$, $V^\prime_{\ti n}$, $H_{\ti n}$, $T^{\ti n}_1$ and $T^{\ti n}_2$ that are produced at each stage of the algorithm.

  Before examining escape probabilites for this boosted DW-algorithm, we introduce some events on which the standard ABM $\tilde X$ may differ too much from $I$. For $j \geq 1,$ define 
\begin{equation}\label{08.0.11.19}
S_j = S_{j,1} \cup S_{j,2} \cup S_{j,3}, 
\end{equation}
where
\begin{description}
\item[$S_{j,1}$ =]  ``the boosted DW-algorithm has an episode of order $j$ of length $\geq j^2 2^{-2j}$ or has at least $j^2$ episodes of order $j$;"
\item[$S_{j,2}$ =] ``for some horizontal episode $[\tilde U^{(\ti \ell)}, \tilde U^{(\ti \ell-1)}]$ for $\tilde X$ of order $j$,
$$
\sup_{\tilde U^{(\ti \ell)}< u_1 < \tilde U^{(\ti \ell-1)}} \vert I_1(u_1) -I_1(\tilde U^{(\ti \ell-1)}) - (\tilde Z_1(u_1)-\tilde Z_1(\tilde U^{(\ti \ell-1)})) \vert \geq 2^{-3 j/2};"
$$
\item[$S_{j,3}$ =] ``for some vertical episode $[\tilde V^{(\ti \ell)}, \tilde V^{(\ti \ell-1)}]$ for $\tilde X$ of order $j$,
$$
\sup_{\tilde V^{(\ti \ell)} < u_2 < \tilde V^{(\ti \ell-1)}} \vert I_2(u_2) - I_2(\tilde V^{(\ti \ell-1)}) - (\tilde Z_2(u_2) - \tilde Z_2 (\tilde V^{(\ti \ell-1)})) \vert \geq 2^{-3j/2}."
$$
\end{description}
The Brownian sheet increments that appear in the definitions of $S_{j,2}$ and $S_{j,3}$ involve areas of order $2^{-2j}$, so the increments $I_i$ and $\tilde Z_i$ are typically of order $2^{-j}$. However, their difference is due to using different white noises in a region with area of order $(2^{-2j})^2$, so $S_{j,2}$ and $S_{j,3}$ are events with low probability.

\begin{lemma}
Suppose that $k_{M,n} \geq k_0$.
Let $h_0 \in \{k_0,\dots,k_{M,n}\}$, let $\tilde G_{h_0}$ be the event ``the boosted DW-algorithm (for $\tilde X$) started at $2^{-k_{M,n}}$ reaches level $2^{-h_0}$ or escapes $\R(2^{-2h_0}).$"
 
   (a) $\tilde G_{h_0}$ is independent of $\IH \vee \sigma (\bar X^{r,t})$;
   
   (b) for $k_0$ sufficiently large, $G_M \cap (\cup_{h_0 \leq j \leq k_{M,n}} S_j)^c \subset \tilde G_{h_0},$ that is, on $G_M$, if none of the $S_j$ occur, then the boosted DW-algorithm started at level $2^{-k_{M,n}}$ reaches level $2^{-h_0}$;
   
   (c) for $k_0$ sufficiently large, $P(\tilde G_{h_0}) \leq c \ 2^{(h_0-k_{M,n})\lambda_1}$.
\label{08.lem5}
\end{lemma}

\proof (a) This is a consequence of the way that $\tilde X$ is constructed during the boosted DW-algorithm: increments of the Brownian sheet $W$ that appear in the definition of $\IH$ and $\Xbar^{r,t}$ either do not overlap with those used to construct $\tilde X$ or, when they do overlap, are replaced by increments from an independent white noise.

(b) It suffices to show that
\begin{align}\nonumber
   \{\hat Q < n-k \} \cap \{\Ybar^{r,t} < 2^{\ell-n} \} &\cap \{\Xbar^{r,t} \in [ (M-1)2^{\ell-n}, M 2^{\ell-n}[ \} \\
	& \cap \tilde G^c_{h_0} \cap (\cup_{h_0 \leq j \leq k_{M,n}} S_j )^c  \subset (A_2(t,n,k_0,v))^c.
\label{08.0.8}
\end{align}
Observe that on $\tilde G_{h_0}^c$, the boosted DW-algorithm constructs a rectangle $\R(\tilde \ultau_{(N)})$ with the following properties :

   (i) $\R(\tilde \tau_{(N)}) \subset \R(2^{-2h_0}) \subset \R(\frac{1}{4})$;

   (ii) if the maximum height achieved by the boosted DW-algorithm is in the interval $[2^{-j},2^{1-j}[$, with $h_0 < j \leq k_{M,n}$, then
\begin{equation}\label{08.0.9}
x_{k_{M,n}-j} + \tilde X < 0\quad \mbox{ on }\quad \partial \R(\tilde \ultau_{(N)}).
\end{equation}
On $(\cup_{1 \leq i \leq k_0} S_i)^c$, within $\R(\tilde \ultau_{(N)})$, the accumulated difference between $\tilde X$ and $I$ cannot be too large. Indeed, for each episode of order $i$, the difference between an increment of $I$ and an increment of $\tilde X$ is at most $2^{-3i/2}$, by definition of $S_{i,2}$ and $S_{i,3}$. There are no more than $i^2$ such episodes on $S_{i,1}^c$. Therefore, at any point within $\R(\tilde \ultau_{(N)})$,
\begin{equation}\label{11.25rd}
\vert \tilde X - I \vert \leq \sum^{k_{M,n}}_{i=j} i^2 2^{-3i/2} \leq c \ j^2\, 2^{-3j/2}.
\end{equation}
Therefore, by (\ref{08.0.9}),
$$
x_{k_{M,n}-j} + I < c \, j^2 \, 2^{-3j/2} \quad \mbox{ on } \partial \R(\tilde \ultau_{(N)}),
$$
or, equivalently, by (\ref{08.0.11.17}),
$$
   2^{-k_{M,n}} + I < c \ j^2 \ 2^{-3j/2} - \sum^{k_{M,n}-1}_{h = j} h^{-2} 2^{-h}.
$$

For $j \in\, ]h_0, k_{M,n}]$ and $h_0$ sufficiently large, this right-hand side is negative, so
\begin{equation}\label{11.24}
    2^{-k_{M,n}} + I < 0  \quad \mbox{ on } \partial \R(\tilde \ultau_{(N)}).
\end{equation}
Suppose that $\tilde \ultau_{(N)} = (a_1,a_2,a_3,a_4)$. We want to deduce that $2^{-n} + X^t < 0$ on $C_1 \cup C_2 \cup C_3 \cup C_4$, where
\begin{align*}
   C_1 &= \{\ti U^\prime_m + a_1 \} \times [\ti V_m - a_4 , \ti V^\prime_m + a_2], \\
	 C_2 &= [\ti U_m - a_3 , \ti U^\prime_m + a_1] \times \{\ti V^\prime_m + a_2 \}, \\
   C_3 &= \{\ti U_m - a_3 \} \times [\ti V_m - a_4 , \ti V^\prime_m + a_2],\\
   C_4 &= [\ti U_m - a_3 , \ti U^\prime_m + a_1] \times \{\ti V_m - a_4 \}.
\end{align*}
We consider first the case where $s = (s_1,s_2) \in C_1$ with $s_1 = \ti U^\prime_m + a_1$ and then there are four cases for $s_2$: $\ti V_m - a_4 \leq s_2 \leq \ti V_m$, $ \ti V_m \leq s_2 \leq 0$, $0\leq s_2 \leq \ti V^\prime_m$, $\ti V^\prime_m \leq s_2 \leq \ti V^\prime_m + a_2$. We only consider the first two cases, since all other cases are similar to these two.
\vskip 12pt

   {\em Case 1:} $\ti V_m - a_4 \leq s_2 \leq \ti V_m$. Then
\begin{align*}
   X^t(s_1, s_2) &= X^t(\ti U^\prime_m + a_1, s_2) = Z^t_1(\ti U^\prime_m + a_1) + Z^t_2(s_2) \\
	   &= Z^t_1(\ti U^\prime_m) + I_1(a_1) + Z^t_2(\ti V_m) + I_2 (s_2 - \ti V_m) \\
		 & = X^t(\ti U^\prime_m, \ti V_m) + I(a_1, s_2 - \ti V_m).
\end{align*}
Now $X^t(\ti U^\prime_m, \ti V_m) \leq M 2^{\ell - n}$, and since $(a_1, s_2 - \ti V_m) \in \partial \R(\tilde \ultau_{(N)})$,
$$
   3M 2^{\ell - n} + I(a_1, s_2 - \ti V_m) < 0
$$
by \eqref{11.24}. Therefore,
\begin{align*}
   2^{-n} + X^t(s_1,s_2) \leq 2^{-n} + M 2^{\ell - n} + I(a_1, s_2 - \ti V_m) \leq 3M 2^{\ell - n} + I(a_1, s_2 - \ti V_m) < 0,
\end{align*}
as was to be proved.
\vskip 12pt

    {\em Case 2:} $ \ti V_m \leq s_2 \leq 0$. Then
\begin{align*}
   X^t(s_1, s_2) &= X^t(\ti U^\prime_m + a_1, s_2) = Z^t_1(\ti U^\prime_m + a_1) + Z^t_2(s_2) \\
	   &= Z^t_1(\ti U^\prime_m) + I_1(a_1) + Z^t_2(s_2) \\
		 &= X^t(\ti U^\prime_m, s_2) + I(a_1,0).
\end{align*}
Because $\ti V_m \leq s_2 \leq 0$, $X^t(\ti U^\prime_m, s_2) \leq M 2^{\ell - n}$ on $G_M$, therefore, since $(a_1,0) \in \partial \R(\tilde \ultau_{(N)})$, 
$$
   2^{-n} + X^t(s_1,s_2) \leq 2^{-n} + M 2^{\ell - n} + I(a_1, 0) < 0,
$$
as was to be proved.
\vskip 12pt

   Handling the remaining 14 cases in the same way, we conclude that $2^{-n} + X^t < 0$ on $C_1 \cup C_2 \cup C_3 \cup C_4$. This means that the DW-algorithm for $X^t$ does not escape the rectangle $[\ti U_m - a_3, \ti U^\prime_m + a_1] \times [\ti V_m - a_4, \ti V^\prime_m + a_2]$.
	
	We now check that $2^{-n} + X^t < 2^{-k_0}$ in $[\ti U_m - a_3, \ti U^\prime_m + a_1] \times [\ti V_m - a_4, \ti V^\prime_m + a_2]$. Indeed, by (ii) above \eqref{08.0.9}, and \eqref{11.25rd}, on this rectangle,
\begin{align*}
   2^{-n} + X^t < 2^{-n} + 2^{1-j} + c j^2 2^{-3j/2} < 3 \cdot 2^{-j} \leq 3\cdot 2^{-(h_0+1)},
\end{align*}
and this is $\leq 2^{-k_0}$ provided $h_0 \geq k_0 + \log_2 \frac32$. This inequality holds for $h_0 > k_0$.
	
	The above considerations show that $A_2(t,n,k_0,v)$ does not occur (by Remark \ref{rem11.8}). This proves \eqref{08.0.8} and completes the proof of (b).




   (c) Let $\tilde G_{h_0,1}$ (resp. $\tilde G_{h_0,2})$ be the event ``the boosted DW-algorithm (for $\ti X$) started at level $2^{-k_{M,n}}$ reaches level $2^{-h_0}$ (resp. escapes $\R(2^{-2h_0})$)," so that $\tilde G_{h_0} = \tilde G_{h_0,1} \cup \tilde G_{h_0,2}$.
	
	   We will consider a standard ABM $Y = (Y(u_1,u_2),\, (u_1,u_2) \in \IR^2)$. From this ABM $Y$, we will construct below an event $A$ with $P(A) \geq 1/2$ and another standard ABM $\bar Y = (\bar Y(u_1,u_2),\, (u_1,u_2) \in \IR^2)$. Let $\bar G_{h_0}$, $\bar G_{h_0,1}$, $\bar G_{h_0,2}$ be defined in the same as were $\tilde G_{h_0}$, $\tilde G_{h_0,1}$, $\tilde G_{h_0,2}$, respectively, but with $\ti X$ replaced by $\bar Y$. Let $G_{h_0,1}$ (resp.~$G_{h_0,2}$) be the event ``the (ordinary) DW-algorithm (of Section \ref{sec2}) for $Y$ started at level $2^{-k_{M,n}}$ reaches level $2^{-h_0}$ (resp.~escapes $ $)," and let $G_{h_0} = G_{h_0,1} \cup G_{h_0,2}$. We will establish the following properties.
		
		(i) $A$ and $\bar Y$ are independent;
		
		(ii) $A \cap \bar G_{h_0,1} = A \cap G_{h_0,1}$;
		
		(iii) $A \cap \bar G_{h_0,2} \subset A \cap G_{h_0,2}$.
		
\noindent With these three properties, we see that
\begin{align*}
   P(\tilde G_{h_0}) = P(\bar G_{h_0}) = \frac{P(\bar G_{h_0} \cap A)}{P(A)} \leq \frac{P(G_{h_0} \cap A)}{P(A)} \leq 2\, P(G_{h_0}) \leq c 2^{(h_0-k_{M,n})\lambda_1},
\end{align*}
where we have used Theorems \ref{prop1} and \ref{thm3}, and this establishes property (c).

   It remains, given $Y$, to construct the event $A$ and the ABM $\bar Y$ so that the properties (i), (ii) and (iii) hold. We use the notations from Section \ref{sec2} in relation to the DW-algorithm for the ABM $Y$. We set
$$
   A = \bigcap_{m=k_0}^{k_{M,n}-1} A_m,
$$
where
$$
   A_m = \{\ultau^{2^{-m}} = \ulinfty \} \cup (\{\ultau^{2^{-m}} <  \ulinfty\} \cap \ti A_m),
$$
Here, $ \ti A_m = \cup_{k=1}^\infty \ti A_{m,k}$, and for $k$ odd,
$$
   \ti A_{m,k} = \{H_{k-1} < 2^{-m},\ H_k \geq 2^{-m} \} \cap (\ti A_{m,k,1} \cup \ti A_{m,k,3}),
$$
and
\begin{align*}
   \ti A_{m,k,1} &= \{Y(\ultau_1^{2^{-m}}, T_2^{k-1}) = 2^{-m}, \, \tau_+^{k,m,1} < \tau_-^{k,m,1} \},\\
	 \ti A_{m,k,3} &= \{Y(-\ultau_3^{2^{-m}}, T_2^{k-1}) = 2^{-m}, \, \tau_+^{k,m,3} < \tau_-^{k,m,3} \},
\end{align*}
and
\begin{align*}
   \tau_+^{k,m,1} &= \inf\{v_1 >\ultau_1^{2^{-m}}: Y(v_1,T_2^{k-1}) = (1+ m^{-2}) 2^{-m}  \},\\
	\tau_-^{k,m,1} &= \inf\{ v_1 >\ultau_1^{2^{-m}}: Y(v_1,T_2^{k-1}) = 0\},
\end{align*}
and similarly,
\begin{align*}
   \tau_+^{k,m,3} &= \inf\{v_1 >\ultau_3^{2^{-m}}: Y(-v_1,T_2^{k-1}) = (1+ m^{-2}) 2^{-m}  \},\\
	\tau_-^{k,m,3} &= \inf\{ v_1 >\ultau_3^{2^{-m}}: Y(-v_1,T_2^{k-1}) = 0\}.
\end{align*}
For $k$ even, $\ti A_{m,k}$ is defined analogously, using $\ultau_2^{2^{-m}}$ and $\ultau_4^{2^{-m}}$.

   In order to get a lower bound on $P(A)$, we notice that $\ti A_m$ is the event ``upon reaching level $2^{-m}$, the relevant component Brownian motion of $Y$, which is at level $2^{-m}$, hits $(1+ m^{-2}) 2^{-m}$ before $0$." By the Markov property and gambler's ruin probabilities, the probability of the complement of $A_m$ is
$$
   P(A_m^c) = P(\{\ultau^{2^{-m}} < \ulinfty \} \cap \ti A_m^c) \leq 2\, \frac{m^{-2} 2^{-m}}{(1+m^{-2})2^{-m}} = \frac{2}{1+m^2}.
$$
Therefore, 
$$
   P(A^c) \leq \sum_{m= k_{0}}^\infty  \frac{2}{1+m^2} \leq \frac{1}{2}
$$
for $k_0$ large enough. Therefore, $P(A) \geq 1/2$.

   Before constructing the process $\bar Y$, we first define a process obtained from a standard Brownian motion $B = (B(u),\, u \in \IR_+)$ by deleting certain intervals. Let $J_1, J_2, \dots$ be a sequence of (random, possibly empty) intervals in $\IR_+$ with measurable endpoints, ordered so that $\sup J_m \leq \inf J_{m+1}$, for all $m \geq 1$. We suppose that the Lebesgue measure of $\IR_+ \setminus \cup_{m=1}^\infty J_m$ is $+\infty$, and we define the time-change
$$
	 D(v) = \inf\{u\geq 0: u - \vert [0,u] \cap \cup_{m=1}^\infty J_m \vert = v \},
$$
where $\vert \cdot \vert$ denotes Lebesgue measure. We then define the ``deleted process" 
$$
   B_D(v) = \int_0^{D(v)} 1_{(\cup_{m=1}^\infty J_m)^c}(u)\, dB_u.
$$

   Since the ABM $Y$ is given by four independent standard Brownian motions $B^1$, $B^2$, $B^3$, $B^4$ (as in \eqref{startp6}), we construct the ``deleted process" associated with each of them and respectively the intervals
\begin{align*}
   J_m^1 &= \left \{\begin{array}{ll}
	    [\ultau_1^{2^{-m}}, \tau_+^{k,m,1}] & \mbox{ if } \{\ultau^{2^{-m}} < \ulinfty \} \cap \ti A_{m,k,1}  \mbox{ occurs,}\\
			\emptyset & \mbox{ otherwise,}
	     \end{array}\right. \\
		J_m^3 &= \left \{\begin{array}{ll}
	    [\ultau_3^{2^{-m}}, \tau_+^{k,m,3}] & \mbox{ if } \{\ultau^{2^{-m}} < \ulinfty \} \cap \ti A_{m,k,3}  \mbox{ occurs,}\\
			\emptyset & \mbox{ otherwise,}
	     \end{array}\right. 
\end{align*}
The intervals $J_m^2$ and $J_m^4$ are defined analogously. The corresponding time changes are denoted $D^i(v)$, $i=1,2,3,4$, and the corresponding ``deleted processes" are $B_D^i = (B_D^i(v_i),\, v_i \in \IR_+)$.

   The ABM $\bar Y$ is now determined by the four Brownian motions $B_D^1,\dots,B_D^4$. By the Markov property of the DW algorithm, it is easy to check that these are independent standard Brownian motions, and they are independent of the event $A$, since both they and $A$ are determined by increments of the Brownian motions $B^1,\dots,B^4$ over disjoint intervals, and the intervals themselves are determined by such increments. Therefore, property (i) holds.
	
	In order to check (ii) and (iii), we note that on the event $A$, the sequence of intervals $[\bar U_k, \bar U^\prime_k]$ and $[\bar V_k, \bar V^\prime_k]$ constructed by the boosted DW algorithm for $\bar Y$ satisfy
$$
   \bar U^\prime_k = D^1(U^\prime_k),\qquad \bar U_k = - D^2(- U_k),
$$
and similarly for $\bar V^\prime_k$, $\bar V_k$, where $[U_k, U^\prime_k]$ and $[V_k, V^\prime_k]$ are the intervals constructed by the DW-algorithm for $Y$, and there is a similar relation for the points $\bar T$ where the successive maxima $\bar H_k$ are attained, and also $\bar H_k = H_k$. Therefore, (ii) and (iii) hold (but note that the ``$\subset$" in (iii) cannot be replaced by ``$=$", which, fortunately, is not needed).
This completes the proof of (c) and of Lemma \ref{08.lem5}. 
\hfill $\Box$
\vskip 16pt

   Let $\IH^t$ be the sigma-field generated by $\IH \vee \sigma(\Ybar^{r,t}, Y_3,Y^\prime_2)$ and white noise increments used by the DW-algorithm for $X^t$ up to escaping the rectangle $[x^t_1 - t_1, y^r_1 - t_1] \times [x^r_2 - t_2, y^t_2 - t_2 ]$.

\begin{lemma}
For $h_0 \leq h_1 \leq k_{M,n},$ let $\tilde G_{h_1}$\index{$\tilde G_{h_1}$} be the event defined in Lemma \ref{08.lem5}, with $h_0$ there replaced by $h_1$. Then:

(a) $G_M \cap S^c_{k_{M,n}} \cap \cdots \cap S^c_{h_1+1} \subset \tilde G_{h_1+1}$;

   (b) On $G_M,$
$$
P(S^c_{k_{M,n}} \cap \cdots \cap S^c_{h_1+1} \cap S_{h_1} \mid \IH^t) \leq c 2^{(h_1-k_{M,n})\lambda_1} e^{-h_1^2}.
$$ 
\label{08.lem6}
\end{lemma}

\proof Using the same proof as in Lemma \ref{08.lem5}(b), one checks that $G_M \cap S_{k_{M,n}}^c \cap \cdots \cap S^c_{h_1+1} \subset \tilde G_{h_1+1}$, which establishes (a). In order to show (b), it suffices to show that on $G_M,$
\begin{equation}\label{08.0.11}
P(\tilde G_{h_1+1} \cap S^c_{k_{M,n}} \cap \cdots \cap S^c_{h_1+1} \cap S_{h_1} \vert \IH^t) \leq c 2^{(h_1-k_{M,n})\lambda_1} e^{-h_1^2}.
\end{equation}

 If $S_{h_1}$ occurs because $S_{h_1,1}$ occurs, then we notice that
$$
   \tilde G_{h_1+1} \cap S^c_{k_{M,n}} \cap \cdots \cap S^c_{h_1+1} \cap S_{h_1,1} \subset \tilde G_{h_1+1} \cap S_{h_1,1}
$$
and these last two events are independent and are independent of $\IH^t$. So
\begin{eqnarray}\nonumber
   P(\tilde G_{h_1+1} \cap S_{h_1,1} \mid \IH^t) &=& P(\tilde G_{h_1+1}) P(S_{h_1,1})\\
   &\leq& c 2^{(h_1-k_{M,n})\lambda_1} e^{-h_1^2},
\label{e11.30a}
\end{eqnarray} 
by Lemma \ref{08.lem5}(c) for the first factor, and Lemma \ref{rdlem9} and standard bounds for Brownian motion for the second factor.
 
   It remains to show that
\begin{equation}\label{08.0.12}
 P(\tilde G_{h_1+1} \cap S_{k_{M,n}}^c \cap \cdots \cap S^c_{h_1+1} \cap S^c_{h_1,1} \cap (S_{h_1,2} \cup S_{h_1,3}) \mid \IH^t ) \leq C \, 2^{(h_1-k_{M,n})\lambda_1} e^{-h_1^2}.
 \end{equation}
The proof of (\ref{08.0.12}) uses ideas similar to those used to prove (\ref{rd10.2a}). However, there is an additional difficulty: when considering variables such as the $E^3_{k, \ell}$ defined before Lemma \ref{rdlem40}, the rectangle involved may be contained in $[0, \tilde U_m] \times [\tilde V_m, \tilde V_m^\prime]$. In this case, it may be correlated with increments used by the DW-algorithm for $X^t$ before reaching level $2^{-k_{M,n}}$, but also with $Y^\prime_2$, $Y^\prime_4$ (used in Lemma \ref{08.lem2} and defined above (\ref{08rd11.8})) and with increments used by the DW-algorithm for $Y^r$ before it exits $\R_r(2^{2(k-n)-2}).$ However, in all these cases, Lemma \ref{rdlem39} applies as it did in the proof of \eqref{rd10.2a}.
This is sufficient to establish (\ref{08.0.12}) and completes the proof of Lemma \ref{08.lem6}. 
\hfill $\Box$
\vskip 16pt

\noindent{\em Proof of Proposition \ref{rdprop42}.} We have already observed that it suffices to prove (\ref{rd10.9}). Set
$$
   F = A_1(t, n) \cap A_2(t, n, k_0, v) \cap A_1( r, n) \cap A_2(r, n, k_0, v).
$$
Let $\hat Q$ be defined as in Proposition \ref{08.prop11.3}, $G_M$ be as defined in \eqref{defG_M}, and $S_j$ as in \eqref{08.0.11.19}. Let $k_0$ be fixed sufficiently large so that the conclusions of Lemma \ref{08.lem5} hold. 
We observe that $F$ is contained in the union of the following four events:
\begin{eqnarray*}
   F_1 &=& F \cap \{\hat Q \geq n-k\},\\
	 F_{2} &=& F \cap \{\hat Q < n-k\} \cap \{ \Ybar^{r,t} \geq 2^{\ell-n }\},\\
   F_3 &=& \cup_{M \geq 1} (A_1(t,n) \cap A_1(r,n) \cap G_M \cap  (\cup_{k_0 \leq j \leq k_{M,n}} S_j )),\\
   F_4 &=& \cup_{M \geq 1} (A_1(t,n) \cap A_1(r,n) \cap G_M \cap  (\cup_{k_0 \leq j \leq k_{M,n}} S_j)^c),
\end{eqnarray*}
so we bound each $P(F_i)$ separately.

By Proposition \ref{08.prop11.3}(b),
\begin{equation}\label{08_11-22}
   P(F_1) \leq c \ 2^{-n-\ell-2k\lambda_1} \exp (-2^{4(n-k)/5})),
\end{equation}
and since $\exp(-2^{4(n-k)/5}) \leq \exp(-2^{4(n-\ell)/5}) \leq c \ 2^{-(n-\ell)\lambda_1},$ we obtain
$$
   P(F_1) \leq c \ 2^{-n-\ell-2k\lambda_1 -(n-\ell)\lambda_1}.
$$

   As we observed below \eqref{11.6rd}, for $n-k \geq k_0$,
$$
   A_2(t, n, k_0, v) \cap A_2(r, n, k_0, v) \subset A_2(t, n, n - k, v) \cap A_2(r, n, n - k, v),
$$
so
\begin{align*}
   P(F_2) & \leq P(A_1(t, n) \cap A_2(t, n, n - k, v) \cap A_1( r, n) \cap A_2(r, n, n - k, v) \\
	 & \qquad \qquad\cap \{\hat Q < n-k\} \cap  \{ \Ybar^{r,t} \geq 2^{\ell-n }\}) \\
	 &\leq C \ 2^{-n-\ell-2k\lambda_1} \exp (-2^{4(n-k)/5}))
\end{align*}
by Lemma \ref{08.lem1}, and this is $\leq c \ 2^{-n-\ell-2k\lambda_1 -(n-\ell)\lambda_1}$ as for $P(F_1)$. 

For $P(F_3)$, we observe using Lemma \ref{08.lem6}(a), that
$$
   F_3 = \cup_{M \geq 1} \cup^{k_{M,n}}_{h_1=k_0} (A_1(t_,n) \cap A_1(r,n) \cap G_M \cap \tilde G_{h_1+1} \cap S^c_{k_{M,n}} \cap \cdots \cap S^c_{h_1+1} \cap S_{h_1}).
$$
We are going to show that
\begin{eqnarray}
P(A_1(t,n) \cap A_1(r,n) \cap G_M \cap \tilde G_{h_1+1} \cap S^c_{k_{M,n}} \cap \cdots \cap S^c_{h_1+1} \cap S_{h_1}) \nonumber\\
 \leq C \ M2^{-n-\ell} \ 2^{-2k \lambda_1} e^{-cM^2} \ 2^{(h_1-k_{M,n})\lambda_1} e^{-h_1^2} .
\label{08.0.13}
\end{eqnarray}
This will give
$$
   P(F_3) \leq c2^{-n- \ell} \sum_{M=1}^{+\infty} M \sum_{h_1=k_0}^{k_{M,n}} 2^{(h_1-k_{M,n})\lambda_1} e^{-h_1^2} 2^{-2k\lambda_1}e^{-c M^2}.
$$
The right-hand side is bounded above by
\begin{eqnarray*}
 &&c  \, 2^{-n-\ell} 2^{-2k \lambda_1} \sum^\infty_{M=1} M \ e^{-c M^2}(3 M 2^{\ell-n})^{\lambda_1} \sum^{k_{M,n}}_{h_1=k_0} 2^{h_1 \lambda_1} e^{-h_1^2}\\
 &&\qquad \leq  c 3^{\lambda_1} 2^{-n-\ell} 2^{-2k \lambda_1} 2^{(\ell-n)\lambda_1} \sum^\infty_{M=1} e^{-cM^2} M^{\lambda_1+1} \sum^{k_{M,n}}_{h_1=k_0} 2^{h_1 \lambda_1} e^{-h_1^2}.
\end{eqnarray*}
The sum over $h_1$ is bounded by a constant times $2^{k_0 \lambda_1}$, so we find that
$$
   P(F_3) \leq C_{k_0,v} \, 2^{-n-\ell-2k \lambda_1-(n-\ell)\lambda_1} \sum_{M=1}^{\infty} e^{-cM^2} M^{\lambda_1+1}
$$
The sum over $M$ converges, so we obtain
$$
   P(F_3) \leq C_{k_0,v} \, 2^{-n-\ell-2k \lambda_1-(n-\ell)\lambda_1}.
$$

   Before proving (\ref{08.0.13}), we consider $P(F_4)$. By Lemma \ref{08.lem5}(b),
$$
   F_4 = \cup_{M \geq 1} (A_1(t,n) \cap A_1(r,n) \cap G_M \cap \tilde G_{k_0} \cap (\bigcup^{k_{M,n}}_{j=k_0} S_j)^c).
$$
We are going to show that
\begin{equation}\label{08.0.14}
   P(A_1(t,n) \cap A_1(r,n) \cap G_M \cap \tilde G_{k_0} \cap (\cup^{k_{M,n}}_{j=k_0} S_j)^c) \leq C 2^{-n-\ell} 2^{(k_0-k_{M,n})\lambda_1} 2^{-2k\lambda_1} e^{-cM^2}.
\end{equation}
This will give
$$
   P(F_4) \leq C 2^{-n-\ell} 2^{-2k\lambda_1} \sum^\infty_{M=1} 2^{(k_0-k_{M,n})\lambda_1} e^{-cM^2}.
$$
The right-hand side is bounded by
\begin{eqnarray*}
   && C 2^{-n-\ell-2k\lambda_1}2^{k_0\lambda_1} \sum^\infty_{M=1} (3 M 2^{\ell-n})^{\lambda_1}e^{-cM^2}\\
   &&\qquad \leq C \  2^{-n-\ell-2k \lambda_1-(n-\ell)\lambda_1} \sum^\infty_{M=1} M^{\lambda_1} e^{-c M^2}.
\end{eqnarray*}
Since the series converges, we obtain
$$
P(F_4) \leq C_{k_0,v}\, 2^{-n-\ell-2k \lambda_1-(n-\ell)\lambda_1}.
$$
Adding up the bounds on $P(F_1), \ldots, P(F_4)$ establishes Proposition \ref{rdprop42}.

   It remains only to prove (\ref{08.0.13}) and (\ref{08.0.14}). We begin with (\ref{08.0.14}). The event on the left-hand side of (\ref{08.0.14}) is contained in
\begin{eqnarray}\nonumber
  && \cup_{i=0}^{M-30} ( A_1(t,n) \cap A_1(r,n) \cap A_2(t,n,n-k,v) \cap A_2(r,n,n-k,v) \cap \{\hat Q < n-k\}\\
&& \qquad \qquad \cap\ \{\Ybar^{r,t} < 2^{\ell-n},\, Y'_2 \geq (M-30-i) 2^{\ell-n},\, Y_3 \geq i \ 2^{\ell-n}\}
	\cap \tilde G_{k_0}).
	\label{11.35a}
\end{eqnarray}
Set
\begin{align}\nonumber
   \tilde F_{1,i} &= A_1(r,n) \cap \{Y_3 \geq i \ 2^{\ell-n} \} \cap A_2(t,n,n-k,v) \cap A_2(r,n,n-k,v) \\
	   &\qquad  \cap \{\hat Q < n-k\}\cap \{\Ybar^{r,t} < 2^{\ell-n}\} \cap \tilde G_{k_0}.
\label{e11.35a}
\end{align}
We write $W(t) = A^t_0 + A^t_1$, where $A^t_0$ and $A^t_1$ are defined in the proof of Lemma \ref{08.lem1}. Looking back to Figure \ref{fig10.4}, we see that $\sigma(A^t_0, Y'_2)$ is conditionally independent of $\IH_1 := \IH \vee \sigma(Y_3,\Ybar^{r,t}, A^t_1,W(r), \tilde G_{k_0})$ given $(x^t_1,y^r_2, x^t_2)$, and $A_1(t,n) = \{A^t_0 \in [-2^{1-n} - A^t_1, -2^{-n} - A^t_1] \}$. Proceeding as when we bounded \eqref{11.18a}, we find that the probability of the event in \eqref{11.35a} is bounded above by
\begin{equation}\label{11.36a}
   \sum_{i=0}^{M-30} c' 2^{-\ell} \exp(-c(M-30-i)^2)\, P(\tilde F_{1,i}).
\end{equation}
We now write $W(r) = W(1,1) + A^r$, so that $A_1(r,n) = \{W(1,1) \in [-2^{1-n} - A^r, -2^{-n} - A^r] \}$, and we observe that $W(1,1)$ is independent of $\sigma(A^r, \tilde F_{2,i})$, where 
\begin{align}\nonumber
   \tilde F_{2,i} &= \{ Y_3 \geq i \ 2^{\ell-n}\} \cap  A_2(t,n,n-k,v) \cap A_2(r,n,n-k,v) \\
	  & \qquad \cap \{\hat Q < n-k\} \cap \{\Ybar^{r,t} < 2^{\ell-n}\} \cap \tilde G_{k_0}.
\label{e11.36aa}
\end{align}
Therefore,
\begin{align}\label{11.37a}
   P(\tilde F_{1,i}) \leq \sup_{x \in \IR} P\{W(1,1) \in [-2^{1-n} +x, -2^{-n} +x] \} \, P(\tilde F_{2,i}) 
	\leq 2^{-n} \, P(\tilde F_{2,i}).
\end{align}
Since $\tilde G_{k_0}$ is independent of the other events that appear in the definition of $\tilde F_{2,i}$, we see from Lemma \ref{08.lem5}(c) that
\begin{equation}\label{11.38a}
   P(\tilde F_{2,i}) \leq c 2^{(k_0 - k_{M,n})\lambda_1} P(\tilde F_{3,i}),
\end{equation}
where
$$
   \tilde F_{3,i} = \{ Y_3 \geq i \ 2^{\ell-n}\} \cap  A_2(t,n,n-k,v) \cap A_2(r,n,n-k,v) \cap \{\hat Q < n-k\}.
$$
Using \eqref{08rd10.15} and the $9v$-robustness property (as in the proof of Lemma \ref{08.lem2}), we see that 
\begin{equation}\label{11.39a}
   P(\tilde F_{3,i}) \leq c' \exp(-ci^2) 2^{-2k\lambda_1}.
\end{equation}
Combining \eqref{11.36a}--\eqref{11.39a}, we conclude that the left-hand side of \eqref{08.0.14} is bounded above by
\begin{align}\nonumber
    &\sum_{i=0}^{M-30} c' 2^{-n-\ell -2k\lambda_1}\, 2^{(k_0 - k_{M,n})\lambda_1} \exp(-c((M-30-i)^2+ i^2)) \\
		&\qquad\qquad \leq \tilde c' 2^{-n-\ell -2k\lambda_1}\, 2^{(k_0 - k_{M,n})\lambda_1} e^{-\tilde c M^2},
\label{e11.39aa}
\end{align}
and this establishes \eqref{08.0.14}. 


   We now turn to the proof of \eqref{08.0.13}. Observe that the event on the left-hand side of (\ref{08.0.13}) is contained in
$$
   \cup_{p=1}^3 \hat G_p,
$$
where
\begin{align*}
   \hat G_p & = A_1(t,n) \cap A_1(r,n) \cap \tilde G_{h_1+1} \cap S^c_{k_{M,n}} \cap \cdots \cap S^c_{h_1+1} \cap \ti S_{h_1,p}\\
	  & \qquad \cap \{\Xbar^{r,t} \in [ (M-1)2^{\ell-n}, M 2^{\ell-n}[\}  \cap \{\Ybar^{r,t} < 2^{\ell-n}\}\\
	    &\qquad \cap \{\hat Q < n-k\} \cap A_2(t,n,n-k,v) \cap A_2(r,n,n-k,v)
\end{align*}
and
$$
   \ti S_{h_1,1} = S_{h_1,1},\qquad \ti S_{h_1,2} = S_{h_1,1}^c \cap S_{h_1,2},\qquad \ti S_{h_1,3} = S_{h_1,1}^c \cap S_{h_1,3}.
$$

   For $\hat G_1$, we write $W(t) = A^t_0 + A^t_1$, where $A^t_0$ and $A^t_1$ are defined in the proof of Lemma \ref{08.lem1}, and we use, as in the proof of Lemma \ref{08.lem2}(b), the fact that
\begin{align*}
   &\{\Xbar^{r,t} \in [ (M-1)2^{\ell-n}, M 2^{\ell-n}[\} \cap \{\Ybar^{r,t} < 2^{\ell-n}\} \\
	&\qquad \subset \cup_{i=0}^{M - 30} \{Y^\prime_2 \geq (M - 30-i)  2^{\ell - n},\,  Y_3 \geq i 2^{\ell - n}\},
\end{align*}
to see that
\begin{align}\nonumber
   P(\hat G_1) &\leq \sum_{i=0}^{M-30} P(\{A^t_0 \in [-2^{1-n} - A^t_1, -2^{-n} - A^t_1] \}\\
	   &\qquad\qquad\qquad \cap  \{Y^\prime_2 \geq (M - 30-i)  2^{\ell - n}\} \cap \hat G_{1,1,i}),
\label{e11.43}
\end{align}
where
\begin{align*}
   \hat G_{1,1,i} &= A_1(r,n) \cap \tilde G_{h_1+1} \cap S^c_{k_{M,n}} \cap \cdots \cap S^c_{h_1+1} \cap S_{h_1,1}\\
	  & \qquad \cap \{Y_3 \geq i 2^{\ell-n}\}  \cap \{\Ybar^{r,t} < 2^{\ell-n}\}\\
	    &\qquad \cap \{\hat Q < n-k\} \cap A_2(t,n,n-k,v) \cap A_2(r,n,n-k,v).
\end{align*}
Proceeding as in \eqref{e11.35a}, we see that 
\begin{equation}\label{e11.44}
   P(\hat G_{1}) \leq \sum_{i=0}^{M-30} c' \, 2^{-\ell} \exp(-c(M-30-i)^2) P(\hat G_{1,1,i}).
\end{equation}
We now proceed as in \eqref{e11.36aa} to see that
\begin{equation}\label{e11.45}
   P(\hat G_{1,1,i}) \leq c 2^{-n} P(\hat G_{1,2,i}),
\end{equation}
where 
\begin{align*}
   \hat G_{1,2,i} &= \tilde G_{h_1+1} \cap S^c_{k_{M,n}} \cap \cdots \cap S^c_{h_1+1} \cap \ti S_{h_1,p}\\
	  & \qquad \cap \{Y_3 \geq i 2^{\ell-n}\}  \cap \{\Ybar^{r,t} < 2^{\ell-n}\}\\
	    &\qquad \cap \{\hat Q < n-k\} \cap A_2(t,n,n-k,v) \cap A_2(r,n,n-k,v).
\end{align*}
We now use \eqref{e11.30a} in the proof of Lemma \ref{08.lem6} to see that
\begin{equation}\label{e11.46}
   P(\hat G_{1,2,i}) \leq c 2^{(h_1-k_{M,n})\lambda_1} e^{-h_1^2} P(\hat G_{1,3,i}),
\end{equation}
where
\begin{align*}
   \hat G_{1,3,i} &=  \{Y_3 \geq i 2^{\ell-n}\}  \cap \{\Ybar^{r,t} < 2^{\ell-n}\}\\
	    &\qquad \cap \{\hat Q < n-k\} \cap A_2(t,n,n-k,v) \cap A_2(r,n,n-k,v).
\end{align*}
Since $\hat G_{1,3,i} \subset \ti F_{3,i}$, we can use the bound in \eqref{11.39a} to conclude that
\begin{align}\label{e11.47}
   P(\hat G_{1,3,i}) \leq c e^{-ci^2}\, 2^{-2k\lambda_1}.
\end{align}
Combining \eqref{e11.44}--\eqref{e11.47} gives
$$
   P(\hat G_1) \leq c \sum_{i=0}^{M-30} 2^{-n-\ell}\, e^{-c(M-30-i)^2-c i^2}\, 2^{(h_1-k_{M,n})\lambda_1} e^{-h_1^2}\, 2^{-2k\lambda_1}.
$$
As in \eqref{e11.39aa}, we conclude that
\begin{equation}\label{e11.49}
   P(\hat G_1) \leq \ti c M \, 2^{-n-\ell}\, 2^{-2k\lambda_1}\, e^{-c M^2}\, 2^{(h_1-k_{M,n})\lambda_1} e^{-h_1^2}.
\end{equation}

   We now consider $\hat G_2$. Supose that the first episode of order $h_1$ that satisfies the condition for $S_{h_1,2}$ is $[\ti U^{(\ti \ell)}, \ti U^{(\ti \ell-1)}]$; it has length $\leq h_1^2 2^{-2 h_1}$ because $S_{h_1,1}^2$ occurs.
	
	   Define
$$
   \hat A_0^t = W([0,t_1 + \ti U^{(\ti \ell)}] \times [y_2^r, x_2^t]),
$$
and let $\hat A_1^t$ be such that $W(t) = \hat A_0^t + \hat A_1^t$. Then, as in \eqref{e11.43},
\begin{align}\nonumber
    P(\hat G_2) &\leq \sum_{i=0}^{M-30} P(\{\hat A^t_0 \in [-2^{1-n} - \hat A^t_1, -2^{-n} - \hat A^t_1] \}\\
	   &\qquad\qquad\qquad \cap  \{Y^\prime_2 \geq (M - 30-i)  2^{\ell - n}\} \cap \hat G_{2,1,i}),
\label{e11.50}
\end{align}
where 
\begin{align*}
   \hat G_{2,1,i} &= A_1(r,n) \cap \tilde G_{h_1+1} \cap S^c_{k_{M,n}} \cap \cdots \cap S^c_{h_1+1} \cap  S_{h_1,1}^c \cap S_{h_1,2}\\
	  & \qquad \cap \{Y_3 \geq i 2^{\ell-n}\}  \cap \{\Ybar^{r,t} < 2^{\ell-n}\}\\
	    &\qquad \cap \{\hat Q < n-k\} \cap A_2(t,n,n-k,v) \cap A_2(r,n,n-k,v).
\end{align*}
Define
\begin{align*}
   Y^{\prime\prime}_2 & = \sup_{y_2^r \leq u_2 \leq x_2^t}  (W(t_1 + \ti U^{(\ti \ell)}, u_2) - W(t_1 + \ti U^{(\ti \ell)}, x_2^t)),\\
	 Y^{\prime\prime\prime}_2 & = \sup_{y_2^r \leq u_2 \leq x_2^t}  (W(x_1^t, u_2) - W(t_1 + \ti U^{(\ti \ell)}, u_2) - W(x_1^t,x_2^t) + W(t_1 + \ti U^{(\ti \ell)}, x_2^t)).
\end{align*}
Then $Y_2^\prime \leq Y^{\prime\prime}_2 + Y^{\prime\prime\prime}_2$ and 
$$
   \{Y^\prime_2 \geq (M - 30-i)  2^{\ell - n}\} \subset \left\{Y^{\prime\prime}_2 \geq \frac{M - 30-i}{2}  2^{\ell - n}\right\} \cup \left\{Y^{\prime\prime\prime}_2 \geq \frac{M - 30-i}{2}  2^{\ell - n}\right\}.
$$
With this decomposition, the probability on the right-hand side of \eqref{e11.50} is bounded by the sum of two probabilities. Proceeding as in \eqref{e11.35a}, we see that
\begin{align*}
   & P(\{\hat A^t_0 \in [-2^{1-n} - \hat A^t_1, -2^{-n} - \hat A^t_1] \}
	    \cap  \{Y^{\prime\prime}_2 \geq \frac{M - 30-i}{2}\,  2^{\ell - n}\} \cap \hat G_{2,1,i})\\
		&\qquad\qquad \leq c\, 2^{-\ell} \exp(-c (M - 30-i)^2) P(\hat G_{2,1,i}) .
\end{align*}
For the second term, there is conditional independence between $\hat A^t_0$ and $Y^{\prime\prime\prime}_2$, so we also have
\begin{align*}
   & P(\{\hat A^t_0 \in [-2^{1-n} - \hat A^t_1, -2^{-n} - \hat A^t_1] \}
	    \cap  \{Y^{\prime\prime\prime}_2 \geq \frac{M - 30-i}{2}\,  2^{\ell - n}\} \cap \hat G_{2,1,i})\\
		&\qquad\qquad \leq c\, 2^{-\ell} \exp(-c (M - 30-i)^2) P(\hat G_{2,1,i}).
\end{align*}
We conclude using \eqref{e11.50} that
\begin{equation}\label{e11.51}
   P(\hat G_2) \leq c \sum_{i=0}^{M-30} 2^{-\ell} \exp(-c (M - 30-i)^2) P(\hat G_{2,1,i}).
\end{equation}
We now proceed as in \eqref{e11.36aa} to see that
\begin{equation}\label{e11.52}
   P(\hat G_{2,1,i}) \leq c 2^{-n} P(\hat G_{2,2,i}),
\end{equation}
where
\begin{align*}
   \hat G_{2,2,i} &= \tilde G_{h_1+1} \cap S^c_{k_{M,n}} \cap \cdots \cap S^c_{h_1+1} \cap  S_{h_1,1}^c \cap S_{h_1,2}\\
	  & \qquad \cap \{Y_3 \geq i 2^{\ell-n}\}  \cap \{\Ybar^{r,t} < 2^{\ell-n}\}\\
	    &\qquad \cap \{\hat Q < n-k\} \cap A_2(t,n,n-k,v) \cap A_2(r,n,n-k,v).
\end{align*}
We now use \eqref{08.0.12} in the proof of Lemma \ref{08.lem6} to see that
\begin{equation}\label{e11.53}
   P(\hat G_{2,2,i}) \leq c 2^{(h_1-k_{M,n})\lambda_1} e^{-h_1^2} P(\hat G_{2,3,i}),
\end{equation}
where $\hat G_{2,3,i}$ is the same event as $\hat G_{1,3,i}$, the probability of which is bounded in \eqref{e11.47}. We conclude from \eqref{e11.50}--\eqref{e11.53} and \eqref{e11.47} that
$$
    P(\hat G_2) \leq c \sum_{i=0}^{M-30} 2^{-n-\ell} \exp(-c (M - 30-i)^2)\, 2^{(h_1-k_{M,n})\lambda_1} e^{-h_1^2} e^{-c i^2}\, 2^{-2k\lambda_1}.
$$
As in \eqref{e11.39aa}, we conclude that
\begin{equation}\label{e11.54}
   P(\hat G_2) \leq c M \, 2^{-n-\ell}\, 2^{-2k\lambda_1}\, e^{-c M^2}\, 2^{(h_1-k_{M,n})\lambda_1} e^{-h_1^2}.
\end{equation}

   We now consider $\hat G_3$. Here, as for $\hat G_1$, we use the decomposition $W(t) = A_0^t + A_1^t$, then we proceed as for $\hat G_1$, but we quote \eqref{08.0.12} instead of \eqref{e11.30a} in the step that corresponds to \eqref{e11.46}, and we obtain, as in \eqref{e11.49},
\begin{equation}\label{e11.55}
   P(\hat G_3) \leq c M \, 2^{-n-\ell}\, 2^{-2k\lambda_1}\, e^{-c M^2}\, 2^{(h_1-k_{M,n})\lambda_1} e^{-h_1^2}.
\end{equation}
Putting together \eqref{e11.49}, \eqref{e11.54} and \eqref{e11.55} proves (\ref{08.0.13}). The proof of Proposition \ref{rdprop42} is complete.
\hfill $\Box$
\vskip 16pt

\end{section}

\eject

\begin{section}{Lower bound on the Hausdorff dimension of the boundaries of bubbles of the Brownian sheet}
\label{sec12}

  In this section we complete the proof of Theorem \ref{thm3a1}, by showing that the Hausdorff dimension of the boundary of every bubble is bounded below by $(3-\lambda_1)/2$. We begin by extending Proposition \ref{rdprop41} to a family of processes.

\begin{prop} Fix $q \in \IR$ and $x \in \,]0,1]$. Let $T=(T_1,T_2)$ be a stopping point with values in $[1,\infty[^2$ such that $W(T) = q+x$. Define a process $W^{(x)} = (W^{(x)}(u_1,u_2),\, (u_1,u_2) \in [1,\infty[^2)$ by
$$
   W^{(x)}(u_1,u_2) = \frac{1}{x} \left[W\left(T_1 + \frac{x^2}{T_2}(u_1-1), T_2 + \frac{x^2}{T_1}(u_2-1)\right) - W(T) \right].
$$ 
Define events $A_1^{(x)}(r,n)$, $A_2^{(x)}(r,n,k_0,v)$, $A_3^{(x)}(r,n,k_0,v)$, $A_4^{(x)}(r,n,k_0)$ and $A^{(x)}(r,n,k_0,v)$ in the same way as $A_1(r,n)$, $A_2(r,n,k_0,v)$, $A_3(r,n,k_0,v)$, $A_4(r,n,k_0)$ and $A(r,n,k_0,v)$, but with $W$ replaced by $1 + W^{(x)}$. 
There are $v \geq 1$, $k_0 \in \IN \setminus \{0\}$ and $c >0$ such that, for all $r \in [2,3]^2$, all sufficiently large $n$, for all $x \in \,]0,1]$ and all stopping points $T$ as above,
$$
   P(A^{(x)}(r,n,k_0,v) \mid \F_T) \geq c\, 2^{-n(1+\lambda_1)}.
$$
\label{prop12.1}
\end{prop}

\begin{remark} For fixed $x_0>0$, if we want a similar statement for stopping points with values in $[x_0,\infty[^2$, instead of $[1,\infty[^2$, then we could simply use the above statement with the Brownian sheet $(t_1,t_2) \mapsto x_0 W(t_1/x_0, t_2/x_0)$.
\end{remark}

\noindent{\em Proof of Proposition \ref{prop12.1}.} Fix $x \in \,]0,1]$. We are interested in the process $1+ W^{(x)}$ in particular because of the following. Consider the one-to-one transformation $\Phi: [1,\infty[^2 \to [T_1,\infty[\, \times [T_2,\infty[$ defined by 
\begin{equation}\label{e12.1.1}
   \Phi (u_1,u_2) = \left(T_1 + \frac{x^2}{T_2} (u_1 - 1), T_2 + \frac{x^2}{T_1} (u_2 - 1) \right).
\end{equation}
Then
$$
   \Phi([1,4]^2) = [T_1, T_1 + \frac{3x^2}{T_2}] \times [T_2, T_2 + \frac{3x^2}{T_1}],
$$
and, a.s., if $(s_1,s_2) = \Phi(u_1,u_2)$, then since $W(T_1,T_2) = q+x$,
\begin{align}\nonumber
   W(s_1,s_2) = q &\Longleftrightarrow q + x + W(s_1,s_2) - W(T_1,T_2) = q \\
	  & \Longleftrightarrow 1+ W^{(x)}(u_1,u_2) = 0.
\label{e12.1.2}
\end{align}
Further, the conditional law of $1+W^{(x)}$ given $\F_T$ is not very different from the conditional law of $1+W(\cdot) - W(1,1)$ given $W(1,1) = 1$, as we shall now make precise.

   One quickly checks that the processes $(X_1^{(x)}(u_1) = W^{(x)}(u_1,1),\, u_1 \in [1,4])$ and $(X_2^{(x)}(u_2) = W^{(x)}(1,u_2),\, u_2 \in [1,4])$ are standard Brownian motions that are conditionally independent given $\F_T$, and
$$
   W^{(x)}(u_1,u_2) = X_1^{(x)}(u_1) + X_2^{(x)}(u_2) + \E^{(x)}(u_1,u_2),
$$
where $\E^{(x)}(u_1,u_2)$ is a Brownian sheet with conditional variance $\frac{x^2}{T_1 T_2} (u_1-1)(u_2-1)$ (given $\F_T$). The variance of a rectangular increment of $\E^{(x)}$ is therefore smaller than that of a standard Brownian sheet (by a factor of $\frac{x^2}{T_1 T_2}$).

   More generally, for $(r_1,r_2) \in [2,3]^2$, we have a local decomposition of $(W^{(x)}(r_1+u_1,r_2+u_2),\, (u_1,u_2) \in [-1,1]^2)$ in the neighborhood of $(r_1,r_2)$ as follows:
\begin{equation}\label{e12.1}
   W^{(x)}(r_1+u_1,r_2+u_2) - W^{(x)}(r_1,r_2) = X_1^{(x),r}(u_1) + X_2^{(x),r}(u_2) + \E^{(x),r}(u_1,u_2),
\end{equation}
where
\begin{align*}
   X_1^{(x),r}(u_1) &=  W^{(x)}(r_1+u_1, r_2) - W^{(x)}(r_1, r_2), \\
	 X_2^{(x),r}(u_1) &=  W^{(x)}(r_1, r_2 +u_2) - W^{(x)}(r_1, r_2), \\
	\E^{(x),r}(u_1,u_2) &= \Delta _{]r_1,r_1+u_1]\times \, ]r_2,r_2+u_2} W^{(x)}.
\end{align*}
We note that $(X_1^{(x),r}(u_1), u_1 \in [-1,1])$ is a (two-sided) Brownian motion with speed $1+ \frac{x^2}{T_1 T_2} (r_2 -1) \in [1,3]$, $(X_2^{(x),r}(u_2), u_2 \in [-1,1])$is a Brownian motion with speed $1+ \frac{x^2}{T_1 T_2} (r_1 -1) \in [1,3]$, and $(\E^{(x),r}(u_1,u_2),\, (u_1,u_1) \in [-1,1]^2)$ is a Brownian sheet whose variance over a rectangle is $\frac{x^2}{T_1 T_2}$ times that area of the rectangle, and this fraction belongs to $]0,1]$. Hence, \eqref{e12.1} provides a better local approximation than one would obtain for $x=1=T_1=T_2$, which would be \eqref{e12.1} for the standard Brownian sheet.

   In order to establish Proposition \ref{prop12.1}, we simply follow the proof of Proposition \ref{rdprop41}, and check that the constant $c$ there can be chosen to work simultaneously for all the processes $1+ W^{(x)}(\cdot) $, since the ABM's $X_i^{(x),r}$ have a speed contained in $[1,3]$ and $\E^{(x),r}$ is a smaller local perturbation of $X_1^{(x),r} + X_2^{(x),r}$ than one has for the Brownian sheet itself.

   Indeed, going first through the proof of Lemma \ref{rdlem42}, we see that in Lemma \ref{lem6.18}, the constant $c_0$ can be chosen so that the conclusion of Lemma \ref{lem6.18} holds for all ABM's which are sums of two independent Brownian motions with speeds in $[1,3]$. Similarly, in Proposition \ref{rdprop38}, the constant $c(v)$ can be chosen so that the conclusion of Proposition \ref{rdprop38} also holds for all ABM's which are sums of two independent Brownian motions with speeds in $[1,3]$. And in Lemma \ref{rdlem40}, the constant $c(v,m)$ can also be chosen so that the conclusion of Lemma \ref{rdlem40} holds for all ABM's which are sums of two independent Brownian motions with speeds in $[1,3]$, and all ``error terms" $E^i_{k,\ell}$ in \eqref{rd10.2b} with variances of increments over a rectangle of area $u_1 u_2$ bounded by $\sigma^2 u_1 u_2$, with $\sigma^2 \leq 1$.
	
	With this variant of Lemma \ref{rdlem40}, the proof of Proposition \ref{rdprop41} carries over to $W^{(x)}$, with a constant $c$ that does not depend on $T \geq (1,1)$ or $x \in \, ]0,1]$. This proves Proposition \ref{prop12.1}.
\hfill $\Box$
\vskip 16pt

   We now extend Proposition \ref{rdprop42} to the family of processes $1+ W^{(x)}$.

\begin{prop} Fix $q \in \IR$ and $x \in \,]0,1]$. Let $T=(T_1,T_2)$ be a stopping point with values in $[1,\infty[^2$ such that $W(T) = q+x$. Define a process $W^{(x)}$ as in Proposition \ref{prop12.1}. For $v \geq 1$, there is $C = C_{v,k_0}$ such that for all large $n \in \IN$, $1\leq k \leq \ell \leq n-k_0$, $(r,t) \in \D_n(k,\ell)$, for all $x \in \,]0,1]$ and all stopping points $T$ as above,
$$
    P(A^{(x)}(t,n,k_0,v) \cap A^{(x)}(r,n,k_0,v) \mid \F_T) \leq C\, 2^{-n - \ell - 2k\lambda_1-(n-\ell)\lambda_1}.
$$
\label{prop12.2}
\end{prop}

\proof As was the case for Proposition \ref{prop12.1}, the idea here is to follow the proof of Proposition \ref{rdprop42} and to check that the constant $C$ there can be chosen so that the conclusion holds simultaneously for all the processes $1+W^{(x)}$. This involves checking that the same is true for the constants that appear in Proposition \ref{08.prop11.3} and in Lemmas \ref{lem11.4} to \ref{08.lem6}. This is indeed the case, since we are simply replacing ABM's with speed $1$ with ABM's with speed in $[1,3]$, and standard Brownian sheet increments with increments of a Brownian sheet with smaller variance. This establishes Proposition \ref{prop12.2}.
\hfill $\Box$
\vskip 16pt

	We continue with a lemma similar to Lemma \ref{lowerlem2}.

\begin{lemma} Fix $q \in \IR$ and $x \in \,]0,1]$. Let $T=(T_1,T_2)$ be a stopping point with values in $[1,\infty[^2$ such that $W(T) = q+x$. Define a process $W^{(x)}$ as in Proposition \ref{prop12.1}. Fix $v \geq 1$, $k_0 \in \IN \setminus \{0\}$, $c > 0$ and $C > 0$ such that the conclusions of Propositions \ref{prop12.1} and \ref{prop12.2} hold. Let $\mu_n^{(x)}$ be the random measure on $[2,3]^2$ defined by
$$
   \mu_n^{(x)}(E) = 2^{-(3-\lambda_1)n} \sum_{t \in \ID_n} \delta_t (E) 1_{A^{(x)}(t,n,k_0,v)},
$$
where $\delta_t(E)$ is defined in Lemma \ref{lowerlem2}. For $0 < \beta < (3-\lambda_1)/2,$ there is $K_\beta < \infty$ such that for all large $n$ and all $x \in \,]0,1]$, 
$$
   P\left\{\mu_n^{(x)}([2, 2+2^{-k_0}]^2) \in \left[\frac{c\, 2^{-2k_0}}{4}, \frac{2 C}{c}\right],\ Z_n^{(x)} \leq K_\beta \, \Big\vert \, \F_T \right\} 
        \geq \frac{c^2\, 2^{-2k_0}}{8C},
$$
where
$$
   Z_n^{(x)} = \int_{[2,2+2^{-k_0}]^2} \int_{[2,2+2^{-k_0}]^2} 
   \frac{1}{(\vert t-s \vert \vee 2^{-2n})^{\beta}}\, \mu_n^{(x)}(dt) \mu_n^{(x)}(ds).
$$
\label{lem12.3}
\end{lemma}

\begin{proof} Let $X_n := X_n^{(x)} = \mu_n^{(x)}([2,2+2^{-k_0}]^2).$ The lower bound on $E(X_n \mid \F_T)$ follows exactly as in the proof of Lemma \ref{lowerlem2}, except that we appeal to Proposition \ref{prop12.1}  instead of Proposition \ref{lowerlem1}(a). 

  The desired upper bound for $E(X_n^2 \mid \F_T)$ will come from an estimate of $E[(Z_n^{(x)})^2 \mid \F_T]$ (that is uniform in $\beta \in \, ]0,(3-\lambda_1)/2]$). Indeed, since $\vert t-s \vert \leq \sqrt{2}\, 2^{-k_0} \leq \sqrt{2}$, an upper bound for this quantity will give an upper bound for $2^{(3-\lambda_1)/4}\, E(X_n^2 \mid \F_T)$. Now, $E[(Z_n^{(x)})^2 \mid \F_T]$ is equal to
$$
   2^{-2(3-\lambda_1)n} \sum_{s,t \in \ID_n \cap [2, 2+2^{-k_0}]^2} P(A^{(x)}(s,n,k_0,v) \cap A^{(x)}(t,n,k_0,v) \mid \F_T) \frac{1}{(\vert t-s \vert \vee 2^{-2n})^{\beta}}\, .
$$
In view of Proposition \ref{prop12.2}, this is bounded above by 
$$
    C \ 2^{-2(3-\lambda_1)n} \sum^{n-k_0}_{\ell=1} \sum^\ell_{k = 1} \sum_{(s,t)\in \ID_n(k, \ell) \cap [2,2+2^{-k_0}]^4} 2^{-n-\ell-2k\lambda_1-(n-\ell)\lambda_1} \frac{1}{2^{-2(n-\ell)\beta}}\, .
$$
Use the inequality card $\ID_n(k, \ell) \cap [2,2+2^{-k_0}]^4 \leq \frac{1}{4} 2^{4n+2\ell+2k-2k_0}$ to see that this is bounded by
$$
\frac{1}{4} C\, 2^{-2k_0} \ 2^{(\lambda_1-3)n}\, 2^{2 \beta n} \sum^{n-k_0}_{\ell=1} 2^{(1+\lambda_1- 2 \beta)\ell} \sum^\ell_{k=1} 2^{2(1-\lambda_1)k}.
$$
The sum over $k$ is bounded by $2^{2(1-\lambda_1)\ell + 1}$ and so we get the uniform (in $n$, $x$ and $\beta$) bound $E[(Z_n^{(x)})^2 \mid \F_T] \leq C \, 2^{-2k_0}$. The remainder of the proof follows that of Lemma \ref{lowerlem2}: the references to the lower bound in Proposition \ref{lowerlem1}(a) are replaced by references to Proposition \ref{prop12.1}. The references to Proposition \ref{lowerlem1}(b) are replaced by references to Proposition \ref{prop12.2}.
\end{proof}

   Let $q$, $x$ and $T=(T_1,T_2)$ be as in Proposition \ref{prop12.2}. Let $\cC_T(q)$ be the component of $\{s \in \rtoo: W(s) > q \}$ that contains $T$. Define $\partial \cC_T^\alpha(q)$ to be the subset of points in $\partial \cC_T(q)$ to which one can get arbitrarily close by following a curve starting at $T$ which is contained in $[T_1, T_1 + 3\alpha/T_2]\times [T_2, T_2 + 3\alpha/T_1]$ along which $W > q$. Let $\cC_{T}^{(x)}$ be the component of $\{u \in [1,\infty[^2: 1+W^{(x)}(u) >0 \}$ that contains $(1,1)$. Define $\partial \cC_{T}^{(x),\alpha}$ to be the subset of $ \partial \cC_{T}^{(x)}$  to which one can get arbitrarily close by following a curve contained in $[1,1+\alpha]^2$ starting at $(1,1)$ along which $1+ W^{(x)} >0$.

\begin{prop} For $0< \beta < (3-\lambda_1)/2$, there exists $c_0 >0$ such that, for all $x \in \,]0,1]$, if $T$ is a stopping point with values in $[1,\infty[^2$ such that $W(T) = q+x$, then
$$
   P\{\mbox{\rm dim} (\partial \cC_T^{x^2}(q)) \geq \beta \mid \F_T \} \geq c_0.
$$
\label{prop12.4}
\end{prop}

\proof Fix $x \in \, ]0,1]$. Using \eqref{e12.1.2}, we see that the map $\Phi$ defined in \eqref{e12.1.1} maps $\partial \cC_T^{(x),3}$ onto $\partial \cC_T^{x^2}(q)$, and clearly, dim$\,\partial \cC_T^{(x),3} =$ dim$\,\partial \cC_T^{x^2}(q)$.

   Let $v$, $k_0$, $c$ and $C$ be as in Lemma \ref{lem12.3} and set $c_0 = c^2\, 2^{-2k_0} / (8C^2) > 0$. We are going to show that
$$
   P\{\mbox{\rm dim} (\partial \cC_T^{(x),3}) \geq \beta \mid \F_T\} \geq c_0.
$$
Let $\mu_n^{(x)}$, $Z_n^{(x)}$, $K_\beta$ and $X_n^{(x)}$ be as in Lemma \ref{lem12.3} and its proof. Let
$$
   F_n^{(x)} = \left\{X_n^{(x)} \in \left[\frac{c\, 2^{-2k_0}}{4}, \frac{2C}{c} \right],\ Z_n^{(x)} \leq K_\beta \right\},
$$
and $F^{(x)}  = \limsup_{n \to \infty} F_n^{(x)}$. By Fatou's Lemma and Lemma \ref{lem12.3}, $P(F^{(x)} \mid \F_T) \geq c_0$, so it suffices to show that on $F^{(x)}$, dim$\, \partial \cC_T^{(x),3}  \geq \beta$. This is done exactly as in the proof of Proposition \ref{lbprop14}. Proposition \ref{prop12.4} is proved.
\hfill $\Box$
\vskip 16pt

\noindent{\em Proof of Theorem \ref{thm3a1}.} Recall that the upper bound on the Hausdorff dimension of $q$-bubbles is established in Proposition \ref{bs_ubprop1}, so it remains to establish the corresponding lower bound.

   It suffices to consider upwards $q$-bubbles. Since each such bubble contains a point $r \in \,]0,\infty[^2$ with rational coordinates, it suffices to fix $r=(r_1,r_2)$, assume that $W(r) > q$ and show that a.s., dim~$\partial \cC_r(q) \geq (3-\lambda_1)/2$, where $\cC_r(q)$ denotes the component of $\{s \in \IR_+^2: W(s) > q\}$ that contains $r$. For simplicity, we only consider the case where $r=(1,1)$.
	
	As in the proof of Theorem \ref{thm7.3},  for $x\in\, ]0,1]$, define 
\begin{align*}
   T_1 &= \inf\{s_1 \geq 1: W(s_1,1) = q \}, \\
	  T_2^{(x)} &= \inf\{s_2 \geq 1: W(T_1,s_2) = q+x \}, \\
		S_1^{(x)} &= \sup\{s_1 < T_1: W(s_1,1) = q+x \}.
\end{align*}
In contrast with Theorem \ref{thm7.3}, it is no longer true that $(1,1)$ and $(T_1, T_2^{(x)})$ are always in the same $q$-bubble. However, this occurs with a probability that is uniformly (in $x$) bounded away from $0$. Indeed, $(1,1)$ and $(S_1^{(x)},1)$ are in the same $q$-bubble, and it is not difficult to check, using the method in the proof of \cite[Theorem 2.1]{DW0}, that with probability uniformly (in $x$) bounded away from $0$, $W > q$ along the concatenation of the two segments $\{S_1^{(x)}\} \times [1,T_2^{(x)}] $ and $[S_1^{(x)}, T_1] \times \{T_2^{(x)}\}$.
   
   We are now back on track with the proof of Theorem \ref{thm7.3}: fix $0 < \beta < (3-\lambda_1)/2$ and let $G^{(x)}$ be the intersection of the events ``$W > q$ along the concatenation of the two segments $\{S_1^{(x)}\} \times [1,T_2^{(x)}] $ and $[S_1^{(x)}, T_1] \times \{T_2^{(x)}\}$" (which is $\F_{(T_1,T_2^{(x)})}$-measurable) and $\{$dim$(\partial \cC_{(T_1,T_2^{(x)})}^{x^2}(q)) \geq \beta\}$ (which is conditionally independent of $\F_{(T_1,T_2^{(x)})}$ given $(T_1,T_2^{(x)})$). By the above considerations and Proposition \ref{prop12.4}, there is $\ti c >0$ such that, for all $x \in \,]0,1]$, $P(G^{(x)}) \geq \ti c$.
	
	Using an easily proved $0$-$1$ law for the stopping point $(T_1,1)$, we conclude that 
$$
   P(\limsup_{n \to \infty} G^{(1/n)}) = 1,
$$
and on this event, $\partial \cC_{(1,1)}(q) \supset \partial \cC_{(T_1,T_2^{(1/n)})}^{(1/n)^2}(q)$ for infinitely many $n$, hence dim$\, \partial \cC_{(1,1)}(q)$  $\geq \beta$. This proves Theorem \ref{thm3a1}.
\hfill $\Box$
\vskip 16pt

\end{section}

\noindent{\sc Acknowledgement.} The research reported in this article began while the first author was visiting the University of California at Los Angeles in Spring 1999. The authors thank Davar Khoshnevisan for several stimulating discussions.

\newpage
\addcontentsline{toc}{section}{Index of notation}
\printindex

\eject


\addcontentsline{toc}{section}{References}

\begin{thebibliography}{99}

\bibitem{adler} Adler, R.J. The uniform dimension of the level sets of a Brownian sheet.
{\em Annals Probab.} {\bf 6}-3 (1978), 509-515. 

\bibitem{AT}  Adler, R.J. \& Taylor, J.E. {\em Random fields and geometry.} Springer Monographs in Mathematics. Springer, New York (2007).

\bibitem{AW} Aza\"is, J.-M. \& Wschebor, M. {\em Level sets and extrema of random processes and fields.} John Wiley \& Sons, Inc., Hoboken, NJ (2009).

\bibitem{beffara} Beffara, V. The dimension of the SLE curves. {\em Ann. Probab.} {\bf 36}-4 (2008), 1421-1452.

\bibitem{Dalang} Dalang, R.C. Level sets and excursions of the Brownian sheet.  In: {\em Topics in spatial stochastic processes} (Martina Franca, 2001),  pp.167-208. Lecture Notes in Math. {\bf 1802}, Springer, Berlin (2003).

\bibitem{DM96} Dalang, R.C. \& Mountford, T.  Nondifferentiability of curves on the Brownian sheet. {\em Annals Probab.} {\bf 24}-1 (1996), 182-195.

\bibitem{DM97} Dalang, R.C. \& Mountford, T.  Points of increase of the Brownian sheet. {\em Probab. Th. Relat. Fields} {\bf 108}-1 (1997), 1-27.

\bibitem{DM0} Dalang, R.C. \& Mountford, T. Level sets, bubbles and excursions of a Brownian sheet. In: {\em Infinite dimensional stochastic analysis} (Amsterdam, 1999), pp.117-128, Verh. Afd. Natuurkd. 1. Reeks. K. Ned. Akad. Wet., 52, R. Neth. Acad. Arts Sci., Amsterdam (2000).

\bibitem{DM01} Dalang, R.C. \& Mountford, T.  Jordan curves in the level sets of additive Brownian motion. {\em Trans. of the A.M.S.} {\bf 353}-9 (2001), 3531-3545.

\bibitem{DM1} Dalang, R.C. \& Mountford, T.  Eccentric behaviours of the Brownian sheet along lines.  
{\em Annals Probab.} {\bf 30} (2003), 293-322.

\bibitem{DM2} Dalang, R.C. \& Mountford, T.   Nonindependence of excursions of the Brownian sheet
and of additive brownian motion. {\em Trans. of the A.M.S.} {\bf 355} (2003), 967-985.

\bibitem{DW0} Dalang, R.C. \& Walsh, J.B. Geography of the level sets of the Brownian sheet.  {\em Probab. Theory Related Fields}  {\bf 96}  (1993), 153-176.

\bibitem{DW} Dalang, R.C. \& Walsh, J.B. The structure of a Brownian bubble. 
{\em Probab. Th. Relat. Fields} {\bf 96} (1993), 475-501.

\bibitem{DW2002} Dalang, R.C. \& Walsh, J.B. Time-reversal in hyperbolic s.p.d.e.'s. {\em Ann. Probab.} {\bf 30}-1 (2002), 213-252.

\bibitem{DZ} Dembo, A. \& Zeitouni, O. Large Deviations Techniques and Applications. Jones and Bartlett Publishers, Boston (1993).

\bibitem{dubedat} Dubedat, J. SLE and the free field: partition functions and couplings. {\em J. Amer. Math. Soc.} {\bf 22}-4 (2009), 995-1054.

\bibitem{E} Ehm, W. Sample function properties of multiparameter
stable processes.  {\it Zeit. Wahr. Theorie} {\bf 56} (1981), 195-228.

\bibitem{F} Falconer, K.J.  {\em The geometry of fractal sets}. Cambridge University Press (1985).

\bibitem{gabor} Gabor, P. Corner percolation on $\ZZ^2$ and the square root of $17$. {\em Annals Probab.} {\bf 336}-5 (2008), 1711-1747.

\bibitem{IM} Ito, K. \& McKean, H.P. {\em Diffusion processes and their
sample paths.}  Springer, Berlin (1965).

\bibitem{kahane} Kahane, J.P. {\em Some random series of functions} (second
edition). Cambridge University Press (1985).

\bibitem{KSh} Karatzas, I.  \& Shreve, S. {\em Brownian motion and stochastic calculus.}    Springer, New York (2000).

\bibitem{kendall} Kendall, W.S. Contours of Brownian processes with several-dimensional times.  {\em Z. Wahrsch. Verw. Gebiete}  {\bf 52}-3  (1980), 267-276.

\bibitem{davar} Khoshnevisan, D. {\em Multiparameter processes. An introduction to random fields.} Springer-Verlag, New York (2002).

\bibitem{KhShi} Khoshnevisan, D.  \& Shi, Z.  Brownian sheet and capacity.  {\em
Annals Probab.} {\bf 27} (1999), 1135-1159.

\bibitem{landkof} Landkof, N.S. {\em Foundations of modern potential theory.} Berlin, Heidelberg, New York: Springer (1972).

\bibitem{LS} Li, W. \& Shao, Q. {\em Gaussian processes: inequalities, small ball probabilities and applications.} In: Stochastic processes: theory and methods. Handbook of Statist. {\bf 19}, pp. 533-597,  North-Holland, Amsterdam (2001).

\bibitem{meyer} Meyer, P.A. Th\'eorie \'el\'ementaire des processus \`a deux indices. In: Korezlioglu, H., Mazziotto, G., Szpirglas, J. (eds). {\em Processus al\'eatoires \`a deux indices.} Lect. Notes Math. {\bf 863}, pp.1-30. Berlin, Heidelberg, New York: Springer (1981).

\bibitem{M} Mountford, T.S.   Estimates of the Hausdorff dimension of the boundary of positive Brownian sheet components.
{\it S\'eminaire de Probabilit\'es} XXVII {\em Lecture Notes in Math.}
vol.~1557, pp.233-255. Berlin, Heidelberg, New York: Springer (1993).

\bibitem{M2} Mountford, T.S.   Quasi-everywhere upper functions.
{\it S\'eminaire de Probabilit\'es} XXVI {\em Lecture Notes in Math.}
vol.~1526, pp.95-106. Berlin, Heidelberg, New York: Springer (1992).

\bibitem{OP} Orey, S. \& Pruitt, W.  Sample functions of the $N$-parameter Wiener process.  {\em Annals Probab.} {\bf 1} (1973), 138-163.

\bibitem{R} Rosen, J.  Joint continuity of the local time for $N$-parameter Wiener process on $\IR^d $. Preprint Univ. of Massachusetts (1981).

\bibitem{RY} Revuz, D. \& Yor, M. {\em Continuous martingales and Brownian motion.} Springer,  Berlin (1991). 


\bibitem{SS2009} Shramm, O. \& Sheffield, S. Contour lines of the two-dimensional discrete Gaussian free field. {\em Acta Math.} {\bf 202}-1 (2009), 21-137. 

\bibitem{SS2013} Shramm, O. \& Sheffield, S. A contour line of the continuum Gaussian free field. {\em Probab. Theory Related Fields} {\bf 157} (2013), 47-80. 

\bibitem{walsh} Walsh, J.B. {\em An introduction to stochastic partial differential equations.} In: Ecole d'\'et\'e de probabilit\'es de Saint-Flour, XIV--1984, Lect.~Notes in Math. {\bf 1180},  pp.~265-439. Springer, Berlin (1986).

\end{thebibliography}
\end{document}